\documentclass[1 [leqno,11pt]{amsart}
\usepackage{amssymb, amsmath,latexsym,amsfonts,amsbsy, amsthm,mathtools,graphicx,CJKutf8,CJKnumb,CJKulem,color}
\usepackage{float}
\usepackage{hyperref}
\numberwithin{equation}{section}
\allowdisplaybreaks
\let\f=\frac

\let\om=\omega

\let\na=\nabla

\let\pa=\partial

\def\R{\mathbb R}
\def\T{\mathbb T}
\def\Z{\mathbb Z}

\newcommand{\beq}{\begin{equation}}
\newcommand{\eeq}{\end{equation}}
\newcommand{\ben}{\begin{eqnarray}}
\newcommand{\een}{\end{eqnarray}}
\newcommand{\beno}{\begin{eqnarray*}}
\newcommand{\eeno}{\end{eqnarray*}}

\newtheorem{theorem}{Theorem}[section]

\newtheorem{lemma}[theorem]{Lemma}
\newtheorem{proposition}[theorem]{Proposition}

\newtheorem{Theorem}{Theorem}[section]

\begin{document}
\begin{CJK*}{UTF8}{gkai}

\title[Nonlinear inviscid damping for 2-D inhomogeneous Euler equations]
{Nonlinear inviscid damping for 2-D inhomogeneous incompressible Euler equations}

\author[Q. Chen]{Qi Chen}
\address{Academy of Mathematics $\&$ Systems Science, The Chinese Academy of Sciences, Beijing 100190, China}
\email{chenqi@amss.ac.cn}

\author[D. Wei]{Dongyi Wei}
\address{School of Mathematical Science, Peking University, Beijing 100871,  China}
\email{jnwdyi@pku.edu.cn}

\author[P. Zhang]{Ping Zhang}
\address{Academy of Mathematics $\&$ Systems Science, The Chinese Academy of Sciences, Beijing 100190, China,  and School of Mathematical Sciences,
University of Chinese Academy of Sciences, Beijing 100049, China.}
\email{zp@amss.ac.cn}

\author[Z. Zhang]{Zhifei Zhang}
\address{School of Mathematical Science, Peking University, Beijing 100871, China}
\email{zfzhang@math.pku.edu.cn}

\date{\today}

\maketitle

\begin{abstract}
We prove the asymptotic stability of shear flows close to the Couette flow for the 2-D inhomogeneous incompressible Euler equations
on $\T\times \R$. More precisely, if the initial velocity is close to the Couette flow and the initial density is close to a positive constant in the Gevrey class 2, then  2-D inhomogeneous incompressible Euler equations
are globally well-posed and the velocity converges strongly to a shear flow close to the Couette flow, and the vorticity will be driven to small scales by a linear evolution and weakly converges as $t\to \infty$. To our knowledge, this is the first global well-posedness result for the 2-D inhomogeneous incompressible Euler equations.
\end{abstract}

\tableofcontents

\section{Introduction}

We consider 2-D inhomogeneous incompressible Euler equations on $\T\times \R$:
\begin{align}\label{eq:IE}\left\{
\begin{aligned}
&\pa_t\rho+{\rm v}\cdot\na \rho=0,\\
&\rho(\pa_t{\rm v}+{\rm v}\cdot\na {\rm v})+\na p=0,\\
&\text{div}\ {\rm v}=0.
\end{aligned}
\right.
\end{align}
Here $\rho(t,,x,y), {\rm v}(t,x,y)=\left({\rm v}^1(t,x,y),{\rm v}^2(t,x,y)\right), p(t,x,y)$ denote the density, the velocity and the hydrostatic  pressure of the ideal fluid respectively (see \cite{Lions} for instance).

When $\rho=1$, the system \eqref{eq:IE} is reduced to the classical incompressible Euler equations.
It is well known that  the 2-D  classical incompressible  Euler equations are globally well-posed for smooth data \cite{Ch, MB02}, see also \cite{MP} about some classical results on the stability of some special solutions to these equations. To our knowledge,  the global well-posedness for the system \eqref{eq:IE} is still an open question even in the case when  the density is close enough to a positive constant. One may check page 159 of \cite{Lions} for the ``state of the art" concerning the system \eqref{eq:IE}.

In this paper, we study the asymptotic stability of  the Couette flow, i.e., $\rho=1, {\rm v}=(y,0)$, which is a steady solution of  \eqref{eq:IE}.
To this end, we introduce the perturbation
\beno
d=\frac 1 \rho-1, \quad u=(u^1,u^2)={\rm  v}-(y,0),\quad \Pi=p.
\eeno
Then $(d, u, \Pi)$ satisfies
\begin{align}\label{eq:IE-pert}\left\{
\begin{aligned}
&\partial_td+u\cdot\nabla d+y\partial_xd=0,\\
& \partial_tu+y\partial_xu+u\cdot\nabla u+\binom{u^2}{0}+(1+d)\nabla \Pi=0,\\
&\text{div}\ u=0,\\
&d|_{t=0}=d_0,\ u|_{t=0}=u_0.
\end{aligned}
\right.
\end{align}
Next we introduce the vorticity $\omega=\partial_xu^2-\partial_yu^1$ and let $\varphi$ be the stream{ }function which  solves $\Delta \varphi=\omega$.
Then $\om$ satisfies
\begin{align}\label{eq:vorticity}
\partial_t\omega+y\partial_x\omega+u\cdot\nabla \omega=-(\partial_xd\partial_y \Pi-\partial_yd\partial_x\Pi).
\end{align}

For the classical 2-D incompressible  Euler equations on $\T\times \R$, the asymptotic stability of the Couette flow has been proved by Bedrossian and Masmoudi in a breakthrough work \cite{BM}.
Roughly speaking, if the initial velocity is a small perturbation of the Couette flow in the Gevrey class $2-$,  then the velocity converges strongly to a shear flow close to the Couette flow, and the vorticity will be driven to small scales by a linear evolution and weakly converges as $t\to \infty$.  This phenomenon is the  {so-called } inviscid damping, which is an analogue to the Landau damping in the plasma physics found by Landau \cite{Lan}. We also refer to \cite{MV} on the breakthrough of nonlinear Landau damping. Ionescu and Jia \cite{IJ} proved the asymptotic stability of the Couette flow in a finite channel when the initial velocity is a small perturbation of the Couette flow in the Gevrey class $2$ and the initial vorticity is supported in the interior of the channel. It is easy to observe in these works that the regularity of the perturbations plays a crucial role on nonlinear inviscid damping and Landau damping, see \cite{LZ1, LZ2, DM} on negative results in the lower regularity and the instability in the Gevrey class $2+$. For general shear flows, the linear inviscid damping was  proved in a series of works \cite{WZZ-CPAM, WZZ-APDE, WZZ-AM}, see also earlier results \cite{Z1, Z2} and some relevant results \cite{BCV, CZ, GNR, WZZ-CMP, IJ-ARMA}.  Finally let us mention recent breakthrough on nonlinear inviscid damping for stable monotone shear flows, which was independently proved by Ionescu-Jia \cite{IJ-Mon} and Masmoudi-Zhao \cite{MZ}.\smallskip

Now we state our main result of this paper.\smallskip

Let us first introduce the Gevrey class$-\frac 1s$ denoted by $\mathcal{G}^{\lambda, s}$, whose norm is defined by
\beno
\|f\|_{\mathcal{G}^{\lambda,s}}^2=\sum_{k\in \Z}\int_{\R}e^{2\lambda\left(1+k^2+\xi^2\right)^{\frac{s}2}}|\widetilde{f}(k,\xi)|^2d\xi,
\eeno
where $\widetilde{f}$ denotes the Fourier transform of $f$, and $s\in [0,1]$ and $\lambda>0$.

\begin{Theorem}\label{thm:main}
Let $\beta_0\in (0,\frac 1 8]$. There exist $\beta_1=\beta_1(\beta_0)$ and $\overline\epsilon=\overline\epsilon(\beta_0)$ such that if the initial data ${d_0},\ \om_0\in \mathcal{G}^{\lambda, \f12}$ with $\int_{\T\times \R}|y||\om_0(x,y)|dxdy<+\infty$ and $\lambda=\beta_0$ satisfy
\beno
{\|d_0\|_{\mathcal{G}^{\lambda, \f12}}+}\|\om_0\|_{\mathcal{G}^{\lambda, \f12}}\le \epsilon\le \overline{\epsilon},\quad \int_{\T\times \R}\om_{0}(x,y)dxdy=0,
\eeno
then   the system \eqref{eq:IE-pert} has  a unique global-in-time smooth solution $(d, u, \Pi)$ which satisfies:

\begin{itemize}

\item[1.] There exists $\om_\infty(x,y)$ such that
\begin{align}\label{est:w-winfty-0}
&\|\om(t,x+ty+\Phi(t,y),y)-\om_\infty(x,y)\|_{\mathcal{G}^{\beta_1,\f12}}\lesssim_{\beta_0} {\f \epsilon {\langle t\rangle}},
\end{align}
where
\beno
\Phi(t,y)=\int_0^t\overline{u}^1(\tau,y)d\tau,\quad  \overline{u}^1(t,y)=\f1 {2\pi}\int_{\T}u^1(t,x,y)dx.
\eeno
% $\Phi(t,y)$ satisfies
% \begin{align}\label{est:Phi-tuinfty-0}
%   & |\Phi(t,y)-tu_{\infty}(y)|\lesssim_{\beta_0} \epsilon,
% \end{align}
% where $u_{\infty}:=\lim_{t\rightarrow \infty}\overline{u}^1(t,y)$ is given explicitly as
% \begin{align}\label{eq:uinfty-def-0}
%    &u_{\infty}(y)=-\partial_y(\Delta^{-1}\omega_{\infty})(y).
% \end{align}

\item[2.] There exists $d_\infty(x,y)$ such that
\begin{align}\label{est:d-dinfty-0}
&\|{d}(t,x+ty+\Phi(t,y),y)-d_\infty(x,y)\|_{\mathcal{G}^{\beta_1,\f12}}\lesssim_{\beta_0} {\f \epsilon {\langle t\rangle}}.
\end{align}

\item[3.]  {There exists $u_\infty(y)\in \mathcal{G}^{\beta_1,\f12}$ so that} the velocity $(u^1,u^2)$ satisfies:
\begin{align}
&\|\overline{u}^1(t,y)-u_\infty(y)\|_{\mathcal{G}^{\beta_1,\f12}}\lesssim_{\beta_0} \f \epsilon {\langle t\rangle^2}, \label{est:u1bar-uinfty-0}\\
&\|u^1(t,x,y)-\overline{u}^1(t,y)\|_{{L^\infty({\T\times\R})}}\lesssim_{\beta_0} \f \epsilon {\langle t\rangle},\label{est:u1-u1bar-0}\\
&\|u^2(t,x,y)\|_{{L^\infty({\T\times\R})}}\lesssim_{\beta_0} \f \epsilon {\langle t\rangle^2}.\label{est:u2-L2-0}
\end{align}

\end{itemize}

\end{Theorem}

Let us give some remarks on our result.\smallskip

\begin{itemize}

\item[1.] To our knowledge,  our result gives the first global well-posedness result for the 2-D inhomogeneous incompressible  Euler equations.

\item[2.] For simplicity, we consider the problem in the domain $\T\times \R$. It seems possible to generalize our result to a finite channel, i.e., $\T\times [-1,1],$ under the assumption that the initial vorticity is supported in the interior of the channel. In this case, the main trouble lies in the estimate of the pressure.

\item[3.] As in \cite{BMV}, our result could be proved via the inviscid limit.

\item[4.] As in \cite{BVW, MZ-AH}, it is quite interesting to study the stability threshold problem for the 2-D inhomogeneous incompressible Navier-Stokes equations at high Reynolds number ($Re$): find the threshold $\beta_1\ge 0$ and $\beta_2\ge 0$ so that if the initial perturbations satisfy
\beno
\|d_0\|_{H^N}\le c\nu^{\beta_1},\quad \|u_0\|_{H^N}\le c\nu^{\beta_2},
\eeno
with $\nu=Re^{-1}\ll1$, then the corresponding system \eqref{eq:IE-pert} related to  the 2-D inhomogeneous incompressible Navier-Stokes equations is globally stable.

 \end{itemize}
 
 Compared with \cite{BM, IJ}, the main difficulty is to control the growth of the density and the pressure. Under nonlinear coordinate transform \eqref{eq:change variable}, the vorticity satisfies 
 \begin{align*}
   & \partial_tf=(1+V_1)\partial_v\phi_{\neq}\partial_zf-(V_3+(1+V_1)\partial_z\phi)\partial_vf+(1+V_1)\{a,P\}.
\end{align*} 
 To control new nonlinear term $\{a,P\}$, we need to construct suitable multipliers for the density $a$ and the pressure $P$, which match their growth mechanism. The multiplier of $a$ should have similar properties as the multiplier of $f$ introduced in \cite{BM, IJ}. Moreover, the multipliers for $a$ and $P$ need to ensure that nonlinear term $\{a,P\}$ can be controlled by using bilinear weighted estimates. See section 2.1 for the details.

 \medskip

Let us end this section with some {\bf notations} that we shall use in the following context:
\begin{align*}
&f_{\neq}={\mathbb{P}_{\neq 0}f}=f-\overline{f}, \qquad \overline{f}{=\mathbb{P}_0f}=\f1 {2\pi}\int_{\T}f(x,y)dx,\\
& \langle a\rangle=(1+a^2)^\f12,\qquad {\langle a,b\rangle=(1+a^2+b^2)^{\f12}}.
\end{align*}
Let $\{f_1,f_2\}=\partial_vf_1\partial_zf_2-\partial_zf_1\partial_vf_2$ be the poisson bracket and {$(\mathcal{F}f)(t,k,\xi)= \widetilde{f}(t,k,\xi)$} denotes the Fourier transform of $f(t,z,v)$ in $(z,v)$.

\section{Main equations and bootstrap proposition}

\subsection{Main equations and sketch of the idea}
The proof of Theorem \ref{thm:main} follows the framework in \cite{IJ}.
To explain main difficulties and new ideas of this paper, we first recall two key ingredients  in \cite{BM, IJ}. \smallskip

{\bf The first key {ingredient}} is to introduce the following  nonlinear coordinate transform:
\begin{align}\label{eq:change variable}
  v=y+\dfrac{1}{t}\int_{0}^{t}\overline{u}^1(\tau,y)\mathrm{d}\tau,\quad z=x-tv.
\end{align}
 Then we define
\begin{align*}
   &f(t,z,v)=\omega(t,x,y),\quad a(t,z,v)=d(t,x,z),\\
   &\phi(t,z,v)=\varphi(t,x,y),\quad P(t,z,v)=\Pi(t,x,y),\\
   &V_1(t,v)=\partial_yv(t,y)-1,\quad V_2(t,v)=\partial_y^2v(t,y),\quad V_3(t,v)=\partial_tv(t,y),\\
   &\mathcal{H}(t,z,v)=t(1+V_1(t,v))\partial_vV_3(t,v)=-V_1(t,v)-{(\mathbb{P}_0 f)(t,v)}.
\end{align*}

It is easy to verify that
\begin{align*}
   & \partial_tf=(1+V_1)\partial_v\phi_{\neq}\partial_zf-(V_3+(1+V_1)\partial_z\phi)\partial_vf+(1+V_1)\{a,P\}.
\end{align*}
Then the new unknowns $f,\ a,\ V_1,\ \mathcal{H}$ satisfy the following system
\begin{align}
   & \partial_tf+V_3\partial_vf=(1+V_1)\{\phi_{\neq},f\}-(1+V_1)\{P,a\},\label{eq:f}\\
   &\partial_ta+V_3\partial_va=(1+V_1)\{\phi_{\neq},a\},\label{eq:a}\\
   &\partial_tV_1+V_3\partial_vV_1=\mathcal{H}/t,\label{eq:V1}\\
   &\partial_t\mathcal{H}+V_3\partial_v\mathcal{H}+\mathcal{H}/t =-(1+V_1){\mathbb{P}_0\big(\{\phi_{\neq},f\}-\{P,a\}\big)}.\label{eq:H}
\end{align}
Moreover, $\phi, V_2, V_3$ satisfy
\begin{align}
   &\partial_z^2\phi+(V_1+1)^2(\partial_v-t\partial_z)^2\phi +V_2(\partial_v-t\partial_z)\phi=f,\label{eq:phi}\\
   &\partial_vV_3=\mathcal{H}/\big(t(V_1+1)\big),\quad V_2=(V_1+1)\partial_vV_1,\label{eq:V2V3}
\end{align}
and the pressure $P$ satisfies
\begin{align}\label{eq:P}
 &\partial_z^2P+(V_1+1)^2(\partial_v-t\partial_z)^2P +V_2(\partial_v-t\partial_z)P +\partial_z(a\partial_zP)\\
   &\qquad+(V_1+1)(\partial_v-t\partial_z) \big(a(V_1+1)(\partial_v-t\partial_z)P\big)+q=0,\nonumber
\end{align}
where $q$ is given by (in the Euler coordinates)
\begin{align*}
   q(t,z,v)=&\text{div}(u\cdot \nabla u)+2\partial_xu^2=2(\partial_{x}\partial_y\varphi)^2 -2\partial_{x}^2\varphi\partial_y^2\varphi+2\partial_x^2\varphi\\
   =&2\partial_x^2\varphi(1-\omega) + 2(\partial_x^2\varphi)^2+2(\partial_y\partial_{x}\varphi)^2.
\end{align*}
In terms of new unknowns, we have
\begin{align}\label{eq:q}
 q=2\partial_z^2\phi(1-f)+2(\partial_z^2\phi)^2 +2\big[(V_1+1)(\partial_v-t\partial_z)\partial_z\phi\big]^2.
\end{align}
\if0We decompose $P=P_1+P_2$, where $P_1$ and $P_2$ solve
\begin{align}
   &\partial_z^2P_1+(V_1+1)^2(\partial_v-t\partial_z)^2P_1 +V_2(\partial_v-t\partial_z)P_1+q=0,\label{eq:P1}\\
   & \partial_z^2P_2+(V_1+1)^2(\partial_v-t\partial_z)^2P_2 +V_2(\partial_v-t\partial_z)P_2 +\partial_z(a\partial_zP)\label{eq:P2}\\
   &\qquad+(V_1+1)(\partial_v-t\partial_z) \big(a(V_1+1)(\partial_v-t\partial_z)P\big)=0.\nonumber
\end{align}\fi

We also define
\begin{align}\label{eq:Theta}
 \Theta(t,z,v)=\big(\partial_z^2+(\partial_v-t\partial_z)^2\big)\phi(t,z,v).
\end{align}

Compared with \cite{BM, IJ},  the equation of $f$ here contains a new nonlinear term $(1+V_1)\{P,a\}$. Moreover, the control for $a$ and $P$ is rather non-trivial.
\smallskip

{\bf The most key ingredient}  is to introduce time-dependent imbalanced weights like $A_k(t,\xi)$ with the key property
\beno
\frac {A_l(t,\eta)} {A_k(t,\xi)}\approx \Big|\frac \eta {l^2}\Big|\frac 1 {1+|t-\eta/l|},
\eeno
when $k\neq l, \xi=\eta+O(1), k=l+O(1)$, and $t$ is resonant for the frequency $(l,\eta)$. Then the energy functional of $f$ is defined by
\begin{align*}
   \mathcal{E}_f(t)=\sum_{k\in \mathbb{Z}}\int_{\mathbb{R}}A_k(t,\xi)^2|\widetilde{f}(t,k,\xi)|^2\mathrm{d}\xi.
   \end{align*}
 Due to the structural similarity of the equations of $f$ and $a$, it seems natural to introduce the same energy functional
 for $a$. However, this idea does not work due to new nonlinear term $(1+V_1)\{P,a\}$ in the equation of $f$, which requires
 one more derivative of $a$ in the process of energy estimate. Therefore, we have to introduce a new energy functional for $a$, which is defined as follows
 \begin{align*}
  &\mathcal{E}_{a}(t)=\sum_{k\in \mathbb{Z}}\int_{\mathbb{R}}A_k^*(t,\xi)^2|\widetilde{a}(t,k,\xi)|^2\mathrm{d}\xi,
  \end{align*}
 where
 \begin{align*}
   A_k^{*}(t,\xi)=A_{k}(t,\xi)\left(1+\dfrac{k^2+|\xi|}{\langle t\rangle^2}\right)^{\f12}.%+\dfrac{\xi^2}{\langle t\rangle^4+\langle t\rangle^{\sigma_1}|\xi|}.
\end{align*}
The construction of new weight $A_k^*(t,\xi)$ is rather tricky.  On one hand, $A_k^*(t,\xi)$ should have similar properties as $A_k(t,\xi)$ so that the weighted bilinear estimates for $A_k(t,\xi)$ still hold for new weight $A_k^*(t,\xi)$, see Appendix B and C. On the other hand, the weights for $a$ and $P$  need to ensure that  nonlinear term $(1+V_1)\{P,a\}$ can be controlled by using bilinear weighted estimates.

The construction of  {the multiplier for} the pressure $P$ is partially  motivated by  the following toy models for $P$ and $\phi$ 
\begin{align*}
&\pa_z^2P+(\partial_v-t\partial_z)^2P=\pa_z^2\phi,\quad \pa_z^2\phi+(\partial_v-t\partial_z)^2\phi=\Theta.
\end{align*}
Thus, we obtain
 \beno
 k^2\langle t-\xi/k\rangle^2\widetilde{P}={\f{\widetilde{\Theta}} {\langle t-\xi/k\rangle^2}}.
 \eeno
 Thanks to
\begin{align*}
   &\langle t\rangle +|(k,\xi)|\lesssim |k|\langle \xi/k\rangle\langle t-\xi/k\rangle,
   \end{align*}
we have
\begin{align*}
   & \dfrac{\langle t\rangle\big(\langle t\rangle +|(k,\xi)|\big)}{|k|\langle \xi/k\rangle^2} \lesssim \dfrac{t\langle t-\xi/k\rangle}{\langle \xi/k\rangle},
\end{align*}
from which, we infer that
 \begin{align*}
 \dfrac{k^2\langle t\rangle^{2}(\langle t\rangle^2+|(k,\xi)|^2)\langle t-\xi/k\rangle^4}{\langle\xi/k\rangle^4}|\widetilde{P}(t,k,\xi)|^2 \lesssim&
 \dfrac{\langle t\rangle^2}{\langle \xi/k\rangle^2\langle t-\xi/k\rangle^2}|\widetilde{\Theta}(t,k,\xi)|^2\\
 \lesssim&\dfrac{\langle t\rangle^2k^2}{|\xi|^2+k^2\langle t\rangle^2}|\widetilde{\Theta}(t,k,\xi)|^2\\
 \lesssim&|\widetilde{f}(t,k,\xi)|^2.
   \end{align*}
Furthermore, there holds that  for $k,l\neq 0(\sigma=k-l, \rho=\xi-\eta)$
\begin{align*}
\begin{aligned}
     A_{k}(t,\xi)\dfrac{\langle t\rangle(\langle t\rangle+|(k,\xi)|)}{|k|\langle\xi/k\rangle^2} \lesssim_{\delta}& A_l(t,\eta)A_{\sigma}(t,\rho) \dfrac{|l|\langle t\rangle \langle t-\eta/l\rangle^2}{|\eta|+|l|\langle t\rangle} \\
     &\times \left\{\langle l,\eta\rangle^{-2}+ \langle \sigma,\rho\rangle^{-2}\right\}.
\end{aligned}
\end{align*}
This property is crucial to control nonlinear terms in the equation of $P$.  Based on the above analysis,  we choose  the following energy functional of $P$ and $\Theta$:
 \begin{align}
  &\mathcal{E}_{P}(t)=\sum_{k\in \mathbb{Z}\setminus \{0\}}\int_{\mathbb{R}}A_k(t,\xi)^2\dfrac{k^2\langle t\rangle^{2}(\langle t\rangle^2+|(k,\xi)|^2)\langle t-\xi/k\rangle^4}{\langle\xi/k\rangle^4}|\widetilde{P}(t,k,\xi)|^2\mathrm{d}\xi \label{eq:energy-EP}\\
   &\qquad\qquad+\int_{\mathbb{R}}A_{0}(t,\xi)^2\dfrac{\langle t\rangle^{2}( \langle t\rangle^2+|\xi|^2)^2|\xi|^2}{\langle\xi\rangle^4}|\widetilde{P}(t,0,\xi)|^2\mathrm{d}\xi,\nonumber\\
&\mathcal{E}_{\Theta}(t)=\sum_{k\in \mathbb{Z}\setminus\{0\}}\int_{\mathbb{R}}A_{k}(t,\xi)^2\dfrac{|k|^2 \langle t\rangle^2}{|\xi|^2+|k|^2 \langle t\rangle^2}|\widetilde{\Theta}(t,k,\xi)|^2\mathrm{d}\xi.\label{eq:energy-ETheta}
   \end{align}
Here and in all that follows, we always denote $|(k,\xi)|^2=k^2+\xi^2$.

\subsection{Energy functional and bootstrap proposition}
Let the weights $A_R, A_{NR}, A_k$ be defined by \eqref{eq:ARANR} and \eqref{eq:Ak}.  We define
\begin{align*}
   & A_k^{*}(t,\xi)=A_{k}(t,\xi)\left(1+\dfrac{k^2+|\xi|}{\langle t\rangle^2}\right)^{\f12},
\end{align*}
and let
\beno
\dot{A}_{*}(t,\xi)=(\partial_tA_{*})(t,\xi),\quad \dot{A}^{*}_k(t,\xi)=(\partial_tA_{k}^{*})(t,\xi).
\eeno
Then we introduce the energy functional
\begin{align}
   &\mathcal{E}_f(t)=\sum_{k\in \mathbb{Z}}\int_{\mathbb{R}}A_k(t,\xi)^2|\widetilde{f}(t,k,\xi)|^2\mathrm{d}\xi,\label{eq:energy-Ef}\\
   &\mathcal{E}_{a}(t)=\sum_{k\in \mathbb{Z}}\int_{\mathbb{R}}A_k^*(t,\xi)^2|\widetilde{a}(t,k,\xi)|^2\mathrm{d}\xi,\label{eq:energy-Ea}\\
  % &\mathcal{E}_{P}(t)=\sum_{k\in \mathbb{Z}\setminus \{0\}}\int_{\mathbb{R}}A_k(t,\xi)^2\dfrac{k^2\langle t\rangle^{2}(\langle t\rangle^2+|(k,\xi)|^2)\langle t-\xi/k\rangle^4}{\langle\xi/k\rangle^4}|\widetilde{P}(t,k,\xi)|^2\mathrm{d}\xi \label{eq:energy-EP}\\
%   &\qquad\qquad+\int_{\mathbb{R}}A_{0}(t,\xi)^2\dfrac{\langle t\rangle^{2}( \langle t\rangle^2+|\xi|^2)^2|\xi|^2}{\langle\xi\rangle^4}|\widetilde{P}(t,0,\xi)|^2\mathrm{d}\xi,\nonumber\\
   &\mathcal{E}_{V_1}(t)=\int_{\mathbb{R}}A_{R}(t,\xi)^2
   |\widetilde{V_1}(t,k,\xi)|^2\mathrm{d}\xi,\label{eq:energy-EV1}\\
   &\mathcal{E}_{\mathcal{H}}(t)=K_{\delta}^2 \int_{\mathbb{R}}A_{NR}(t,\xi)^2(\langle t\rangle/\langle\xi\rangle)^{3/2}|\widetilde{\mathcal{H}}(t,k,\xi)|^2\mathrm{d}\xi,
   \label{eq:energy-EH} %\\
 %  &\mathcal{E}_{\Theta}(t)=\sum_{k\in \mathbb{Z}\setminus\{0\}}\int_{\mathbb{R}}A_{k}(t,\xi)^2\dfrac{|k|^2 \langle t\rangle^2}{|\xi|^2+|k|^2 \langle t\rangle^2}|\widetilde{\Theta}(t,k,\xi)|^2\mathrm{d}\xi,\label{eq:energy-ETheta}
   \end{align}
and $\mathcal{E}_{P}(t), \mathcal{E}_{\Theta}(t)$ are given respectively by \eqref{eq:energy-EP} and \eqref{eq:energy-ETheta}. We also introduce
 \begin{align}
   &\mathcal{B}_f(t)=\int_{1}^{t}\sum_{k\in \mathbb{Z}}\int_{\mathbb{R}}|\dot{A}_k(s,\xi)|A_k(s,\xi) |\widetilde{f}(s,k,\xi)|^2\mathrm{d}\xi\mathrm{d}s,\label{eq:energy-Bf}\\
     &\mathcal{B}_{a}(t)=\int_{1}^{t}\sum_{k\in \mathbb{Z}}\int_{\mathbb{R}}|\dot{A}^*_k(s,\xi)|A^*_k(s,\xi) |\widetilde{a}(s,k,\xi)|^2\mathrm{d}\xi\mathrm{d}s,\label{eq:energy-Ba}\\
   &\mathcal{B}_{P}(t)= \int_{1}^{t}\int_{\mathbb{R}}|\dot{A}_{0}(s,\xi)|A_{0}(s,\xi) \dfrac{\langle s\rangle^{2}(\langle s\rangle^2+|\xi|^2)^2|\xi|^2}{\langle\xi\rangle^4}|\widetilde{P}(s,0,\xi)|^2 \mathrm{d}\xi\mathrm{d}s\label{eq:energy-BP}\\
   & +\int_{1}^{t}\sum_{k\in \mathbb{Z}\setminus \{0\}}\int_{\mathbb{R}}|\dot{A}_k(s,\xi)|A_k(s,\xi) \dfrac{k^2\langle s\rangle^{2}(\langle s\rangle^2+|(k,\xi)|^2)\langle s-\xi/k\rangle^4}{\langle \xi/k\rangle^4}|\widetilde{P}(s,k,\xi)|^2\mathrm{d}\xi\mathrm{d}s ,\nonumber\\
    &\mathcal{B}_{V_1}(t)=\int_{1}^{t}\int_{\mathbb{R}}|\dot{A}_R(s,\xi) |A_R(s,\xi)|\widetilde{V_1}(s,k,\xi)|^2\mathrm{d}\xi,\label{eq:energy-BV1}\\
   &\mathcal{B}_{\mathcal{H}}(t)=K_{\delta}^2 \int_{\mathbb{R}}|\dot{A}_{NR}(s,\xi)|A_{NR}(s,\xi)(\langle s\rangle/\langle\xi\rangle)^{3/2} |\widetilde{\mathcal{H}}(s,k,\xi)|^2\mathrm{d}\xi\mathrm{d}s,\label{eq:energy-BH}\\
   &\mathcal{B}_{\Theta}(t)=\int_{1}^{t}\sum_{k\in \mathbb{Z}\setminus\{0\}}\int_{\mathbb{R}}|\dot{A}_k(s,\xi)|A_k(s,\xi) \dfrac{|k|^2 \langle s\rangle^2}{|\xi|^2+|k|^2 \langle s\rangle^2}|\widetilde{\Theta}(s,k,\xi)|^2\mathrm{d}\xi\mathrm{d}s.\label{eq:energy-BTheta}
\end{align}

Now the key bootstrap proposition is stated as follows.

\begin{proposition}\label{prop:Bootstrap}
  Assume $T\geq 1$ and let $\omega\in C\big([0,1]:\mathcal{G}^{2\delta_0,1/2})$, $d\in C([0,1]:\mathcal{G}^{2\delta_0,1/2}\big)$ be a sufficiently smooth solution of the system \eqref{eq:IE-pert} with $\|\overline{\omega}\|_{H^{10}}\ll 1$ for all $t\in[0,T]$.  Assume that $\epsilon_1$ is a sufficiently small constant depending on $\delta_0$ and $\delta$ so that
  \begin{align*}
     & \sum_{g\in \{f,V_1,\mathcal{H},\Theta\}}\big[\mathcal{E}_{g}(t)+ \mathcal{B}_{g}(t)\big] \leq \epsilon_1^3\quad \text{for any }\ t\in{[0,1]},\\
     &\mathcal{E}_{a}(t)+ \mathcal{B}_{a}(t)\leq \epsilon_1^3\quad \text{for any}\ t\in{[0,1]},\\
   &\mathcal{E}_{P}(t)+ \mathcal{B}_{P}(t)\leq \epsilon_1^3\quad \text{for any}\ t\in {[0,1]},
  \end{align*}
  and
  \begin{align*}
   & \sum_{g\in\{f,V_1,\mathcal{H},\Theta\}}\big[\mathcal{E}_{g}(t)+ \mathcal{B}_{g}(t)\big]\leq \epsilon_1^2 \quad \text{for any}\ t\in[1,T],\\
   &\mathcal{E}_{a}(t)+ \mathcal{B}_{a}(t)\leq \epsilon_1^2\quad \text{for any}\ t\in[1,T],\\
   &\mathcal{E}_{P}(t)+ \mathcal{B}_{P}(t)\leq \epsilon_1^2\quad \text{for any}\ t\in[1,T].
\end{align*}
Then for any $t\in[1,T]$, there holds the following improved bounds
\begin{align*}
  & \sum_{g\in\{f,V_1,\mathcal{H},\Theta\}}\big[\mathcal{E}_{g}(t)+ \mathcal{B}_{g}(t)\big]\leq \epsilon_1^2/2\quad \text{for any}\ t\in[1,T],\\
   &\mathcal{E}_{a}(t)+ \mathcal{B}_{a}(t)\leq \epsilon_1^2/2\quad \text{for any}\ t\in[1,T],\\
   &\mathcal{E}_{P}(t)+ \mathcal{B}_{P}(t)\leq \epsilon_1^2/2\quad \text{for any}\ t\in[1,T].
\end{align*}
Moreover, for $g\in\{f,\Theta,a,P\}$, we have the stronger bounds
\begin{align}\nonumber
  \sum_{g\in\{f,\Theta,a,P\}}\big[\mathcal{E}_{g}(t)+ \mathcal{B}_{g}(t)\big]\leq {C}\epsilon_1^3\quad \text{for any}\ t\in[1,T].
  \end{align}
\end{proposition}

Under the bootstrap assumptions, we have the following estimates for the coordinate functions (Lemma 4.2 in \cite{IJ}, more precisely, in arXiv:1808.04026v1).

 \begin{lemma}\label{lem:4.2JiaHao}
 For any $t\in [1,T]$ and $F\in\{V_1,V_1^2,\langle\partial_v\rangle^{-1}V_2\}$, we have
\begin{align}\label{eq:boot-V1}
\begin{aligned}
& \int_{\mathbb{R}}A_R(t,\xi)^2|\widetilde{F}(t,\xi)|^2\mathrm{d}\xi\lesssim_{\delta} \epsilon_1^2,\\
&\int_{1}^{t}\int_{\mathbb{R}}|\dot{A}_R(s,\xi)|A_R(s,\xi)|\widetilde{F}(s,\xi)|^2 \mathrm{d}\xi\mathrm{d}s\lesssim_{\delta} \epsilon_1^2.
\end{aligned}
\end{align}
 Moreover, for any $t\in[1,T]$, there holds
\begin{align}\label{eq:boot-H}
\begin{aligned}
& \int_{\mathbb{R}}A_{NR}(t,\xi)^2\big(1+\langle\xi\rangle^{-3/2}\langle t\rangle^{3/2} \big) | \widetilde{\mathcal{H}}(t,\xi)|^2\mathrm{d}\xi\lesssim_{\delta} \epsilon_1^2,\\
&\int_{1}^{t}\int_{\mathbb{R}}|\dot{A}_{NR}(s,\xi)|A_{NR}(s,\xi) \big(1+\langle\xi\rangle^{-3/2}\langle s\rangle^{3/2} \big) | \widetilde{\mathcal{H}}(s,\xi)|^2 \mathrm{d}\xi\mathrm{d}s\lesssim_{\delta} \epsilon_1^2,
\end{aligned}
\end{align}
and
\begin{align}\label{eq:boot-V3}
\begin{aligned}
& \int_{\mathbb{R}}A_{NR}(t,\xi)^2\big( \langle t\rangle^2+\langle\xi\rangle^{-3/2}\langle t\rangle^{7/2} \big) | \widetilde{\partial_vV_3}(t,\xi)|^2\mathrm{d}\xi\lesssim_{\delta} \epsilon_1^2,\\
&\int_{1}^{t}\int_{\mathbb{R}}|\dot{A}_{NR}(s,\xi)|A_{NR}(s,\xi) \big( \langle s\rangle^2+\langle\xi\rangle^{-3/2}\langle s\rangle^{7/2} \big) | \widetilde{\partial_vV_3}(s,\xi)|^2\mathrm{d}\xi\mathrm{d}s\lesssim_{\delta} \epsilon_1^2.
\end{aligned}
\end{align}
 \end{lemma}

\if0\begin{proof}
The inequalities \eqref{eq:boot-V1} and \eqref{eq:boot-H} have been proved in . Due to the lack of the compact support condition on $V_3$, we only need to prove \eqref{eq:boot-V3}.

Thanks to the bootstrap assumption $\mathcal{E}_{V_1}+\mathcal{B}_{V_1}\lesssim\epsilon_1^2$, we get by Lemma \ref{lem:8.1JiaHao} (i) and  Lemma \ref{lem:8.2JiaHao} that
\begin{align}\label{est:boot-V3-prove1}
  \begin{aligned}
  & \int_{\mathbb{R}}A_R^2(t,\xi)|\widetilde{(V_1^n)}(t,\xi)|^2\mathrm{d}\xi\lesssim_{\delta} \epsilon_1^{2n}\lesssim_{\delta} 2^{-2n}\epsilon_1^2,\\
&\int_{1}^{t}\int_{\mathbb{R}}|\dot{A}_R(s,\xi)|A_R(s,\xi)|\widetilde{(V_1^n)}(s,\xi)|^2 \mathrm{d}\xi\mathrm{d}s\lesssim_{\delta} \epsilon_1^{2n}\lesssim_{\delta} 2^{-2n}\epsilon_1^{2}.
  \end{aligned}
\end{align}
Due to $\mathcal{H}/t=(V_1+1)\partial_vV_3$, we have
  \begin{align*}
     &\partial_vV_3=\sum_{n\geq0 }(-1)^{n}V_1^n\cdot (\mathcal{H}/t).
  \end{align*}
Then by using Lemma \ref{lem:8.1JiaHao} (i) and  Lemma \ref{lem:8.2JiaHao} again, the desired bound \eqref{eq:boot-V3} follows from \eqref{eq:boot-H}, \eqref{est:boot-V3-prove1} and $A_{NR}(t,\xi)\leq A_{R}(t,\xi)$.
\end{proof}\fi

The following lemma will be constantly used (Lemma 4.3 in \cite{IJ}).

 \begin{lemma}\label{lem:tr-1}
   Let $a,b\in{\mathbb{R}^{n}}$ with ${n\geq1},\ \beta\in[0,1]$. Then
%   \begin{align}
%\langle a\rangle \geq \beta\langle b\rangle\qquad &\text{implies}\qquad |a|\geq \beta|b|-1,\label{eq:tr-1}\\
%| a| \geq \beta| b|\qquad &\text{implies}\qquad \langle a\rangle\geq \beta\langle b\rangle,\label{eq:tr-2}
%   \end{align}
%   and
   \begin{align}
   \langle b\rangle \geq \beta\langle a-b\rangle\qquad &\text{implies}\qquad \langle a\rangle^{1/2}\leq \langle b\rangle^{1/2}+(1-\sqrt{\beta}/2)\langle a-b\rangle^{1/2}.\label{eq:tr-3}
   \end{align}
 \end{lemma}

 \subsection{Proof of Theorem \ref{thm:main}}
 With  the key bootstrap Proposition \ref{prop:Bootstrap} at hand,  the proof of Theorem \ref{thm:main}  is quite similar to those in \cite{IJ}.
 Let us just give a sketch.\smallskip

 The first step is to prove a local regularity lemma(see Lemma 3.1 in \cite{IJ}), which ensures that the bootstrap assumptions in Proposition \ref{prop:Bootstrap} hold in a time interval $[0,T]$ for some $T\ge 1$. Then we apply  Proposition \ref{prop:Bootstrap} to conclude that for any $t\in [0,\infty)$, there holds
  \begin{align*}
  & \sum_{g\in\{f,V_1,\mathcal{H},\Theta\}}\big[\mathcal{E}_{g}(t)+ \mathcal{B}_{g}(t)\big]\leq \epsilon_1^2,\\
   &\mathcal{E}_{a}(t)+ \mathcal{B}_{a}(t)\leq \epsilon_1^2,\quad \mathcal{E}_{P}(t)+ \mathcal{B}_{P}(t)\leq \epsilon_1^2,\\
   & \sum_{g\in\{f,\Theta\}}\big[\mathcal{E}_{g}(t)+ \mathcal{B}_{g}(t)\big]\lesssim_{\delta} \epsilon_1^3.
\end{align*}
Notice that  $A_k(t,\xi)\geq \mathrm{e}^{\delta_0\langle k,\xi\rangle^{1/2}}$ and $A_R(t,\xi)\geq A_{NR}(t,\xi)\geq \mathrm{e}^{\delta_0\langle\xi\rangle^{1/2}}$ for any $(t,\xi,k)\in[0,\infty)\times \mathbb{R}\times \mathbb{Z}$. Then we deduce that
\begin{align}\label{est:f-Theta-V1-00}
  & \big\|\mathrm{e}^{\delta_0\langle k,\xi\rangle^{1/2}}\widetilde{f}(t,k,\xi)\big\|_{L^2_{k,\xi}}+ \big\|\mathbf{1}_{k\neq 0}\mathrm{e}^{(\delta_0/2)\langle k,\xi\rangle^{1/2}}\widetilde{\Theta}(t,k,\xi)\big\|_{L^2_{k,\xi}}\\
  &\qquad +\big\|\mathrm{e}^{\delta_0\langle\xi\rangle^{1/2}}\widetilde{V_1}(t,\xi)\big\|_{L^2_{\xi}}\lesssim \epsilon_1,\nonumber
\end{align}
and
\begin{align}\label{est:f-Theta-00}
   &\big\|\mathrm{e}^{\delta_0\langle k,\xi\rangle^{1/2}}\widetilde{f}(t,k,\xi)\big\|_{L^2_{k,\xi}}+ \big\|\mathbf{1}_{k\neq 0}\mathrm{e}^{(\delta_0/2)\langle k,\xi\rangle^{1/2}}\widetilde{\Theta}(t,k,\xi)\big\|_{L^2_{k,\xi}} \lesssim_{\delta} \epsilon_1^{3/2},
\end{align}
and
{\begin{equation}\label{eq:a-P-00}
\begin{aligned} & \big\|\mathrm{e}^{\delta_0\langle k,\xi\rangle^{1/2}} \widetilde{a}(t,k,\xi)\big|_{L^2_{k,\xi}} +\langle t\rangle^4 \big\|\mathrm{e}^{(\delta_0/2)\langle k,\xi\rangle^{1/2}}\widetilde{\partial_zP}(t,k,\xi)\big\|_{L^2_{k,\xi}} \\
&\qquad+\langle t\rangle^3\big\|\mathrm{e}^{(\delta_0/2)\langle\xi\rangle^{1/2}}\widetilde{\partial_v\overline{P}}(t,\xi)\big\|_{L^2_{\xi}} \lesssim \epsilon_1.
\end{aligned}
\end{equation}}

The second step is to show that for some $\delta_1=\delta_1(\beta_0)>0$,
\begin{align}\label{est:prtv-00}
   & \left\|\mathrm{e}^{\delta_1\langle \xi\rangle^{1/2}}\widetilde{\partial_tv}(t,\xi)\right\|_{L^2_{\xi}}\lesssim \epsilon_1^2\langle t\rangle^{-2}.
\end{align}

Notice that
\begin{align}\label{eq:partv-formula-00}
   & \partial_tv(t,y)= \dfrac{1}{t}\left[-\dfrac{1}{t}\int_{0}^{t}\overline{u}^1(\tau,y)\mathrm{d}\tau +\overline{u}^1(t,y)\right]=\dfrac{1}{t^2}\int_{0}^{t}\int_{\tau}^{t}\partial_s\overline{u}^1(s,y) \mathrm{d}s\mathrm{d}\tau,
\end{align}
and
\begin{align*}
   & \partial_t\overline{u}^1(t,y)+\mathbb{P}_0(u^2\partial_yu^1)(t,y)+ \mathbb{P}_0(d\partial_x\Pi)(t,y)=0.
\end{align*}
Thanks to $\mathbb{P}_0(u^2\partial_yu^1)=-\mathbb{P}_0(\partial_x\varphi\partial_y^2\varphi)$, we have
\begin{align}\label{eq:u1-bar-0}
   & \partial_t\overline{u}^1(t,y) =F_1(t,v(t,y))+F_2(t,v(t,y)),
\end{align}
where
\begin{equation}\nonumber%\label{eq:F1F2-define-00}
  \begin{aligned}
  &F_1(t,v)=  |V_1+1|^2\mathbb{P}_0\big(\partial_z\phi (\partial_v-t\partial_z)^2 \mathbb{P}_{\neq}\phi\big)(t,v) +V_2\mathbb{P}_0\big(\partial_z\phi(\partial_v-t\partial_z)\mathbb{P}_{\neq}\phi\big)(t,v),\\
  &F_2(t,v)=-\mathbb{P}_0\big(a\partial_zP\big)(t,v).
  \end{aligned}
\end{equation}
By \eqref{est:f-Theta-V1-00}, Lemma \ref{lem:8.1JiaHao} and $V_2=(V_1+1)\partial_vV_1$, we have
\begin{align*}
   & \big\|\mathrm{e}^{(\delta_0/4)\langle \xi\rangle^{1/2}}\widetilde{F_1}(t,\xi)\big\|_{L^2_{\xi}}\lesssim \epsilon_1^2\langle t\rangle^{-3},
\end{align*}
and by \eqref{eq:a-P-00} and Lemma \ref{lem:8.1JiaHao}, we have
\begin{align*}
   &\big\|\mathrm{e}^{(\delta_0/4)\langle \xi\rangle^{1/2}}\widetilde{F_2}(t,\xi)\big\|_{L^2_{\xi}}\lesssim \epsilon_1^2\langle t\rangle^{-3}.
\end{align*}
Notice that $\partial_v\mathcal{Y}(t,v)=(1/V')(t,v)$, where $\mathcal{Y}(t,\cdot)$ is the inverse of the function $y\rightarrow v(t,y)$. By \eqref{est:f-Theta-V1-00}, we can get, for some constant $K_3=K_3(\beta_0)$,
\begin{align}\label{eq:y-Y-00}
  & |D^{\alpha}_v\mathcal{Y}(t,v)|\leq K_3^m(m+1)^{2m},\qquad |D^{\alpha}_yv(t,y)|\leq K_3^m(m+1)^{2m},
\end{align}
for all $m\geq 1$ and $|\alpha|\in[1,m]$. Then can deduce that
\begin{align}\label{eq:part-baru1-00}
  & \left\|\mathrm{e}^{\delta_1\langle \xi\rangle^{1/2}}\widetilde{\partial_t\overline{u}^1}(t,\xi)\right\|_{L^2_{\xi}}\lesssim \epsilon_1^2\langle t\rangle^{-2},
\end{align}
for some $\delta_1=\delta_1(\beta_0)>0$. From \eqref{eq:partv-formula-00} and \eqref{eq:part-baru1-00}, we get
\begin{align*}
   & \left\|\mathrm{e}^{\delta_1\langle \xi\rangle^{1/2}}\widetilde{\partial_tv}(t,\xi)\right\|_{L^2_{\xi}} \lesssim \dfrac{1}{\langle t\rangle^2}\int_{0}^{t}\int_{\tau}^{t}\left\|\mathrm{e}^{\delta_1\langle \xi\rangle^{1/2}}\widetilde{\partial_t\overline{u}^1}(s,\xi)\right\|_{L^2_{\xi}}\mathrm{d}s \mathrm{d}\tau\lesssim \epsilon_1^2\langle t\rangle^{-2},
\end{align*}
which gives \eqref{est:prtv-00}. Consequently, $v_{\infty}:=\lim_{t\rightarrow\infty}v(t,y)$ exists in $\mathcal{G}^{\delta_1,1/2}$ and
\begin{align}\label{eq:v-vinfty-00}
   & \left\|\mathrm{e}^{\delta_1\langle \xi\rangle^{1/2}}\big[\widetilde{v}(t,\xi)-\widetilde{v}_{\infty}(\xi)\big]\right\|_{L^2_{\xi}}\lesssim \epsilon_1^2\langle t\rangle^{-1}.
\end{align}

The third step is to prove the convergence of the profiles $f$ and $a$. Recall that
\begin{align}\label{eq:partf-00}
   & \partial_tf=-V_3\partial_vf+(1+V_1)\{\phi_{\neq},f\}-(1+V_1)\{P,a\}.
\end{align}
Using the bounds \eqref{est:f-Theta-V1-00} on $\mathcal{E}_{\Theta}$, we have
\begin{align}\label{eq:phi-00}
   & \left\|\mathbf{1}_{k\neq 0}\mathrm{e}^{(\delta_0/3)\langle k,\xi\rangle^{1/2}}\widetilde{\phi}(t,k,\xi)\right\|_{L^2_{k,\xi}}\lesssim \epsilon_1\langle t\rangle^{-2}.
\end{align}
We use \eqref{est:prtv-00}, \eqref{eq:y-Y-00} to conclude
\begin{align}\label{eq:V3-00}
   &  \left\|\mathrm{e}^{\delta'_1\langle \xi\rangle^{1/2}}\widetilde{V_3}(t,\xi)\right\|_{L^2_{\xi}}\lesssim \epsilon_1\langle t\rangle^{-2},
\end{align}
for some $\delta'_1=\delta'_1(\beta_0)>0$. Using \eqref{eq:partf-00}, \eqref{eq:phi-00}, \eqref{eq:V3-00}, \eqref{est:f-Theta-V1-00} and \eqref{eq:a-P-00}, we have
\begin{align}\nonumber%\label{est:prtf-00}
   & \left\|\mathrm{e}^{\delta_2\langle \xi\rangle^{1/2}}\widetilde{\partial_tf}(t,\xi)\right\|_{L^2_{\xi}}\lesssim \epsilon_1^2\langle t\rangle^{-2}
\end{align}
for some $\delta_2=\delta_2(\delta_0)>0$. In particular, $f(t,z,v)$ converges to $f_{\infty}(z,v)$ in $\mathcal{G}^{\delta_2,1/2}$ with
\begin{align}\label{eq:f-finfty-00}
    & \left\|\mathrm{e}^{\delta_2\langle \xi\rangle^{1/2}}\big[\widetilde{f}(t,k,\xi)-\widetilde{f}_{\infty}(k,\xi)\big] \right\|_{L^2_{k,\xi}}\lesssim \epsilon_1^2\langle t\rangle^{-1}.
 \end{align}
 Thus, we infer from \eqref{eq:v-vinfty-00} and \eqref{eq:f-finfty-00} that
 \begin{align}\nonumber%\label{eq:wtof-00}
    &\omega(t,x+tv(t,y),y)=f(t,x,v(t,y))
 \end{align}
 converges to $f_{\infty}(x,v_{\infty}(y))$ with
 \begin{align}\label{eq:finfty-00}
    & \left\|\mathrm{e}^{\delta'_2\langle k,\xi\rangle^{1/2}} \big[\mathcal{F}(\omega(t,x+tv(t,y),y))(t,k,\xi)- \mathcal{F}(f_{\infty}(x,v_{\infty}(y)))(k,\xi)\big]\right\|_{L^2_{k,\xi}}\lesssim \epsilon_1^2\langle t\rangle^{-1}
 \end{align}
 for some $\delta_2'=\delta_2'(\delta_0)>0$.

 Similarly,  we can deduce by using \eqref{eq:a}, \eqref{eq:phi-00}, \eqref{eq:V3-00}, \eqref{est:f-Theta-V1-00} and \eqref{eq:a-P-00} that
\begin{align}\nonumber%\label{est:prta-00}
   & \left\|\mathrm{e}^{\delta_3\langle \xi\rangle^{1/2}}\widetilde{\partial_ta}(t,\xi)\right\|_{L^2_{\xi}}\lesssim \epsilon_1^2\langle t\rangle^{-2}
\end{align}
for some $\delta_3=\delta_3(\delta_0)>0$.
In particular, $a(t,z,v)$ converges to $a_{\infty}(z,v)$ in $\mathcal{G}^{\delta_3,1/2}$ with
\begin{align}\label{eq:a-ainfty-00}
    & \left\|\mathrm{e}^{\delta_3\langle \xi\rangle^{1/2}}\big[\widetilde{a}(t,k,\xi)-\widetilde{a}_{\infty}(k,\xi)\big] \right\|_{L^2_{k,\xi}}\lesssim \epsilon_1^2\langle t\rangle^{-1}.
 \end{align}
 Moreover, we have
  \begin{align}\label{eq:ainfty-00}
    & \left\|\mathrm{e}^{\delta'_3\langle k,\xi\rangle^{1/2}} \big[\mathcal{F}(d(t,x+tv(t,y),y))(t,k,\xi)- \mathcal{F}(a_{\infty}(x,v_{\infty}(y)))(k,\xi)\big]\right\|_{L^2_{k,\xi}}\lesssim \epsilon_1^2\langle t\rangle^{-1}
 \end{align}
 for some $\delta_3'=\delta_3'(\delta_0)>0$.\smallskip

 Finally, we are in a position to prove \eqref{est:w-winfty-0}-\eqref{est:u2-L2-0}.
 Let $\omega_\infty(x,y)=f_{\infty}(x,v_{\infty}(y))$ and $d_{\infty}(x,y)=a_{\infty}(x,v_{\infty})$. The bounds \eqref{est:w-winfty-0} and \eqref{est:d-dinfty-0} follow from \eqref{eq:ainfty-00}, \eqref{eq:f-finfty-00} and the definition of $v(t,y)$ and $\Phi(t,y)$. Moreover, let $u_{\infty}(y)=\lim_{t\rightarrow \infty}\overline{u}^{1} (t,y)$. The existence of the limit in $\mathcal{G}^{\delta_1,1/2}$ follows from \eqref{eq:part-baru1-00}, and the bound \eqref{est:u1bar-uinfty-0} follows from the definition.

 Notice that
 \begin{equation}\label{eq:u1u2-00}
    \begin{aligned}
    & u^{1}(t,x,y)=-\partial_y\varphi(t,x,y)=-V_1[(\partial_v-t\partial_z)\phi](t,z,v),\\
    & u^{2}(t,x,y)=\partial_x\varphi(t,x,y)=\partial_z\phi(t,z,v).
 \end{aligned}
 \end{equation}
By \eqref{est:f-Theta-V1-00}, \eqref{eq:phi-00} and Lemma \ref{lem:8.1JiaHao}, we have
 \begin{align*}
   & \left\|\mathbf{1}_{k\neq 0}\mathrm{e}^{(\delta_0/4)\langle k,\xi\rangle^{1/2}}\widetilde{(V_1[(\partial_v-t\partial_z)\phi])}(t,k,\xi)\right\|_{L^2_{k,\xi}}\lesssim \epsilon_1\langle t\rangle^{-1},\\
    &\left\|\mathrm{e}^{(\delta_0/4)\langle k,\xi\rangle^{1/2}}\widetilde{\partial_z\phi}(t,k,\xi)\right\|_{L^2_{k,\xi}}\lesssim \epsilon_1\langle t\rangle^{-2}.
 \end{align*}
 Thus, we obtain
 \begin{align*}
    &  \left\|V_1[(\partial_v-t\partial_z)\phi_{\neq}]\right\|_{\mathcal{G}^{\delta_4,1/2}}+ \langle t\rangle \left\|\partial_z\phi\right\|_{\mathcal{G}^{\delta_4,1/2}}\lesssim \epsilon_1\langle t\rangle^{-1}
 \end{align*}
 for some $\delta_4=\delta_4(\delta_0)>0$. Hence,
\begin{align*}
   & \left\|V_1[(\partial_v-t\partial_z)\phi_{\neq}]\right\|_{L^\infty_{z,v}}+ \langle t\rangle \left\|\partial_z\phi\right\|_{L^\infty_{z,v}}\lesssim \epsilon_1\langle t\rangle^{-1},
\end{align*}
which along with \eqref{eq:u1u2-00} yields \eqref{est:u1-u1bar-0} and \eqref{est:u2-L2-0}.

\section{Improved control of the pressure}

In this section, we prove an improved control for the pressure under the bootstrap assumptions in Proposition \ref{prop:Bootstrap}.
This part is completely new.

 \begin{proposition}\label{prop:Improved-P}
   With the definitions and assumptions in Proposition \ref{prop:Bootstrap}, there exits $c(\delta)>0$, such that if $\epsilon_1\leq c(\delta)$, then it holds that
   \begin{align}\nonumber%\label{est:Improved-P}
      & \mathcal{E}_P(t)+\mathcal{B}_{P}(t) \lesssim_{\delta}\epsilon_1^4  +\mathcal{E}_{\Theta}(t)+\mathcal{B}_{\Theta}(t).\quad \text{for any}\ t\in[1,T].
   \end{align}
 \end{proposition}
 \smallskip

 Let us first introduce the following weighted norms:
 \begin{align*}
   \|g\|_{A}^2=& \sum_{k\in \mathbb{Z}\setminus \{0\}}\int_{\mathbb{R}}A_k(t,\xi)^2\dfrac{\langle t\rangle^{2}(\langle t\rangle^2+|(k,\xi)|^2)\langle t-\xi/k\rangle^2}{\langle\xi/k\rangle^4}|\widetilde{g}(t,k,\xi)|^2\mathrm{d}\xi \nonumber\\
   &+\int_{1}^{t}\sum_{k\in \mathbb{Z}\setminus \{0\}}\int_{\mathbb{R}}|\dot{A}_k(s,\xi)|A_k(s,\xi)
   \dfrac{\langle s\rangle^{2}(\langle s\rangle^2+|(k,\xi)|^2)\langle s-\xi/k\rangle^2}{\langle \xi/k\rangle^4}|\widetilde{g}(s,k,\xi)|^2\mathrm{d}\xi\mathrm{d}s,
%\label{eq:norm-A}
\end{align*}
and
\begin{align}\nonumber%\label{eq:norm-B}
\begin{aligned}
   \|g\|_{B}^2=& \sum_{k\in \mathbb{Z}\setminus \{0\}}\int_{\mathbb{R}}A_k(t,\xi)^2\dfrac{\langle t\rangle^{2}(\langle t\rangle^2+|(k,\xi)|^2)}{k^2\langle\xi/k\rangle^4} |\widetilde{g}(t,k,\xi)|^2\mathrm{d}\xi \\
   &+\int_{1}^{t}\sum_{k\in \mathbb{Z}\setminus \{0\}}\int_{\mathbb{R}}|\dot{A}_k(s,\xi)|A_k(s,\xi) \dfrac{\langle s\rangle^{2}(\langle s\rangle^2+|(k,\xi)|^2)}{k^2\langle \xi/k\rangle^4}|\widetilde{g}(s,k,\xi)|^2\mathrm{d}\xi\mathrm{d}s.
\end{aligned}
\end{align}
It is easy to see that $\|(\partial_z,\partial_v-t\partial_z)g\|_{B}\approx \|g\|_{A}$. For the  zero mode,  we introduce the norm
\begin{align}\nonumber%\label{eq:norm-A0}
  \begin{aligned}
   \|{g}\|_{A0}^2=& \int_{\mathbb{R}}A_0(t,\xi)^2\dfrac{\langle t\rangle^{2}(\langle t\rangle^2+|\xi|^2)^2}{\langle\xi\rangle^4}|\widetilde{g}(t,0,\xi)|^2\mathrm{d}\xi \\
   &+\int_{1}^{t}\int_{\mathbb{R}}|\dot{A}_0(s,\xi)|A_0(s,\xi) \dfrac{\langle s\rangle^{2}(\langle s\rangle^2+|\xi|^2)^2}{\langle\xi\rangle^4}|\widetilde{g}(s,0,\xi)|^2\mathrm{d}\xi\mathrm{d}s.
\end{aligned}
\end{align}

 In this section, we always assume the bootstrap assumptions in Proposition \ref{prop:Bootstrap}.

\subsection{Some product estimates}

\begin{lemma}\label{lem:ABnorm-product}
 It holds that
  \begin{align}
     & \|V_1h_{\neq}\|_{B}+\|V_1^2h_{\neq}\|_{B}\lesssim_{\delta} \epsilon_1\|h_{\neq}\|_{B},\label{eq:ABnorm-product-1}\\
     & \|V_1h_{\neq}\|_{A}+\|V_1^2h_{\neq}\|_{A}\lesssim_{\delta} \epsilon_1\|h_{\neq}\|_{A},\label{eq:ABnorm-product-3}\\
     &\|V_2h_{\neq}\|_{B}\lesssim_{\delta} \epsilon_1\|h_{\neq}\|_{B}\lesssim \epsilon_1\|h_{\neq}\|_{A}.\label{eq:ABnorm-product-2}
  \end{align}
\end{lemma}

\begin{proof}
For \eqref{eq:ABnorm-product-1}, in view of Lemma \ref{lem:8.1JiaHao} and \eqref{eq:boot-V1},  it suffices to prove that (for $k\neq 0$, $t>1$)
  \begin{align}\label{est:ABnorm-prove1}
  \begin{aligned}
     A_{k}(t,\xi)\dfrac{\langle t\rangle(\langle t\rangle+|(k,\xi)|)}{|k|\langle \xi/k\rangle^2} \lesssim_{\delta}& A_{R}(t,\xi-\eta)A_{k}(t,\eta) \dfrac{\langle t\rangle(\langle t\rangle+|(k,\eta)|)}{|k|\langle \eta/k\rangle^2} \\ &\times \left\{\langle\xi-\eta\rangle^{-2} +\langle k,\eta\rangle^{-2}\right\},
  \end{aligned}
  \end{align}
  and
  \begin{align}\label{est:ABnorm-prove2}
  \begin{aligned}
     |(\dot{A}_kA_{k})(t,\xi)|^{1/2}&\dfrac{\langle t\rangle(\langle t\rangle+|(k,\xi)|)}{|k|\langle \xi/k\rangle^2}  \lesssim_{\delta}\left[|(\dot{A}_{R}/A_{R})(t,\xi-\eta)|^{1/2} +|(\dot{A}_k/A_{k})(t,\eta)|^{1/2}\right]\\
     &\times A_{R}(t,\xi-\eta)A_{k}(t,\eta) \dfrac{\langle t\rangle(\langle t\rangle+|(k,\eta)|)}{|k|\langle \eta/k\rangle^2}  \left\{\langle\xi-\eta\rangle^{-2} +\langle k,\eta\rangle^{-2}\right\}.
  \end{aligned}
  \end{align}
 Notice that
   \begin{align}\label{est:ABnorm-prove3}
      &\dfrac{\langle t\rangle(\langle t\rangle+|(k,\xi)|)}{|k|\langle \xi/k\rangle^2}  \lesssim_{\delta} \dfrac{\langle t\rangle(\langle t\rangle+|(k,\eta)|)}{|k|\langle \eta/k\rangle^2} \mathrm{e}^{\delta \min(\langle\xi-\eta\rangle,\langle k,\eta\rangle)^{1/2}}.
  \end{align}
 Then \eqref{est:ABnorm-prove1}-\eqref{est:ABnorm-prove2} follow from \eqref{est:8.3JiaHao-1}, \eqref{est:8.3JiaHao-2} and \eqref{est:ABnorm-prove3}. This finishes the proof of \eqref{eq:ABnorm-product-1}.

 For \eqref{eq:ABnorm-product-3}, by Lemma \ref{lem:8.1JiaHao} and Lemma \ref{lem:4.2JiaHao}, we only need to prove that
    \begin{align*}
     A_{k}(t,\xi)&\dfrac{\langle t\rangle(\langle t\rangle+|(k,\xi)|)\langle t-\xi/k\rangle}{\langle \xi/k\rangle^2} \lesssim_{\delta}A_{R}(t,\xi-\eta)A_{k}(t,\eta) \dfrac{\langle t\rangle(\langle t\rangle+|(k,\eta)|)\langle t-\eta/k\rangle}{\langle \eta/k\rangle^2} \\
     &\times\left\{\langle\xi-\eta\rangle^{-2} +\langle k,\eta\rangle^{-2}\right\},
  \end{align*}
  and
  \begin{align*}
  \begin{aligned}
     |(\dot{A}_kA_{k})(t,\xi)|^{1/2}&\dfrac{\langle t\rangle(\langle t\rangle+|(k,\xi)|)\langle t-\xi/k\rangle}{\langle \xi/k\rangle^2} \lesssim_{\delta}\left[|(\dot{A}_{R}/A_{R})(t,\xi-\eta)|^{1/2} +|(\dot{A}_k/A_{k})(t,\eta)|^{1/2}\right]\\
     &\times A_{R}(t,\xi-\eta)A_{k}(t,\eta) \dfrac{\langle t\rangle(\langle t\rangle+|(k,\eta)|)\langle t-\eta/k\rangle}{\langle \eta/k\rangle^2} \left\{\langle\xi-\eta\rangle^{-2} +\langle k,\eta\rangle^{-2}\right\},
  \end{aligned}
  \end{align*}
  which can be deduced  from \eqref{est:8.3JiaHao-1}, \eqref{est:8.3JiaHao-2} and
  \begin{align*}
     & \dfrac{\langle t\rangle(\langle t\rangle+|(k,\xi)|)\langle t-\xi/k\rangle}{\langle \xi/k\rangle^2}\lesssim_{\delta} \dfrac{\langle t\rangle(\langle t\rangle+|(k,\eta)|)\langle t-\eta/k\rangle}{\langle \eta/k\rangle^2}\mathrm{e}^{\delta \min(\langle\xi-\eta\rangle,\langle k,\eta\rangle)^{1/2}}.
  \end{align*}

For \eqref{eq:ABnorm-product-2}, it suffices to prove that
  \begin{align}\label{est:ABnorm-prove5}
  \begin{aligned}
     A_{k}(t,\xi)\dfrac{\langle t\rangle(\langle t\rangle+|(k,\xi)|)}{\langle \xi/k\rangle^2} \lesssim_{\delta}& \dfrac{ A_{R}(t,\xi-\eta)}{\langle \xi-\eta\rangle} A_{k}(t,\eta) \dfrac{\langle t\rangle(\langle t\rangle+|(k,\eta)|)}{\langle \eta/k\rangle^2} \\
      &\times\left\{\langle\xi-\eta\rangle^{-2} +\langle k,\eta\rangle^{-2}\right\},
  \end{aligned}
  \end{align}
  and
  \begin{align}\label{est:ABnorm-prove6}
  \begin{aligned}
     |(\dot{A}_k&A_{k})(t,\xi)|^{1/2}\dfrac{\langle t\rangle(\langle t\rangle+|(k,\xi)|)}{\langle \xi/k\rangle^2} \lesssim_{\delta}\left[|(\dot{A}_{R}/A_{R})(t,\xi-\eta)|^{1/2} +|(\dot{A}_k/A_{k})(t,\eta)|^{1/2}\right]\\
     &\times\dfrac{ A_{R}(t,\xi-\eta)}{\langle\xi-\eta\rangle}A_{k}(t,\eta)\dfrac{\langle t\rangle(\langle t\rangle+|(k,\eta)|)}{\langle \eta/k\rangle^2} \left\{\langle\xi-\eta\rangle^{-2} +\langle k,\eta\rangle^{-2}\right\}.
  \end{aligned}\end{align}
  In view of \eqref{est:8.3JiaHao-2},  \eqref{est:ABnorm-prove5} and \eqref{est:ABnorm-prove6}  can be deduced from the bound
  \begin{align}\label{est:ABnorm-prove7}
     A_{k}(t,\xi)\dfrac{\langle t\rangle+|(k,\xi)|}{\langle \xi/k\rangle^2} \lesssim_{\delta} &\dfrac{ A_{R}(t,\xi-\eta)}{\langle \xi-\eta\rangle} A_{k}(t,\eta) \dfrac{\langle t\rangle+|(k,\eta)|}{\langle \eta/k\rangle^2}\mathrm{e}^{-20\sqrt{\delta}\min(\langle\xi-\eta\rangle ,\langle k,\eta\rangle)^{1/2}},
  \end{align}
 which is a consequence of  \eqref{est:8.3JiaHao-1} and
  \begin{align*}
     &\dfrac{\langle\eta/k\rangle}{\langle\xi/k\rangle} \dfrac{\langle t\rangle +|\xi|+|k|}{\langle t\rangle+|\eta|+|k|} \dfrac{\langle\eta/k\rangle\langle\xi-\eta\rangle}{\langle\xi/k\rangle}\lesssim_{\delta}\mathrm{e}^{(\lambda(t)/40)\min(\langle\xi-\eta\rangle ,\langle k,\eta\rangle)^{1/2}}.
  \end{align*}
\end{proof}

\begin{lemma}\label{lem:A0norm-product}
  It holds that
  \begin{align}
     & \|V_1\bar{h}\|_{A0}+\|V_1^2\bar{h}\|_{A0}\lesssim_{\delta} \epsilon_1\|\bar{h}\|_{A0},\label{eq:A0norm-product-1}
  \end{align}
\end{lemma}
\begin{proof}
   In view of Lemma \ref{lem:8.1JiaHao} and  \eqref{eq:boot-V1},  it suffices to prove that
   \begin{align}\label{est:A0norm-prove1}
      & A_{0}(t,\xi)\dfrac{\langle t\rangle \big(\langle t\rangle^2 +|\xi|^2\big)}{\langle\xi\rangle^2}\lesssim_{\delta} A_R(t,\xi-\eta)A_{0}(t,\eta)\dfrac{\langle t\rangle\big(\langle t\rangle^2 +|\eta|^2\big)}{\langle \eta\rangle^2}\left\{\langle \xi-\eta\rangle^{-2} +\langle \eta\rangle^{-2}\right\},
   \end{align}
   and
   \begin{align}\label{est:A0norm-prove2}
   \begin{aligned}
      |(\dot{A}_0A_{0})(t,\xi)|^{1/2}& \dfrac{\langle t\rangle \big(\langle t\rangle^2 +|\xi|^2\big)}{\langle\xi\rangle^2}\lesssim_{\delta} \big[|({\dot{A}_R/A_{R}})(t,\xi-\eta)|^{1/2}+|({\dot{A}_0/A_0})(t,\eta)|^{1/2}\big] \\
      &\times A_R(t,\xi-\eta)A_{0}(t,\eta)\dfrac{\langle t\rangle\big(\langle t\rangle^2 +|\eta|^2\big)}{\langle \eta\rangle^2}\left\{\langle \xi-\eta\rangle^{-2} +\langle \eta\rangle^{-2}\right\}.
   \end{aligned}
   \end{align}
   Notice that
   \begin{align}\label{est:A0norm-prove3}
      & \dfrac{\langle t\rangle \big(\langle t\rangle^2 +|\xi|^2\big)}{\langle\xi\rangle^2} \lesssim_{\delta}   \dfrac{\langle t\rangle \big(\langle t\rangle^2 +|\eta|^2\big)}{\langle\eta\rangle^2} \mathrm{e}^{\delta\min(\langle\xi-\eta\rangle,\langle\eta\rangle)^{1/2}}.
   \end{align}
 Then the desired bounds \eqref{est:A0norm-prove1} and \eqref{est:A0norm-prove2} follow from \eqref{est:A0norm-prove3} and Lemma \ref{lem:8.3JiaHao}.
\end{proof}

\begin{lemma}\label{lem:ABnorm-product-a}
 It holds that
  \begin{align}\label{eq:ABnorm-product-a}
     &\|[ah]_{\neq}\|_{A} +\|[a(V_1+1)h]_{\neq}\|_{A}+\|\overline{ah}\|_{A0}\lesssim_{\delta}\epsilon_1 (\|h_{\neq}\|_{A}+\|\bar{h}\|_{A0}).
  \end{align}
\end{lemma}
\begin{proof}
  \textbf{Step 1.} We first estimate $\|[ah]_{\neq}\|_{A}$ and we will prove that
  \begin{align}\label{eq:product-(ah)neq}
    & \|[ah]_{\neq}\|_{A}\lesssim_{\delta} (\mathcal{E}_{a}(t)+\mathcal{B}_{a}(t))^{1/2}  (\|h_{\neq}\|_{A}+\|\bar{h}\|_{A0})\lesssim_{\delta} \epsilon_1 (\|h_{\neq}\|_{A}+\|\bar{h}\|_{A0}).
  \end{align} Let $(\sigma,\rho)=(k-l,\xi-\eta)$. Thanks to Lemma \ref{lem:8.1JiaHao} {and \eqref{est:par-t-A*-1}}, it suffices to prove that

  (i) for $k\neq 0,\ l\neq 0$, we have
  \begin{align}\label{eq:product-a-prove1}
  \begin{aligned}
      A_k(t,\xi)\dfrac{\langle t\rangle\big(\langle t\rangle+|(k,\xi)|\big)\langle t-\xi/k\rangle}{\langle\xi/k\rangle^2} \lesssim_{\delta} & A_{l}(t,\eta)\dfrac{\langle t\rangle\big(\langle t\rangle+|(l,\eta)|\big)\langle t-\eta/l\rangle}{\langle\eta/l\rangle^2} A^*_{\sigma}(t,\rho)\\
     &\times\left\{\langle l,\eta\rangle^{-2}+\langle \sigma,\rho\rangle^{-2}\right\},
  \end{aligned}
  \end{align}
  and
   \begin{align}\label{eq:product-a-prove2}
   \begin{aligned}
      |(\dot{A}_{k} A_k)(t,\xi)|^{\f12}&\dfrac{\langle t\rangle\big(\langle t\rangle+|(k,\xi)|\big)\langle t-\xi/k\rangle}{\langle\xi/k\rangle^2} \lesssim_{\delta} \big[|(\dot{A}_{l}/A_l)(t,\eta)|^{\f12}+|(\dot{A}_\sigma/A_{\sigma}) (t,\rho)|^{\f12}\big]\\
     &\times A_{l}(t,\eta)\dfrac{\langle t\rangle\big(\langle t\rangle+|(l,\eta)|\big)\langle t-\eta/l\rangle}{\langle\eta/l\rangle^2} A^*_{\sigma}(t,\rho)\left\{\langle l,\eta\rangle^{-2}+\langle \sigma,\rho\rangle^{-2}\right\}.
  \end{aligned}
  \end{align}

   (ii) for $k\neq0,\ l=0$, we have
  \begin{align}\label{eq:product-a-prove3}
  \begin{aligned}
      A_k(t,\xi)\dfrac{\langle t\rangle\big(\langle t\rangle+|(k,\xi)|\big)\langle t-\xi/k\rangle}{\langle\xi/k\rangle^2} \lesssim_{\delta} & A_{0}(t,\eta)\dfrac{\langle t\rangle\big(\langle t\rangle+|\eta|\big)^2}{\langle\eta\rangle^2} A^*_{k}(t,\rho)\\
     &\times\left\{\langle\eta\rangle^{-2}+\langle k,\rho\rangle^{-2}\right\},
  \end{aligned}
  \end{align}
  and
   \begin{align}\label{eq:product-a-prove4}
   \begin{aligned}
      |(\dot{A}_{k} A_k)(t,\xi)|^{\f12}&\dfrac{\langle t\rangle\big(\langle t\rangle+|(k,\xi)|\big)\langle t-\xi/k\rangle}{\langle\xi/k\rangle^2} \lesssim_{\delta} \big[|(\dot{A}_{0}/A_0)(t,\eta)|^{\f12}+|(\dot{A}_k/A_{k})(t,\rho)|^{\f12}\big]\\
     &\times A_{0}(t,\eta)\dfrac{\langle t\rangle\big(\langle t\rangle+|\eta|\big)^2}{\langle\eta\rangle^2} A^*_{k}(t,\rho)\left\{\langle\eta\rangle^{-2}+\langle k,\rho\rangle^{-2}\right\}.
  \end{aligned}
  \end{align}

 Let us first  prove \eqref{eq:product-a-prove1} and \eqref{eq:product-a-prove2}.\smallskip

 \textbf{Case 1.} $|(\sigma,\rho)|\leq |(l,\eta)|$. By \eqref{est:8.3JiaHao-mod3} with $j=2$, we have
  \begin{align}\label{eq:product-a-pr-1}
     &A_k(t,\xi)\langle t-\xi/k\rangle\lesssim_{\delta} A_l(t,\eta)\langle t-\eta/l\rangle A_{\sigma}(t,\rho) \mathrm{e}^{-(\lambda(t)/20)\langle \sigma\rho\rangle^{1/2}}.
  \end{align}
 Thanks to
  \begin{align*}
  & \dfrac{\langle t\rangle +|(k,\xi)|}{\langle t\rangle +|(l,\eta)|} \dfrac{\langle \eta/l\rangle^2}{\langle \xi/k\rangle^2}\lesssim \langle \sigma,\rho\rangle^3,
  \end{align*}
we get
  \begin{align}\label{eq:product-a-pr-2}
     & \dfrac{\langle t\rangle \big(\langle t\rangle +|(k,\xi)|\big)}{\langle\xi/k\rangle^2 }\lesssim_{\delta} \dfrac{\langle t\rangle \big(\langle t\rangle +|(l,\eta)|\big)}{\langle\eta/l\rangle^2 } \mathrm{e}^{\delta\langle \sigma ,\rho\rangle^{1/2}}.
  \end{align}
  By \eqref{eq:product-a-pr-1}, \eqref{eq:product-a-pr-2}, $A_{\sigma}(t,\rho)\leq A_{\sigma}^*(t,\rho)$ and \eqref{est:8.3JiaHao-mod2}, we deduce {\eqref{eq:product-a-prove1}} and \eqref{eq:product-a-prove2} for Case 1.\smallskip

  \textbf{Case 2.} $|(l,\eta)|\leq|(\sigma,\rho)|$ and $\sigma\neq 0$.  By \eqref{est:8.3JiaHao-mod3} with $j=2$, we have
  \begin{align}\label{eq:product-a-pr-3}
     & A_k(t,\xi)\langle t-\xi/k\rangle\lesssim_{\delta} A_l(t,\eta)A_{\sigma}(t,\rho)\langle t-\rho/\sigma\rangle \mathrm{e}^{-(\lambda(t)/20)\langle l,\eta\rangle^{1/2}},
  \end{align}
  Thanks to $\langle t\rangle +|(k,\xi)|\lesssim (\langle t\rangle +|k|)\langle\xi/k\rangle
  $, a direct calculation gives
  \begin{align*}
    \dfrac{\langle t\rangle \big(\langle t\rangle +|(k,\xi)|\big)\langle t-\rho/\sigma\rangle}{\langle\xi/k\rangle^2}  & \lesssim \left(1+\dfrac{|k|}{\langle t\rangle}\right)\dfrac{\langle t\rangle^2\langle t-\rho/k\rangle}{\langle\xi/k\rangle} \lesssim \left(1+\dfrac{|k|}{\langle t\rangle}\right)\left(\langle t\rangle^3 +\dfrac{\langle t\rangle^2|\rho|}{|k|\langle\xi/k\rangle}\right).
  \end{align*}
 Notice that
  \begin{align*}
     & \left(1+\dfrac{|k|}{\langle t\rangle}\right) \lesssim \left(1+\dfrac{|\sigma|}{\langle t\rangle}\right){ |l|},
  \end{align*}
  and  thanks to $\langle t-\eta/l\rangle \langle \eta/l\rangle \gtrsim \langle t\rangle$, we have
  \begin{align*}
     \langle t\rangle^3 \lesssim &\langle t\rangle \big(\langle t\rangle+|(l,\eta)|\big)\langle t-\eta/l\rangle\langle \eta/l\rangle\lesssim  \dfrac{\langle t\rangle\big(\langle t\rangle+|(l,\eta)|\big)\langle t-\eta/l\rangle}{\langle \eta/l\rangle^2} \langle l,\eta\rangle^3,
  \end{align*}
and
  \begin{align*}
     \dfrac{\langle t\rangle^2|\rho|}{|k|\langle\xi/k\rangle}& \lesssim \dfrac{\langle t\rangle^2 |\rho|}{\langle\xi\rangle}\lesssim \langle t\rangle^2 \langle \eta\rangle \lesssim  \dfrac{\langle t\rangle\big(\langle t\rangle+|(l,\eta)|\big)\langle t-\eta/l\rangle}{\langle \eta/l\rangle^2} \langle l,\eta\rangle^3.
  \end{align*}
  The above inequalities ensure that
  \begin{align*}
     \dfrac{\langle t\rangle \big(\langle t\rangle +|(k,\xi)|\big)\langle t-\rho/\sigma\rangle}{\langle\xi/k\rangle^2 } &\lesssim \dfrac{\langle t\rangle(\langle t\rangle+|(l,\eta)|)\langle t-\eta/l\rangle}{\langle \eta/l\rangle^2} \bigg(1+\dfrac{|\sigma|}{\langle t\rangle}\bigg)\langle l,\eta\rangle^4,
  \end{align*}
which gives
  \begin{align}\label{eq:product-a-pr-4}
     & \dfrac{\langle t\rangle \big(\langle t\rangle +|(k,\xi)|\big)}{\langle\xi/k\rangle^2}\lesssim_{\delta} \dfrac{|l|\langle t\rangle \big(\langle t\rangle +|(l,\eta)|\big)\langle t-\eta/l\rangle}{\langle\eta/l\rangle^2\langle t-\rho/\sigma\rangle}\left(1+ \dfrac{\sigma^2}{\langle t\rangle^2}\right)^{1/2}  \mathrm{e}^{\delta\langle l,\eta\rangle^{1/2}}.
  \end{align}
  By \eqref{eq:product-a-pr-3}, \eqref{eq:product-a-pr-4}, $A_{\sigma}(t,\rho)\left(1+\sigma^2/\langle t\rangle^2\right)^{1/2}\leq A_{\sigma}^*(t,\rho)$ and \eqref{est:8.3JiaHao-mod2}, we deduce \eqref{eq:product-a-prove1} and \eqref{eq:product-a-prove2} for Case 2.\smallskip

\textbf{Case 3.} $|(l,\eta)|\leq|\rho|$ and $\sigma=k-l= 0$. By \eqref{est:8.3JiaHao-mod3}  with $j=2$, we have
  \begin{align}\label{eq:product-a-pr-5}
     & A_k(t,\xi)\langle t-\xi/k\rangle\lesssim_{\delta} A_k(t,\eta)A_{0}(t,\rho)\big(\langle t\rangle+|\rho|\big) \mathrm{e}^{-(\lambda(t)/20)\langle k,\eta\rangle^{1/2}}.
  \end{align}
A  direct calculations gives
  \begin{align*}
     \dfrac{\langle t\rangle \big(\langle t\rangle +|(k,\xi)|\big)\big(\langle t\rangle+|\rho|\big)}{\langle\xi/k\rangle^2 } &\lesssim  \left(1+\dfrac{|k|}{\langle t\rangle}\right)\left(\langle t\rangle^3 +\dfrac{\langle t\rangle^2|\rho|}{\langle\xi/k\rangle}\right)\lesssim |k|\left(\langle t\rangle^3 +\dfrac{\langle t\rangle^2|\rho|}{\langle\xi/k\rangle}\right)\\
     &\lesssim |k|\langle t\rangle^3 +\dfrac{|k|^2\langle t\rangle^2 \langle \rho\rangle}{\langle\xi\rangle} \lesssim |k|^2\langle \eta\rangle\langle t\rangle^3.
  \end{align*}
Thanks to $\langle t-\eta/k\rangle \langle \eta/k\rangle \gtrsim \langle t\rangle$, we have
  \begin{align*}
     |k|^2\langle\eta\rangle\langle t\rangle^3 \lesssim &|k|^2\langle\eta\rangle\langle t\rangle \big(\langle t\rangle+|(k,\eta)|\big)\langle t-\eta/k\rangle\langle \eta/k\rangle \\
     \lesssim& \dfrac{\langle t\rangle(\langle t\rangle+|(k,\eta)|)\langle t-\eta/k\rangle}{\langle \eta/k\rangle^2} \langle k,\eta\rangle^6.
  \end{align*}
The above inequalities ensure that
  \begin{align*}
     \dfrac{\langle t\rangle \big(\langle t\rangle +|(k,\xi)|\big)\big(\langle t\rangle+|\rho|\big)}{\langle\xi/k\rangle^2 } &\lesssim \dfrac{\langle t\rangle(\langle t\rangle+|(k,\eta)|)\langle t-\eta/k\rangle}{\langle \eta/k\rangle^2} \langle k,\eta\rangle^6,
  \end{align*}
which gives
  \begin{align}\label{eq:product-a-pr-6}
     & \dfrac{\langle t\rangle \big(\langle t\rangle +|(k,\xi)|\big)}{\langle\xi/k\rangle^2 }\lesssim_{\delta} \dfrac{|k|\langle t\rangle \big(\langle t\rangle +|(k,\eta)|\big)\langle t-\eta/k\rangle}{\langle\eta/k\rangle^2\big(\langle t\rangle+|\rho|\big)} \mathrm{e}^{\delta\langle k,\eta\rangle^{1/2}}.
  \end{align}
By \eqref{eq:product-a-pr-5}, \eqref{eq:product-a-pr-6}, $A_{0}(t,\rho)\leq A_{0}^*(t,\rho)$ and \eqref{est:8.3JiaHao-mod2}, we deduce  \eqref{eq:product-a-prove1} and \eqref{eq:product-a-prove2} for Case 3.\smallskip

 Next we  prove  \eqref{eq:product-a-prove3} and \eqref{eq:product-a-prove4}.\smallskip

  \textbf{Case 1.} $|(k,\rho)|\leq |\eta|$. By \eqref{est:8.3JiaHao-mod3} with $j=2$, we have
  \begin{align}\label{eq:product-a-pr-7}
     &A_k(t,\xi)\langle t-\xi/k\rangle\lesssim_{\delta} A_0(t,\eta)\big(\langle t\rangle+|\eta|\big) A_{k}(t,\rho) \mathrm{e}^{-(\lambda(t)/20)\langle k,\rho\rangle^{1/2}}.
  \end{align}
 A  direct calculation shows
  \begin{align*}
     \dfrac{\langle t\rangle \big(\langle t\rangle +|(k,\xi)|\big)}{\langle\xi/k\rangle^2} &\lesssim   \dfrac{\langle t\rangle \big(\langle t\rangle +|\eta|\big)}{ \langle\eta\rangle^2} \dfrac{\langle \eta\rangle^2}{\langle \xi/k\rangle^2}\dfrac{\langle t\rangle +|(k,\xi)|}{\langle t\rangle +|\eta|}\\
     &\lesssim   \dfrac{\langle t\rangle \big(\langle t\rangle +|\eta|\big)}{ \langle\eta\rangle^2}  \dfrac{|k|^2\langle \eta\rangle^2}{\langle \xi\rangle^2}\dfrac{\langle t\rangle +|(k,\xi)|}{\langle t\rangle +|\eta|} \lesssim   \dfrac{\langle t\rangle \big(\langle t\rangle +|\eta|\big)}{ \langle\eta\rangle^2}\langle k,\rho\rangle^5,
  \end{align*}
which gives
\begin{align}\label{eq:product-a-pr-8}
     & \dfrac{\langle t\rangle \big(\langle t\rangle +|(k,\xi)|\big)}{\langle\xi/k\rangle^2 }\lesssim_{\delta} \dfrac{\langle t\rangle\big(\langle t\rangle+|\eta|\big)}{\langle \eta\rangle^2}\mathrm{e}^{\delta\langle k,\rho\rangle^{1/2}}.
  \end{align}
 Then  by \eqref{eq:product-a-pr-7}, \eqref{eq:product-a-pr-8}, $A_{k}(t,\rho)\leq A_{k }^*(t,\rho)$ and \eqref{est:8.3JiaHao-mod2}, we  deduce \eqref{eq:product-a-prove3} and \eqref{eq:product-a-prove4} for Case 1.\smallskip

  \textbf{Case 2.} $ |\eta|\leq |(k,\rho)|$. By \eqref{est:8.3JiaHao-mod3}  with $j=2$, we have
  \begin{align}\label{eq:product-a-pr-9}
     &A_k(t,\xi)\langle t-\xi/k\rangle\lesssim_{\delta} A_0(t,\eta) A_{k}(t,\rho)\langle t-\rho/k\rangle \mathrm{e}^{-(\lambda(t)/20)\langle \eta\rangle^{1/2}},
  \end{align}
 A direct calculations gives
  \begin{align*}
     \dfrac{\langle t\rangle \big(\langle t\rangle +|(k,\xi)|\big)\langle t-\rho/k\rangle}{\langle\xi/k\rangle^2} &\lesssim  \left(1+\dfrac{|k|}{\langle t\rangle}\right)\left(\langle t\rangle^3 +\dfrac{\langle t\rangle^2|\rho|}{|k|\langle\xi/k\rangle}\right) .
  \end{align*}
Notice that
\begin{align*}
    \langle t\rangle^3 +\dfrac{\langle t\rangle^2|\rho|}{|k|\langle\xi/k\rangle} \lesssim \langle t\rangle^3 +\dfrac{\langle t\rangle^2|\rho|}{\langle\xi\rangle}\lesssim \langle t\rangle^3\langle \eta \rangle\lesssim \dfrac{\langle t\rangle\big(\langle t\rangle+|\eta|\big)^2}{\langle \eta\rangle^2}\langle \eta\rangle^3.
  \end{align*}
 Then we have
  \begin{align}\nonumber
     \dfrac{\langle t\rangle \big(\langle t\rangle +|(k,\xi)|\big)}{\langle\xi/k\rangle^2}\lesssim& \dfrac{\langle t\rangle\big(\langle t\rangle+|\eta|\big)^2}{\langle \eta\rangle^2\langle t-\rho/k\rangle}\left(1+\dfrac{k^2}{\langle t\rangle^2}\right)^{1/2}\langle \eta\rangle^3\\ \label{eq:product-a-pr-10}
    \lesssim_{\delta}& \dfrac{\langle t\rangle\big(\langle t\rangle+|\eta|\big)^2}{\langle \eta\rangle^2\langle t-\rho/k\rangle}\left(1+\dfrac{k^2}{\langle t\rangle^2}\right)^{1/2}\mathrm{e}^{\delta\langle \eta\rangle^{1/2}}.
  \end{align}
  Then by \eqref{eq:product-a-pr-9}, \eqref{eq:product-a-pr-10},  $A_{k}(t,\rho)\left(1+k^2/\langle t\rangle^2\right)^{1/2}\leq A_{k }^*(t,\rho)$ and \eqref{est:8.3JiaHao-mod2}, we deduce  \eqref{eq:product-a-prove3} and \eqref{eq:product-a-prove4} for Case 2.\smallskip

 \noindent\textbf{Step 2.} We estimate $\|[a(V_1+1)h]_{\neq}\|_{A}$ and we will prove that
  \begin{align}\label{eq:product-(aV1h)neq}
    & \|[a(V_1+1)h]_{\neq}\|_{A}\lesssim_{\delta} (\mathcal{E}_{a}(t)+\mathcal{B}_{a}(t))^{1/2}  (\|h_{\neq}\|_{A}+\|\bar{h}\|_{A0})\lesssim_{\delta} \epsilon_1 (\|h_{\neq}\|_{A}+\|\bar{h}\|_{A0}).
  \end{align}
By Step 1, we have
\begin{align*}
   & \|[a(V_1+1)h]_{\neq}\|_{A}\lesssim_{\delta} (\mathcal{E}_{a}(t)+\mathcal{B}_{a}(t))^{1/2}  (\|[(V_1+1)h]_{\neq}\|_{A}+\|(V_1+1)\bar{h}\|_{A0}),
\end{align*}
By Lemma \ref{lem:ABnorm-product} and Lemma \ref{lem:A0norm-product}, we get
\begin{align*}
   & \|[(V_1+1)h]_{\neq}\|_{A}+\|(V_1+1)\bar{h}\|_{A0} \lesssim \|h_{\neq}\|_{A}+\|\bar{h}\|_{A0}.
\end{align*}
Then we finish the proof of \eqref{eq:product-(aV1h)neq}.\smallskip

\noindent\textbf{Step 3.} We estimate $\|\overline{ah}\|_{A0}$ and we will prove that
  \begin{align}\label{eq:product-(ah)0}
    & \|\overline{ah}\|_{A0}\lesssim_{\delta} (\mathcal{E}_{a}(t)+\mathcal{B}_{a}(t))^{1/2}  (\|h_{\neq}\|_{A}+\|\bar{h}\|_{A0})\lesssim_{\delta} \epsilon_1 (\|h_{\neq}\|_{A}+\|\bar{h}\|_{A0}).
  \end{align}

By Lemma \ref{lem:8.1JiaHao} {and \eqref{est:par-t-A*-1}}, it suffices to prove that \smallskip

  (i) for $k\neq 0$, we have
  \begin{align}\label{eq:product-a-A0prove1}
  \begin{aligned}
      A_0(t,\xi)\dfrac{\langle t\rangle\big(\langle t\rangle+|\xi|\big)^2}{\langle\xi\rangle^2} \lesssim_{\delta} & A_{k}(t,\eta)\dfrac{\langle t\rangle\big(\langle t\rangle+|(k,\eta)|\big)\langle t-\eta/k\rangle}{\langle\eta/k\rangle^2} A^*_{-k}(t,\rho)\\
     &\times\left\{\langle k,\eta\rangle^{-2}+\langle k,\rho\rangle^{-2}\right\},
  \end{aligned}
  \end{align}
  and
   \begin{align}\label{eq:product-a-A0prove2}
   \begin{aligned}
      |(\dot{A}_{0}& A_0)(t,\xi)|^{1/2}\dfrac{\langle t\rangle\big(\langle t\rangle+|\xi|\big)^2}{\langle\xi\rangle^2} \lesssim_{\delta} \big[|(\dot{A}_{k}/A_k)(t,\eta)|^{1/2}+|(\dot{A}_{-k}/A_{-k})(t,\rho)|^{1/2}\big]\\
     &\times A_{k}(t,\eta)\dfrac{\langle t\rangle\big(\langle t\rangle+|(k,\eta)|\big)\langle t-\eta/k\rangle}{\langle\eta/k\rangle^2} A^*_{-k}(t,\rho)\left\{\langle k,\eta\rangle^{-2}+\langle k,\rho\rangle^{-2}\right\}.
  \end{aligned}
  \end{align}

   (ii) for $k= 0$, we have
  \begin{align}\label{eq:product-a-A0prove3}
      A_0(t,\xi)\dfrac{\langle t\rangle\big(\langle t\rangle+|\xi|\big)^2}{\langle\xi\rangle^2} \lesssim_{\delta} & A_{0}(t,\eta)\dfrac{\langle t\rangle\big(\langle t\rangle+|\eta|\big)^2}{\langle\eta\rangle^2}  A^*_{0}(t,\rho)\left\{\langle \eta\rangle^{-2}+\langle \rho\rangle^{-2}\right\},
  \end{align}
  and
   \begin{align}\label{eq:product-a-A0prove4}
   \begin{aligned}
      |(\dot{A}_{0} A_0)(t,\xi)|^{1/2}&\dfrac{\langle t\rangle\big(\langle t\rangle+|\xi|\big)^2}{\langle\xi\rangle^2} \lesssim_{\delta} \big[|(\dot{A}_{0}/A_0)(t,\eta)|^{1/2}+|(\dot{A}_{0}/A_{0})(t,\rho)|^{1/2}\big]\\
     &\times A_{0}(t,\eta)\dfrac{\langle t\rangle\big(\langle t\rangle+|\eta|\big)^2}{\langle\eta\rangle^2}  A^*_{0}(t,\rho)\left\{\langle \eta\rangle^{-2}+\langle \rho\rangle^{-2}\right\}.
  \end{aligned}
  \end{align}

  Let us first  prove \eqref{eq:product-a-A0prove1} and \eqref{eq:product-a-A0prove2}.\smallskip

   \textbf{Case 1.} $|(k,\rho)|\leq |(k,\eta)|$. By \eqref{est:8.3JiaHao-mod3} with $j=2$, we have
  \begin{align}\label{eq:product-a-A0kk-1}
     &A_0(t,\xi)\big(\langle t\rangle+|\xi|\big) \lesssim_{\delta} A_k(t,\eta)\langle t-\eta/k\rangle A_{-k}(t,\rho) \mathrm{e}^{-(\lambda(t)/20)\langle k,\rho\rangle^{1/2}},
  \end{align}
 Notice that
  \begin{align*}
  & \dfrac{\langle t\rangle +|\xi|}{\langle t\rangle +|(k,\eta)|} \dfrac{\langle \eta/k\rangle^2}{\langle \xi\rangle^2}\lesssim  \langle k,\rho\rangle^3,
  \end{align*}
which gives
  \begin{align}\label{eq:product-a-A0kk-2}
     & \dfrac{\langle t\rangle \big(\langle t\rangle +|\xi|\big)}{ \langle\xi\rangle^2}\lesssim_{\delta} \dfrac{\langle t\rangle \big(\langle
      t\rangle+|(k,\eta)|\big)}{\langle\eta/k\rangle^2} \mathrm{e}^{\delta\langle k ,\rho\rangle^{1/2}}.
  \end{align}
  By \eqref{eq:product-a-A0kk-1}, \eqref{eq:product-a-A0kk-2}, $A_{-k}(t,\rho)\leq A_{-k}^*(t,\rho)$ and \eqref{est:8.3JiaHao-mod2}, we deduce  \eqref{eq:product-a-A0prove1} and \eqref{eq:product-a-A0prove2} for Case 1.\smallskip

 \textbf{Case 2.} $ |(k,\eta)|\leq |(k,\rho)|$. By \eqref{est:8.3JiaHao-mod3} in Lemma \ref{lem:8.3JiaHao-mod1} (taking $j=2$), we have
  \begin{align}\label{eq:product-a-A0kk-3}
     &A_0(t,\xi)\big(\langle t\rangle+|\xi|\big)\lesssim_{\delta} A_k(t,\eta) A_{-k}(t,\rho)\langle t+\rho/k\rangle \mathrm{e}^{-(\lambda(t)/20)\langle k, \eta\rangle^{1/2}}.
  \end{align}
 A  direct calculation gives
  \begin{align*}
     \dfrac{\langle t\rangle \big(\langle t\rangle +|\xi|\big)\langle t+\rho/k\rangle}{\langle\xi\rangle^2} &\lesssim  \langle t\rangle^3/\langle\xi\rangle^2 +\langle t\rangle^2\langle \rho/k\rangle/\langle\xi\rangle^2 +\langle t\rangle^2/\langle\xi\rangle +\langle t\rangle\langle \rho/k\rangle/\langle\xi\rangle\\
     &\lesssim\langle t\rangle^3\left(1+\dfrac{\langle\rho/k\rangle}{\langle\xi\rangle}\right)\lesssim \langle t\rangle^3\langle k,\eta\rangle.
  \end{align*}
   Thanks to $\langle t-\eta/k\rangle\langle \eta/k\rangle\gtrsim \langle t\rangle$, we have
  \begin{align*}
     \langle t\rangle^3\lesssim \langle t\rangle \big(\langle t\rangle+|(k,\eta)|\big)\langle t-\eta/k\rangle\langle \eta/k\rangle \lesssim \dfrac{\langle t\rangle \big(\langle
      t\rangle+|(k,\eta)|\big)\langle t-\eta/k\rangle}{\langle\eta/k\rangle^2} \langle k,\eta\rangle^3.
  \end{align*}
The above inequalities ensure that
  \begin{align*}
     \dfrac{\langle t\rangle \big(\langle t\rangle +|\xi|\big)\langle t+\rho/k\rangle}{\langle\xi\rangle^2}\lesssim \dfrac{\langle t\rangle \big(\langle
      t\rangle+|(k,\eta)|\big)\langle t-\eta/k\rangle}{\langle\eta/k\rangle^2} \langle k,\eta\rangle^4,
  \end{align*}
which gives
\begin{align}\label{eq:product-a-A0kk-4}
     & \dfrac{\langle t\rangle \big(\langle t\rangle +|\xi|\big)}{\langle\xi\rangle^2}\lesssim_{\delta} \dfrac{\langle t\rangle \big(\langle
      t\rangle+|(k,\eta)|\big)\langle t-\eta/k\rangle}{\langle\eta/k\rangle^2\langle t+\rho/k\rangle} \mathrm{e}^{\delta\langle k ,\eta\rangle^{1/2}}.
  \end{align}
   By \eqref{eq:product-a-A0kk-3}, \eqref{eq:product-a-A0kk-4}, $A_{-k}(t,\rho)\leq A_{-k}^*(t,\rho)$ and \eqref{est:8.3JiaHao-mod2}, we deduce \eqref{eq:product-a-A0prove1} and \eqref{eq:product-a-A0prove2} for Case 2.\smallskip

 Next we prove  \eqref{eq:product-a-A0prove3} and \eqref{eq:product-a-A0prove4}. By \eqref{est:8.3JiaHao-mod3}  with $j=1$ and $k=l=0$, we have
  \begin{align}\label{eq:product-a-A000-1}
     &A_0(t,\xi)\lesssim_{\delta} A_0(t,\eta) A_{0}(t,\rho) \mathrm{e}^{-(\lambda(t)/20)\min(\langle \rho\rangle,\langle \eta\rangle)^{1/2}},
  \end{align}
and we also have
  \begin{align}\label{eq:product-a-A000-2}
     \dfrac{\langle t\rangle\big(\langle t\rangle+|\xi|\big)^2}{\langle\xi\rangle^2}\lesssim_{\delta} \dfrac{\langle t\rangle\big(\langle t\rangle+|\eta|\big)^2}{\langle\eta\rangle^2} \mathrm{e}^{\delta\min(\langle \rho\rangle,\langle \eta\rangle)^{1/2}}.
  \end{align}
Then by \eqref{eq:product-a-A000-1}, \eqref{eq:product-a-A000-2}, $A_{0}(t,\rho)\leq A_{0}^*(t,\rho)$ and \eqref{est:8.3JiaHao-mod2}, we deduce  \eqref{eq:product-a-A0prove3} and \eqref{eq:product-a-A0prove4}.\smallskip

This finishes the proof of this lemma.
\end{proof}

\subsection{Elliptic estimates}
In this subsection, we prove the following elliptic estimates.

\begin{lemma}\label{lem:elliptic}
Let $h(t,z,v),\ H(t,z,v),\ H_1(t,z,v),\ H_2(t,z,v)$ satisfy
\begin{align*}
   &\partial_z^2h+
   (V_1+1)^2(\partial_v-t\partial_z)^2h+V_2(\partial_v-t\partial_z)h =H+\partial_zH_1+(V_1+1)(\partial_v-t\partial_z)H_2.
\end{align*}
There exists $c_1(\delta)>0$ such that if $\epsilon_1\leq c_1(\delta)$, then we have
\begin{align*}
   & \|(\partial_z,\partial_v-t\partial_z)h_{\neq}\|_{A}\lesssim_{\delta} \|H_{\neq}\|_{B}+\|H_{1,\neq}\|_{A} +\|H_{2,\neq}\|_{A}.
\end{align*}
\end{lemma}

\begin{proof}
We rewrite the equation as
\begin{align*}
    \partial_z^2h_{\neq}+ (\partial_v-t\partial_z)^2h_{\neq} =&-(V_1^2+2V_1)(\partial_v-t\partial_z)^2 h_{\neq}{-}V_2(\partial_v-t\partial_z)h_{\neq}\\
    &+H_{\neq}+\partial_zH_{1,\neq}+(V_1+1)(\partial_v-t\partial_z)H_{2,\neq}.
\end{align*}
By the definitions of $\|\cdot\|_{A}$ and $\|\cdot\|_{B}$, we have
\begin{align*}
   & \|\partial_z^2h_{\neq}+ (\partial_v-t\partial_z)^2h_{\neq} \|_{B}\approx\|(\partial_z,\partial_v-t\partial_z)h_{\neq}\|_{A}.
\end{align*}
Then we infer that
\begin{align*}
     \|(\partial_z,\partial_v-t\partial_z)h_{\neq}\|_{A}\lesssim_{\delta} &\|(V_1^2+2V_1)(\partial_v-t\partial_z)^2 h_{\neq}\|_{B}+\|V_2(\partial_v-t\partial_z)h_{\neq}\|_{B}\\
     &+\|H_{\neq}\|_{B}+\|\partial_zH_{1,\neq}\|_{B} +\|(V_1+1)(\partial_v-t\partial_z)H_{2,\neq}\|_{B}.
\end{align*}
By Lemma \ref{lem:ABnorm-product}, we have
\begin{align*}
   & \|(V_1^2+2V_1)(\partial_v-t\partial_z)^2 h_{\neq}\|_{B} \lesssim_{\delta} \epsilon_1\|(\partial_v-t\partial_z)^2h_{\neq}\|_{B}\lesssim_{\delta} \epsilon_1 \|(\partial_v-t\partial_z)h_{\neq}\|_{A},\\
   &\|V_2(\partial_v-t\partial_z)h_{\neq}\|_{B}\lesssim_{\delta} \epsilon_1\|(\partial_v-t\partial_z)h_{\neq}\|_{A},\\
   &\|(V_1+1)(\partial_v-t\partial_z)H_{2,\neq}\|_{B} \lesssim_{\delta} (1+\epsilon_1)\|(\partial_v-t\partial_z)H_{2,\neq}\|_{B} \lesssim_{\delta} \|H_{2,\neq}\|_{A}.
\end{align*}
Then by taking $\epsilon_1>0$ sufficiently small(dependent on $\delta$), we obtain
\begin{align*}
   & \|(\partial_z,\partial_v-t\partial_z)h_{\neq}\|_{A}\lesssim_{\delta} \|H_{\neq}\|_{B}+\|\partial_zH_{1,\neq}\|_{B} +\|H_{2,\neq}\|_{A}.
\end{align*}
Thanks to $\|\partial_zH_{1,\neq}\|_{B} \leq \|H_{1,\neq}\|_{A}$, we finish the proof of this lemma.
\end{proof}

\subsection{Estimates for nonlinear term $q$}

Recall that $ q=2\partial_z^2\phi(1-f)+2(\partial_z^2\phi)^2 +2\big[(V_1+1)(\partial_v-t\partial_z)\partial_z\phi\big]^2.$ Let
\begin{align}\label{eq:Qj}
   &Q_1=-\partial_z^2\phi f,\quad Q_2=(\partial_z^2\phi)^2,\quad Q_3=\big[(V_1+1)(\partial_v-t\partial_z)\partial_z\phi\big]^2.
\end{align}

\begin{lemma}\label{lem:q0-nonzero}
 It holds that for $j\in\{1,2,3\}$,
   \begin{align}\label{est:q0}
   \begin{aligned}
      &\sum_{k\in \mathbb{Z}\setminus \{0\}}\int_{\mathbb{R}}A_k(t,\xi)^2\dfrac{\langle t\rangle^{2}(\langle t\rangle^2+|(k,\xi)|^2)}{k^2 \langle\xi/k\rangle^4}|\widetilde{Q_j}(t,k,\xi)|^2\mathrm{d}\xi \lesssim_{\delta}\epsilon_1^4,
   \end{aligned}
   \end{align}
   and
   \begin{align}\label{est:q1}
   \begin{aligned}
      &\sum_{k\in \mathbb{Z}\setminus \{0\}}\int_{1}^{t}\int_{\mathbb{R}}| (\dot{A}_{k} A_k)(s,\xi)|\dfrac{\langle s\rangle^{2}(\langle s\rangle^2+|(k,\xi)|^2)}{k^2\langle\xi/k\rangle^4}|\widetilde{Q_j}(s,k,\xi)|^2\mathrm{d}\xi\mathrm{d}s  \lesssim_{\delta}\epsilon_1^4.
   \end{aligned}
   \end{align}
\end{lemma}

\begin{proof}
Let $(\sigma,\rho)=(k-l,\xi-\eta)$.\smallskip

\noindent\textbf{Step 1.} Let us claim that for $k,l\neq 0$,
\begin{align}\label{eq:q0-nonzero-mul1}
\begin{aligned}
     A_{k}(t,\xi)\dfrac{\langle t\rangle(\langle t\rangle+|(k,\xi)|)}{|k|\langle\xi/k\rangle^2} \lesssim_{\delta} A_l(t,\eta)&A_{\sigma}(t,\rho) \dfrac{|l|\langle t\rangle \langle t-\eta/l\rangle^2}{|\eta|+|l|\langle t\rangle} \\
     &\times \left\{\langle l,\eta\rangle^{-2}+ \langle \sigma,\rho\rangle^{-2}\right\},
\end{aligned}
\end{align}
and
\begin{align}\label{eq:q0-nonzero-mul11}
\begin{aligned}
    |(\dot{A}_{k} A_{k})(t,\xi)|^{1/2}&\dfrac{\langle t\rangle(\langle t\rangle+|(k,\xi)|)}{|k| \langle\xi/k\rangle^2} \lesssim_{\delta} \big[ |(\dot{A}_{l}/ A_{l})(t,\eta)|^{1/2} + |(\dot{A}_{\sigma} /A_{\sigma})(t,\rho)|^{1/2}\big]\\
     &\times A_l(t,\eta)A_{\sigma}(t,\rho) \dfrac{|l|\langle t\rangle \langle t-\eta/l\rangle^2}{|\eta|+|l|\langle t\rangle} \left\{\langle l,\eta\rangle^{-2}+ \langle \sigma,\rho\rangle^{-2}\right\}.
\end{aligned}
\end{align}
Thanks to
\begin{align*}
   &\langle t\rangle +|(k,\xi)|\lesssim |k|\langle \xi/k\rangle\langle t-\xi/k\rangle,\quad  \langle t\rangle+|\eta/l| \lesssim \langle t-\eta/l\rangle\langle \eta/l\rangle,
\end{align*}
we have
\begin{align}\label{eq1}
   & \dfrac{\langle t\rangle\big(\langle t\rangle +|(k,\xi)|\big)}{|k|\langle \xi/k\rangle^2\langle t-\xi/k\rangle} \dfrac{\langle t\rangle+|\eta/l|}{\langle t\rangle \langle t-\eta/l\rangle^2} \lesssim \dfrac{\langle \eta/l\rangle}{\langle \xi/k\rangle\langle t-\eta/l\rangle},
\end{align}
which implies
\begin{align*}
   & \dfrac{\langle t\rangle\big(\langle t\rangle +|(k,\xi)|\big)}{|k|\langle \xi/k\rangle^2\langle t-\xi/k\rangle} \dfrac{\langle t\rangle+|\eta/l|}{\langle t\rangle \langle t-\eta/l\rangle^2} \lesssim \dfrac{\langle \sigma,\rho\rangle^2}{\langle t-\eta/l\rangle}.
\end{align*}

For $\sigma\neq 0$, we have
\begin{align*}
   &\dfrac{\langle \eta/l\rangle\langle t-\rho/\sigma\rangle}{\langle \xi/k\rangle\langle t-\eta/l\rangle} \lesssim \dfrac{\langle \eta/l\rangle\big(\langle t\rangle+|\rho/\sigma|\big)}{\langle \xi/k\rangle\langle t-\eta/l\rangle} \lesssim \dfrac{\langle \eta/l\rangle\langle t\rangle}{\langle t-\eta/l\rangle} +\dfrac{\langle \eta/l\rangle\langle \rho/\sigma\rangle}{\langle \xi/k\rangle} \lesssim \langle l,\eta\rangle^2.
\end{align*}
Thus, we conclude that for $\sigma\neq 0$,
\begin{align}\label{eq:q0-nonzeromulprove-0}
    & \dfrac{\langle t\rangle\big(\langle t\rangle +|(k,\xi)|\big)}{|k|\langle \xi/k\rangle^2\langle t-\xi/k\rangle} \dfrac{\langle t\rangle+|\eta/l|}{\langle t\rangle \langle t-\eta/l\rangle^2} \lesssim_{\delta}\min\left( \dfrac{\mathrm{e}^{\delta \langle \sigma,\rho\rangle^{1/2}}}{\langle t-\eta/l\rangle} ,\dfrac{\mathrm{e}^{\delta \langle l,\eta\rangle^{1/2}}}{\langle t-\rho/\sigma\rangle}\right).
 \end{align}
 Then for $\sigma\neq0$, \eqref{eq:q0-nonzero-mul1} follows from \eqref{eq:q0-nonzeromulprove-0} and \eqref{est:8.3JiaHao-mod3} with $j=2$; \eqref{eq:q0-nonzero-mul11} follows from \eqref{eq:q0-nonzeromulprove-0}, \eqref{est:8.3JiaHao-mod3} with $j=2$ and \eqref{est:8.3JiaHao-mod2}.

 For $\sigma =0$, we have\begin{align*}
   & \dfrac{\langle \eta/l\rangle\big(\langle t\rangle+|\rho|\big)}{\langle \xi/k\rangle\langle t-\eta/l\rangle} \lesssim \dfrac{\langle \eta/l\rangle\langle t\rangle}{\langle t-\eta/l\rangle} +\dfrac{\langle \eta/l\rangle| \rho|}{\langle \xi/k\rangle} \lesssim \langle l,\eta\rangle^2.
\end{align*}Thus, we conclude that for $\sigma= 0$,
\begin{align}\label{eq:q0-nonzeromulprove-1}
    & \dfrac{\langle t\rangle\big(\langle t\rangle +|(k,\xi)|\big)}{|k|\langle \xi/k\rangle^2\langle t-\xi/k\rangle} \dfrac{\langle t\rangle+|\eta/l|}{\langle t\rangle \langle t-\eta/l\rangle^2} \lesssim_{\delta}\min\left( \dfrac{\mathrm{e}^{\delta \langle \sigma,\rho\rangle^{1/2}}}{\langle t-\eta/l\rangle} ,\dfrac{\mathrm{e}^{\delta \langle l,\eta\rangle^{1/2}}}{\langle t\rangle+|\rho|}\right).
 \end{align}
 Then for $\sigma=0$, \eqref{eq:q0-nonzero-mul1} follows from \eqref{eq:q0-nonzeromulprove-1} and \eqref{est:8.3JiaHao-mod3} with $j=2$; \eqref{eq:q0-nonzero-mul11} follows from \eqref{eq:q0-nonzeromulprove-1}, \eqref{est:8.3JiaHao-mod3} with $j=2$ and \eqref{est:8.3JiaHao-mod2}.
 \smallskip

\noindent\textbf{Step 2}. Estimate for $Q_1$.\smallskip

  By Lemma \ref{lem:8.1JiaHao}, the bootstrap assumption for $\Theta$ and $f$, $[\partial_z^2+(\partial_v-t\partial_z)^2]\phi=\Theta$, the estimate for $Q_1$ follows from \eqref{eq:q0-nonzero-mul1} and \eqref{eq:q0-nonzero-mul11}.\smallskip

\noindent\textbf{Step 3}. Estimate for $Q_2$.\smallskip

 Thanks to
  \begin{align*}
     & \dfrac{\langle t\rangle \langle t-\rho/\sigma\rangle^2}{\langle t\rangle +|\rho/\sigma|}\gtrsim \dfrac{\langle t\rangle \langle t-\rho/\sigma\rangle}{\langle t\rangle +|\rho/\sigma|}\gtrsim 1,
  \end{align*}
 we get by \eqref{eq:q0-nonzero-mul1} and \eqref{eq:q0-nonzero-mul11} that for $k,l,\sigma\neq 0$
  \begin{align*}
\begin{aligned}
     A_{k}(t,\xi)\dfrac{\langle t\rangle(\langle t\rangle+|(k,\xi)|)}{|k|\langle\xi/k\rangle^2} \lesssim_{\delta} A_l(t,\eta)&\dfrac{|l|\langle t\rangle \langle t-\eta/l\rangle^2}{|\eta|+|l|\langle t\rangle}A_{\sigma}(t,\rho) \dfrac{|\sigma|\langle t\rangle \langle t-\rho/\sigma\rangle^2}{|\rho|+|\sigma|\langle t\rangle} \\
     &\times \left\{\langle l,\eta\rangle^{-2}+ \langle \sigma,\rho\rangle^{-2}\right\},
\end{aligned}
\end{align*}
and
\begin{align*}
\begin{aligned}
    |(\dot{A}_{k} A_{k})(t,\xi)|^{1/2}&\dfrac{\langle t\rangle(\langle t\rangle+|(k,\xi)|)}{| k|\langle\xi/k\rangle^2} \lesssim_{\delta} \big[ |(\dot{A}_{l}/ A_{l})(t,\eta)|^{1/2} + |(\dot{A}_{\sigma} /A_{\sigma})(t,\rho)|^{1/2}\big]\\
     &\times A_l(t,\eta)\dfrac{|l|\langle t\rangle \langle t-\eta/l\rangle^2}{|\eta|+|l|\langle t\rangle}A_{\sigma}(t,\rho)  \dfrac{|\sigma|\langle t\rangle \langle t-\rho/\sigma\rangle^2}{|\rho|+|\sigma|\langle t\rangle} \left\{\langle l,\eta\rangle^{-2}+ \langle \sigma,\rho\rangle^{-2}\right\}.
\end{aligned}
\end{align*}
Thanks to  Lemma \ref{lem:8.1JiaHao} and the bootstrap assumption for $\Theta$, we deduce \eqref{est:q0} and \eqref{est:q1} for $Q_2$.\smallskip

\noindent\textbf{Step 4.} Estimate for  $Q_3$.\smallskip

We denote
\beno
Q_3^{(0)}= (\partial_v-t\partial_z)\partial_z\phi,\quad Q_3^{(1)}=V_1(\partial_v-t\partial_z)\partial_z\phi,\quad Q_3^{(2)}:=(V_1+1)(\partial_v-t\partial_z)\partial_z\phi.
\eeno
 We claim that for $j\in\{0,1,2\}$,
\begin{align}\label{eq:Q3-012}
  \begin{aligned}
     &\sum_{k\in\mathbb{Z}\setminus \{0\}} \int_{\mathbb{R}}A_{k}(t,\xi)^2\dfrac{|k|^2\langle t\rangle^2\langle t-\xi/k\rangle^2}{|\xi|^2+|k|^2\langle t\rangle^2}|\widetilde{Q^{(j)}_3}(t,k,\xi)|^2\mathrm{d}\xi\lesssim_{\delta}\epsilon_1^2,\\
     &\sum_{k\in\mathbb{Z}\setminus \{0\}} \int_{1}^{t}\int_{\mathbb{R}}|( \dot{A}_{k} A_{k})(s,\xi)| \dfrac{|k|^2\langle s\rangle^2\langle s-\xi/k\rangle^2}{|\xi|^2+|k|^2\langle s\rangle^2}|\widetilde{Q^{(j)}_3}(s,k,\xi)|^2\mathrm{d}\xi\mathrm{d}s\lesssim_{\delta}\epsilon_1^2.
  \end{aligned}
\end{align}
Due to $\Theta=[\partial_z^2+(\partial_v-t\partial_z)^2]\phi$, \eqref{eq:Q3-012} holds for $j=0$.
  In view of Lemma \ref{lem:8.3JiaHao} and the fact
  \begin{align*}
     & \dfrac{|k|\langle t\rangle \langle t-\xi/k\rangle}{|\xi|+|k|\langle t\rangle} \lesssim_{\delta} \dfrac{|k|\langle t\rangle \langle t-\eta/k\rangle}{|\eta|+|k|\langle t\rangle} \mathrm{e}^{\delta\min(\langle \xi-\eta\rangle,\langle k,\eta\rangle
     )^{1/2}},
  \end{align*}
  we have
  \begin{align}\label{est:q0-nonzero-prove1}
     A_{k}(t,\xi)\dfrac{|k|\langle t\rangle \langle t-\xi/k\rangle}{|\xi|+|k|\langle t\rangle} &\lesssim_{\delta}  A_{R}(t,\xi-\eta) \dfrac{|k|\langle t\rangle \langle t-\eta/k\rangle}{|\eta|+|k|\langle t\rangle}A_{k}(t,\eta)\big\{ \langle k,\eta\rangle^{-2} +\langle \xi-\eta\rangle^{-2}\big\},
  \end{align}
  and
  \begin{align}\label{est:q0-nonzero-prove2}
  \begin{aligned}
    |(\dot{A}_{k} A_{k})(t,\xi)|^{1/2}&\dfrac{|k|\langle t\rangle \langle t-\xi/k\rangle}{|\xi|+|k|\langle t\rangle}\lesssim_{\delta} \left[|(\dot{A}_{R}/ A_{R})(t,\xi-\eta)|^{1/2}+|(\dot{A}_{k} /A_{k})(t,\eta)|^{1/2}\right]\\& \times A_{R}(t,\xi-\eta)\cdot A_{k}(t,\eta)\dfrac{|k|\langle t\rangle \langle t-\eta/k\rangle}{|\eta|+|k|\langle t\rangle}\cdot\big\{\langle \xi-\eta\rangle^{-2}+\langle k,\eta\rangle^{-2}\big\}.
     \end{aligned}
  \end{align}
 Then by (ii) of Lemma \ref{lem:8.1JiaHao}, \eqref{est:q0-nonzero-prove1}, \eqref{est:q0-nonzero-prove2}, \eqref{eq:boot-V1} and \eqref{eq:Q3-012}($j=0$), we prove \eqref{eq:Q3-012} for $j=1$. Thanks to $Q_3^{(2)}=Q_3^{(0)}+Q_3^{(1)}$, \eqref{eq:Q3-012} holds for $j=2$.

  In view of Lemma \ref{lem:8.1JiaHao} again, thanks to $Q_3=Q_3^{(2)}\times Q_{3}^{(2)}$ and \eqref{eq:Q3-012},  it suffices to prove that for $k,l,\sigma\neq 0$,
  \begin{align}\label{est:q0-nonzero-prove3}
  \begin{aligned}
     A_{k}(t,\xi)\dfrac{\langle t\rangle(\langle t\rangle+|(k,\xi)|)}{|k|\langle \xi/k\rangle^2} \lesssim_{\delta} A_l(t,\eta)A_{\sigma}(t,\rho)& \dfrac{|l|\langle t\rangle \langle t-\eta/l\rangle}{|\eta|+|l|\langle t\rangle} \dfrac{|\sigma|\langle t\rangle \langle t-\rho/\sigma\rangle}{|\rho|+|\sigma|\langle t\rangle} \\
     &\times\left\{\langle l,\eta\rangle^{-2}+ \langle \sigma,\rho\rangle^{-2}\right\},
  \end{aligned}
  \end{align}
  and
  \begin{align}\label{est:q0-nonzero-prove4}
  \begin{aligned}
    |(\dot{A}_{k} &A_{k})(t,\xi)|^{1/2}\dfrac{\langle t\rangle(\langle t\rangle+|(k,\xi)|)}{|k|\langle \xi/k\rangle^2} \lesssim_{\delta} \big[ |(\dot{A}_{l} /A_{l})(t,\eta)|^{1/2}+  |(\dot{A}_{\sigma}/ A_{\sigma})(t,\rho)|^{1/2}\big]\\
     &\times A_l(t,\eta)A_{\sigma}(t,\rho)\dfrac{|l|\langle t\rangle \langle t-\eta/l\rangle}{|\eta|+|l|\langle t\rangle} \dfrac{|\sigma|\langle t\rangle \langle t-\rho/\sigma\rangle}{|\rho|+|\sigma|\langle t\rangle}\left\{\langle l,\eta\rangle^{-2}+ \langle \sigma,\rho\rangle^{-2}\right\},
  \end{aligned}
  \end{align}
 By symmetry, we only consider the case $|(\sigma,\rho)|\leq|(l,\eta)|$. By applying \eqref{est:8.3JiaHao-mod3} with $j=2$ and \eqref{est:8.3JiaHao-mod2}, it suffices to prove that
 \begin{align}\label{est:q0-nonzero-prove5}
    & \dfrac{\langle t\rangle(\langle t\rangle+|(k,\xi)|)}{|k|\langle \xi/k\rangle^2} \dfrac{|\eta/l|+\langle t\rangle}{\langle t\rangle \langle t-\eta/l\rangle} \dfrac{|\rho/\sigma|+\langle t\rangle}{\langle t\rangle \langle t-\rho/\sigma\rangle}\dfrac{\langle t-\eta/l\rangle}{\langle t-\xi/k\rangle} \lesssim \langle \sigma,\rho \rangle^3.
 \end{align}
First of all, we have
\begin{align}\label{est:q0-nonzero-1}
   & \dfrac{\langle t\rangle(\langle t\rangle+|(k,\xi)|)}{|k|\langle \xi/k\rangle^2} \dfrac{|\eta/l|+\langle t\rangle}{\langle t\rangle \langle t-\eta/l\rangle} \dfrac{|\rho/\sigma|+\langle t\rangle}{\langle t\rangle \langle t-\rho/\sigma\rangle}\dfrac{\langle t-\eta/l\rangle}{\langle t-\xi/k\rangle}\nonumber\\
   &= \dfrac{\langle t\rangle+|(k,\xi)|}{|k|\langle\xi/k\rangle^2} \dfrac{|\eta/l|+\langle t\rangle}{ \langle t-\xi/k\rangle\langle t-\rho/\sigma\rangle} \dfrac{|\rho/\sigma|+\langle t\rangle}{\langle t\rangle }\nonumber\\
  & \lesssim  \dfrac{\langle t\rangle+|(k,\xi)|}{|k|\langle \xi/k\rangle^2} \dfrac{|\eta/l|+\langle t\rangle}{ \langle t-\xi/k\rangle\langle t-\rho/\sigma\rangle}\langle \sigma,\rho\rangle\nonumber\\
  &\lesssim  \langle \sigma,\rho\rangle\left(\dfrac{|\eta/l|+\langle t\rangle}{\langle \xi/k\rangle \langle t-\xi/k\rangle}+ \dfrac{\langle t\rangle|\eta/l|}{|k|\langle \xi/k\rangle^2 \langle t-\xi/k\rangle} +\dfrac{\langle t\rangle^2}{|k|\langle\xi/k\rangle^2 \langle t-\xi/k\rangle\langle t-\rho/\sigma\rangle} \right).
\end{align}
We also have
\begin{align}
   & \dfrac{|\eta/l|+\langle t\rangle}{\langle \xi/k\rangle \langle t-\xi/k\rangle} \lesssim \dfrac{|\eta/l|}{\langle \xi/k\rangle}+\dfrac{\langle t\rangle}{\langle\xi/k\rangle\langle t-\xi/k\rangle}\lesssim\langle\sigma,\rho\rangle^2,\label{est:q0-nonzero-2}\\
   & \dfrac{\langle t\rangle|\eta/l|}{|k|\langle \xi/k\rangle^2 \langle t-\xi/k\rangle}= \dfrac{|\eta/l|}{|k|\langle\xi/k\rangle}\dfrac{\langle t\rangle}{\langle\xi/k\rangle \langle t-\xi/k\rangle} \lesssim \langle \sigma,\rho\rangle.\label{est:q0-nonzero-3}
\end{align}

For $\dfrac{\langle t\rangle^2}{|k|\langle \xi/k\rangle^2 \langle t-\xi/k\rangle\langle t-\rho/\sigma\rangle}$,
we estimate it by considering three cases: \smallskip

Case 1. $|t-\rho/\sigma|\leq |\rho/(10\sigma)|$. In this case, we have $t\approx \rho/\sigma$ and then
\begin{align*}
   &\dfrac{\langle t\rangle^2}{|k|\langle \xi/k\rangle^2 \langle t-\xi/k\rangle\langle t-\rho/\sigma\rangle}\lesssim \langle t\rangle^2 \lesssim \langle \rho/\sigma\rangle^2 \lesssim \langle \sigma,\rho\rangle^2.
\end{align*}

Case 2. $|t-\xi/k|\leq |\xi/(10k)|$. In this case, we have $t\approx \xi/k$ and then
\begin{align*}
   &\dfrac{\langle t\rangle^2}{|k|\langle \xi/k\rangle^2 \langle t-\xi/k\rangle\langle t-\rho/\sigma\rangle}\lesssim \dfrac{\langle t\rangle^2}{|k| \langle\xi/k\rangle^2} \lesssim 1.
\end{align*}

Case 3. $|t-\rho/\sigma|\geq |\rho/(10\sigma)|$ and $|t-\xi/k|\geq |\xi/(10k)|$. In this case, we have $|t-\rho/\sigma|\geq t/11$ and $|t-\xi/k|\geq t/11$ and then
\begin{align*}
   &\dfrac{\langle t\rangle^2}{|k|\langle \xi/k\rangle^2 \langle t-\xi/k\rangle\langle t-\rho/\sigma\rangle}\lesssim  \dfrac{\langle t\rangle^2}{\langle t-\xi/k\rangle\langle t-\rho/\sigma\rangle} \lesssim 1.
\end{align*}
Thus, we arrive at
\begin{align}\label{est:q0-nonzero-4}
   &\dfrac{\langle t\rangle^2}{|k|\langle \xi/k\rangle^2 \langle t-\xi/k\rangle\langle t-\rho/\sigma\rangle}\lesssim \langle \sigma,\rho\rangle^2.
\end{align}
Then \eqref{est:q0-nonzero-prove5} follows from \eqref{est:q0-nonzero-1}, \eqref{est:q0-nonzero-2}, \eqref{est:q0-nonzero-3} and \eqref{est:q0-nonzero-4}.
\end{proof}

\subsection{Proof of Proposition \ref{prop:Improved-P}}

{We first estimate the zero mode.}
Recall that
\begin{align*}
   &  \partial_tu+y\partial_xu+u\cdot\nabla u+\binom{u^2}{0}+(1+\tilde{a})\nabla \Pi=0,
\end{align*}
which implies
\begin{align*}
   & \overline{u\cdot\nabla u^2} +\partial_y\overline{\Pi}+ \overline{\tilde{a}\partial_y\Pi}=0.
\end{align*}
Thank to $u=(-\partial_y\varphi,\partial_x\varphi)$, we have
\begin{align*}
   u\cdot\nabla u^2&=-\partial_y\varphi\partial_x^2\varphi +\partial_x\varphi \partial_x\partial_y\varphi=-\partial_x(\partial_y\varphi\partial_x\varphi) +2\partial_x\varphi\partial_x\partial_y\varphi\\
   &=-\partial_x(\partial_y\varphi\partial_x\varphi) +\partial_y((\partial_x\varphi)^2),
\end{align*}
which gives
\begin{align}\label{eq:Pi-zero-xy}
   &\partial_y\overline{\big[(\partial_x\varphi)^2+\Pi\big]} +\overline{\tilde{a}\partial_y\Pi}=0.
\end{align}
In terms of $(t,z,v)$ variable, we may write
\begin{align}\label{eq:Pi-zero-vz}
   &\partial_v\overline{\big[(\partial_z\phi)^2+P\big]} +\overline{a(\partial_v-t\partial_z)P}=0.
\end{align}

\begin{lemma}\label{lem:parxphi2}
 It holds that
  \begin{align*}
     &\|\partial_v\overline{(\partial_z\phi)^2}\|_{A0}
     \lesssim_{\delta}\mathcal{E}_{\Theta}(t)+\mathcal{B}_{\Theta}(t)\lesssim_{\delta}\epsilon_1^2.
  \end{align*}
\end{lemma}
\begin{proof}
Let $\rho= \xi-\eta$. In view of Lemma \ref{lem:8.1JiaHao}, it suffices to prove that ($l\neq 0$)
  \begin{align}\label{eq:parxphi2-prove1}
  \begin{aligned}
     A_0(t,\xi)\dfrac{\langle t\rangle (\langle t\rangle^2 +|\xi|^2)|\xi|}{\langle \xi\rangle^2} \lesssim_{\delta}& A_{-l}(t,\rho)\dfrac{|l|^2\langle t\rangle\langle t+\rho/l\rangle^2}{|\rho|+|l|\langle t\rangle} A_{l}(t,\eta)\dfrac{|l|^2\langle t\rangle\langle t-\eta/l\rangle^2}{|\eta|+|l|\langle t\rangle}\\
     &\times \left\{\langle l,\rho\rangle^{-2}+ \langle l,\eta\rangle^{-2}\right\}.
  \end{aligned}
  \end{align}
  and
  \begin{align}\label{eq:parxphi2-prove2}
     \begin{aligned}
     |(\dot{A}_{0} &A_0)(t,\xi)|^{1/2}\dfrac{\langle t\rangle (\langle t\rangle^2 +|\xi|^2)|\xi|}{\langle \xi\rangle^2} \lesssim_{\delta} \big[|(\dot{A}_{-l}/A_{-l})(t,\rho)|^{1/2}+|(\dot{A}_{l}/A_{l})(t,\eta)|^{1/2} \big]\\
     &\times A_{-l}(t,\rho)\dfrac{|l|^2\langle t\rangle\langle t+\rho/l\rangle^2}{|\rho|+|l|\langle t\rangle} A_{l}(t,\eta)\dfrac{|l|^2\langle t\rangle\langle t-\eta/l\rangle^2}{|\eta|+|l|\langle t\rangle}\left\{\langle l,\rho\rangle^{-2}+ \langle l,\eta\rangle^{-2}\right\}.
  \end{aligned}
  \end{align}
  By the symmetry, we only need to consider the case of  $|(-l,\rho)|\leq |(l,\eta)|$. By Lemma \ref{lem:8.3JiaHao-mod1}, we have
  \begin{align}\label{eq:parxphi2-prove3}
     &A_{0}^{(2)}(t,\xi)\lesssim_{\delta}A_{-l}(t,\rho)A_l^{(2)}(t,\eta) \mathrm{e}^{-(\lambda(t)/20)\langle l,\rho\rangle^{1/2}},\nonumber\\
     \Leftrightarrow\quad& A_{0}(t,\xi)(\langle t\rangle+|\xi|) \lesssim_{\delta} A_{-l}(t,\rho)A_{l}(t,\eta)\langle t-\eta/l\rangle \mathrm{e}^{-(\lambda(t)/20)\langle l,\rho\rangle^{1/2}}.
  \end{align}
Notice that
\begin{align*}
   & \langle t\rangle \langle t+\rho/l\rangle^2\langle \rho/l\rangle^2 \gtrsim\left(\langle t\rangle +\langle \rho/l\rangle\right)^3\gtrsim \left(\langle t\rangle +\langle \rho/l\rangle\right)^2(|\rho| +|l|\langle t\rangle)/|l|,
\end{align*}
and $\langle t\rangle\langle t-\eta/l\rangle \gtrsim (|l|\langle t\rangle +\langle \eta\rangle)/|l|$. Then we infer that
\begin{align}\label{eq:parxphi2-prove4}
   &\dfrac{|l|^2\langle t\rangle
    \langle t+\rho/l\rangle^2}{|\rho|+|l|\langle t\rangle}\langle \rho/l\rangle^2 \gtrsim |l|(\langle t\rangle +\langle \rho/l\rangle)^2\gtrsim \langle t\rangle^2,\nonumber\\
   & \dfrac{|l|^2\langle t\rangle \langle t-\eta/l\rangle}{|\eta|+|l|\langle t\rangle}\gtrsim |l|\gtrsim 1,\nonumber\\
   \Rightarrow\quad& \dfrac{\langle t\rangle (\langle t\rangle +|\xi|)|\xi|}{\langle\xi\rangle^2} \lesssim \langle t\rangle^2 \lesssim \dfrac{|l|^2\langle t\rangle \langle t+\rho/l\rangle^2}{|\rho|+|l|\langle t\rangle}\dfrac{|l|^2\langle t\rangle\langle t-\eta/l\rangle}{|\eta|+|l|\langle t\rangle}\times \langle \rho/l\rangle^2.
\end{align}
Then the bounds \eqref{eq:parxphi2-prove1} and \eqref{eq:parxphi2-prove2} follow from \eqref{est:8.3JiaHao-mod2}  with $k=0$, \eqref{eq:parxphi2-prove3} and \eqref{eq:parxphi2-prove4}.
\end{proof}

\begin{lemma}\label{lem:Improved-P-zero}
It holds that
  \begin{align*}
     & \|\partial_v\overline{P}\|_{A0}\lesssim_{\delta} \epsilon_1^2.
  \end{align*}
\end{lemma}

\begin{proof}
 Using \eqref{eq:Pi-zero-vz},  we get by \eqref{eq:ABnorm-product-a} and Lemma \ref{lem:parxphi2} that
 \begin{align*}
     \|\partial_v\overline{P}\|_{A0}&\lesssim \|\partial_v\overline{(\partial_z\phi)^2}\|_{A0}+ \|\overline{a(\partial_v-t\partial_z)P}\|_{A0} \\
     &\lesssim_{\delta} \epsilon_1^2+ \epsilon_1\big(\|(\partial_v-t\partial_z)P_{\neq}\|_{A}+ \|\overline{(\partial_v-t\partial_z)P}\|_{A0}\big)\\
     &\lesssim_{\delta} \epsilon_1^2+ \epsilon_1\big(\|(\partial_v-t\partial_z)P_{\neq}\|_{A}+ \|\partial_v\overline{P}\|_{A0}\big).
 \end{align*}
 This along with \eqref{eq:energy-EP} and \eqref{eq:energy-BP} gives
 \begin{align}\label{eq:p-zeromode}
    &\|\partial_v\overline{P}\|_{A0}\lesssim \epsilon_1^2+ \epsilon_1\big(\mathcal{E}_{P}(t)+\mathcal{B}_{P}(t)\big)^{1/2}.
 \end{align}
 Then the lemma follows from the bootstrap assumption.
 \end{proof}

\if0\subsubsection{Estimate for non-zero mode}

\begin{lemma}\label{lem:P1con-P2}
 If $\epsilon_1>0$ sufficiently small (dependent on $\delta$), then it holds that
  \begin{align*}
  \|(\partial_z,\partial_v-t\partial_z)P_{2,\neq}\|_{A} &\lesssim_{\delta} \epsilon_1^2+ \|(\partial_z,\partial_v-t\partial_z)P_{1,\neq}\|_{A} .
  \end{align*}
\end{lemma}\fi
\begin{proof}[Proof of Proposition \ref{prop:Improved-P}]
  Recall that  by \eqref{eq:P}
\begin{align*}
   \partial_z^2P+(V_1+1)^2(\partial_v-t\partial_z)^2P &+V_2(\partial_v-t\partial_z)P =-\partial_z(a\partial_zP)\\
   &-(V_1+1)(\partial_v-t\partial_z) \big(a(V_1+1)(\partial_v-t\partial_z)P\big)-q.
\end{align*}
Then by Lemma \ref{lem:ABnorm-product}, Lemma \ref{lem:ABnorm-product-a}, and Lemma \ref{lem:elliptic}, we have
\begin{align*}
   \|(\partial_z,\partial_v-t\partial_z)P_{\neq}\|_{A}&\lesssim_{\delta} \|[a\partial_zP]_{\neq}\|_{A}+\|[a(V_1+1)(\partial_v-t\partial_z)P]_{\neq}\|_{A}+\|q_{\neq}\|_{B}\\
   &\lesssim_{\delta} \epsilon_1\big(\|\partial_zP_{\neq}\|_{A}+\|(\partial_v-t\partial_z)P_{\neq}\|_{A} \big) +\epsilon_1\|\partial_v\bar{P}\|_{A0}+\|q_{\neq}\|_{B}.
   %\\&\lesssim_{\delta}\epsilon_1\big(\|\partial_zP_{1,\neq}\|_{A}+ \|(\partial_v-t\partial_z)P_{1,\neq}\|_{A}\big)+\epsilon_1\|\partial_v\overline{P}\|_{A0}\\
   %&\qquad+ \epsilon_1\|(\partial_z,\partial_v-t\partial_z)P_{2,\neq}\|_{A}.
\end{align*}
Then by taking $\epsilon_1>0$ sufficiently small, we get by Lemma \ref{lem:Improved-P-zero}  that
\begin{align*}
  \|(\partial_z,\partial_v-t\partial_z)P_{\neq}\|_{A} &\lesssim_{\delta} \epsilon_1\|\partial_v\bar{P}\|_{A0}+\|q_{\neq}\|_{B}
  \lesssim_{\delta} \epsilon_1^2+\|q_{\neq}\|_{B}.
  %\epsilon_1^2+ \epsilon_1\|(\partial_z,\partial_v-t\partial_z)P_{1,\neq}\|_{A} +\epsilon_1\big(\mathcal{E}_P(t)+\mathcal{B}_{P}(t)\big)^{1/2}.
\end{align*}
%\end{proof}
% \ref{lem:elliptic} and
%\begin{lemma}\label{lem:Improved-P1-nonzero}
%If $\epsilon_1>0$ sufficiently small (dependent on $\delta$), then it holds that
%  \begin{align*}
%     &\|(\partial_z,\partial_v-t\partial_z)P_{1,\neq}\|_{A}\lesssim_{\delta} \epsilon_1^2 +\big(\mathcal{E}_{\Theta}(t)+\mathcal{B}_{\Theta}(t)\big)^{1/2}.
% \end{align*}
%\end{lemma}
%\begin{proof}Notice that $P_{1}$ satisfies \eqref{eq:P1},
Then by Lemma \ref{lem:q0-nonzero} and the fact that $q=2\big(\partial^2_z\phi+Q_1+Q_2+Q_3)$, we have
\begin{align}\label{eq:P1-nonzero-prove1}
   & \|(\partial_z,\partial_v-t\partial_z)P_{\neq}\|_{A} \lesssim_{\delta} \epsilon_1^2+\|\partial_z^2\phi_{\neq}\|_{B}+\sum_{j=1,2,3}\|Q_{j,\neq}\|_{B}\lesssim_{\delta} \epsilon_1^2+\|\partial_z^2\phi_{\neq}\|_{B}.
\end{align}

\if0By the definition of norm $\|\cdot\|_{B}$  and Lemma , we have
\begin{align}\label{eq:P1-nonzero-prove2}
   & \sum_{j=1,2,3}\|Q_{j,\neq}\|_{B}\lesssim_{\delta} \epsilon_1^2.
\end{align}

If $|t-\xi/k|\leq |\xi/(10k)|$, we have $t\approx \xi/k$ and then
\begin{align*}
   & \big(\langle t\rangle +|(k,\xi)|\big)\big(|\xi/k|+\langle t\rangle\big) \lesssim     \big(\langle \xi/k\rangle +|(k,\xi)|\big)\langle \xi/k\rangle \lesssim |k|\langle t-\xi/k\rangle^2\langle \xi/k\rangle^2.
\end{align*}
While if $|t-\xi/k|\geq |\xi/(10k)|$, we have $\langle t-\xi/k\rangle\approx \langle t\rangle +|\xi/k|$ and then
\begin{align*}
      & \big(\langle t\rangle +|(k,\xi)|\big)\big(|\xi/k|+\langle t\rangle\big) \lesssim   \langle k, \xi\rangle\big(\langle t\rangle +|\xi/k|\big)\big(|\xi/k|+\langle t\rangle\big) \lesssim |k|\langle t-\xi/k\rangle^2\langle \xi/k\rangle^2.
\end{align*}
Thus, we conclude
\begin{align*}
   & \big(\langle t\rangle +|(k,\xi)|\big)\big(|\xi/k|+\langle t\rangle\big) \lesssim |k|\langle t-\xi/k\rangle^2\langle \xi/k\rangle^2,
\end{align*}
which gives\fi
By \eqref{eq1} with $(k,\xi)=(l,\eta)$, we have
\begin{align}\label{eq:P1-nonzero-prove3}
   & \dfrac{\langle t\rangle \big( \langle t\rangle+|(k,\xi)|\big)}{|k|\langle \xi/k\rangle^2\langle t-\xi/k\rangle^2}\lesssim \dfrac{\langle t\rangle }{|\xi/k|+\langle t\rangle}= \dfrac{|k|\langle t\rangle }{|\xi|+|k|\langle t\rangle}.
\end{align}
By the definition of norm $\|\cdot\|_{B}$ again and $(\partial_z^2+(\partial_v-t\partial_z)^2)\phi=\Theta$,  we get by \eqref{eq:P1-nonzero-prove3} that
\begin{align*}
  \|\partial_z^2\phi_{\neq}\|_{B}^2 =& \sum_{k\in \mathbb{Z}\setminus \{0\}}\int_{\mathbb{R}}A_k(t,\xi)^2\dfrac{k^2\langle t\rangle^{2}(\langle t\rangle^2+|(k,\xi)|^2)}{\langle\xi/k\rangle^4}|\widetilde{\phi}(t,k,\xi)|^2\mathrm{d}\xi \\
   &+\int_{1}^{t}\sum_{k\in \mathbb{Z}\setminus \{0\}}\int_{\mathbb{R}}|\dot{A}_k(s,\xi)|A_k(s,\xi) \dfrac{k^2\langle s\rangle^{2}(\langle s\rangle^2+|(k,\xi)|^2)}{\langle \xi/k\rangle^4}|\widetilde{\phi}(s,k,\xi)|^2\mathrm{d}\xi\mathrm{d}s\\
    =& \sum_{k\in \mathbb{Z}\setminus \{0\}}\int_{\mathbb{R}}A_k(t,\xi)^2\dfrac{\langle t\rangle^{2}(\langle t\rangle^2+|(k,\xi)|^2)}{k^2\langle\xi/k\rangle^4\langle t-\xi/k\rangle^4} |\widetilde{\Theta}(t,k,\xi)|^2\mathrm{d}\xi \\
   &+\int_{1}^{t}\sum_{k\in \mathbb{Z}\setminus \{0\}}\int_{\mathbb{R}}|\dot{A}_k(s,\xi)|A_k(s,\xi) \dfrac{\langle s\rangle^{2}(\langle s\rangle^2+|(k,\xi)|^2)}{k^2\langle \xi/k\rangle^4\langle s-\xi/k\rangle^4}|\widetilde{\Theta}(s,k,\xi)|^2\mathrm{d}\xi\mathrm{d}s\\
   \lesssim& \sum_{k\in \mathbb{Z}\setminus \{0\}}\int_{\mathbb{R}}A_k(t,\xi)^2\dfrac{|k|^2\langle t\rangle^2 }{|\xi|^2+|k|^2\langle t\rangle^2}|\widetilde{\Theta}(t,k,\xi)|^2\mathrm{d}\xi \\
   &+\int_{1}^{t}\sum_{k\in \mathbb{Z}\setminus \{0\}}\int_{\mathbb{R}}|\dot{A}_k(s,\xi)|A_k(s,\xi) \dfrac{|k|^2\langle s\rangle^2 }{|\xi|^2+|k|^2\langle s\rangle^2}|\widetilde{\Theta}(s,k,\xi)|^2\mathrm{d}\xi\mathrm{d}s\\
   =&\mathcal{E}_{\Theta}(t)+\mathcal{B}_{\Theta}(t).
\end{align*}
This along with \eqref{eq:P1-nonzero-prove1} and Lemma \ref{lem:Improved-P-zero} gives
\if0\begin{align*}
   &\|(\partial_z,\partial_v-t\partial_z)P_{\neq}\|_{A}\lesssim_{\delta} \epsilon_1^2 +\big(\mathcal{E}_{\Theta}(t)+\mathcal{B}_{\Theta}(t)\big)^{1/2}.
\end{align*}
This finishes the proof of the lemma.
\subsubsection{}. It follows from
Then by , we have Lemma \ref{lem:P1con-P2} and Lemma \ref{lem:Improved-P1-nonzero} that\fi
\begin{align*}
   &\|\partial_v\overline{P}\|_{A0}+ \|(\partial_z,\partial_v-t\partial_z)P_{\neq}\|_{A}
   \lesssim_{\delta} \epsilon_1^2 +\big(\mathcal{E}_{\Theta}(t)+\mathcal{B}_{\Theta}(t)\big)^{1/2}.
\end{align*}
Thanks to $\|\partial_v\overline{P}\|_{A0}^2+ \|(\partial_z,\partial_v-t\partial_z)P_{\neq}\|_{A}^2=\mathcal{E}_{P}(t)+\mathcal{ B}_{P}(t)$, we deduce by taking $\epsilon_1>0$ sufficiently small (depends on $\delta$) that
\begin{align*}
   & \mathcal{E}_P(t)+\mathcal{B}_{P}(t) \lesssim_{\delta}\epsilon_1^4  +\mathcal{E}_{\Theta}(t)+\mathcal{B}_{\Theta}(t).
\end{align*}

This finishes the proof of Proposition \ref{prop:Improved-P}.\end{proof}

 \section{Improved control of the vorticity}

 In this section, we prove an improved control for the vorticity part $f$ under the bootstrap assumptions in Proposition \ref{prop:Bootstrap}.

 \begin{proposition}\label{prop:Improved-f}
   With the definitions and assumptions in Proposition \ref{prop:Bootstrap}, we have
   \begin{align}\nonumber%\label{est:Improved-f}
      & \mathcal{E}_f(t)+\mathcal{B}_f(t)\lesssim_{\delta} \epsilon_1^3\quad \text{for any}\ t\in[1,T].
   \end{align}
 \end{proposition}

It is easy to find that
 \begin{align*}
    \dfrac{\mathrm{d}}{\mathrm{d}t}\mathcal{E}_f(t)=&\sum_{k\in\mathbb{Z}}
    \int_{\mathbb{R}}2\dot{A}_k(t,\xi)A_k(t,\xi)|\widetilde{f}(t,k,\xi)|^2\mathrm{d}\xi\\
    &+2\mathbf{Re}\sum_{k\in\mathbb{Z}}\int_{\mathbb{R}}A_k(t,\xi)^2\partial_t\widetilde{f} (t,k,\xi)\overline{\widetilde{f}(t,k,\xi)}\mathrm{d}\xi.
 \end{align*}
 Then we get by $\partial_tA_{k}\leq 0$ that for any $t\in[1,T]$, we have
 \begin{align}
    &\mathcal{E}_f(t)+2\mathcal{B}_{f}(t) =\mathcal{E}_{f}(1) +2\mathbf{Re}\int_{1}^{t}\sum_{k\in\mathbb{Z}}\int_{\mathbb{R}}A_k(s,\xi)^2 \partial_s\widetilde{f} (s,k,\xi)\overline{\widetilde{f}(s,k,\xi)}\mathrm{d}\xi\mathrm{d}s.
 \end{align}
Thus, it suffices to prove that
 \begin{align}\label{eq:nonlinear-vor}
   & \left|2\mathbf{Re}\int_{1}^{t}\sum_{k\in\mathbb{Z}}\int_{\mathbb{R}}A_k(s,\xi)^2\partial_s\widetilde{f} (s,k,\xi)\overline{\widetilde{f}(s,k,\xi)}\mathrm{d}\xi\mathrm{d}s\right| \lesssim_{\delta} \epsilon_1^3.
 \end{align}
 Recall that $ \partial_tf+V_3\partial_vf=(1+V_1)\{\mathbb{P}_{\neq 0}\phi,f\}-(1+V_1)\{P,a\}$, which gives
 \begin{align*}
    & \partial_sf=\mathcal{N}_1+\mathcal{N}_2+\mathcal{N}_3+\mathcal{N}_4+\mathcal{N}_5,
 \end{align*}
 where
 \begin{align}\label{eq:N1-N5}
 \begin{aligned}
    & \mathcal{N}_1=(V_1+1)\partial_v\mathbb{P}_{\neq 0}\phi
    \partial_zf,\quad \mathcal{N}_2=-(1+V_1)\partial_z\mathbb{P}_{\neq 0}\phi\partial_vf,\quad \mathcal{N}_3=-V_3\partial_vf,\\
    &\mathcal{N}_4=(V_1+1)\partial_vP\partial_za,\quad \mathcal{N}_5=-(1+V_1)\partial_zP\partial_va.
    \end{aligned}
 \end{align}
 Then the bound \eqref{eq:nonlinear-vor} follows from the following Lemma \ref{lem:N1}, Lemma \ref{lem:N2}, Lemma \ref{lem:N3}, Lemma \ref{lem:N4}
 and Lemma \ref{lem:N5}.\smallskip

 \subsection{Nonlinear estimate for $\mathcal{N}_1-\mathcal{N}_3$}

 The following two lemmas have been  proved in \cite{IJ}.
\begin{lemma}\label{lem:N1}
 It holds that for any $t\in[1,T]$, we have
  \begin{align*}
     & \left|2\mathbf{Re}\int_{1}^{t}\sum_{k\in\mathbb{Z}}\int_{\mathbb{R}}A_k(s,\xi)^2 \widetilde{\mathcal{N}_1}(s,k,\xi)\overline{\widetilde{f}(s,k,\xi)}\mathrm{d}\xi\mathrm{d}s\right| \lesssim_{\delta} \epsilon_1^3.
  \end{align*}
\end{lemma}

\begin{lemma}\label{lem:N2}
It holds that for any $t\in[1,T]$, we have
  \begin{align*}
     &\left|2\mathbf{Re}\int_{1}^{t}\sum_{k\in\mathbb{Z}}\int_{\mathbb{R}}A_k(s,\xi)^2 \widetilde{\mathcal{N}_2}(s,k,\xi)\overline{\widetilde{f}(s,k,\xi)}\mathrm{d}\xi\mathrm{d}s\right| \lesssim_{\delta} \epsilon_1^3.
  \end{align*}
\end{lemma}

For $\mathcal{N}_3$,  we need to give a proof  since $V_3$ has no compact support.

\begin{lemma} \label{lem:N3}
 It holds that for any $t\in[1,T]$, we have
  \begin{align}\label{eq:N3bound}
     &\left|2\mathbf{Re}\int_{1}^{t}\sum_{k\in\mathbb{Z}}\int_{\mathbb{R}}A_k(s,\xi)^2 \widetilde{\mathcal{N}_3}(s,k,\xi)\overline{\widetilde{f}(s,k,\xi)}\mathrm{d}\xi\mathrm{d}s\right| \lesssim_{\delta} \epsilon_1^3.
  \end{align}
\end{lemma}

\begin{proof}
We write
\begin{align*}
   & \left|2\mathbf{Re}\int_{1}^{t}\sum_{k\in\mathbb{Z}}\int_{\mathbb{R}}A_k(s,\xi)^2 \widetilde{N_3}(s,k,\xi) \overline{\widetilde{f}(s,k,\xi)}\mathrm{d}\xi\mathrm{d}s\right|\\
   &=\left|2\mathbf{Re}\sum_{k\in\mathbb{Z}}\int_{1}^{t}\int_{\mathbb{R}^2}A_k(s,\xi)^2 \widetilde{V_3}(s,\xi-\eta)i\eta\widetilde{f}(s,k,\eta)\overline{\widetilde{f}(s,k,\xi)}\mathrm{d} \xi\mathrm{d}\eta\mathrm{d}s\right|\\
   &=\left| \sum_{k\in\mathbb{Z}}\int_{1}^{t}\int_{\mathbb{R}^2}\big[\eta
    A_k(s,\xi)^2-\xi A_k(s,\eta)^2\big] \widetilde{V_3}(s,\xi-\eta)\widetilde{f}(s,k,\eta)\overline{\widetilde{f}(s,k,\xi)}\mathrm{d} \xi\mathrm{d}\eta\mathrm{d}s\right|.
\end{align*}

For $i\in\{0,1,2,3\}$, we define
\begin{align}\label{eq:set-Sigma}
  \Sigma_{i}=\{\big((k,\xi),(l,\eta)\big)\in R_i:k=l\},
\end{align}
where $R_i$ are defined in \eqref{eq:R0}-\eqref{eq:R3}. Then we denote that for $i\in\{0,1,2,3\}$
\begin{align*}
   \mathcal{W}_i=\int_{1}^{t}\sum_{k\in\mathbb{Z}}\int_{\mathbb{R}^2} & \mathbf{1}_{|\rho|\geq 1} \mathbf{1}_{\Sigma_i}\big((k,\xi),(k,\eta)\big)|\eta A_k(s,\xi)^2-\xi A_k(s,\eta)^2| | \widetilde{V_3}(s,\xi-\eta)|\\
   &\times|\widetilde{f}(s,k,\eta)||\widetilde{f}(s,k,\xi)|\mathrm{d} \xi\mathrm{d}\eta\mathrm{d}s,
\end{align*}
and
\begin{align}\label{eq:W4}
\begin{aligned}
   \mathcal{W}_4=\int_{1}^{t}\sum_{k\in\mathbb{Z}}\int_{\mathbb{R}^2} & \mathbf{1}_{|\rho|\leq 1}|\eta A_k(s,\xi)^2-\xi A_k(s,\eta)^2| | \widetilde{V_3}(s,\xi-\eta)|\\
   &\times|\widetilde{f}(s,k,\eta)||\widetilde{f}(s,k,\xi)|\mathrm{d} \xi\mathrm{d}\eta\mathrm{d}s.
\end{aligned}
\end{align}

For $\mathcal{W}_i$, $i\in\{0,1\}$, we get by (i) of Lemma \ref{lem:8.6JiaHao} and \eqref{eq:boot-V3} that($\rho=\xi-\eta$)
\begin{align*}
  \mathcal{W}_i\lesssim_{\delta} &\int_{1}^{t}\sum_{k\in\mathbb{Z}}\int_{\mathbb{R}^2} \mathbf{1}_{|\rho|\geq 1}\big[\langle \rho\rangle\langle s\rangle+\langle \rho\rangle^{1/4}\langle s\rangle^{7/4}\big]A_{NR}(s,\rho)\mathrm{e}^{-(\delta_0/200)\langle \rho\rangle^{1/2}}|\widetilde{V_3}(s,\rho)|\\
  &\times\sqrt{|(A_k\dot{A}_k)(s,\eta)|}|\widetilde{f}(s,k,\eta)| \sqrt{|(A_k\dot{A}_k)(s,\xi)|}|\widetilde{f}(s,k,\xi)|\mathrm{d}\xi\mathrm{d}\eta\mathrm{d}s\\
  \lesssim_{\delta}&\left\|\big[\langle s\rangle+\langle \rho\rangle^{-3/4}\langle s\rangle^{7/4}\big]|\rho|A_{NR}(s,\rho)\mathrm{e}^{-(\delta_0/300)\langle \rho\rangle^{1/2}}\widetilde{V_3}(s,\rho)\right\|_{L^\infty_s L^2_{\rho}}\\
  &\times\left\|\sqrt{|(A_k\dot{A}_k)(s,\eta)|}\widetilde{f}(s,k,\eta)\right\|_{L^2_sL^2_{k,\eta}} \cdot\left\|\sqrt{|(A_k\dot{A}_k)(s,\xi)|}\widetilde{f}(s,k,\xi)\right\|_{L^2_sL^2_{k,\xi}}\\
  \lesssim_{\delta}&\epsilon_1^3.
\end{align*}
For $\mathcal{W}_2$,, we get by (ii) of Lemma \ref{lem:8.6JiaHao}  and \eqref{eq:boot-V3}  that
\begin{align*}
  {\mathcal{W}_2}\lesssim_{\delta} &\int_{1}^{t}\sum_{k\in\mathbb{Z}}\int_{\mathbb{R}^2} \mathbf{1}_{|\rho|\geq 1}\big[\langle \rho\rangle\langle s\rangle+\langle \rho\rangle^{1/4}\langle s\rangle^{7/4}\big]\sqrt{|(A_{NR}\dot{A}_{NR})(s,\rho)|}|\widetilde{V_3}(s,\rho)|\\
  &\times A_k(s,\eta)\mathrm{e}^{-(\delta_0/200)\langle k,\eta\rangle^{1/2}}|\widetilde{f}(s,k,\eta)| \sqrt{|(A_k\dot{A}_k)(s,\xi)|}|\widetilde{f}(s,k,\xi)|\mathrm{d}\xi\mathrm{d}\eta\mathrm{d}s\\
  \lesssim_{\delta}&\left\|\big[\langle s\rangle+\langle \rho\rangle^{-3/4}\langle s\rangle^{7/4}\big]|\rho|\sqrt{|(A_{NR}\dot{A}_{NR})(s,\rho)|} \widetilde{V_3}(s,\rho)\right\|_{L^2_s L^2_{\rho}}\\
  &\times\left\|A_k(s,\eta)\mathrm{e}^{-(\delta_0/300)\langle k,\eta\rangle^{1/2}}\widetilde{f}(s,k,\eta)\right\|_{L^\infty_sL^2_{k,\eta}} \cdot\left\|\sqrt{|(A_k\dot{A}_k)(s,\xi)|}\widetilde{f}(s,k,\xi)\right\|_{L^2_sL^2_{k,\xi}}\\
  \lesssim_{\delta}&\epsilon_1^3.
\end{align*}
The case of $i=3$ is the same as the case of $i=2$ by the symmetry.

For $\mathcal{W}_4$, we get by Lemma \ref{lem:rholeq1}, $\mathbf{1}_{|\rho|\leq1} A_{NR}(s,\rho)\gtrsim_{\delta} {\mathbf{1}_{|\rho|\leq1}}$ and \eqref{eq:boot-V3}   that
\begin{align*}
  \mathcal{W}_4\lesssim_{\delta} &\int_{1}^{t}\sum_{k\in\mathbb{Z}}\int_{\mathbb{R}^2} \mathbf{1}_{|\rho|\leq 1}\langle s\rangle^{7/4}|\rho|\widetilde{V}_3(s,\rho)\sqrt{|(A_k\dot{A}_k)(s,\eta)|} |\widetilde{f}(s,k,\eta)| \\
  &\times \sqrt{|(A_k\dot{A}_k)(s,\xi)|}|\widetilde{f}(s,k,\xi)|\mathrm{d}\xi\mathrm{d}\eta\mathrm{d}s\\
  \lesssim_{\delta}&\left\|\big[\langle s\rangle+\langle \rho\rangle^{-3/4}\langle s\rangle^{7/4}\big]|\rho|A_{NR}(s,\rho)\mathrm{e}^{-(\delta_0/300)\langle \rho\rangle^{1/2}}\widetilde{V_3}(s,\rho)\right\|_{L^\infty_s L^2_{\rho}}\\
  &\times\left\|\sqrt{|(A_k\dot{A}_k)(s,\eta)|}\widetilde{f}(s,k,\eta)\right\|_{L^2_sL^2_{k,\eta}} \cdot\left\|\sqrt{|(A_k\dot{A}_k)(s,\xi)|}\widetilde{f}(s,k,\xi)\right\|_{L^2_sL^2_{k,\xi}}\\
  \lesssim_{\delta}&\epsilon_1^3,
\end{align*}

 Thus, the desired bound \eqref{eq:N3bound} follows.
  \end{proof}

\subsection{Nonlinear estimate for $\mathcal{N}_4$}

\begin{lemma}\label{lem:N4}
It holds that for any $t\in[1,T]$, we have
  \begin{align}\label{est:N4bound}
     &\left|2\mathbf{Re}\int_{1}^{t}\sum_{k\in\mathbb{Z}}\int_{\mathbb{R}}A_k(s,\xi)^2\widetilde{\mathcal{N}_4}(s,k,\xi)\overline{\widetilde{f}(s,k,\xi)}
     \mathrm{d}\xi\mathrm{d}s\right| \lesssim_{\delta} \epsilon_1^3 .
  \end{align}
\end{lemma}

 Let
  \begin{align}\label{eq:H5H6}
     & H_5=\partial_vP,\quad H_6=(V_1+1)\partial_vP.
  \end{align}
  We first prove the following lemma.

\begin{lemma}\label{lem:par-v-P}
  For any $t\in[1,T]$ and $j\in\{5,6\}$, we have
  \begin{align}\label{est:par-v-P}
    \begin{aligned}
    &\sum_{k\in \mathbb{Z}\setminus\{0\}} \int_{\mathbb{R}}A_{k}(t,\xi)^2 \dfrac{\langle t\rangle^2 \big(\langle t\rangle^2 +|(k,\xi)|^2\big)\langle t-\xi/k\rangle^4}{\langle\xi/k\rangle^4\langle \xi\rangle^2/k^2}|\widetilde{H_j}(t,k,\xi)|^2\mathrm{d}\xi\lesssim_{\delta}\epsilon_{1}^2,\\
    &\int_{1}^{T}\sum_{k\in \mathbb{Z}\setminus\{0\}}\int_{\mathbb{R}}|(\dot{A}_{k}A_{k})(s,\xi)| \dfrac{\langle s\rangle^2 \big(\langle s\rangle^2 +|(k,\xi)|^2\big)\langle s-\xi/k\rangle^4}{\langle\xi/k\rangle^4\langle \xi\rangle^2/k^2}|\widetilde{H_j}(s,k,\xi)|^2\mathrm{d}\xi\mathrm{d}s\lesssim_{\delta}\epsilon_1^2,
    \end{aligned}
  \end{align}
and
\begin{align}\label{est:par-v-P-zero}
   \begin{aligned}
    &\int_{\mathbb{R}}A_{0}(t,\xi)^2\dfrac{\langle t\rangle^2 \big(\langle t\rangle^2 +|\xi|^2\big)^2}{\langle \xi\rangle^4}|\widetilde{H_j}(t,0,\xi)|^2\mathrm{d}\xi\lesssim_{\delta}\epsilon_{1}^2,\\
    &\int_{1}^{T}\int_{\mathbb{R}}|(\dot{A}_{0} A_{0})(s,\xi)| \dfrac{\langle s\rangle^2 \big(\langle
     s\rangle^2 +|\xi|^2\big)^2}{\langle \xi\rangle^4}|\widetilde{H_j}(s,0,\xi)|^2\mathrm{d}\xi\mathrm{d}s\lesssim_{\delta} \epsilon_1^2.
    \end{aligned}
\end{align}
\end{lemma}
\begin{proof}
  The bounds on $H_5$ follow directly from the bootstrap assumption on $\mathcal{E}_{P}(t)$ and $\mathcal{B}_{P}(t)$.

 Notice that $H_6=H_5+V_1H_5$. By Lemma \ref{lem:8.1JiaHao} (ii), \eqref{eq:boot-V1},  it suffices to prove the following multiplier bounds for $k\neq 0$,
  \begin{align}\label{est:par-v-P-prove-1}
  \begin{aligned}
     A_{k}(t,\xi)&\dfrac{\langle t\rangle(\langle t\rangle +|(k,\xi)|)\langle t-\xi/k\rangle^2}{\langle\xi/k\rangle^2\langle \xi\rangle/|k|}\lesssim_{\delta}A_{R}(t,\xi-\eta) A_{k}(t,\eta) \\& \times\dfrac{\langle t\rangle(\langle t\rangle +|(k,\eta)|)\langle t-\eta/k\rangle^2}{\langle\eta/k\rangle^2\langle \eta\rangle/|k|}\cdot\big\{\langle \xi-\eta\rangle^{-2}+\langle k,\eta\rangle^{-2}\big\}.
     \end{aligned}
  \end{align}
  and
   \begin{align}
    |(\dot{A}_{k} A_{k})(t,\xi)|^{\f12}&\dfrac{\langle t\rangle(\langle t\rangle +|(k,\xi)|)\langle t-\xi/k\rangle^2}{\langle\xi/k\rangle^2\langle \xi\rangle/|k|}\lesssim_{\delta} \left[|(\dot{A}_{R} /A_{R})(t,\xi-\eta)|^{\f12}+|(\dot{A}_{k}/ A_{k})(t,\eta)|^{\f12}\right]\nonumber\\
    & \times A_{R}(t,\xi-\eta) A_{k}(t,\eta)\dfrac{\langle t\rangle(\langle t\rangle +|(k,\eta)|)\langle t-\eta/k\rangle^2}{\langle\eta/k\rangle^2\langle \eta\rangle/|k|}\cdot\big\{\langle \xi-\eta\rangle^{-2}+\langle k,\eta\rangle^{-2}\big\},\label{est:par-v-P-prove-2}
  \end{align}
 and for zero mode (i.e., $k=0$)
  \begin{align}\label{est:par-v-P-prove-3}
  \begin{aligned}
     A_{0}(t,\xi)&\dfrac{\langle t\rangle(\langle t\rangle^2 +|\xi|^2)}{\langle \xi\rangle^2}\lesssim_{\delta}A_{R}(t,\xi-\eta)\cdot A_{0}(t,\eta)\dfrac{\langle t\rangle(\langle t\rangle
     ^2 +|\eta|^2)}{\langle \eta\rangle^2}\cdot\big\{\langle \xi-\eta\rangle^{-2}+\langle \eta\rangle^{-2}\big\}.
     \end{aligned}
  \end{align}
  and
   \begin{align}\label{est:par-v-P-prove-4}
  \begin{aligned}
    |(\dot{A}_{0} A_{0})(t,\xi)|^{1/2}&\dfrac{\langle t\rangle(t^2 +|\xi|^2)}{\langle \xi\rangle^2}\lesssim_{\delta} \left[|(\dot{A}_{R} / A_{R})(t,\xi-\eta)|^{1/2}+|(\dot{A}_{0}/ A_{0})(t,\eta)|^{1/2}\right]\\& \times A_{R}(t,\xi-\eta)\cdot A_{0}(t,\eta)\dfrac{\langle t\rangle(\langle t\rangle
    ^2 +|\eta|^2)}{\langle \eta\rangle^2}\cdot\big\{\langle \xi-\eta\rangle^{-2}+\langle \eta\rangle^{-2}\big\},
     \end{aligned}
  \end{align}

 By considering the cases $|\xi-\eta|\leq 10|(k,\eta)|$ and $|\xi-\eta|\geq 10|(k,\eta)|$, it is easy to see that ($k\neq 0$)
  \begin{align}\label{est:par-v-P-prove-5}
     & \dfrac{\langle t\rangle\big(\langle t\rangle +|(k,\xi)|\big)\langle t-\xi/k\rangle^2}{\langle\xi/k\rangle^2\langle \xi\rangle/|k|} \lesssim_{\delta} \dfrac{\langle t\rangle\big(\langle t\rangle +|(k,\eta)|\big)\langle t-\eta/k\rangle^2}{\langle\eta/k\rangle^2\langle \eta\rangle/|k|}\mathrm{e}^{\delta\min(\langle \xi-\eta\rangle,\langle k,\eta\rangle)^{1/2}},
  \end{align}
 and by considering the cases  $|\xi-\eta|\leq 10|\eta|$ and $|\xi-\eta|\geq 10|\eta|$, we have
  \begin{align}\label{est:par-v-P-prove-6}
   \dfrac{\langle t\rangle(\langle t\rangle^2 +|\xi|^2)}{\langle \xi\rangle^2} &  \lesssim_{\delta}\dfrac{\langle t\rangle(\langle t\rangle^2 +|\eta|^2)}{\langle \eta\rangle^2}\mathrm{e}^{\delta\min(\langle \xi-\eta\rangle,\langle \eta\rangle)^{1/2}}.
  \end{align}
 % If $10|k,\eta|\leq|\xi-\eta|$, we consider $|\xi-t k|\leq 10|\xi,\eta|$ and $|\xi-t k|\geq 10|\xi,\eta|$.
 Then the bound \eqref{est:par-v-P-prove-1} follows from \eqref{est:8.3JiaHao-1} and \eqref{est:par-v-P-prove-5}; the bounds \eqref{est:par-v-P-prove-2} follows from \eqref{est:8.3JiaHao-1}, \eqref{est:8.3JiaHao-2} and \eqref{est:par-v-P-prove-5}; the bound \eqref{est:par-v-P-prove-3} follows from \eqref{est:8.3JiaHao-1} and \eqref{est:par-v-P-prove-6}; the bound \eqref{est:par-v-P-prove-4} follows from \eqref{est:8.3JiaHao-1}, \eqref{est:8.3JiaHao-2} and \eqref{est:par-v-P-prove-6}.
\end{proof}

We now turn to the proof of \eqref{est:N4bound}.

\begin{proof}
We write
\begin{align*}
   & \left|2\mathbf{Re}\int_{1}^{t}\sum_{k\in\mathbb{Z}}\int_{\mathbb{R}}A_k(s,\xi)^2\widetilde{N_4} (s,k,\xi)\overline{\widetilde{f}(s,k,\xi)}\mathrm{d}\xi\mathrm{d}s\right|\\
   &=2\left|\mathbf{Re}\sum_{k,l\in\mathbb{Z}}\int_{1}^{t}\int_{\mathbb{R}^2}A_k(s,\xi)^2\widetilde{H_6}(s,k-l,\xi-\eta)il\widetilde{a}(s,l,\eta)
   \overline{\widetilde{f}(s,k,\xi)}\mathrm{d} \xi\mathrm{d}\eta\mathrm{d}s\right|.
\end{align*}
We introduce the sets
\begin{align}
   R_0=\Big\{&\big((k,\xi),(l,\eta)\big)\in(\mathbb{Z}\times \mathbb{R})^2: \label{eq:R0}\\
   &\min(\langle k,\xi\rangle,\langle l,\eta\rangle,\langle k-l,\xi-\eta\rangle)\geq \dfrac{\langle k,\xi\rangle+\langle l,\eta\rangle+\langle k-l,\xi-\eta\rangle}{20}\Big\},\nonumber\\
   R_1=\Big\{&\big((k,\xi),(l,\eta)\big)\in(\mathbb{Z}\times \mathbb{R})^2: \langle k-l,\xi-\eta\rangle\leq \dfrac{\langle k,\xi\rangle+\langle l,\eta\rangle+\langle k-l,\xi-\eta\rangle}{10}\Big\},\label{eq:R1}\\
   R_2=\Big\{&\big((k,\xi),(l,\eta)\big)\in(\mathbb{Z}\times \mathbb{R})^2: \langle l,\eta\rangle\leq \dfrac{\langle k,\xi\rangle+\langle l,\eta\rangle+\langle k-l,\xi-\eta\rangle}{10}\Big\},\label{eq:R2}\\
   R_3=\Big\{&\big((k,\xi),(l,\eta)\big)\in(\mathbb{Z}\times \mathbb{R})^2: \langle k,\xi\rangle\leq \dfrac{\langle k,\xi\rangle+\langle l,\eta\rangle+\langle k-l,\xi-\eta\rangle}{10}\Big\}.\label{eq:R3}
\end{align}
Then we denote that for {$j=0,1,2,3$},
\begin{align*}
   \mathcal{U}^{(1)}_j=\int_{1}^{t}\sum_{k,l\in\mathbb{Z}}\int_{\mathbb{R}^2} & \mathbf{1}_{R_j}\big((k,\xi),(l,\eta)\big)|l| A_k(s,\xi)^2 | \widetilde{H_6}(s,k-l,\xi-\eta)|\\
   &\times\mathbf{1}_{k-l\neq 0}
   \cdot|\widetilde{a}(s,l,\eta)|\cdot|\widetilde{f}(s,k,\xi)|\mathrm{d} \xi\mathrm{d}\eta\mathrm{d}s,
\end{align*}
and
\begin{align*}
   \mathcal{U}^{(1)}_4=\int_{1}^{t}\sum_{k\in\mathbb{Z}}\int_{\mathbb{R}^2} & |k| A_k(s,\xi)^2 | \widetilde{H_6}(s,0,\xi-\eta)|\cdot |\widetilde{a}(s,k,\eta)|\cdot|\widetilde{f}(s,k,\xi)|\mathrm{d} \xi\mathrm{d}\eta\mathrm{d}s.
\end{align*}

Let $(\sigma,\rho)=(k-l,\xi-\eta)$. For $j=0$,  by Lemma \ref{lem:8.4JiaHao} (i) and \eqref{est:par-t-A*-1}, we have
\begin{align}\label{est:U0(1)}
  \mathcal{U}_0^{(1)} \lesssim_{\delta}&\int_{0}^{t}\sum_{k,l\in\mathbb{Z}}\int_{\mathbb{R}^2} \sqrt{|(A_k\dot{A}_{k})(s,\xi)|}|\widetilde{f}(s,k,\xi)| \sqrt{|(A_l\dot{A}_{l})(s,\eta)|}|\widetilde{a}(s,l,\eta)|\dfrac{\langle s\rangle}{|\rho/\sigma|+\langle s\rangle}\nonumber\\
  &\times \mathbf{1}_{\sigma\neq 0}\dfrac{\langle s-\rho/\sigma\rangle^2}{\langle\rho\rangle/\sigma^2}A_{\sigma}(s,\rho)| \widetilde{H_6}(s,\sigma,\rho)|\mathrm{e}^{-(\delta_0/200)\langle \sigma,\rho\rangle^{1/2}}\mathrm{d}\xi\mathrm{d}\eta\mathrm{d}s\nonumber\\
  \lesssim_{\delta}&\left\| \sqrt{|(A_k\dot{A}_{k})(s,\xi)|}\widetilde{f}(s,k,\xi)\right\|_{L_s^2L^2_{k,\xi}} \left\|\sqrt{|(A_l^*\dot{A}^*_{l})(s,\eta)|}\widetilde{a}(s,l,\eta)\right\|_{L_s^2L^2_{l,\eta}}\\
  &\times \left\| \mathbf{1}_{\sigma\neq 0}A_{\sigma}(s,\rho)\dfrac{\langle s\rangle}{|\rho/\sigma|+\langle s\rangle}\dfrac{\langle s-\rho/\sigma\rangle^2}{\langle\rho\rangle/\sigma^2} \widetilde{H_6}(s,\sigma,\rho)\mathrm{e}^{-(\delta_0/300)\langle \sigma,\rho\rangle^{1/2}}\right\|_{L^{\infty}_{s}L^2_{\sigma,\rho}}.\nonumber
\end{align}
Thanks to
\begin{align*}
   & \langle \rho/\sigma\rangle^2|\sigma|= \langle\rho/\sigma\rangle|(\sigma,\rho)|\leq \big(|\rho/\sigma|+\langle s\rangle\big)\big(\langle s\rangle +|(\sigma,\rho)|\big),
\end{align*}
we have
\begin{align}\label{est:N4bound-prove-1}
  \dfrac{\langle s\rangle}{|\rho/\sigma|+\langle s\rangle}\dfrac{\langle s-\rho/\sigma\rangle^2}{\langle\rho\rangle/\sigma^2} &\lesssim \dfrac{\langle s\rangle\big(\langle s\rangle+|(\sigma,\rho)|\big) \langle s-\rho/\sigma\rangle^2}{\langle\rho/\sigma\rangle^2\langle \rho\rangle/|\sigma|}.
\end{align}
which along with Lemma \ref{lem:par-v-P}  gives
\begin{align*}
  \left\| \mathbf{1}_{\sigma\neq 0}A_{\sigma}(s,\rho)\dfrac{\langle s\rangle}{|\rho/\sigma|+\langle s\rangle}\dfrac{\langle s-\rho/\sigma\rangle^2}{\langle\rho\rangle/\sigma^2} \widetilde{H_6}(s,\sigma,\rho)\mathrm{e}^{-(\delta_0/300)\langle \sigma,\rho\rangle^{1/2}}\right\|_{L^{\infty}_{s}L^2_{\sigma,\rho}} & \lesssim_{\delta}\epsilon_1,
\end{align*}
from which and \eqref{est:U0(1)}, we infer that
\begin{align*}
   \mathcal{U}_0^{(1)} &\lesssim_{\delta}  \epsilon_1^3.
\end{align*}

For $j=1$, we get by Lemma \ref{lem:A*R0R1R2-1-R1} and  Lemma \ref{lem:par-v-P} that
\begin{align*}
  \mathcal{U}_1^{(1)} \lesssim_{\delta}&\int_{0}^{t}\sum_{k,l\in\mathbb{Z}}\int_{\mathbb{R}^2} \sqrt{|(A_k\dot{A}_{k})(s,\xi)|}|\widetilde{f}(s,k,\xi)| \sqrt{|(A^*_l\dot{A}^*_{l})(s,\eta)|}|\widetilde{a}(s,l,\eta)|\\
  &\times \mathbf{1}_{\sigma\neq 0} \langle s\rangle^3 A_{\sigma}(s,\rho)| \widetilde{H_6}(s,\sigma,\rho)|\mathrm{e}^{-(\delta_0/200)\langle \sigma,\rho\rangle^{1/2}}\mathrm{d}\xi\mathrm{d}\eta\mathrm{d}s\\
  \lesssim_{\delta}&\left\| \sqrt{|(A_k\dot{A}_{k})(s,\xi)|}\widetilde{f}(s,k,\xi)\right\|_{L_s^2L^2_{k,\xi}} \left\|\sqrt{|(A_l^*\dot{A}^*_{l})(s,\eta)|}\widetilde{a}(s,l,\eta)\right\|_{L_s^2L^2_{l,\eta}}\\
  &\times \left\| \mathbf{1}_{\sigma\neq 0}A_{\sigma}(s,\rho)\dfrac{\langle s\rangle\big(\langle s\rangle+|(\sigma,\rho)|\big) \langle s-\rho/\sigma\rangle^2}{\langle\rho/\sigma\rangle^2\langle \rho\rangle/|\sigma|} \widetilde{H_6}(s,\sigma,\rho)\mathrm{e}^{-(\delta_0/300)\langle \sigma,\rho\rangle^{1/2}}\right\|_{L^{\infty}_{s}L^2_{\sigma,\rho}}\\
  \lesssim_{\delta}& \epsilon_1^3.
\end{align*}

Similarly, for $j=2$ we get by Lemma \ref{lem:8.4JiaHao} (ii), Lemma \ref{lem:par-v-P} and \eqref{est:N4bound-prove-1} that
\begin{align*}
  \mathcal{U}_2^{(1)} \lesssim_{\delta}&\int_{0}^{t}\sum_{k,l\in\mathbb{Z}}\int_{\mathbb{R}^2} \mathbf{1}_{\sigma\neq 0}\sqrt{|(A_{\sigma}\dot{A}_{\sigma})(s,\rho)|}\dfrac{\langle s\rangle}{|\rho/\sigma|+\langle s\rangle}\dfrac{\langle s-\rho/\sigma\rangle^2}{\langle\rho\rangle/\sigma^2} |\widetilde{H_6}(s,\sigma,\rho)| \\
  &\times \sqrt{|(A_k\dot{A}_{k})(s,\xi)|}|\widetilde{f}(s,k,\xi)|A_{l}(s,\eta) \mathrm{e}^{-(\delta_0/200)\langle l,\eta\rangle^{1/2}}|\widetilde{a}(s,l,\eta)|\mathrm{d}\xi\mathrm{d}\eta\mathrm{d}s\\
  \lesssim_{\delta}&\left\| \sqrt{|(A_k\dot{A}_{k})(s,\xi)|}\widetilde{f}(s,k,\xi)\right\|_{L_s^2L^2_{k,\xi}} \left\|A_l(s,\eta)\mathrm{e}^{-(\delta_0/300)\langle l,\eta\rangle^{1/2}}\widetilde{a}(s,l,\eta)\right\|_{L^{\infty}_sL^2_{l,\eta}}\\
  &\times \left\| \mathbf{1}_{\sigma\neq 0}\sqrt{|(A_{\sigma}\dot{A}_{\sigma})(s,\rho)|}\dfrac{\langle s\rangle}{|\rho/\sigma|+\langle s\rangle}\dfrac{\langle s-\rho/\sigma\rangle^2}{\langle\rho\rangle/\sigma^2}\widetilde{H_6}(s,\sigma,\rho)\right\|_{L_{s}^{2}L^2_{\sigma,\rho}}\\
  \lesssim_{\delta}&\left\| \sqrt{|(A_k\dot{A}_{k})(s,\xi)|}\widetilde{f}(s,k,\xi)\right\|_{L_s^2L^2_{k,\xi}} \left\|A_l^*(s,\eta)\mathrm{e}^{-(\delta_0/300)\langle l,\eta\rangle^{1/2}}\widetilde{a}(s,l,\eta)\right\|_{L^{\infty}_sL^2_{l,\eta}}\\
  &\times \left\| \mathbf{1}_{\sigma\neq 0}\sqrt{|(A_{\sigma}\dot{A}_{\sigma})(s,\rho)|}\dfrac{\langle s\rangle\big(\langle s\rangle+|(\sigma,\rho)|\big) \langle s-\rho/\sigma\rangle^2}{\langle\rho/\sigma\rangle^2\langle \rho\rangle/|\sigma|}\widetilde{H_6}(s,\sigma,\rho)\right\|_{L_{s}^{2}L^2_{\sigma,\rho}}\\
  \lesssim_{\delta}&\epsilon_1^3.
\end{align*}

The case $j=3$ is similar to the case $j=2$. Now we deal with the case of $j=4$. We get by Lemma \ref{lem:8.3JiaHao-mod} that
\begin{align}\label{est:Ak-A0-1}
   A_{k}(t,\xi) &\lesssim_{\delta} \left(\dfrac{\mathbf{1}_{t\in I_{k,\xi}} |\xi|/k^2}{\langle t-\xi/k\rangle}+1\right) A_{NR}(t,\rho)A_{k}(t,\eta)\mathrm{e}^{-(\lambda(t)/20)\min(\langle \rho\rangle,\langle k,\eta\rangle)^{1/2}},\nonumber\\
   &\lesssim_{\delta} \left(\dfrac{\mathbf{1}_{t\in I_{k,\xi}} \langle t\rangle/|k|}{\langle t-\xi/k\rangle}+1\right) A_{0}(t,\rho)A_{k}(t,\eta)\mathrm{e}^{-(\lambda(t)/20)\min(\langle \rho\rangle,\langle k,\eta\rangle)^{1/2}},
\end{align}
by noticing that $A_{0}(t,\rho)=\mathrm{e}^{\lambda(t)\langle\rho\rangle^{1/2}}+A_{NR}(t,\rho)$.

For $t\in I_{k,\xi}$, we get by \eqref{est:7.4JiaHao-2} and \eqref{est:7.4JiaHao-4} that
\begin{align*}
   &\dfrac{\langle\rho\rangle^2|k|}{\langle t\rangle\big(\langle t\rangle^{2}+|\rho|^2\big) }A_{k}(t,\xi)^2\\
   \lesssim_{\delta}&\dfrac{1}{\langle t-\xi/k\rangle} A_{0}(t,\rho)A_{k}(t,\xi)A_{k}(t,\eta)\mathrm{e}^{-(\lambda(t)/20)\min(\langle \rho\rangle,\langle k,\eta\rangle)^{1/2}}\\
   \lesssim_{\delta}&  \mathbf{1}_{\langle\rho\rangle\leq \langle k,\eta\rangle} A_{0}(t,\rho)\sqrt{|(\dot{A}_{k} A_{k})(t,\xi)|}\sqrt{|(\dot{A}_{k} A_{k})(t,\eta)|}\mathrm{e}^{-(\delta_0/200)\langle \rho\rangle^{1/2}}\\
   &+ \mathbf{1}_{ \langle k,\eta\rangle \leq\langle\rho\rangle} \sqrt{|( \dot{A}_0 A_{0})(t,\rho)|}\sqrt{|(\dot{A}_{k} A_{k})(t,\xi)|}A_{k}(t,\eta)\mathrm{e}^{-(\delta_0/200)\langle k,\eta\rangle^{1/2}}\\
   \lesssim_{\delta}&  \mathbf{1}_{\langle\rho\rangle\leq \langle k,\eta\rangle} A_{0}(t,\rho)\sqrt{|(\dot{A}_{k} A_{k})(t,\xi)|}\sqrt{|(\dot{A}^*_{k} A^*_{k})(t,\eta)|}\mathrm{e}^{-(\delta_0/200)\langle \rho\rangle^{1/2}}\\
   &+ \mathbf{1}_{ \langle k,\eta\rangle \leq\langle\rho\rangle} \sqrt{|( \dot{A}_0 A_{0})(t,\rho)|}\sqrt{|(\dot{A}_{k} A_{k})(t,\xi)|}A^*_{k}(t,\eta)\mathrm{e}^{-(\delta_0/200)\langle k,\eta\rangle^{1/2}},
\end{align*}
and for $t\notin I_{k,\xi}$, (as $(1+|k|/\langle t\rangle)A_{k}(t,\eta)\lesssim A_{k}^*(t,\eta)$)
\begin{align*}
   &\dfrac{\langle\rho\rangle^2|k|}{\langle t\rangle\big(\langle t\rangle^{2}+|\rho|^2\big) }A_{k}(t,\xi)^2\\
   \lesssim_{\delta}&  \left(\dfrac{\langle\rho\rangle^{1/4}|k|}{\langle t\rangle^{5/4}}\mathbf{1}_{ \langle k,\eta\rangle \leq\langle\rho\rangle}+\dfrac{\langle\rho\rangle^2|k|}{\langle t\rangle^3 }\mathbf{1}_{\langle\rho\rangle\leq \langle k,\eta\rangle}\right) A_{0}(t,\rho)A_k(t,\xi)A_{k}(t,\eta)\mathrm{e}^{-(\lambda(t)/20)\min(\langle \rho\rangle,\langle k,\eta\rangle)^{1/2}}\\
   \lesssim_{\delta}&  \left(\dfrac{\langle\rho\rangle^{1/4}}{\langle t\rangle^{5/4}}\mathbf{1}_{ \langle k,\eta\rangle \leq\langle\rho\rangle}+\dfrac{1}{\langle t\rangle^2 }\mathbf{1}_{\langle\rho\rangle\leq \langle k,\eta\rangle}\right) A_{0}(t,\rho)A_k(t,\xi)A_{k}^*(t,\eta)\mathrm{e}^{-(\delta_0/200)\min(\langle \rho\rangle,\langle k,\eta\rangle)^{1/2}}\\
   \lesssim_{\delta}&  \mathbf{1}_{ \langle k,\eta\rangle \leq\langle\rho\rangle}\sqrt{|( \dot{A}_0 A_{0})(t,\rho)|}\sqrt{|(\dot{A}_{k} A_{k})(t,\xi)|}A^*_{k}(t,\eta)\mathrm{e}^{-(\delta_0/200)\min(\langle \rho\rangle,\langle k,\eta\rangle)^{1/2}}\\&+
   \mathbf{1}_{\langle\rho\rangle\leq \langle k,\eta\rangle} A_{0}(t,\rho)\sqrt{|(\dot{A}_{k} A_{k})(t,\xi)|}\sqrt{|(\dot{A}^*_{k} A^*_{k})(t,\eta)|}\mathrm{e}^{-(\delta_0/200)\min(\langle \rho\rangle,\langle k,\eta\rangle)^{1/2}}.
\end{align*}
Then we arrive at
\begin{align*}
   &\dfrac{\langle\rho\rangle^2|k|}{\langle t\rangle\big(\langle t\rangle^{2}+|\rho|^2\big) }A_{k}(t,\xi)^2\\
   \lesssim_{\delta}&   A_{0}(t,\rho)\sqrt{|(\dot{A}_{k} A_{k})(t,\xi)|}\sqrt{|(\dot{A}^*_{k} A^*_{k})(t,\eta)|}\mathrm{e}^{-(\delta_0/200)\min(\langle \rho\rangle,\langle k,\eta\rangle)^{1/2}}\\
   &+ \sqrt{|( \dot{A}_0 A_{0})(t,\rho)|}\sqrt{|(\dot{A}_{k} A_{k})(t,\xi)|}A^*_{k}(t,\eta)\mathrm{e}^{-(\delta_0/200)\min(\langle \rho\rangle,\langle k,\eta\rangle)^{1/2}}.
\end{align*}
Thus, we obtain
\begin{align*}
  \mathcal{U}_4^{(1)} \lesssim_{\delta}&\int_{0}^{t}\sum_{k\in\mathbb{Z}}\int_{\mathbb{R}^2} \dfrac{\langle s\rangle\big(\langle s\rangle^{2}+|\rho|^2\big) }{\langle\rho\rangle^2}  A_{0}(s,\rho)|\widetilde{H}_6(s,0,\rho)| \sqrt{|(\dot{A}_k A_{k})(s,\xi)|}|\widetilde{f}(s,k,\xi)|\\
  &\times\sqrt{|(\dot{A}^*_{k} A^*_{k})(s,\eta)|} \mathrm{e}^{-(\delta_0/200)\min(\langle \rho\rangle,\langle k,\eta\rangle)^{1/2}}|\widetilde{a}(s,k,\eta)|\mathrm{d}\xi\mathrm{d}\eta\mathrm{d}s\\
  & + \int_{0}^{t}\sum_{k\in\mathbb{Z}}\int_{\mathbb{R}^2} \dfrac{\langle s\rangle\big(\langle s\rangle^{2}+|\rho|^2\big) }{\langle\rho\rangle^2} \sqrt{ (\dot{A}_{0}  A_{0})(s,\rho)|}|\widetilde{H}_6(s,0,\rho)| \sqrt{|(\dot{A}_k A_{k})(s,\xi)|}|\widetilde{f}(s,k,\xi)|\\
  &\times A^*_{k}(s,\eta) \mathrm{e}^{-(\delta_0/200)\min(\langle \rho\rangle,\langle k,\eta\rangle)^{1/2}}|\widetilde{a}(s,k,\eta)|\mathrm{d}\xi\mathrm{d}\eta\mathrm{d}s\\
  \lesssim_{\delta}&\left\| \sqrt{|(\dot{A}_k A_{k})(s,\xi)|}\widetilde{f}(s,k,\xi)\right\|_{L_s^2 L^2_{k,\xi}} \left\|\sqrt{|(\dot{A}^*_{k} A^*_{k})(s,\eta)|}\widetilde{a}(s,k,\eta)\right\|_{L^{2}_sL^2_{k,\eta}}\\
  &\times \left\| A_0(s,\rho)\dfrac{\langle s\rangle\big(\langle s\rangle^{2}+|\rho|^2\big) }{\langle\rho\rangle^2}\widetilde{H_6}(s,0,\rho) \right\|_{L_{s}^{\infty}L^2_{\rho}}\\
  +&\left\| \sqrt{|(\dot{A}_k A_{k})(s,\xi)|}\widetilde{f}(s,k,\xi)\right\|_{L_s^2 L^2_{k,\xi}} \left\| A^*_{k}(s,\eta)\widetilde{a}(s,k,\eta)\right\|_{L^{\infty}_sL^2_{k,\eta}}\\
  &\times \left\| \sqrt{|(\dot{A}_{0} A_{0})(s,\rho)|}\dfrac{\langle s\rangle\big(\langle s\rangle^{2}+|\rho|^2\big) }{\langle\rho\rangle^2}\widetilde{H_6}(s,0,\rho)\right\|_{L_{s}^{2}L^2_{\rho}}\\
   \lesssim_{\delta}&\epsilon_1^3.
\end{align*}

 This shows that  $\mathcal{U}^{(1)}_j\lesssim_{\delta}\epsilon_1^3$ for all $j\in\{0,1,2,3,4\}$, and the desired bound \eqref{est:N4bound} follows.
\end{proof}

\subsection{Nonlinear estimate for $\mathcal{N}_5$}

\begin{lemma}\label{lem:N5}
It holds that for any $t\in[1,T]$ we have
  \begin{align}\label{est:N5bound}
     &\left|2\mathbf{Re}\int_{1}^{t}\sum_{k\in\mathbb{Z}}\int_{\mathbb{R}}A_k(s,\xi)^2\widetilde{\mathcal{N}_5}(s,k,\xi)\overline{\widetilde{f}(s,k,\xi)}
     \mathrm{d}\xi\mathrm{d}s\right| \lesssim_{\delta} \epsilon_1^3 .
  \end{align}
\end{lemma}

Let
\begin{align}\label{eq:H7H8}
  &H_7=\partial_zP,\quad H_8=(1+V_1)\partial_zP.
\end{align}
We first prove the following lemma.

\begin{lemma}\label{lem:par-z-P}
  For any $t\in[1,T]$ and $j\in\{7,8\}$ we have
  \begin{align}\label{est:par-z-P}
    \begin{aligned}
    &\sum_{k\in \mathbb{Z}\setminus\{0\}}\int_{\mathbb{R}}A_{k}(t,\xi)^2\dfrac{\langle t\rangle^4(\langle t\rangle^2+|(k,\xi)|^2
    ) \langle t-\xi/k\rangle^4}{(\langle t\rangle^2+|\xi/k|^2
    )\langle\xi/k\rangle^4}|\widetilde{H_j}(t,k,\xi)|^2\mathrm{d}\xi\lesssim_{\delta}\epsilon_{1}^2,\\
    &\int_{1}^{T}\sum_{k\in \mathbb{Z}\setminus\{0\}}\int_{\mathbb{R}}|(\dot{A}_{k}A_{k})(s,\xi)| \dfrac{\langle s\rangle^4\big(\langle s\rangle
    ^2+|(k,\xi)|^2\big)\langle s-\xi/k\rangle^4}{(\langle s\rangle^2+|\xi/k|^2)\langle\xi/k\rangle^4}|\widetilde{H_j}(s,k,\xi)|^2\mathrm{d}\xi\mathrm{d}s \lesssim_{\delta}\epsilon_1^2.
    \end{aligned}
  \end{align}
\end{lemma}

\begin{proof}
  The bounds on $H_7$ follow directly from the bootstrap assumption on $\mathcal{E}_{P}(t)$ and $\mathcal{B}_{P}(t)$  and
   \begin{align*}
     & \dfrac{\langle t\rangle^4(\langle t\rangle^2+|(k,\xi)|^2
    ) \langle t-\xi/k\rangle^4}{(\langle t\rangle^2+|\xi/k|^2
    )\langle\xi/k\rangle^4}\leq \dfrac{\langle t\rangle^2(\langle t\rangle^2+|(k,\xi)|^2
    ) \langle t-\xi/k\rangle^4}{\langle\xi/k\rangle^4}.
  \end{align*}

  Notice that $H_8=H_7+V_1H_7$. By Lemma \ref{lem:8.1JiaHao} (ii) and \eqref{eq:boot-V1}, it suffices to prove the following multiplier bounds for $k\neq 0$,
  \begin{align}\label{est:par-z-P-prove-1}
  \begin{aligned}
    & A_{k}(t,\xi)\dfrac{\langle t\rangle^2(\langle t\rangle+|(k,\xi)|
    ) \langle t-\xi/k\rangle^2}{\big(\langle t\rangle+|\xi/k|
    \big)\langle\xi/k\rangle^2}\\&\lesssim_{\delta}A_{R}(t,\xi-\eta)\cdot A_{k}(t,\eta)\dfrac{\langle t\rangle^2(\langle t\rangle+|(k,\eta)|
    ) \langle t-\eta/k\rangle^2}{\big(\langle t\rangle+|\eta/k|
    \big)\langle\eta/k\rangle^2}\cdot\big\{\langle \xi-\eta\rangle^{-2}+\langle k,\eta\rangle^{-2}\big\}.
     \end{aligned}
  \end{align}
  and
   \begin{align}\label{est:par-z-P-prove-2}
    |(\dot{A}_{k} A_{k})(t,\xi)|^{\f12}&\dfrac{\langle t\rangle^2(\langle t\rangle+|(k,\xi)|
    ) \langle t-\xi/k\rangle^2}{\big(\langle t\rangle+|\xi/k|
    \big)\langle\xi/k\rangle^2}\lesssim_{\delta} \left[|(\dot{A}_{R}/ A_{R})(t,\xi-\eta)|^{\f12}+|(\dot{A}_{k} /A_{k})(t,\eta)|^{\f12}\right]\nonumber\\& \times A_{R}(t,\xi-\eta)\cdot A_{k}(t,\eta)\dfrac{\langle t\rangle^2(\langle t\rangle+|(k,\eta)|
    ) \langle t-\eta/k\rangle^2}{\big(\langle t\rangle+|\eta/k|
    \big)\langle\eta/k\rangle^2}\cdot\big\{\langle \xi-\eta\rangle^{-2}+\langle k,\eta\rangle^{-2}\big\},
  \end{align}

 By considering the cases $|\xi-\eta|\leq 10|(k,\eta)|$ and $|\xi-\eta|\geq 10|(k,\eta)|$, it is easy to see that ($k\neq 0$)
  \begin{align}\label{est:par-z-P-prove-3}
     & \dfrac{\langle t\rangle^2(\langle t\rangle+|(k,\xi)|
    ) \langle t-\xi/k\rangle^2}{\big(\langle t\rangle+|\xi/k|
    \big)\langle\xi/k\rangle^2} \lesssim_{\delta} \dfrac{\langle t\rangle^2(\langle t\rangle+|(k,\eta)|
    ) \langle t-\eta/k\rangle^2}{\big(\langle t\rangle+|\eta/k|
    \big)\langle\eta/k\rangle^2}\mathrm{e}^{\delta\min(\langle \xi-\eta\rangle,\langle k,\eta\rangle)^{1/2}}.
  \end{align}
Then the bound \eqref{est:par-z-P-prove-1} follows from \eqref{est:8.3JiaHao-1} and \eqref{est:par-z-P-prove-3}; the bound \eqref{est:par-z-P-prove-2} follows from \eqref{est:8.3JiaHao-1}, \eqref{est:8.3JiaHao-2} and \eqref{est:par-z-P-prove-3}.
\end{proof}

We now turn to the proof of \eqref{est:N5bound}.

\begin{proof}
We write
\begin{align*}
   & \left|2\mathbf{Re}\int_{1}^{t}\sum_{k\in\mathbb{Z}}\int_{\mathbb{R}}A_k(s,\xi)^2 \widetilde{N_5} (s,k,\xi)\overline{\widetilde{f}(s,k,\xi)}\mathrm{d}\xi\mathrm{d}s\right|\\
   &=2\left| \mathbf{Re}\sum_{k,l\in\mathbb{Z}}\int_{1}^{t}\int_{\mathbb{R}^2}A_k(s,\xi)^2\widetilde{H_8}(s,k-l,\xi-\eta)i\eta\widetilde{a}(s,l,\eta)
   \overline{\widetilde{f}(s,k,\xi)}\mathrm{d} \xi\mathrm{d}\eta\mathrm{d}s\right|.
\end{align*}

Let the sets $R_0,R_1,R_2,R_3$ be defined by \eqref{eq:R0}-\eqref{eq:R3} and we denote that for $j=0,1,2,3$,
\begin{align*}
   \mathcal{V}^{(1)}_j=\int_{1}^{t}\sum_{k,l\in\mathbb{Z}}\int_{\mathbb{R}^2} &\mathbf{1}_{R_j}\big((k,\xi),(l,\eta)\big)|\eta| A_k(s,\xi)^2 | \widetilde{H_8}(s,k-l,\xi-\eta)|\\
   &\times|\widetilde{a}(s,l,\eta)||\widetilde{f}(s,k,\xi)|\mathrm{d} \xi\mathrm{d}\eta\mathrm{d}s.
\end{align*}

Let $(\sigma,\rho)=(k-l,\xi-\eta)$. For $j=0$, we get by Lemma \ref{lem:8.5JiaHao} (i) and  \eqref{est:par-t-A*-1}  that
\begin{align}\label{est:V0(1)}
 { \mathcal{V}_0^{(1)}} \lesssim_{\delta}&\int_{0}^{t}\sum_{k,l\in\mathbb{Z}}\int_{\mathbb{R}^2} \sqrt{|(A_k\dot{A}_{k})(s,\xi)|}|\widetilde{f}(s,k,\xi)| \sqrt{|(A_l\dot{A}_{l})(s,\eta)|}|\widetilde{a}(s,l,\eta)|\langle s-\rho/\sigma\rangle^2\nonumber\\
  &\times \mathbf{1}_{\sigma\neq 0}\dfrac{|\sigma|\langle s\rangle^2}{|\rho/\sigma|^2+\langle s\rangle^2}A_{\sigma}(s,\rho)| \widetilde{H_8}(s,\sigma,\rho)|\mathrm{e}^{-(\delta_0/200)\langle \sigma,\rho\rangle^{1/2}}\mathrm{d}\xi\mathrm{d}\eta\mathrm{d}s\nonumber\\
  \lesssim_{\delta}&\left\| \sqrt{|(A_k\dot{A}_{k})(s,\xi)|}\widetilde{f}(s,k,\xi)\right\|_{L_s^2L^2_{k,\xi}} \left\|\sqrt{|(A_l^*\dot{A}^*_{l})(s,\eta)|}\widetilde{a}(s,l,\eta)\right\|_{L_s^2L^2_{l,\eta}}\\
  &\times \left\| \mathbf{1}_{\sigma\neq 0}A_{\sigma}(s,\rho)\langle s-\rho/\sigma\rangle
  ^2\dfrac{|\sigma|\langle s\rangle^2}{|\rho/\sigma|^2+\langle s\rangle^2} \widetilde{H_8}(s,\sigma,\rho)\mathrm{e}^{-(\delta_0/300)\langle \sigma,\rho\rangle^{1/2}}\right\|_{L^{\infty}_{s}L^2_{\sigma,\rho}}.\nonumber
\end{align}
Thanks to $|\sigma|\langle\rho/\sigma\rangle^2\leq \big(|\rho/\sigma|+\langle s\rangle\big) \big(\langle s\rangle+|(\sigma,\rho)|\big)$, we have
\begin{align}\label{est:N5bound-prove-1}
  \langle s-\rho/\sigma\rangle^2\dfrac{|\sigma|\langle s\rangle^2}{|\rho/\sigma|^2+\langle s\rangle^2} &\lesssim \dfrac{\langle s\rangle^2(\langle s\rangle+|(\sigma,\rho)|
    ) \langle s-\rho/\sigma\rangle^2}{\big(\langle s\rangle+|\rho/\sigma|
    \big)\langle\rho/\sigma\rangle^2}.
\end{align}
This along with  Lemma \ref{lem:par-z-P}  gives
\begin{align*}
  \left\| \mathbf{1}_{\sigma\neq 0}A_{\sigma}(s,\rho) \langle s-\rho/\sigma\rangle^2\dfrac{|\sigma|\langle s\rangle^2}{|\rho/\sigma|^2+\langle s\rangle^2} \widetilde{H_8}(s,\sigma,\rho)\mathrm{e}^{-(\delta_0/300)\langle \sigma,\rho\rangle^{1/2}}\right\|_{L^{\infty}_{s}L^2_{\sigma,\rho}} & \lesssim_{\delta}\epsilon_1,
\end{align*}
from which and \eqref{est:V0(1)}, we infer that
\begin{align*}
   \mathcal{V}_0^{(1)} &\lesssim_{\delta}  \epsilon_1^3.
\end{align*}

For $j=1$,  we get by Lemma \ref{lem:A*R0R1R2-1-R1} and Lemma \ref{lem:par-z-P} that
\begin{align*}
  &\mathcal{V}_1^{(1)} \lesssim_{\delta}\int_{0}^{t}\sum_{k,l\in\mathbb{Z}}\int_{\mathbb{R}^2} \sqrt{|(A_k\dot{A}_{k})(s,\xi)|}|\widetilde{f}(s,k,\xi)| \sqrt{|(A^*_l\dot{A}^*_{l})(s,\eta)|}|\widetilde{a}(s,l,\eta)|\\
  &\times \mathbf{1}_{\sigma\neq 0}\langle s\rangle^3 A_{\sigma}(s,\rho)| \widetilde{H_8}(s,\sigma,\rho)|\mathrm{e}^{-(\delta_0/200)\langle \sigma,\rho\rangle^{1/2}}\mathrm{d}\xi\mathrm{d}\eta\mathrm{d}s\\
  \lesssim_{\delta}&\left\| \sqrt{|(A_k\dot{A}_{k})(s,\xi)|}\widetilde{f}(s,k,\xi)\right\|_{L_s^2L^2_{k,\xi}} \left\|\sqrt{|(A_l^*\dot{A}^*_{l})(s,\eta)|}\widetilde{a}(s,l,\eta)\right\|_{L_s^2L^2_{l,\eta}}\\
  &\times \left\| \mathbf{1}_{\sigma\neq 0}A_{\sigma}(s,\rho)\dfrac{\langle s\rangle^2(\langle s\rangle+|(\sigma,\rho)|
    ) \langle s-\rho/\sigma\rangle^2}{\big(\langle s\rangle+|\rho/\sigma|
    \big)\langle\rho/\sigma\rangle^2} \widetilde{H_8}(s,\sigma,\rho)\mathrm{e}^{-(\delta_0/300)\langle \sigma,\rho\rangle^{1/2}}\right\|_{L^{\infty}_{s}L^2_{\sigma,\rho}}
  \lesssim_{\delta} \epsilon_1^3.
\end{align*}

Similarly, for $j=2$, we get by Lemma \ref{lem:8.5JiaHao} (ii), Lemma \ref{lem:par-z-P} and \eqref{est:N5bound-prove-1} that
\begin{align*}
  \mathcal{V}_2^{(1)} \lesssim_{\delta}&\int_{0}^{t}\sum_{k,l\in\mathbb{Z}}\int_{\mathbb{R}^2} \mathbf{1}_{\sigma\neq 0}\sqrt{|(A_{\sigma}\dot{A}_{\sigma})(s,\rho)|}\dfrac{|\sigma|\langle s\rangle^2}{|\rho/\sigma|^2+\langle s\rangle^2}\langle s-\rho/\sigma\rangle^2 |\widetilde{H_8}(s,\sigma,\rho)| \\
  &\times \sqrt{|(A_k\dot{A}_{k})(s,\xi)|}|\widetilde{f}(s,k,\xi)|A_{l}(s,\eta) \mathrm{e}^{-(\delta_0/200)\langle l,\eta\rangle^{1/2}}|\widetilde{a}(s,l,\eta)|\mathrm{d}\xi\mathrm{d}\eta\mathrm{d}s\\
  \lesssim_{\delta}&\left\| \sqrt{|(A_k\dot{A}_{k})(s,\xi)|}\widetilde{f}(s,k,\xi)\right\|_{L_s^2L^2_{k,\xi}} \left\|A_l(s,\eta)\mathrm{e}^{-(\delta_0/300)\langle l,\eta\rangle^{1/2}}\widetilde{a}(s,l,\eta)\right\|_{L^{\infty}_sL^2_{l,\eta}}\\
  &\times \left\| \mathbf{1}_{\sigma\neq 0}\sqrt{|(A_{\sigma}\dot{A}_{\sigma})(s,\rho)|}\dfrac{|\sigma|\langle s\rangle^2}{|\rho/\sigma|^2+\langle s\rangle^2}\langle s-\rho/\sigma\rangle^2\widetilde{H_8}(s,\sigma,\rho)\right\|_{L_{s}^{2}L^2_{\sigma,\rho}}\\
  \lesssim_{\delta}&\left\| \sqrt{|(A_k\dot{A}_{k})(s,\xi)|}\widetilde{f}(s,k,\xi)\right\|_{L_s^2L^2_{k,\xi}} \left\|A_l^*(s,\eta)\mathrm{e}^{-(\delta_0/300)\langle l,\eta\rangle^{1/2}}\widetilde{a}(s,l,\eta)\right\|_{L^{\infty}_sL^2_{l,\eta}}\\
  &\times \left\| \mathbf{1}_{\sigma\neq 0}\sqrt{|(A_{\sigma}\dot{A}_{\sigma})(s,\rho)|}\dfrac{\langle s\rangle^2(\langle s\rangle+|(\sigma,\rho)|
    ) \langle s-\rho/\sigma\rangle^2}{\big(\langle s\rangle+|\rho/\sigma|
    \big)\langle\rho/\sigma\rangle^2}\widetilde{H_8}(s,\sigma,\rho)\right\|_{L_{s}^{2}L^2_{\sigma,\rho}}  \lesssim_{\delta}\epsilon_1^3.
\end{align*}
The case $j=3$ is similar to the case $j=2$. Thus, $\mathcal{V}^{(1)}_j\lesssim_{\delta}\epsilon_1^3$ for all $j\in\{0,1,2,3\}$, and the desired bound \eqref{est:N5bound} follows.
\end{proof}

\section{Improved control of the stream function}

 In this section, we prove an improved control for the stream function under the bootstrap assumptions in Proposition \ref{prop:Bootstrap}.

\begin{proposition}\label{prop:Improved-Theta}
  With the definitions and assumptions in Proposition \ref{prop:Bootstrap}, we have
  \begin{align}\nonumber%\label{est:Improved-Theta}
     & \mathcal{E}_{\Theta}(t)+\mathcal{B}_{\Theta}(t)\lesssim_{\delta} \epsilon_1^3 \qquad \text{for any}\ t\in[1,T].
  \end{align}
\end{proposition}
Recall the elliptic equation \eqref{eq:phi}
\begin{align*}
  &\partial_z^2\phi+(V_1+1)^2(\partial_v-t\partial_z)^2\phi+V_2(\partial_v- t\partial_z)\phi=f,
\end{align*}
which gives
\begin{align}\nonumber%\label{eq:phi2}
  & \partial_z^2\phi+(\partial_v-t\partial_z)^2\phi =f-(2V_1+V_1^2)(\partial_v-t\partial_z)^2\phi- V_2(\partial_v- t\partial_z)\phi.
\end{align}
Therefore, by the definition of \eqref{eq:Theta}, we get
\begin{align}\nonumber%\label{eq:Theta1}
  \Theta &=f-(2V_1+V_1^2)(\partial_v-t\partial_z)^2\phi- V_2(\partial_v- t\partial_z)\phi=f+g_{11}+g_{12},
\end{align}
where
\begin{align}\nonumber%\label{eq:g11g12}
   & g_{11}= -(2V_1+V_1^2)(\partial_v-t\partial_z)^2\phi,\qquad g_{12}=- V_2(\partial_v- t\partial_z)\phi.
\end{align}

Proposition \ref{prop:Improved-Theta} follows from the following lemma,  which has been proved in \cite{IJ}.

\begin{lemma}\label{lem:5.3JiaHao}
  For any $t\in[1,T]$ and $G\in\{f,g_{11},g_{1,2}\}$, we have
  \begin{align}\nonumber%\label{eq:5.2JiaHao-1}
    &\sum_{k\in\mathbb{Z}\setminus \{0\}}\int_{\mathbb{R}}A_{k}(t,\xi)^2\dfrac{k^2\langle t\rangle^2}{|\xi|^2+k^2\langle t\rangle^2}|\widetilde{G}(t,k,\xi)|^2\mathrm{d}\xi\lesssim_{\delta }\epsilon_1^3,
  \end{align}
  and
  \begin{align}\nonumber%\label{eq:5.2JiaHao-2}
    &
    \int_{1}^{T} \sum_{k\in\mathbb{Z}\setminus \{0\}}\int_{\mathbb{R}}|\dot{A}_{k}(s,\xi)|A_{k}(s,\xi)\dfrac{k^2\langle s\rangle^2}{|\xi|^2+k^2\langle s\rangle^2}|\widetilde{G}(s,k,\xi)|^2\mathrm{d}\xi \mathrm{d}s\lesssim_{\delta }\epsilon_1^3.
  \end{align}
\end{lemma}

%\begin{proof}
%  The proof of this Lemma is identically the same as Lemma 5.2 in JiaHao's paper.
%
%  For the case $G=f$, due to Proposition \ref{prop:Improved-f}, we have
%  \begin{align*}
%    &\sum_{k\in\mathbb{Z}\setminus \{0\}}\int_{\mathbb{R}}A_{k}^2(t,\xi)\dfrac{k^2\langle t\rangle^2}{|\xi|^2+k^2\langle t\rangle^2}
%|\widetilde{G}(t,k,\xi)|^2\mathrm{d}\xi\lesssim_{\delta } \mathcal{E}_{f}(t)\lesssim_{\delta}\epsilon_1^3,\\
%    &
%    \int_{1}^{T} \sum_{k\in\mathbb{Z}\setminus \{0\}}\int_{\mathbb{R}}|\dot{A}_{k}(s,\xi)|A_{k}(s,\xi)\dfrac{k^2\langle s\rangle^2}{|\xi|^2+
%k^2\langle s\rangle^2}|\widetilde{G}(s,k,\xi)|^2\mathrm{d}\xi \mathrm{d}s\lesssim_{\delta }\mathcal{B}_{f}(t)\lesssim_{\delta}\epsilon_1^3.
%  \end{align*}
%
%  If $G\in\{g_{11},g_{12}\}$, using the bootstrap assumption on $V_1,\ V_2$ and $\phi$ (namely \eqref{eq:boot-V1} and \eqref{eq:boot-g}), we have
%    \begin{align*}
%    &\sum_{k\in\mathbb{Z}\setminus \{0\}}\int_{\mathbb{R}}A_{k}^2(t,\xi)\dfrac{k^2\langle t\rangle^2}{|\xi|^2+
%k^2\langle t\rangle^2}|\widetilde{G}(t,k,\xi)|^2\mathrm{d}\xi\lesssim_{\delta }\epsilon_1^4,\\
%    &
%    \int_{1}^{T} \sum_{k\in\mathbb{Z}\setminus \{0\}}\int_{\mathbb{R}}|\dot{A}_{k}(s,\xi)|A_{k}(s,\xi)\dfrac{k^2\langle s\rangle^2}{|\xi|^2+
%k^2\langle s\rangle^2}|\widetilde{G}(s,k,\xi)|^2\mathrm{d}\xi \mathrm{d}s\lesssim_{\delta }\epsilon_1^4.
%  \end{align*}
%\end{proof}

 \section{Improved control of the density}

  In this section, we prove an improved control for the density  $a$ under the bootstrap assumptions in Proposition \ref{prop:Bootstrap}.

 \begin{proposition}\label{prop:Improved-a}
   With the definitions and assumptions in Proposition \ref{prop:Bootstrap}, we have
   \begin{align}\nonumber
      & \mathcal{E}_a(t)+\mathcal{B}_a(t)\lesssim_{\delta} \epsilon_1^3\quad \text{for any}\ t\in[1,T].
   \end{align}
 \end{proposition}

 It is easy to find that
 \begin{align*}
    \dfrac{\mathrm{d}}{\mathrm{d}t}\mathcal{E}_a(t)=&\sum_{k\in\mathbb{Z}}
    \int_{\mathbb{R}}2\dot{A}^{*}_k(t,\xi)A_k^{*}(t,\xi)|\widetilde{a}(t,k,\xi)|^2\mathrm{d}\xi\\
    &+2\mathbf{Re}\sum_{k\in\mathbb{Z}}\int_{\mathbb{R}}A_k^*(t,\xi)^2 \partial_t\widetilde{a} (t,k,\xi)\overline{\widetilde{a}(t,k,\xi)}\mathrm{d}\xi,
 \end{align*}
 which gives
 \begin{align}
    &\mathcal{E}_a(t)+2\mathcal{B}_{a}(t) =\mathcal{E}_{a}(1) +2\mathbf{Re}\int_{1}^{t}\sum_{k\in\mathbb{Z}}\int_{\mathbb{R}}A^*_k(s,\xi)^2 \partial_s\widetilde{a} (s,k,\xi)\overline{\widetilde{a}(s,k,\xi)}\mathrm{d}\xi\mathrm{d}s.
 \end{align}
 Since  $\dot{A}_k^*(t,\xi)\le 0$,  it suffices to prove that
 \begin{align}
   & \left|2\mathbf{Re}\int_{1}^{t}\sum_{k\in\mathbb{Z}}\int_{\mathbb{R}}A_k^*(s,\xi)^2 \partial_s\widetilde{a} (s,k,\xi)\overline{\widetilde{a}(s,k,\xi)}\mathrm{d}\xi\mathrm{d}s\right| \lesssim_{\delta} \epsilon_1^3.
 \end{align}
Recall that
  \begin{align*}
    & \partial_sa=\mathcal{N}_6+\mathcal{N}_7+\mathcal{N}_8,
 \end{align*}
where
 \begin{align*}
    & \mathcal{N}_6=(1+V_1)\partial_v\mathbb{P}_{\neq 0}\phi
    \partial_za,\quad \mathcal{N}_7=-(1+V_1)\partial_z\mathbb{P}_{\neq 0}\phi\partial_va,\quad \mathcal{N}_8=-V_3\partial_va.
 \end{align*}

 Proposition \ref{prop:Improved-a} is a direct consequence of  Lemma \ref{lem:N6bound}, Lemma \ref{lem:N7bound} and
 Lemma \ref{lem:N8bound}.

\subsection{Nonlinear estimate for $\mathcal{N}_6$}

Let \begin{align}\label{eq:H12}
      &H_1=\partial_v\mathbb{P}_{\neq 0}\phi,\quad H_2=(1+V_1)\partial_v\mathbb{P}_{\neq0}\phi.
    \end{align}
 Recall the following lemma from \cite{IJ}.
\begin{lemma}\label{lem:4.5JiaHao}
  For any $t\in[1,T)$ and $j\in\{1,2\}$, we have
    \begin{align*}
   & \sum_{k\in\mathbb{Z}\setminus \{0\}}\int_{\mathbb{R}}A_k(t,\xi)^2\dfrac{\langle t\rangle^2}{|\xi/k|^2+\langle t\rangle^2}\dfrac{\langle t-\xi/k\rangle^4}{(\langle \xi\rangle/k^2)^2}|\widetilde{H_j}(t,k,\xi)|^2\mathrm{d}\xi \lesssim_{\delta} \epsilon_1^2,\\
   & \int_{1}^{t}\sum_{k\in\mathbb{Z}\setminus \{0\}}\int_{\mathbb{R}}|\dot{A}_{k}(s,\xi)|A_k(s,\xi)\dfrac{\langle s\rangle^2}{|\xi/k|^2+\langle s\rangle^2}\dfrac{\langle s-\xi/k\rangle^4}{(\langle \xi\rangle/k^2)^2}|\widetilde{H_j}(s,k,\xi)|^2\mathrm{d}\xi\mathrm{d}s \lesssim_{\delta} \epsilon_1^2.
\end{align*}
\end{lemma}

We then prove  the following lemma.

\begin{lemma}\label{lem:N6bound}
 It holds that for any $t\in[1,T]$, we have
  \begin{align}\label{eq:N6bound}
     &\left|2\mathbf{Re}\int_{1}^{t}\sum_{k\in\mathbb{Z}}\int_{\mathbb{R}}A_k^{*}(s,\xi)^2 \widetilde{\mathcal{N}_6}(s,k,\xi)\overline{\widetilde{a}(s,k,\xi)}\mathrm{d}\xi\mathrm{d}s\right| \lesssim_{\delta} \epsilon_1^3.
  \end{align}
\end{lemma}

\begin{proof}
We write
\begin{align*}
   & \left|2\mathbf{Re}\int_{1}^{t}\sum_{k\in\mathbb{Z}}\int_{\mathbb{R}}A_k^*(s,\xi)^2 \widetilde{N_6} (s,k,\xi)\overline{\widetilde{a}(s,k,\xi)}\mathrm{d}\xi\mathrm{d}s\right|\\
   &=2\left| \mathbf{Re}\sum_{k,l\in\mathbb{Z}}\int_{1}^{t}\int_{\mathbb{R}^2}A_k^*(s,\xi)^2 \widetilde{H_2}(s,k-l,\xi-\eta)il\widetilde{a}(s,l,\eta)\overline{\widetilde{a}(s,k,\xi)}\mathrm{d} \xi\mathrm{d}\eta\mathrm{d}s\right|\\
   &=\left| \sum_{k,l\in\mathbb{Z}}\int_{1}^{t}\int_{\mathbb{R}^2}[l A_k^*(s,\xi)^2-kA_l^{*}(s,\eta)^2 ] \widetilde{H_2}(s,k-l,\xi-\eta)\widetilde{a}(s,l,\eta)\overline{\widetilde{a}(s,k,\xi)}\mathrm{d} \xi\mathrm{d}\eta\mathrm{d}s\right|,
\end{align*}
where the second identity used the symmetrization($H_2$ is real-valued).

Let the sets $R_0,R_1,R_2,R_3$ be defined by \eqref{eq:R0}-\eqref{eq:R3} and we denote that for $j=0, 1, 2,3$,
\begin{align*}
   \mathcal{U}_j=\int_{1}^{t}\sum_{k,l\in\mathbb{Z}}\int_{\mathbb{R}^2} & \mathbf{1}_{R_j}\big((k,\xi),(l,\eta)\big)|l A_k^*(s,\xi)^2-kA_l^{*}(s,\eta)^2 | | \widetilde{H_2}(s,k-l,\xi-\eta)|\\
   &\times|\widetilde{a}(s,l,\eta)||\widetilde{a}(s,k,\xi)|\mathrm{d} \xi\mathrm{d}\eta\mathrm{d}s.
\end{align*}

Let $(\sigma,\rho)=(k-l,\xi-\eta)$.  For $j=0,1$, we get by Lemma \ref{lem:A*R0R1R2-1} (i) and  Lemma \ref{lem:4.5JiaHao}  that
\begin{align*}
  \mathcal{U}_j \lesssim_{\delta}&\int_{0}^{t}\sum_{k,l\in\mathbb{Z}}\int_{\mathbb{R}^2} \sqrt{|(A_k^{*}\dot{A}_{k}^*)(s,\xi)|}|\widetilde{a}(s,k,\xi)| \sqrt{|(A_l^{*}\dot{A}_{l}^*)(s,\eta)|}|\widetilde{a}(s,l,\eta)|\dfrac{\langle s\rangle}{|\rho/\sigma|+\langle s\rangle}\\
  &\times \mathbf{1}_{\sigma\neq 0}\dfrac{\langle s-\rho/\sigma\rangle^2}{\langle\rho\rangle/\sigma^2}A_{\sigma}(s,\rho) |\widetilde{H_2}(s,\sigma,\rho)|\mathrm{e}^{-(\delta_0/201)\langle \sigma,\rho\rangle^{1/2}}\mathrm{d}\xi\mathrm{d}\eta\mathrm{d}s\\
  \lesssim_{\delta}&\left\| \sqrt{|(A_k^{*}\dot{A}_{k}^*)(s,\xi)|}\widetilde{a}(s,k,\xi)\right\|_{L_s^2L^2_{k,\xi}} \left\|\sqrt{|(A_l^{*}\dot{A}_{l}^*)(s,\eta)|}\widetilde{a}(s,l,\eta)\right\|_{L_s^2L^2_{l,\eta}}\\
  &\times \left\| \mathbf{1}_{\sigma\neq 0}A_{\sigma}(s,\rho)\dfrac{\langle s\rangle}{|\rho/\sigma|+\langle s\rangle}\dfrac{\langle s-\rho/\sigma\rangle^2}{\langle\rho\rangle/\sigma^2} \widetilde{H_2}(s,\sigma,\rho)\mathrm{e}^{-(\delta_0/300)\langle \sigma,\rho\rangle^{1/2}}\right\|_{L^{\infty}_{s}L^2_{\sigma,\rho}}
  \lesssim_{\delta}\epsilon_1^3.
\end{align*}
Similarly, for $j=2$, we get by Lemma \ref{lem:A*R0R1R2-1} (ii) and  Lemma \ref{lem:4.5JiaHao}  that
\begin{align*}
  \mathcal{U}_2 \lesssim_{\delta}&\int_{0}^{t}\sum_{k,l\in\mathbb{Z}}\int_{\mathbb{R}^2} \mathbf{1}_{\sigma\neq 0}\sqrt{|(A_{\sigma}\dot{A}_{\sigma})(s,\rho)|}\dfrac{\langle s\rangle}{|\rho/\sigma|+\langle s\rangle}\dfrac{\langle s-\rho/\sigma\rangle^2}{\langle\rho\rangle/\sigma^2} |\widetilde{H_2}(s,\sigma,\rho)| \\
  &\times \sqrt{|(A_k^{*}\dot{A}_{k}^*)(s,\xi)|}|\widetilde{a}(s,k,\xi)|A_{l}(s,\eta) \mathrm{e}^{-(\delta_0/201)\langle l,\eta\rangle^{1/2}}|\widetilde{a}(s,l,\eta)|\mathrm{d}\xi\mathrm{d}\eta\mathrm{d}s\\
  \lesssim_{\delta}&\left\| \sqrt{|(A_k^{*}\dot{A}_{k}^*)(s,\xi)|}\widetilde{a}(s,k,\xi)\right\|_{L_s^2L^2_{k,\xi}} \left\|A_l^{*}(s,\eta)\mathrm{e}^{-(\delta_0/300)\langle l,\eta\rangle^{1/2}}\widetilde{a}(s,l,\eta)\right\|_{L^{\infty}_sL^2_{l,\eta}}\\
  &\times \left\| \mathbf{1}_{\sigma\neq 0}\sqrt{|(A_{\sigma}\dot{A}_{\sigma})(s,\rho)|}\dfrac{\langle s\rangle}{|\rho/\sigma|+\langle s\rangle}\dfrac{\langle s-\rho/\sigma\rangle^2}{\langle\rho\rangle/\sigma^2}\widetilde{H_2}(s,\sigma,\rho)\right\|_{L_{s}^{2}L^2_{\sigma,\rho}}  \lesssim_{\delta}\epsilon_1^3.
\end{align*}
The case $j=3$ is identical to the case $j=2$ by the symmetry. Thus, $\mathcal{U}_j\lesssim_{\delta}\epsilon_1^3$ for all $j\in\{0,1,2,3\}$, and the desired bound \eqref{eq:N6bound} follows.
\end{proof}

\subsection{Nonlinear estimate for $\mathcal{N}_7$}

Let
\begin{align}\label{eq:H34}
   &H_3=\partial_z\mathbb{P}_{\neq 0}\phi,\qquad H_4=(1+V_1)\partial_z\mathbb{P}_{\neq 0}\phi.
\end{align}
Recall  the following lemma from \cite{IJ}.

\begin{lemma}\label{lem:4.7JiaHao}
  For any $t\in[1,T)$ and $j\in\{3,4\}$, we have
  \begin{align*}
   & \sum_{k\in\mathbb{Z}\setminus \{0\}}\int_{\mathbb{R}}A_k(t,\xi)^2\dfrac{k^2\langle t\rangle^4\langle t-\xi/k\rangle^4}{\big(|\xi/k|^2+\langle t\rangle^2\big)^2}|\widetilde{H_j}(t,k,\xi)|^2\mathrm{d}\xi \lesssim_{\delta} \epsilon_1^2,\\
   & \int_{1}^{t}\sum_{k\in\mathbb{Z}\setminus \{0\}}\int_{\mathbb{R}}|\dot{A}_{k}(s,\xi)|A_k(s,\xi)\dfrac{k^2\langle s\rangle^4\langle s-\xi/k\rangle^4}{\big(|\xi/k|^2+\langle s\rangle^2\big)^2}|\widetilde{H_j}(s,k,\xi)|^2\mathrm{d}\xi\mathrm{d}s \lesssim_{\delta} \epsilon_1^2.
\end{align*}
\end{lemma}

\begin{lemma}\label{lem:N7bound}
 It holds that for any $t\in[1,T]$, we have
  \begin{align}\label{eq:N7bound}
     &\left|2\mathbf{Re}\int_{1}^{t}\sum_{k\in\mathbb{Z}}\int_{\mathbb{R}}A_k^{*}(s,\xi)^2 \widetilde{\mathcal{N}_7}(s,k,\xi)\overline{\widetilde{a}(s,k,\xi)}\mathrm{d}\xi\mathrm{d}s\right| \lesssim_{\delta} \epsilon_1^3.
  \end{align}
\end{lemma}

\begin{proof}

We write
\begin{align*}
   & \left|2\mathbf{Re}\int_{1}^{t}\sum_{k\in\mathbb{Z}}\int_{\mathbb{R}}A_k^*(s,\xi)^2 \widetilde{N_7}(s,k,\xi) (s,k,\xi)\overline{\widetilde{a}(s,k,\xi)}\mathrm{d}\xi\mathrm{d}s\right|\\
   &=2\left| \mathbf{Re}\sum_{k,l\in\mathbb{Z}}\int_{1}^{t}\int_{\mathbb{R}^2}A_k^*(s,\xi)^2 \widetilde{H_4}(s,k-l,\xi-\eta)i\eta\widetilde{a}(s,l,\eta)\overline{\widetilde{a}(s,k,\xi)}\mathrm{d} \xi\mathrm{d}\eta\mathrm{d}s\right|\\
   &=\left| \sum_{k,l\in\mathbb{Z}}\int_{1}^{t}\int_{\mathbb{R}^2}[\eta
    A_k^*(s,\xi)^2-\xi A_l^{*}(s,\eta)^2 ] \widetilde{H_4}(s,k-l,\xi-\eta)\widetilde{a}(s,l,\eta)\overline{\widetilde{a}(s,k,\xi)}\mathrm{d} \xi\mathrm{d}\eta\mathrm{d}s\right|,
\end{align*}
where the second identity used the symmetrization ($H_4$ is real-valued). We denote that for $j=0,1,2,3,$
\begin{align*}
   \mathcal{V}_j=\int_{1}^{t}\sum_{k,l\in\mathbb{Z}}\int_{\mathbb{R}^2} & \mathbf{1}_{R_j}\big((k,\xi),(l,\eta)\big)|\eta A_k^*(s,\xi)^2-\xi A_l^{*}(s,\eta)^2 | | \widetilde{H_4}(s,k-l,\xi-\eta)|\\
   &\times|\widetilde{a}(s,l,\eta)||\widetilde{a}(s,k,\xi)|\mathrm{d} \xi\mathrm{d}\eta\mathrm{d}s.
\end{align*}

Let $(\sigma,\rho)=(k-l,\xi-\eta)$. For $j=0,1$, we get by Lemma \ref{lem:A*R0R1R2-2} (i) and Lemma \ref{lem:4.7JiaHao}  that
\begin{align*}
  \mathcal{V}_j \lesssim_{\delta}&\int_{0}^{t}\sum_{k,l\in\mathbb{Z}}\int_{\mathbb{R}^2} \sqrt{|(A_k^{*}\dot{A}_{k}^*)(s,\xi)|}|\widetilde{a}(s,k,\xi)| \sqrt{|(A_l^{*}\dot{A}_{l}^*)(s,\eta)|}|\widetilde{a}(s,l,\eta)|\\
  &\times \mathbf{1}_{\sigma\neq 0}\cdot\dfrac{|\sigma|\langle s\rangle^2}{|\rho/\sigma|^2+ \langle s\rangle^2} \langle s-\rho/\sigma\rangle^2 A_{\sigma}(s,\rho)|\widetilde{H_4}(s,\sigma,\rho)|\mathrm{e}^{-(\delta_0/201)\langle \sigma,\rho\rangle^{1/2}}\mathrm{d}\xi\mathrm{d}\eta\mathrm{d}s\\
  \lesssim_{\delta}&\left\| \sqrt{|(A_k^{*}\dot{A}_{k}^*)(s,\xi)|}\widetilde{a}(s,k,\xi)\right\|_{L_s^2L^2_{k,\xi}} \cdot \left\|\sqrt{|(A_l^{*}\dot{A}_{l}^*)(s,\eta)|}\widetilde{a}(s,l,\eta)\right\|_{L_s^2L^2_{l,\eta}}\\
  &\times \left\| \mathbf{1}_{\sigma\neq 0}\cdot\dfrac{\sigma\langle s\rangle^2}{|\rho/\sigma|^2+ \langle s\rangle^2} \langle s-\rho/\sigma\rangle^2 \widetilde{H_4}(s,\sigma,\rho)\mathrm{e}^{-(\delta_0/300)\langle \sigma,\rho\rangle^{1/2}}\right\|_{L^{\infty}_{s}L^2_{\sigma,\rho}}
  \lesssim_{\delta}\epsilon_1^3.
\end{align*}
Similarly, for $j=2$ we get by Lemma \ref{lem:A*R0R1R2-2} (ii) and  Lemma \ref{lem:4.7JiaHao}  that
\begin{align*}
  \mathcal{V}_2 \lesssim_{\delta}&\int_{0}^{t}\sum_{k,l\in\mathbb{Z}}\int_{\mathbb{R}^2} \mathbf{1}_{\sigma\neq 0}\cdot\dfrac{|\sigma|\langle s\rangle^2}{|\rho/\sigma|^2+ \langle s\rangle^2} \langle s-\rho/\sigma\rangle^2\sqrt{|(A_{\sigma}\dot{A}_{\sigma})(s,\rho)|} |\widetilde{H_4}(s,\sigma,\rho)| \\
  &\times \sqrt{|(A_k^{*}\dot{A}_{k}^*)(s,\xi)|}|\widetilde{a}(s,k,\xi)|A_{l}(s,\eta) \mathrm{e}^{-(\delta_0/201)\langle l,\eta\rangle^{1/2}}|\widetilde{a}(s,l,\eta)|\mathrm{d}\xi\mathrm{d}\eta\mathrm{d}s\\
  \lesssim_{\delta}&\left\| \sqrt{|(A_k^{*}\dot{A}_{k}^*)(s,\xi)|}\widetilde{a}(s,k,\xi)\right\|_{L_s^2L^2_{k,\xi}} \cdot\left\|A_l^{*}(s,\eta)\mathrm{e}^{-(\delta_0/300)\langle l,\eta\rangle^{1/2}}\widetilde{a}(s,l,\eta)\right\|_{L^{\infty}_sL^2_{l,\eta}}\\
  &\times \left\| \mathbf{1}_{\sigma\neq 0}\cdot\dfrac{\sigma\langle s\rangle^2}{|\rho/\sigma|^2+ \langle s\rangle^2} \langle s-\rho/\sigma\rangle^2\sqrt{|(A_{\sigma}\dot{A}_{\sigma})(s,\rho)|} \widetilde{H_4}(s,\sigma,\rho)\right\|_{L_{s}^{2}L^2_{\sigma,\rho}}
  \lesssim_{\delta}\epsilon_1^3.
\end{align*}
The case $j=3$ is identical to the case $j=2$ by the symmetry. Thus, $\mathcal{V}_j\lesssim_{\delta}\epsilon_1^3$ for all $j\in\{0,1,2,3\}$, and the desired bound \eqref{eq:N7bound} follows.
\end{proof}

\subsection{Nonlinear estimate for $\mathcal{N}_8$}

\begin{lemma}\label{lem:N8bound}
It holds that for any $t\in[1,T]$, we have
  \begin{align}\label{eq:N8bound}
     &\left|2\mathbf{Re}\int_{1}^{t}\sum_{k\in\mathbb{Z}}\int_{\mathbb{R}}A_k^{*}(s,\xi)^2 \widetilde{\mathcal{N}_8}(s,k,\xi)\overline{\widetilde{a}(s,k,\xi)}\mathrm{d}\xi\mathrm{d}s\right| \lesssim_{\delta} \epsilon_1^3.
  \end{align}
\end{lemma}

\begin{proof}
As before, we write
\begin{align*}
   & \left|2\mathbf{Re}\int_{1}^{t}\sum_{k\in\mathbb{Z}}\int_{\mathbb{R}}A_k^*(s,\xi)^2 \widetilde{N_8}(s,k,\xi) \overline{\widetilde{a}(s,k,\xi)}\mathrm{d}\xi\mathrm{d}s\right|\\
   &=2\left| \mathbf{Re}\sum_{k\in\mathbb{Z}}\int_{1}^{t}\int_{\mathbb{R}^2}A_k^*(s,\xi)^2 \widetilde{V_3}(s,\xi-\eta)i\eta\widetilde{a}(s,k,\eta)\overline{\widetilde{a}(s,k,\xi)}\mathrm{d} \xi\mathrm{d}\eta\mathrm{d}s\right|\\
   &=\left| \sum_{k\in\mathbb{Z}}\int_{1}^{t}\int_{\mathbb{R}^2}[\eta
    A_k^*(s,\xi)^2-\xi A_k^{*}(s,\eta)^2 ] \widetilde{V_3}(s,\xi-\eta)\widetilde{a}(s,k,\eta)\overline{\widetilde{a}(s,k,\xi)}\mathrm{d} \xi\mathrm{d}\eta\mathrm{d}s\right|.
\end{align*}

For $i\in\{0,1,2,3\}$, we recall the sets $\Sigma_i$ defined by \eqref{eq:set-Sigma}:
\begin{align*}
  \Sigma_{i}=\{\big((k,\xi),(l,\eta)\big)\in R_i:k=l\}.
\end{align*}
We denote that for  $j\in \{0,1,2,3\}$,
\begin{align*}
   \mathcal{W}'_j=\int_{1}^{t}\sum_{k\in\mathbb{Z}}\int_{\mathbb{R}^2} & \mathbf{1}_{|\rho|\geq 1} \mathbf{1}_{\Sigma_j}\big((k,\xi),(k,\eta)\big)|\eta A_k^*(s,\xi)^2-\xi A_k^{*}(s,\eta)^2 | | \widetilde{V_3}(s,\xi-\eta)|\\
   &\times|\widetilde{a}(s,k,\eta)||\widetilde{a}(s,k,\xi)|\mathrm{d} \xi\mathrm{d}\eta\mathrm{d}s,
\end{align*}
and
\begin{align}\label{eq:W'4}
\begin{aligned}
   \mathcal{W}'_4=\int_{1}^{t}\sum_{k\in\mathbb{Z}}\int_{\mathbb{R}^2} & \mathbf{1}_{|\rho|\leq 1}|\eta A_k^*(s,\xi)^2-\xi A_k^{*}(s,\eta)^2 | | \widetilde{V_3}(s,\xi-\eta)|\\
   &\times|\widetilde{a}(s,k,\eta)||\widetilde{a}(s,k,\xi)|\mathrm{d} \xi\mathrm{d}\eta\mathrm{d}s.
\end{aligned}
\end{align}

Let $\rho=\xi-\eta$. For  $j\in\{0,1\}$, we get by  Lemma \ref{lem:A*S0S1S2-1} (i)  and \eqref{eq:boot-V3} that
\begin{align*}
  \mathcal{W}'_j\lesssim_{\delta} &\int_{1}^{t}\sum_{k\in\mathbb{Z}}\int_{\mathbb{R}^2} \mathbf{1}_{|\rho|\geq 1}\big[\langle \rho\rangle\langle s\rangle+\langle \rho\rangle^{1/4}\langle s\rangle^{7/4}\big]A_{NR}(s,\rho)\mathrm{e}^{-(\delta_0/201)\langle \rho\rangle^{1/2}}|\widetilde{V_3}(s,\rho)|\\
  &\times\sqrt{|(A_k^*\dot{A}_k^*)(s,\eta)|}|\widetilde{a}(s,k,\eta)| \sqrt{|(A_k^*\dot{A}^*_k)(s,\xi)|}|\widetilde{a}(s,k,\xi)|\mathrm{d}\xi\mathrm{d}\eta\mathrm{d}s\\
  \lesssim_{\delta}&\left\|\big[\langle s\rangle+\langle \rho\rangle^{-3/4}\langle s\rangle^{7/4}\big]|\rho|A_{NR}(s,\rho)\mathrm{e}^{-(\delta_0/300)\langle \rho\rangle^{1/2}}\widetilde{V_3}(s,\rho)\right\|_{L^\infty_s L^2_{\rho}}\\
  &\times\left\|\sqrt{|(A_k^*\dot{A}^*_k)(s,\eta)|}\widetilde{a}(s,k,\eta)\right\|_{L^2_sL^2_{k,\eta}} \cdot\left\|\sqrt{|(A_k^*\dot{A}_k^*)(s,\xi)|}\widetilde{a}(s,k,\xi)\right\|_{L^2_sL^2_{k,\xi}}\\
  \lesssim_{\delta}&\epsilon_1^3.
\end{align*}
For $j=2$, we get by  Lemma \ref{lem:A*S0S1S2-1} (ii) and  and \eqref{eq:boot-V3}  that
\begin{align*}
  \mathcal{W}'_j\lesssim_{\delta} &\int_{1}^{t}\sum_{k\in\mathbb{Z}}\int_{\mathbb{R}^2} \mathbf{1}_{|\rho|\geq 1}\big[\langle \rho\rangle\langle s\rangle+\langle \rho\rangle^{1/4}\langle s\rangle^{7/4}\big]\sqrt{|(A_{NR}\dot{A}_{NR})(s,\rho)|}|\widetilde{V_3}(s,\rho)|\\
  &\times A_k^*(s,\eta)\mathrm{e}^{-(\delta_0/201)\langle k,\eta\rangle^{1/2}}|\widetilde{a}(s,k,\eta)| \sqrt{|(A_k^*\dot{A}_k^*)(s,\xi)|}|\widetilde{a}(s,k,\xi)|\mathrm{d}\xi\mathrm{d}\eta\mathrm{d}s\\
  \lesssim_{\delta}&\left\|\big[\langle s\rangle+\langle \rho\rangle^{-3/4}\langle s\rangle^{7/4}\big]|\rho|\sqrt{|(A_{NR}\dot{A}_{NR})(s,\rho)|} \widetilde{V_3}(s,\rho)\right\|_{L^2_s L^2_{\rho}}\\
  &\times\left\|A_k^*(s,\eta)\mathrm{e}^{-(\delta_0/300)\langle k,\eta\rangle^{1/2}}\widetilde{a}(s,k,\eta)\right\|_{L^\infty_sL^2_{k,\eta}} \cdot\left\|\sqrt{|(A_k^*\dot{A}_k^*)(s,\xi)|}\widetilde{a}(s,k,\xi)\right\|_{L^2_sL^2_{k,\xi}}\\
  \lesssim_{\delta}&\epsilon_1^3.
\end{align*}
The case $j=3$ is identical to the case $j=2$ by the symmetry. For $\mathcal{W}'_4$, we get by Lemma \ref{lem:rholeq1} and  $ A_{NR}(s,\rho)\gtrsim_{\delta} 1$ that
\begin{align*}
  \mathcal{W}'_4\lesssim_{\delta} &\int_{1}^{t}\sum_{k\in\mathbb{Z}}\int_{\mathbb{R}^2} \mathbf{1}_{|\rho|\leq 1}\langle s\rangle^{7/4}|\rho|\widetilde{V}_3(s,\rho)\sqrt{|(A_k^*\dot{A}^*_k)(s,\eta)|}|\widetilde{a}(s,k,\eta)| \\
  &\times \sqrt{|(A^*_k\dot{A}^*_k)(s,\xi)|}|\widetilde{a}(s,k,\xi)|\mathrm{d}\xi\mathrm{d}\eta\mathrm{d}s\\
  \lesssim_{\delta}&\left\|\big[\langle s\rangle+\langle \rho\rangle^{-3/4}\langle s\rangle^{7/4}\big]|\rho|A_{NR}(s,\rho)\mathrm{e}^{-(\delta_0/300)\langle \rho\rangle^{1/2}}\widetilde{V_3}(s,\rho)\right\|_{L^\infty_s L^2_{\rho}}\\
  &\times\left\|\sqrt{|(A^*_k\dot{A}^*_k)(s,\eta)|}\widetilde{a}(s,k,\eta)\right\|_{L^2_sL^2_{k,\eta}} \cdot\left\|\sqrt{|(A^*_k\dot{A}^*_k)(s,\xi)|}\widetilde{a}(s,k,\xi)\right\|_{L^2_sL^2_{k,\xi}}\\
  \lesssim_{\delta}&\epsilon_1^3.
\end{align*}

Thus, $\mathcal{W}'_j\lesssim_{\delta}\epsilon_1^3$ for all $j\in\{0,1,2,3,4\}$, and the desired bound \eqref{eq:N8bound} follows.
 \end{proof}

\section{Improved control of the coordinate functions}

 In this section, we prove an improved control for the coordinate functions under the bootstrap assumptions in Proposition \ref{prop:Bootstrap}.

 \begin{proposition}\label{prop:Improved-V1H}
   With the definitions and assumptions in Proposition \ref{prop:Bootstrap}, we have
   \begin{align}\nonumber%\label{est:Improved-V1H}
      &\mathcal{E}_{V_1}(t)+\mathcal{E}_{\mathcal{H}}(t)+\mathcal{B}_{V_1}(t)+\mathcal{B}_{\mathcal{H}}(t)\leq \epsilon_1^2/20+ C\epsilon_1^3\quad \text{for any}\ t\in[1,T].
   \end{align}
 \end{proposition}

\smallskip

  We denote
   \begin{align}\label{eq:L1}
    &\mathcal{L}_1(t)=2\mathbf{Re}\int_{1}^{t}\int_{\mathbb{R}}A_{R}^2(s,\xi)\partial_s \widetilde{V_1}(s,\xi)\overline{\widetilde{V_1}(s,\xi)}\mathrm{d}\xi\mathrm{d}s,
 \end{align}
 and
  \begin{align}\label{eq:L2}
  \begin{aligned}
   \mathcal{L}_2(t)=&K_{\delta}^2 2\mathbf{Re}\int_{1}^{t}\int_{\mathbb{R}}A_{NR}(s,\xi)^2\big( \langle s\rangle /\langle \xi\rangle\big)^{3/2}\partial_s \widetilde{\mathcal{H}}(s,\xi)\overline{\widetilde{\mathcal{H}}(s,\xi)}\mathrm{d}\xi\mathrm{d}s\\
    &+K^2_{\delta} \int_{1}^{t}\int_{\mathbb{R}}A_{NR}(s,\xi)^2\dfrac{3}{2}\big( s\langle s\rangle^{-1/2} \langle \xi\rangle\big)^{-3/2} |\widetilde{\mathcal{H}}(s,\xi)|^2\mathrm{d}\xi\mathrm{d}s.
 \end{aligned}
 \end{align}
As in \cite{IJ},  it suffices to prove that for any $t\in[1,T]$,
 \begin{align}\label{est:Improved-V1H-prove}
  -\mathcal{B}_{V_1}(t)-\mathcal{B}_{\mathcal{H}}(t)+\mathcal{L}_1(t)+ \mathcal{L}_{2}(t)\leq \epsilon_1^2/30.
 \end{align}

 Using the equations \eqref{eq:V1} and \eqref{eq:H}, we extract the quadratic components of $\mathcal{L}_1$ and $\mathcal{L}_2$. We define
\begin{align}\label{eq:L12}
  &\mathcal{L}_{1,2}(t)=2\mathbf{Re} \int_{1}^{t}\int_{\mathbb{R}} \dfrac{A_R(s,\xi)^2}{s}\widetilde{H}(s,\xi) \overline{\widetilde{V_1}(s,\xi)}\mathrm{d}\xi\mathrm{d}s,
\end{align}
and
\begin{align}\label{eq:L22}
  \mathcal{L}_{2,2}(t)&=K_{\delta}^2 \int_{1}^{t}\int_{\mathbb{R}}\left\{ -A_{NR}(s,\xi)^2\dfrac{2\langle s\rangle^{3/2}}{s\langle\xi\rangle^{3/2}}|\widetilde{H}(s,\xi)|^2 +A_{NR}(s,\xi)^2\dfrac{3s/2}{\langle s\rangle^{1/2} \langle\xi\rangle^{3/2}}|\widetilde{H}(s,\xi)|\right\}\mathrm{d}\xi\mathrm{d}s\nonumber\\
  &=-K_{\delta}^2 \int_{1}^{t}\int_{\mathbb{R}} A_{NR}(s,\xi)^2\dfrac{2+s^2/2 }{s\langle\xi\rangle^{3/2}\langle s\rangle^{1/2}}|\widetilde{H}(s,\xi)|^2 \mathrm{d}\xi\mathrm{d}s.
\end{align}
Then the desired bound \eqref{est:Improved-V1H-prove} follows from Lemma \ref{lem:6.2JiaHao} and Lemma \ref{lem:6.3JiaHao-mod}.

\begin{lemma}\cite{IJ}\label{lem:6.2JiaHao}
  For any $t\in[1,T]$, we have
  \begin{align*}
     & -\mathcal{B}_{V_1}(t)-\mathcal{B}_{\mathcal{H}}(t)+\mathcal{L}_{1,2}(t)+ \mathcal{L}_{2,2}(t)\leq \epsilon_1^2/40.
  \end{align*}
\end{lemma}

We denote
\begin{align}\label{eq:F1G124}
  \begin{aligned}
  &F_1=-V_3\partial_vV_1, \quad G_1=-V_3\partial_v\mathcal{H},\\
  &G_2=(V_1+1) \big[- {\mathbb{P}_0(\partial_v\mathbb{P}_{\neq 0}\phi\partial_zf)} +{\mathbb{P}_0 (\partial_z\phi\partial_vf)}\big],\\
   &G_4=(V_1+1) \big[{\mathbb{P}_0 (\partial_vP\partial_za)} -{\mathbb{P}_0( \partial_zP\partial_va)}\big].
  \end{aligned}
\end{align}

The following lemma is devoted to the estimates for the cubic and higher order contributions.

\begin{lemma}\label{lem:6.3JiaHao-mod}
  For any $t\in[1,T]$, we have
  \begin{align}\label{est:F1bound}
     &\left|2\mathbf{Re}\int_{1}^{t}\int_{\mathbb{R}} A_R(s,\xi)^2\widetilde{F_1}(s,\xi)\overline{\widetilde{V_1}(s,\xi)}\mathrm{d}\xi \mathrm{d}s \right|\lesssim_{\delta}\epsilon_1^3,
  \end{align}
  and
  \begin{align}
     & \left|2\mathbf{Re}\int_{1}^{t}\int_{\mathbb{R}} A_{NR}(s,\xi)^2\big(\langle s\rangle/\langle\xi\rangle\big)^{3/2}\widetilde{G_j}(s,\xi) \overline{\widetilde{\mathcal{H}}(s,\xi)}\mathrm{d}\xi \mathrm{d}s \right|\lesssim_{\delta}\epsilon_1^3,\qquad j\in\{1,2\}\label{est:G12bound}\\
     & \left|2\mathbf{Re}\int_{1}^{t}\int_{\mathbb{R}} A_{NR}(s,\xi)^2\big(\langle s\rangle/\langle\xi\rangle\big)^{3/2}\widetilde{G_4}(s,\xi) \overline{\widetilde{\mathcal{H}}(s,\xi)}\mathrm{d}\xi \mathrm{d}s \right|\lesssim_{\delta}\epsilon_1^3.\label{est:G4bound}
  \end{align}
\end{lemma}

\begin{proof}

The bound  \eqref{est:G12bound} for $j=2$ was proved in \cite{IJ}(see Lemma 6.3).\smallskip

\noindent{\bf Step 1.} Proof of  \eqref{est:F1bound}.\smallskip

  We write
  \begin{align*}
     &\left|2\mathbf{Re}\int_{1}^{t}\int_{\mathbb{R}} A_R(s,\xi)^2\widetilde{F_1}(s,\xi)\overline{\widetilde{V_1}(s,\xi)}\mathrm{d}\xi \mathrm{d}s \right|\\
     &=\left|\int_{1}^{t}\int_{\mathbb{R}}\int_{\mathbb{R}} \big[\eta A_R(s,\xi)^2-\xi A_R(s,\eta)^2\big]\widetilde{V_3}(s,\xi-\eta)\widetilde{V_1}(s,\eta) \overline{\widetilde{V_1}(s,\xi)}\mathrm{d}\xi\mathrm{d}\eta \mathrm{d}s \right|.
  \end{align*}
  We define the sets
  \begin{align}\label{eq:Si}
  \begin{aligned}
   S_0=\Big\{&(k,\eta)\in\mathbb{R}^2: \min(\langle \xi\rangle,\langle \eta\rangle,\langle \xi-\eta\rangle)\geq \dfrac{\langle \xi\rangle+\langle \eta\rangle+\langle \xi-\eta\rangle}{20}\Big\},\\
   S_1=\Big\{&(k,\eta)\in\mathbb{R}^2: \langle \xi-\eta\rangle\leq \dfrac{\langle k,\xi\rangle+\langle l,\eta\rangle+\langle \xi-\eta\rangle}{10}\Big\},\\
   S_2=\Big\{&(k,\eta)\in\mathbb{R}^2: \langle \eta\rangle\leq \dfrac{\langle \xi\rangle+\langle \eta\rangle+\langle \xi-\eta\rangle}{10}\Big\},\\
   S_3=\Big\{&(k,\eta)\in\mathbb{R}^2: \langle \xi\rangle\leq \dfrac{\langle \xi\rangle+\langle \eta\rangle+\langle \xi-\eta\rangle}{10}\Big\},
  \end{aligned}
  \end{align}
  and then we denote that  for $i=0, 1, 2, 3$,
   \begin{align*}
      \mathcal{I}_i=\int_{1}^{t}\int_{\mathbb{R}^2} &  \mathbf{1}_{S_i}(\xi,\eta)|\eta A_R(s,\xi)^2-\xi A_R(s,\eta)^2| | \widetilde{V_3}(s,\xi-\eta)||\widetilde{V_1}(s,\eta)||\widetilde{V_1}(s,\xi)|\mathrm{d} \xi\mathrm{d}\eta\mathrm{d}s.
   \end{align*}
\if0   and
      \begin{align*}
      \mathcal{I}_4=\int_{1}^{t}\int_{\mathbb{R}^2} & \mathbf{1}_{|\xi-\eta|\leq1} |\eta A_R^2(s,\xi)-\xi A_R^{2}(s,\eta) | | \widetilde{V_3}(s,\xi-\eta)||\widetilde{V_1}(s,\eta)||\widetilde{V_1}(s,\xi)|\mathrm{d} \xi\mathrm{d}\eta\mathrm{d}s.
   \end{align*}\fi
   It suffices to prove that for $i\in\{0,1,2,3\}$,
   \begin{align}\label{est:F1bound-prove1}
      \mathcal{I}_i\lesssim_{\delta}\epsilon_1^3.
   \end{align}

   Let $\rho=\xi-\eta$. By \eqref{est:8.9JiaHao-1}, we have
   \begin{align*}
      &\mathbf{1}_{(\xi,\eta)\in S_0\cup S_1}\left|\eta A_{R}(s,\xi)^2- \xi A_{R}(s,\eta)^2\right|\\
     &\lesssim_{\delta} s^{1.6}|\rho| \sqrt{|(A_R\dot{A}_R)(s,\xi)|} \sqrt{|(A_R\dot{A}_R)(s,\eta)|} A_{NR}(s,\rho)\mathrm{e}^{-(\lambda(s)/40)\langle \rho\rangle^{1/2}}.
   \end{align*}
\if0 and we get by Lemma \ref{lem:8.9JIaHao-mod} and $\mathbf{1}_{|\rho|\leq 1}A_{NR}(t,\rho)\gtrsim_{\delta}1$ that
     \begin{align*}
      &\mathbf{1}_{|\rho|\leq 1}\left|\eta A^2_{R}(s,\xi)- \xi A_{R}^2(s,\eta) \right|\\
    & \lesssim_{\delta}\mathbf{1}_{|\rho|\leq 1}s^{1.6}|\rho| \sqrt{|(A_R\dot{A}_R)(s,\xi)|} \sqrt{|(A_R\dot{A}_R)(s,\eta)|} \\
    & \lesssim_{\delta} s^{1.6}|\rho| \sqrt{|(A_R\dot{A}_R)(s,\xi)|} \sqrt{|(A_R\dot{A}_R)(s,\eta)|} A_{NR}(s,\rho)\mathrm{e}^{-(\lambda(s)/40)\langle \rho\rangle^{1/2}}.
   \end{align*}\fi
   Therefore, for $j=0,1$, we have
   \begin{align*}
      \mathcal{I}_j\lesssim_{\delta}&\left\| \sqrt{|(A_R\dot{A}_R)(s,\xi)|} \widetilde{V_1}(s,\xi)\right\|_{L^2_sL^2_{\xi}} \left\| \sqrt{|(A_R\dot{A}_R)(s,\eta)|} \widetilde{V_1}(s,\eta)\right\|_{L^2_sL^2_{\eta}}\\
      &\times \left\| s^{1.6}A_{NR}(s,\rho)|\rho|\mathrm{e}^{-(\lambda(s)/50)\langle \rho\rangle^{1/2}}\widetilde{V_3}(s,\rho)\right\|_{L^\infty_sL^2_{\rho}},
  \end{align*}
  and the desired estimate \eqref{est:F1bound-prove1}($j=0,1$) follows from \eqref{eq:boot-V1} and \eqref{eq:boot-V3}.\smallskip

If $(\xi,\eta)\in S_2$, then $10\langle \eta\rangle\leq \langle \xi\rangle+\langle \eta\rangle+\langle \rho\rangle\leq 2(\langle \eta\rangle+\langle \rho\rangle)$, $\langle \rho\rangle\geq4 \langle \eta\rangle\geq4$, $ |\rho|\geq 3$, $\langle\xi\rangle\leq \langle \eta\rangle+\langle \rho\rangle\leq \frac{5}{4}\langle \rho\rangle\leq 2|\rho|$. By \eqref{est:8.9JiaHao-2}, we have
  \begin{align*}
      &\mathbf{1}_{(\xi,\eta)\in S_2}\left|\eta A_{R}(s,\xi)^2- \xi A_{R}(s,\eta)^2\right|\lesssim\mathbf{1}_{(\xi,\eta)\in S_2}\langle\eta\rangle A_{R}(s,\xi)^2\\
     &\lesssim_{\delta}s^{1.1}\langle\xi\rangle^{0.6} \sqrt{|(A_R\dot{A}_R)(s,\xi)|} \sqrt{|(A_{NR}\dot{A}_{NR})(s,\rho)|} A_{R}(s,\eta)\mathrm{e}^{-(\lambda(s)/40)\langle \eta\rangle^{1/2}}\\
     &\lesssim_{\delta}s^{1.1}|\rho|\langle\rho\rangle^{-0.4} \sqrt{|(A_R\dot{A}_R)(s,\xi)|} \sqrt{|(A_{NR}\dot{A}_{NR})(s,\rho)|} A_{R}(s,\eta)\mathrm{e}^{-(\lambda(s)/40)\langle \eta\rangle^{1/2}},
   \end{align*}
which gives
      \begin{align*}
      \mathcal{I}_2\lesssim_{\delta}&\left\| \sqrt{|(A_R\dot{A}_R)(s,\xi)|} \widetilde{V_1}(s,\xi)\right\|_{L^2_sL^2_{\xi}} \left\| s^{1.1}|\rho|\langle\rho\rangle^{-0.4}\sqrt{|(A_R\dot{A}_R)(s,\rho)|} \widetilde{V_3}(s,\rho)\right\|_{L^2_sL^2_{\rho}}\\
      &\times \left\| A_{R}(s,\eta)\mathrm{e}^{-(\lambda(s)/50)\langle \eta\rangle^{1/2}}\widetilde{V_1}(s,\eta)\right\|_{L^\infty_sL^2_{\eta}},
  \end{align*}
  and the desired estimate \eqref{est:F1bound-prove1}($j=2$) follows from \eqref{eq:boot-V1} and \eqref{eq:boot-V3}.
   The argument for  $j=3$ is similar to $j=2$.\smallskip

 \noindent{\bf Step 2.} Proof of  \eqref{est:G12bound} for $j=1$.\smallskip

  We write
  \begin{align*}
     &\left|2\mathbf{Re}\int_{1}^{t}\int_{\mathbb{R}} A_{NR}(s,\xi)^2\big(\langle s\rangle/\langle\xi\rangle\big)^{3/2}\widetilde{G_1}(s,\xi) \overline{\widetilde{\mathcal{H}}(s,\xi)}\mathrm{d}\xi \mathrm{d}s \right|= \Big| \int_{1}^{t}\int_{\mathbb{R}^2}\langle s\rangle^{3/2}\\
     &\quad\times \big[\eta A_{NR}(s,\xi)^2\langle\xi\rangle^{-3/2}-\xi A_{NR}(s,\eta)^2\langle\eta\rangle^{-3/2}\big]\widetilde{V_3}(s,\xi-\eta)
     \widetilde{\mathcal{H}}(s,\eta) \overline{\widetilde{\mathcal{H}}(s,\xi)}\mathrm{d}\xi\mathrm{d}\eta \mathrm{d}s \Big|.
  \end{align*}
  We denote that for $j=0,1,2,3$,
  \begin{align}\label{eq:Jn0123}
  \begin{aligned}
  \mathcal{J}_j=\int_{1}^{t}\int_{\mathbb{R}^2}\mathbf{1}_{(\xi,\eta)\in S_j} \langle s\rangle^{3/2}\big|&\eta A_{NR}(s,\xi)^2\langle\xi\rangle^{-3/2}-\xi A_{NR}(s,\eta)^2\langle\eta\rangle^{-3/2}\big|\\
  &\times|\widetilde{V_3}(s,\xi-\eta)||
     \widetilde{\mathcal{H}}(s,\eta) || \widetilde{\mathcal{H}}(s,\xi)|\mathrm{d}\xi\mathrm{d}\eta \mathrm{d}s.
  \end{aligned}
  \end{align}
\if0 and
    \begin{align}\label{eq:J4}
    \begin{aligned}
  \mathcal{J}_4=\int_{1}^{t}\int_{\mathbb{R}^2}\mathbf{1}_{|\xi-\eta|\leq 1} \langle s\rangle^{3/2}\big|&\eta A_{NR}^2(s,\xi)\langle\xi\rangle^{-3/2}-\xi A_{NR}^2(s,\eta)\langle\eta\rangle^{-3/2}\big|\\
  &\times|\widetilde{V_3}(s,\xi-\eta)||
     \widetilde{\mathcal{H}}(s,\eta) || \widetilde{\mathcal{H}}(s,\xi)|\mathrm{d}\xi\mathrm{d}\eta \mathrm{d}s.
  \end{aligned}
  \end{align}\fi
  Let $\rho=\xi-\eta$.  It suffices to prove that for $j\in\{0,1,2,3\}$
  \begin{align}\label{est:G12bound1-prove1}
     & \mathcal{J}_i\lesssim_{\delta} \epsilon_1^3.
  \end{align}

 It follows from \eqref{est:8.9JiaHao-1}(taking $\alpha=3/2$) that
       \begin{align*}
      \begin{aligned}
    &\mathbf{1}_{(\xi,\eta)\in S_0\cup S_1}\left|\eta A_{NR}^2(t,\xi)\langle\xi\rangle^{-3/2}- \xi A_{NR}(t,\eta)^2\langle\eta\rangle^{-3/2}\right|\\
    &\lesssim_{\delta}t^{1.6}|\rho|\dfrac{\sqrt{|(A_{NR}\dot{A}_{NR})(t,\xi)|} }{\langle\xi\rangle^{3/4}} \dfrac{\sqrt{|(A_{NR}\dot{A}_{NR})(t,\eta)|}}{\langle\eta\rangle^{3/4}}\cdot A_{NR}(t,\rho)\mathrm{e}^{-(\lambda(t)/40)\langle \rho\rangle^{1/2}}.
  \end{aligned}
  \end{align*}
\if0  and we get by Lemma \ref{lem:8.9JIaHao-mod}(taking $\alpha=3/2$) and $\mathbf{1}_{|\rho|\leq 1}A_{NR}(t,\rho)\gtrsim_{\delta}1$ that
     \begin{align*}
      &\mathbf{1}_{|\rho|\leq 1}\left|\eta A^2_{NR}(s,\xi)\langle\xi\rangle^{-3/2}- \xi A_{NR}^2(s,\eta)\langle\eta\rangle^{-3/2} \right|\\
     &\lesssim_{\delta}\mathbf{1}_{|\rho|\leq 1}s^{1.6}|\rho| \dfrac{\sqrt{|(A_{NR}\dot{A}_{NR})(s,\xi)|}}{\langle\xi\rangle^{3/4}} \dfrac{\sqrt{|(A_{NR}\dot{A}_{NR})(s,\eta)|}}{\langle\eta\rangle^{3/4}} \\
     &\lesssim_{\delta}s^{1.6}|\rho| \dfrac{\sqrt{|(A_{NR}\dot{A}_{NR})(s,\xi)|}}{\langle\xi\rangle^{3/4}} \dfrac{\sqrt{|(A_{NR}\dot{A}_{NR})(s,\eta)|}}{\langle\eta\rangle^{3/4}} A_{NR}(s,\rho)\mathrm{e}^{-(\lambda(s)/40)\langle \rho\rangle^{1/2}}.
   \end{align*}\fi
   Therefore, we deduce that for $j=0,1$,
  \begin{align*}
     \mathcal{J}_j\lesssim_{\delta}&  \left\| s^{3/4}\dfrac{\sqrt{|(A_{NR}\dot{A}_{NR})(t,\xi)|} }{\langle\xi\rangle^{3/4}}\widetilde{\mathcal{H}}(s,\xi)\right\|_{L^2_sL^2_\xi}\left\| s^{3/4}\dfrac{\sqrt{|(A_{NR}\dot{A}_{NR})(t,\eta)|}}{\langle\eta\rangle^{3/4}} \widetilde{\mathcal{H}}(s,\eta)\right\|_{L^2_sL^2_{\eta}}\\
     &\times\left\| s^{1.6} |\rho|A_{NR}(t,\rho)\mathrm{e}^{-(\lambda(t)/50)\langle \rho\rangle^{1/2}}\widetilde{V_3}(s,\rho)\right\|_{L^\infty_sL^2_{\rho}},
  \end{align*}
  and the desired conclusion \eqref{est:G12bound1-prove1}($j\in\{0,1\}$) follows from \eqref{eq:boot-H} and \eqref{eq:boot-V3}.

If $(\xi,\eta)\in S_2$, then $ |\rho|\geq 3$. It follows from \eqref{est:8.9JiaHao-3} that
\begin{align*}
   &\left|\eta A_{NR}(s,\xi)^2\langle\xi\rangle^{-3/2}- \xi A_{NR}(s,\eta)^2\langle\eta\rangle^{-3/2} \right| \lesssim_{\delta}\langle \eta\rangle A_{NR}(t,\xi)^2\langle\xi\rangle^{-3/2}\\
&\lesssim_{\delta}t^{1.1}\langle \xi\rangle^{-1.9}\sqrt{|(A_{NR}\dot{A}_{NR})(t,\xi)|} \sqrt{|(A_{NR}\dot{A}_{NR})(t,\rho)|} A_{NR}(t,\eta)\mathrm{e}^{-(\lambda(t)/40)\langle \eta\rangle^{1/2}}\\
&\lesssim_{\delta}t^{1.1}\sqrt{|(A_{NR}\dot{A}_{NR})(t,\xi)|} \sqrt{|(A_{NR}\dot{A}_{NR})(t,\rho)|} A_{NR}(t,\eta)\mathrm{e}^{-(\lambda(t)/40)\langle \eta\rangle^{1/2}}\\
&\lesssim_{\delta}t^{1.1}|\rho|\langle \rho\rangle^{-0.1}\sqrt{|(A_{NR}\dot{A}_{NR})(t,\xi)|} \sqrt{|(A_{NR}\dot{A}_{NR})(t,\rho)|} A_{NR}(t,\eta)\mathrm{e}^{-(\lambda(t)/41)\langle \eta\rangle^{1/2}}
\end{align*}
for any $(\xi,\eta)\in S_2$. Therefore, we arrive at
\begin{align*}
   \mathcal{J}_2\lesssim_{\delta}&  \left\| \sqrt{|(A_{NR}\dot{A}_{NR})(t,\xi)|} \widetilde{\mathcal{H}}(s,\xi)\right\|_{L^2_sL^2_\xi}\left\| s^{1.1} |\rho|\langle \rho\rangle^{-0.1} \sqrt{|(A_{NR}\dot{A}_{NR})(t,\rho)|}\widetilde{V_3}(s,\rho)\right\|_{L^2_sL^2_{\rho}}\\
     &\times\left\| A_{NR}(t,\eta) \mathrm{e}^{-(\lambda(t)/50)\langle \eta\rangle^{1/2}}\widetilde{\mathcal{H}}(s,\eta)\right\|_{L^\infty_sL^2_{\eta}},
\end{align*}
and the desired bound  \eqref{est:G12bound1-prove1}($j=2$) follows from \eqref{eq:boot-H} and \eqref{eq:boot-V3}.
    The proof for  $j=3$ is similar to $j=2$.\smallskip

\noindent{\bf Step 3.} Proof of \eqref{est:G4bound}.\smallskip

  By the Cauchy-{Schwarz} inequality, we have
  \begin{align*}
    \Big|\int_{1}^{t}\int_{\mathbb{R}} A_{NR}(s,\xi)^2 &\big(\langle s\rangle /\langle \xi\rangle\big)^{3/2}\widetilde{G_4}(s,\xi) \overline{\widetilde{\mathcal{H}}(s,\xi)}\mathrm{d}\xi\mathrm{d}s\Big|^2\\
    &\lesssim \mathcal{B}_{\mathcal{H}}(t)\int_{1}^{t}\int_{\mathbb{R}} |\dot{A}_{NR}(s,\xi)|^{-1}A_{NR}(s,\xi)^3\big(\langle s\rangle/\langle\xi\rangle\big)^{3/2}|\widetilde{G_4}(s,\xi)|^2\mathrm{d}\xi \mathrm{d}s.
  \end{align*}
  In view of the bootstrap assumption on $\mathcal{B}_{\mathcal{H}}(t)$, it suffices to prove that
  \begin{align}\label{est:G4bound-prove}
     \int_{1}^{t}\int_{\mathbb{R}} |\dot{A}_{NR}(s,\xi)|^{-1}A_{NR}(s,\xi)^3\big(\langle s\rangle/\langle\xi\rangle\big)^{3/2}|\widetilde{G_4}(s,\xi)|^2\mathrm{d}\xi \mathrm{d}s\lesssim _{\delta} \epsilon_1^4.
  \end{align}

  Let $G_4^{(0)}={\mathbb{P}_0(\partial_vP\partial_za) -\mathbb{P}_0 (\partial_zP\partial_va)=\mathbb{P}_0 \big[ \partial_z(\partial_vPa)-\partial_v(\partial_zPa)\big] =-\partial_v\mathbb{P}_0 (\partial_zPa)}$.
\if0Notice that
  \begin{align*}
    G_4^{(0)}(t,v) &=\dfrac{1}{2\pi}\int_{\mathbb{T}} \big[ \partial_vP(t,z,v)\partial_za(t,z,v) - \partial_zP(t,z,v)\partial_va(t,z,v)\big] \mathrm{d}z\\
    &=C\sum_{k\in \mathbb{Z}} \int_{\mathbb{R}} \mathrm{e}^{\mathrm{i}v(\rho+\eta)}  \big[ \widetilde{\partial_vP}(t,k,\rho)\widetilde{\partial_za}(t,-k,\eta) - \widetilde{\partial_zP}(t,k,\rho)\widetilde{\partial_va}(t,-k,\eta)\big] \mathrm{d}\rho\mathrm{d}\eta,
  \end{align*}
which gives
  \begin{align}\nonumber%\label{eq:G40tilde}
  \begin{aligned}
     \widetilde{G_4^{(0)}}(t,\xi)& =C\sum_{k\in\mathbb{Z}}\int_{\mathbb{R}} k\xi\cdot\widetilde{P}(t,k,\xi-\eta)\widetilde{a}(t,-k,\eta)\mathrm{d}\eta\\
     &=C\sum_{k\in\mathbb{Z}\setminus\{0\}}\int_{\mathbb{R}} k\xi\cdot\widetilde{P}(t,k,\eta)\widetilde{a}(t,-k,\xi-\eta)\mathrm{d}\eta.
     \end{aligned}
  \end{align}\fi
We first prove that for any $t\in[1,T]$,
  \begin{align}\label{est:G40bound-prove1}
    & \int_{\mathbb{R}}|\dot{A}_{NR}|^{-2}A_{NR}(t,\xi)^4\big( \langle t\rangle/\langle \xi\rangle\big)^{3/2}|\widetilde{G_4^{(0)}}(t,\xi)|^2\mathrm{d}\xi \lesssim_{\delta} \epsilon_1^4,
  \end{align}
  and
  \begin{align}\label{est:G40bound-prove2}
    & \int_{1}^{t}\int_{\mathbb{R}}|\dot{A}_{NR}(s,\xi)|^{-1}A_{NR}(t,\xi)^3\big( \langle s\rangle/\langle \xi\rangle\big)^{3/2}|\widetilde{G_4^{(0)}}(s,\xi)|^2\mathrm{d}\xi \mathrm{d}s \lesssim_{\delta}\epsilon_1^4.
  \end{align}
 For this,  we use the following multiplier bounds
   \begin{align}\label{est:G4bound-mul-1}
   \begin{aligned}
     \dfrac{A_{NR}(t,\xi)^2}{|\dot{A}_{NR}(t,\xi)|}\dfrac{\langle t\rangle^{3/4}}{\langle\xi\rangle^{3/4}}|\xi|\lesssim_{\delta} & A_{k}(t,\eta)\dfrac{\langle t\rangle \big(\langle t\rangle+|(k,\eta)|\big)\langle t-\eta/k\rangle^2}{\langle\eta/k\rangle^2}\\
     &\times A_{-k}^{*}(t,\xi-\eta) \left\{\langle\xi-\eta\rangle^{-2}+\langle\eta\rangle^{-2}\right\},
   \end{aligned}
   \end{align}
   and
   \begin{align}\label{est:G4bound-mul-2}
   \begin{aligned}
     \dfrac{A_{NR}(t,\xi)^{3/2}}{|\dot{A}_{NR}(t,\xi)|^{1/2}}&\dfrac{\langle t\rangle^{3/4}}{\langle\xi\rangle^{3/4}}|\xi|\lesssim_{\delta} \left[|(\dot{A}_k/A_k)(t,\eta)|^{1/2} +|(\dot{A}^*_{-k}/A^*_{-k})(t,\xi-\eta)|^{1/2}\right]A_{k}(t,\eta)\\ & \times \dfrac{\langle t\rangle \big(\langle t\rangle+|(k,\eta)|\big)\langle t-\eta/k\rangle^2}{\langle \eta/k\rangle^2}A_{-k}^{*}(t,\xi-\eta) \left\{\langle\xi-\eta\rangle^{-2}+\langle\eta\rangle^{-2}\right\},
   \end{aligned}\end{align}
   for any $t\in[1,T],\ k\in\mathbb{Z}\setminus\{0\}$, and $\xi,\eta\in\mathbb{R}$. The estimate \eqref{est:G4bound-mul-1} follows from \eqref{est:8.7JiaHao-mod}, while the estimate \eqref{est:G4bound-mul-2} follows from \eqref{est:8.7JiaHao-mod}, \eqref{est:8.7JiaHao-2} and \eqref{est:par-t-A*-1}. By Lemma \ref{lem:8.1JiaHao} with $g=\partial_zP$, $h=a$, \eqref{est:G40bound-prove1} and \eqref{est:G40bound-prove2} follow from the multiplier bounds \eqref{est:G4bound-mul-1} and \eqref{est:G4bound-mul-2}.

%   \textbf{Step 2.} We prove now similar bounds on the function $G_4$
%   \begin{align}\label{est:G4bound-prove1}
%    & \int_{\mathbb{R}}|\dot{A}_{NR}|^{-2}A_{NR}^4(t,\xi)\big( \langle t\rangle/\langle \xi\rangle\big)^{3/2}
%|\widetilde{G_4}(t,\xi)|^2\mathrm{d}\xi \lesssim_{\delta} \epsilon_1^4,
%  \end{align}
%  and

Next we prove \eqref{est:G4bound-prove}.
 Notice that $G_4=V_1G_4^{(0)}+G_4^{(0)}$. In view of  Lemma \ref{lem:8.1JiaHao}, \eqref{est:G40bound-prove1} and \eqref{est:G40bound-prove2},
 it suffices to prove the following multiplier estimate
%  \begin{align}\label{est:G4bound-mul-3}
%     &\dfrac{A_{NR}^2(t,\xi)}{|\dot{A}_{NR}(t,\xi)|}\dfrac{\langle t\rangle^{3/4}}{\langle\xi\rangle^{3/4}}\lesssim_{\delta}
%\dfrac{A_{NR}^2(t,\eta)}{|\dot{A}_{NR}(t,\eta)|}\dfrac{\langle t\rangle^{3/4}}{\langle\eta\rangle^{3/4}}A_{R}(t,\xi-\eta)
%\left\{\langle\xi-\eta\rangle^{-2}+\langle\eta\rangle^{-2}\right\},
%   \end{align}
%   and
   \begin{align}\nonumber%\label{est:G4bound-mul-4}
   \begin{aligned}
     \dfrac{A_{NR}(t,\xi)^{3/2}}{|\dot{A}_{NR}(t,\xi)|^{1/2}}&\dfrac{\langle t\rangle^{3/4}}{\langle\xi\rangle^{3/4}}\lesssim_{\delta} \left[|(\dot{A}_{NR}/A_{NR})(t,\eta)|^{1/2} +|(\dot{A}_{R}/A_R)(t,\xi-\eta)|^{1/2}\right]\\ & \times \dfrac{A_{NR}(t,\eta)^2}{|\dot{A}_{NR}(t,\eta)|}\dfrac{\langle t\rangle^{3/4}}{\langle\eta\rangle^{3/4}}A_{R}(t,\xi-\eta) \left\{\langle\xi-\eta\rangle^{-2}+\langle\eta\rangle^{-2}\right\},
   \end{aligned}
   \end{align}
 which can be deduced by using  \eqref{est:8.8JiaHao}, \eqref{est:8.2JiaHao-2}, \eqref{est:7.4JiaHao-5} and $A_{NR}(t,\rho)\leq A_R(t,\rho)$.
 \end{proof}

\appendix

\section{Main weights and basic properties}\label{sec:mainweights}

\subsection{Main weights}

Let us recall some main weights in \cite{IJ}, which are introduced to match nonlinear transient growth of the system.

For fixed $\delta_0>0$, we define the function $\lambda :[0,\infty)\rightarrow [\delta_0,3\delta_0/2]$ by
\begin{align}\label{eq:lambda}
   &\lambda(0)=\dfrac{3\delta_0}{2},\quad \lambda'(t)=-\dfrac{\delta_0\sigma_0^2}{\langle t\rangle^{1+\sigma_0}},
\end{align}
for small positive constant $\sigma_0$.\if0 and define
\begin{align}\label{eq:sigma0}
   &\sigma_1=2\sigma_0 .
\end{align}\fi

Take small $\delta >0$ with $\delta\ll \delta_0$.  For $\eta >\delta^{-10}$, we define $k_0(\eta)=\lfloor \sqrt{\delta^3\eta}\rfloor$.
For $l\in\{1,...,k_0(\eta)\}$, we define
 \begin{align*}
    &t_{l,\eta}=\dfrac{1}{2}\Big(\dfrac{\eta}{l+1}+\dfrac{\eta}{l}\Big),\quad t_{0,\eta}=2\eta,\quad I_{l,\eta}=[t_{l,\eta},t_{l-1,\eta}].
 \end{align*}
  Notice that $|I_{l,\eta}|\approx \f{\eta}{l^2}$ and
 \begin{align}\label{eq:(7.2)JiaHao}
   &\delta^{-3/2}\sqrt{\eta}/2\leq t_{k_0(\eta),\eta}\leq ...\leq t_{l,\eta} \leq \eta/l \leq t_{l-1,\eta}\leq ...\leq t_{0,\eta}=2\eta.
 \end{align}

Now we define the weights $w_{NR}$ and $w_{R}$.  For $|\eta|\leq \delta^{-10}$,  we define
\begin{align}\nonumber
   &w_{NR}(t,\eta)=1,\quad w_{R}(t,\eta)=1;
 \end{align}
For $\eta >\delta^{-10}$, if $t\geq t_{0,\eta}=2\eta$, then we define
\begin{align}\nonumber
  w_{NR}(t,\eta)=1,\quad w_{R}(t,\eta)=1;
\end{align}
For $k\in\{1,...,k_{0}(\eta)\}$, we define
\begin{align}\nonumber
   \begin{aligned}
      &w_{NR}(t,\eta)=\left(\dfrac{1+\delta^2|t-\eta/k|}{1+\delta^2|t_{k-1,\eta}-\eta/k|} \right)^{\delta_{0}}w_{NR}(t_{k-1,\eta},\eta)\quad \text{if}\ t\in [\eta/k,t_{k-1,\eta}],\\
      &w_{NR}(t,\eta)=\left( \dfrac{1}{1+\delta^2|t-\eta/k|}\right)^{1+\delta_0} w_{NR}(\eta/l,\eta)\quad \text{if}\ t\in [t_{k,\eta},\eta/k],
   \end{aligned}
\end{align}
and
\begin{align}\nonumber
   w_{R}(t,\eta)=\left\{\begin{aligned}
   &w_{NR}(t,\eta)\dfrac{1+\delta^2|t-\eta/k|}{1+\delta^2\eta/(8k^2)}\quad \text{if} \ |t-\eta/k|\leq \eta/(8k^2),\\
   &w_{NR}(t,\eta)\qquad \text{if}\ t\in I_{k,\eta},\,|t-\eta/k|\geq \eta/(8k^2).
   \end{aligned}\right.
\end{align}
For $t\leq t_{k_0(\eta),\eta}$, we define
\begin{align}\label{eq:(7.11)JiaHao}
   w_{NR}(t,\eta)=w_{R}(t,\eta)=\big(\mathrm{e}^{-\delta\sqrt{\eta}}\big)^{\beta} w_{NR}(t_{k_0(\eta),\eta},\eta)^{1-\beta}
\end{align}
if $t=(1-\beta)t_{k_0(\eta),\eta}$, $\beta\in[0,1]$.

For $\eta<-\delta^{-10}$,  we define
\beno
w_{R}(t,\eta)=w_{R}(t,|\eta|),\quad w_{NR}(t,\eta)=w_{NR}(t,|\eta|),
\eeno
 and the resonant intervals $I_{k,\eta}=I_{-k,-\eta}$. Thus, the resonant intervals $I_{k,\eta}$ are defined for $(k,\eta)\in\mathbb{Z}\times\mathbb{R}$ satisfying $|\eta|>\delta^{-10},1\leq |k|\leq \sqrt{\delta^3|\eta|}$, and $\eta/k>0$. {{(Otherwise $I_{k,\eta}=\emptyset$)}}.\smallskip

 We have the following basic properties for the weights:  for $t\in I_{k,\eta}$, we have
\begin{align}\label{eq:(7.7)JiaHao}
   &w_{R}(t,\eta)\approx w_{NR}(t,\eta)\left[\dfrac{k^2}{\delta^2\eta}(1+\delta^2|t-\eta/k|)\right],
\end{align}
and
\begin{align}\label{eq:(7.8)JiaHao}
  & \dfrac{\partial_tw_{NR}(t,\eta)}{w_{NR}(t,\eta)}\approx  \dfrac{\partial_tw_{R}(t,\eta)}{w_{R}(t,\eta)}\approx \dfrac{\delta^2}{1+\delta^2|t-\eta/k|}.
\end{align}

%\item It holds that
%\begin{align}\label{eq:(7.6)JiaHao}
%  &\dfrac{w_{NR}(t_{k,\eta},\eta)}{w_{NR}(t_{k-1,\eta},\eta)} \approx \left(\dfrac{k^2}{\delta^2\eta}\right)^{1+2\delta_0},
%\qquad w_{R}(t_{k,\eta},\eta)=w_{NR}(t_{k,\eta},\eta).
%\end{align}

%In particular, we have
%\begin{align}\label{eq:(7.10)JiaHao}
%&w_{NR}(t_{k_0(\eta),\eta},\eta)=w_{R}(t_{k_0(\eta),\eta},\eta)\in [X_{\delta}(\eta)^4,X_{\delta}(\eta)^{1/4}],
%\end{align}
%where $X_{\delta}(\eta)=\mathrm{e}^{-\delta^{3/2}\ln(\delta^{-1})\sqrt{\eta}}$.

%\item For any  $t_1\in[0,t_{k_0(\eta),\eta}],\, t_2\in[0,+\infty)$, we have
%\begin{align}\label{eq:(7.12)JiaHao}
%\dfrac{w_{NR}(t_1,\eta)}{w_{NR}(t_2,\eta)}\lesssim \mathrm{e}^{4\delta^{5/2}|t_1-t_2|}.
%\end{align}

Next we define the weights $w_k(t,\eta)$, which crucially distinguish the way resonant and nonresonant modes grow around the critical times $\eta/k$, by the formula
\begin{align}\label{eq:(7.13)JiaHao}
 w_k(t,\eta)=\left\{\begin{aligned}
 &w_{NR}(t,\eta)\quad \text{if}\quad t\notin I_{k,\eta},\\
 &w_{R}(t,\eta)\quad \text{if}\quad t\in I_{k,\eta}.
 \end{aligned}\right.
\end{align}

 Fix $\varphi:\mathbb{R}\rightarrow [0,1]$ an even smooth function supported in $[-8/5,6/5]$ and equal to $1$ in $[-5/4,5/4]$ and let $d_0=\int_{\mathbb{R}}\varphi(x)\mathrm{d}x$. For $k\in\mathbb{Z}$ and $Y\in\{NR,R,k\}$, let
\begin{align}\label{eq:(7.14)JiaHao}
\begin{aligned}
&b_{Y}(t,\xi)=\int_{\mathbb{R}}w_{Y}(t,\rho)\varphi \left(\dfrac{\xi-\rho}{L_{\kappa}(t,\xi)}\right)\dfrac{1}{d_0L_{\kappa}(t,\xi)}\mathrm{d}\rho,\\
&L_{\kappa}(t,\xi)=1+\dfrac{\kappa\langle\xi\rangle}{\langle\xi\rangle^{1/2}+\kappa t}.
\end{aligned}
\end{align}
Here $\kappa$ is a sufficiently small positive constant depending only on $\delta$.

Now we define
\begin{align}\label{eq:ARANR}
   &A_R(t,\xi)=\dfrac{\mathrm{e}^{\lambda(t)\langle\xi\rangle^{1/2}}}{b_R(t,\xi)} \mathrm{e}^{\sqrt{\delta}\langle\xi\rangle^{1/2}}, \quad A_{NR}(t,\xi)=\dfrac{\mathrm{e}^{\lambda(t)\langle\xi\rangle^{1/2}}}{b_{NR}(t,\xi)} \mathrm{e}^{\sqrt{\delta}\langle\xi\rangle^{1/2}},
\end{align}
and
\begin{align}\label{eq:Ak}
   & A_k(t,\xi)=\mathrm{e}^{\lambda(t)\langle k,\xi\rangle^{1/2}} \left( \dfrac{\mathrm{e}^{\sqrt{\delta}\langle\xi\rangle^{1/2}}}{b_k(t,\xi)} + \mathrm{e}^{\sqrt{\delta}|k|^{1/2}} \right).
\end{align}

\subsection{Basic bounds on the weights}

Let us first collect some basic bounds on the weights $w_{Y}, b_{Y}$ and $A_{Y}$ with $Y\in \{NR, R, k\}$ from \cite{IJ}.

\begin{lemma}\label{lem:7.1JiaHao}
  For all $t\geq 0$, $\xi,\eta\in\mathbb{R}$, and $k\in\mathbb{Z}$, we have
  \begin{align}\label{est:7.1JiaHao-1}
     &\dfrac{w_{NR}(t,\xi)}{w_{NR}(t,\eta)}+  \dfrac{w_{R}(t,\xi)}{w_{R}(t,\eta)}+\dfrac{w_{k}(t,\xi)}{w_{k}(t,\eta)} \lesssim_{\delta}\mathrm{e}^{\sqrt{\delta}|\eta-\xi|^{1/2}}.
  \end{align}
 Moreover, if $|\xi-\eta|\leq 10L_1(t,\eta)$, then we have the stronger bound
  \begin{align}\label{est:7.1JiaHao-2}
     &\dfrac{w_{NR}(t,\xi)}{w_{NR}(t,\eta)}+  \dfrac{w_{R}(t,\xi)}{w_{R}(t,\eta)}+\dfrac{w_{k}(t,\xi)}{w_{k}(t,\eta)} \lesssim_{\delta}1.
  \end{align}
\if0If $\min(|\xi|,|\eta|)\geq 2\delta^{-10}$, $|\xi-\eta|\leq \min(|\xi|,|\eta|)/3$, and $t\geq \max(t_{k_0(\xi)-4},t_{k_0(\eta)-4},\eta)$, then we also have the stronger bound
  \begin{align}\label{est:7.1JiaHao-3}
     &\max\left\{\dfrac{w_{NR}(t,\xi)}{w_{NR}(t,\eta)},  \dfrac{w_{R}(t,\xi)}{w_{R}(t,\eta)}, \dfrac{w_{k}(t,\xi)}{w_{k}(t,\eta)}\right\} \le \mathrm{e}^{\sqrt{\delta}|\eta-\xi|^{1/2}}.
  \end{align}\fi
\end{lemma}

\begin{lemma}\label{lem:7.2JiaHao}

  For $t\geq 0,\xi\in\mathbb{R},k\in \mathbb{Z}$, and $Y\in\{NR,R,k\}$, we have
  \begin{align}
    &b_{Y}(t,\xi)\approx_{\delta}w_Y(t,\xi),\label{est:7.2JiaHao-1} \\
    &|\partial_{\xi}b_{Y}(t,\xi)|\lesssim_{\delta} b_Y(t,\xi)\dfrac{1}{L_{\kappa}(t,\xi)},\label{est:7.2JiaHao-2} \\
    &\dfrac{ b_{Y}(t,\xi)}{b_{Y}(t,\eta)}\lesssim_{\delta} \mathrm{e}^{\sqrt{\delta}|\eta-\xi|^{1/2}}.\label{est:7.2JiaHao-3}
  \end{align}

 \if0 (ii) For $t\geq 0$, let
   \begin{align*}
     I_{t}^{*} &=\big\{(k,\xi)\in\mathbb{Z}\times\mathbb{R}: 1\leq |k|\leq \delta^2t\ \text{and}\ |\xi-tk|\leq t/6\big\},\\
     I_{t}^{**} &=\big\{(k,\xi)\in\mathbb{Z}\times\mathbb{R}: 1\leq |k|\leq \delta^4t\ \text{and}\ |\xi-tk|\leq t/12\big\}.
   \end{align*}
   Then it holds that
   \begin{align}
     &w_{k}(t,\xi)=w_{NR}(t,\xi)\ \text{and}\ b_{k}(t,\xi)=b_{NR}(t,\xi)\quad \text{if}\ (k,\xi)\notin I_{t}^{*},\label{est:7.2JiaHao-4}\\
      &b_{k}(t,\xi)\approx_{\delta} w_{k}(t,\xi)\approx_{\delta} w_{R}(t,\xi)\approx_{\delta}b_{R}(t,\xi)\quad \text{if}\ (k,\xi)\in I_{t}^{*},\label{est:7.2JiaHao-5}\\
       &b_{k}(t,\xi)\approx_{\delta} w_{k}(t,\xi)\approx_{\delta} w_{NR}(t,\xi)\approx_{\delta}b_{NR}(t,\xi)\quad \text{if}\ (k,\xi)\notin I_{t}^{**},\label{est:7.2JiaHao-6}
   \end{align}\fi
\end{lemma}

\begin{lemma}\label{lem:7.3JiaHao}

(i) Assume that $t \in[0,\infty),\ k\in\mathbb{Z}$, and $Y\in\{NR,R,k\}$. Then for any $\xi,\eta\in\mathbb{R}$ satisfying $|\eta|\geq |\xi|/8$ (or $|(k,\eta)|\geq |(k,\xi)|/8$ if $Y=k$), we have
  \begin{align}\label{est:7.3JiaHao-1}
    & \dfrac{A_{Y}(t,\xi)}{A_Y(t,\eta)}\lesssim_{\delta}\mathrm{e}^{0.9\lambda(t) |\xi-\eta|^{1/2}}.
  \end{align}

 \noindent (ii) Assume that $t \in[0,\infty),\ k,l\in\mathbb{Z}$, $\xi,\eta\in\mathbb{R}$ satisfying $|(l,\eta)|\geq |(k,\xi)|/8$. If $t\notin I_{k,\xi}$ or if $t\in I_{k,\xi}\cap I_{l,\eta}$, then
  \begin{align}\label{est:7.3JiaHao-2}
     & \dfrac{A_{k}(t,\xi)}{A_l(t,\eta)}\lesssim_{\delta}\mathrm{e}^{0.9\lambda(t) |(k-l,\xi-\eta)|^{1/2}}.
  \end{align}
  If $t\in I_{k,\xi}$ and $t\notin I_{l,\eta}$, then
  \begin{align}\label{est:7.3JiaHao-3}
     & \dfrac{A_{k}(t,\xi)}{A_l(t,\eta)}\lesssim_{\delta} \dfrac{|\xi|}{k^2}\dfrac{1}{1+|t-\xi/k|}\mathrm{e}^{0.9\lambda(t) |(k-l,\xi-\eta)|^{1/2}}.
  \end{align}
\end{lemma}

\begin{lemma}\label{lem:7.4JiaHao}
(i) For all $t\geq 0,\xi\in \mathbb{R}$, and $Y\in\{NR,R\}$, we have
\begin{align}\label{est:7.4JiaHao-1}
   & -\dfrac{\partial_tA_Y(t,\xi)}{A_Y(t,\xi)}\approx_{\delta}\left[\dfrac{\langle k,\xi\rangle^{1/2}}{\langle t\rangle^{1+\sigma_0} }+\dfrac{\partial_tw_{Y}(t,\xi)}{w_{Y}(t,\xi)}\right],
\end{align}
  and for any $k\in\mathbb{Z}$, we have
  \begin{align}\label{est:7.4JiaHao-2}
     &-\dfrac{\partial_tA_k(t,\xi)}{A_k(t,\xi)}\approx_{\delta}\left[\dfrac{\langle k,\xi\rangle^{1/2}}{\langle t\rangle^{1+\sigma_0} }+\dfrac{\partial_tw_{k}(t,\xi)}{w_{k}(t,\xi)} \dfrac{1}{1+\mathrm{e}^{\sqrt{\delta}(|k|^{1/2}-\langle \xi\rangle^{1/2})}w_k(t,\xi)}\right].
  \end{align}

\noindent  (ii) For all $t\geq 0,\xi\in \mathbb{R}$, and $Y\in\{NR,R\}$, we have
  \begin{align}\label{est:7.4JiaHao-3}
     & \big|(\dot{A}_{Y}/A_{Y})(t,\xi)\big|\lesssim_{\delta} \big|(\dot{A}_{Y}/A_{Y})(t,\eta)\big|\mathrm{e}^{4\sqrt{\delta}| \xi-\eta|^{1/2}}.
  \end{align}
  Moreover, if $k,l\in\mathbb{Z}$ then
  \begin{align}\label{est:7.4JiaHao-4}
     & \big|(\dot{A}_{k}/A_{k})(t,\xi)\big|\lesssim_{\delta} \big|(\dot{A}_{l}/A_{l})(t,\eta)\big|\mathrm{e}^{4\sqrt{\delta}| k-l,\xi-\eta|^{1/2}}.
  \end{align}
  Finally, if $\xi\in\mathbb{R}$ and $k\in \mathbb{Z}$ satisfy $|k|\leq \langle\xi\rangle+10$, then
  \begin{align}\label{est:7.4JiaHao-5}
     & \big|(\dot{A}_{k}/A_{k})(t,\xi)\big|\approx_{\delta} \big|(\dot{A}_{NR}/A_{NR})(t,\xi)\big| \approx_{\delta}\big|(\dot{A}_{R}/A_{R})(t,\xi)\big|.
  \end{align}
\end{lemma}
As a consequence of \eqref{est:7.4JiaHao-2}, we have (see (8.40) in \cite{IJ}, more precisely arXiv:1808.04026v1)
\begin{align}\label{est:8.40}
     &\left|\dfrac{\partial_tA_{\sigma}(t,\rho)}{A_{\sigma}(t,\rho)}\right|\gtrsim_{\delta}\frac{1}{\langle t-\rho/\sigma\rangle},\quad \text{if}\ 0<t<2|\rho/\sigma|,\ \sigma\neq0.
  \end{align}

\if0The following lemma is a small modification of Lemma \ref{lem:7.3JiaHao}.

\begin{lemma}\label{lem:A*xi/A*eta-0}
(i) Assume that $t \in[0,\infty),\ k\in\mathbb{Z}$. Then for any $\xi,\eta\in\mathbb{R}$ satisfying $|(k,\eta)|\geq |(k,\xi)|/8$, we have
  \begin{align}\label{est:A*xi/A*eta-1}
    & \dfrac{A_{k}^*(t,\xi)}{A_k^{*}(t,\eta)}\lesssim_{\delta}\mathrm{e}^{0.91\lambda(t) |\xi-\eta|^{1/2}}.
  \end{align}

\noindent(ii) Assume that $t\in[0,\infty),\ k,l\in\mathbb{Z}$, $\xi,\eta\in\mathbb{R}$ satisfying $|(l,\eta)|\geq |(k,\xi)|/8$. If $t\notin I_{k,\xi}$ or if $t\in I_{k,\xi}\cap I_{l,\eta}$, then
  \begin{align}\label{est:A*xi/A*eta-2}
     & \dfrac{A_{k}^*(t,\xi)}{A_l^{*}(t,\eta)}\lesssim_{\delta}\mathrm{e}^{0.91\lambda(t) |(k-l,\xi-\eta)|^{1/2}}.
  \end{align}
  If $t\in I_{k,\xi}$ and $t\notin I_{l,\eta}$, then
  \begin{align}\label{est:A*xi/A*eta-3}
     & \dfrac{A_{k}^*(t,\xi)}{A_l^{*}(t,\eta)}\lesssim_{\delta} \dfrac{|\xi|}{k^2}\dfrac{1}{1+|t-\xi/k|}\mathrm{e}^{0.91\lambda(t) |(k-l,\xi-\eta)|^{1/2}}.
  \end{align}
\end{lemma}\fi

Notice that for any $k,l\in\mathbb{Z}$, $\xi,\eta\in\mathbb{R}$, we have $\dfrac{k^2+|\xi|}{\langle t\rangle^2}\lesssim\dfrac{l^2+|\eta|}{\langle t\rangle^2} +\dfrac{|k-l|^2+|\xi-\eta|}{\langle t\rangle^2}$, \if0
  \begin{align*}
     & \dfrac{1+\dfrac{k^2}{\langle t\rangle^2}+\dfrac{\xi^2}{\langle t\rangle^4+\langle t\rangle^{\sigma_1}|\xi|}}{\left(1+\dfrac{l^2}{\langle t\rangle^2}+\dfrac{\eta^2}{\langle t\rangle^4+\langle t\rangle^{\sigma_1}|\eta|}\right)} \\
     &\lesssim  1+ \dfrac{\bigg|\dfrac{k^2}{\langle t\rangle^2}-\dfrac{l^2}{\langle t\rangle^2}\bigg|}{1+\dfrac{l^2}{\langle t\rangle^2}} +\dfrac{\left|\dfrac{\xi^2}{\langle t\rangle^4+\langle t\rangle^{\sigma_1}|\xi|}-\dfrac{\eta^2}{\langle t\rangle^4+\langle t\rangle^{\sigma_1}|\eta|}\right|}{1+\dfrac{\eta^2}{\langle t\rangle^4+\langle t\rangle^{\sigma_1}|\eta|}} \\
     &\lesssim 1+\dfrac{|k-l|^2}{\langle t\rangle^2}+ \dfrac{ |\xi^2-\eta^2|\langle t\rangle^4 +|\xi\eta|\langle t\rangle^{\sigma_1}|\xi-\eta|}{(\langle t\rangle^4+\langle t\rangle^{\sigma_1}|\eta|+\eta^2)(\langle t\rangle^4+\langle t\rangle^{\sigma_1}|\xi|)} \\
    & \lesssim 1+\dfrac{|k-l|^2}{\langle t\rangle^2}+ \dfrac{ |\eta||\xi-\eta|}{\langle t\rangle^4+\langle t\rangle^{\sigma_1}|\eta|+\eta^2}+ \dfrac{|\xi-\eta|^2\langle t\rangle^4}{(\langle t\rangle^4+\langle t\rangle^{\sigma_1}|\eta|+\eta^2)(\langle t\rangle^4+\langle t\rangle^{\sigma_1}|\xi|)} \\
     & \lesssim 1+\dfrac{|k-l|^2}{\langle t\rangle^2}+\dfrac{|\xi-\eta|}{\langle t\rangle^2}+ \dfrac{|\xi-\eta|^2}{\langle t\rangle^4+\langle t\rangle^{\sigma_1}|\xi|},
  \end{align*}\fi
which gives
  \begin{align}\label{est:weight-more-1}
    \dfrac{\left(1+\dfrac{k^2+|\xi|}{\langle t\rangle^2}\right)^{\f12}}{\left(1+\dfrac{l^2+|\eta|}{\langle t\rangle^2} \right)^{\f12}}
    \lesssim& \left(1+\dfrac{|k-l|^2+|\xi-\eta|}{\langle t\rangle^2}\right)^{\f12} \lesssim\left\{\begin{array}{l}
                                                                                                    \langle k-l,\xi-\eta\rangle, \\
                                                                                                    1+\dfrac{|k-l|^2+|\xi-\eta|}{\langle t\rangle^2}.
                                                                                                  \end{array}
\right.
  \end{align}
%\Vvdash\Vvdash  here we used $\sigma_1=2\sigma_0$. Then the lemma follows from \eqref{est:weight-more-1} and Lemma \ref{lem:7.3JiaHao}.

Similar to Lemma \ref{lem:7.4JiaHao}, we have

\begin{lemma}\label{lem:par-t-A*}
  For any $t\ge 0, k\in\mathbb{Z},\xi\in \R$, we have
  \begin{align}
     &-\dfrac{\partial_tA^{*}_k(t,\xi)}{A^{*}_k(t,\xi)}\approx_{\delta}\left[\dfrac{\langle k,\xi\rangle^{1/2}}{\langle t\rangle^{1+\sigma_0} }+\dfrac{\partial_tw_{k}(t,\xi)}{w_{k}(t,\xi)} \dfrac{1}{1+\mathrm{e}^{\sqrt{\delta}(|k|^{1/2}-\langle \xi\rangle^{1/2})}w_k(t,\xi)}\right],\label{est:par-t-A*}\\
     &-\dot{A}^{*}_{k}(t,\xi)\approx_{\delta} -\dot{A}_k(t,\xi)\left(1+\dfrac{k^2+|\xi|}{\langle t\rangle^2}\right)^{\f12}.\label{est:par-t-A*-1}
  \end{align}
  Moreover, for any $t\ge 0, k,l\in \Z, \xi,\eta\in \R$,  we have
  \begin{align}\label{est:par-t-A*-2}
     & \big|(\dot{A}_{k}^{*}/A_{k}^*)(t,\xi)\big|\lesssim_{\delta} \big|(\dot{A}_{l}^{*}/A_{l}^*)(t,\eta)\big|\mathrm{e}^{4\sqrt{\delta}| k-l,\xi-\eta|^{1/2}}.
  \end{align}
\end{lemma}

\begin{proof}
Recall that
\begin{align*}
   A_k^*(t,\xi)&=A_k(t,\xi) \left(1+\dfrac{k^2+|\xi|}{\langle t\rangle^2}\right)^{\f12}.
\end{align*}
Then we have
  \begin{align}\label{est:pa-t-weight-1}
   \dfrac{\dot{A}^*_k(t,\xi)}{A^*_{k}(t,\xi)}&=\dfrac{\dot{A}_{k}(t,\xi)}{A_{k}(t,\xi)}+ \dfrac{\partial_t\left(1+\dfrac{k^2+|\xi|}{\langle t\rangle^2}\right)^{\f12}}{\left(1 +\dfrac{k^2+|\xi|}{\langle t\rangle^2}\right)^{\f12}}=\dfrac{\dot{A}_{k}(t,\xi)}{A_{k}(t,\xi)}-
   \dfrac{\dfrac{2t(k^2+|\xi|)}{\langle t\rangle^4}}{2\left(1 +\dfrac{k^2+|\xi|}{\langle t\rangle^2}\right)}.
 \end{align}
 It is easy to verify that
 \begin{align}\label{est:pa-t-weight-2}
       \dfrac{\dfrac{2t(k^2+|\xi|)}{\langle t\rangle^4}}{2\left(1 +\dfrac{k^2+|\xi|}{\langle t\rangle^2}\right)}
&\lesssim \dfrac{\langle k\rangle^{1/2}}{\langle t\rangle^{3/2}}+ \dfrac{\langle\xi\rangle^{1/2}}{\langle t\rangle^2} \lesssim_{\delta}\dfrac{\langle k,\xi\rangle^{1/2}}{\langle t\rangle^{1+\sigma_0}}.
 \end{align}
 By \eqref{est:7.4JiaHao-2}, we know
 \begin{align}\label{eq:JH-par-t-A}
    & -\dfrac{\partial_tA_{k}(t,\xi)}{A_k(t,\xi)}\approx_{\delta} \left[\dfrac{\langle k,\xi\rangle^{1/2}}{\langle t\rangle^{1+\sigma_0} }+\dfrac{\partial_tw_{k}(t,\xi)}{w_{k}(t,\xi)} \dfrac{1}{1+\mathrm{e}^{\sqrt{\delta}(|k|^{1/2}-\langle \xi\rangle^{1/2})}w_k(t,\xi)}\right].
 \end{align}
 This along with \eqref{est:pa-t-weight-1} and \eqref{est:pa-t-weight-2} gives  \eqref{est:par-t-A*}.

By \eqref{eq:JH-par-t-A} again, we get
 \begin{align*}
    &-\dfrac{\partial_tA_{k}(t,\xi)}{A_k(t,\xi)}\approx_{\delta} -\dfrac{\partial_tA^*_{k}(t,\xi)}{A^*_k(t,\xi)},
 \end{align*}
 which gives
 \begin{align*}
    & -\partial_tA^*_{k}(t,\xi)\approx_{\delta}  -\partial_tA_{k}(t,\xi)\dfrac{A^*_k(t,\xi)}{A_k(t,\xi)} \approx_{\delta} -\dot{A}_k(t,\xi)\left(1+\dfrac{k^2+|\xi|}{\langle t\rangle^2}\right)^{\f12}.
 \end{align*}
 This proves \eqref{est:par-t-A*-1}.

 Thanks to the fact that
 \begin{align*}
    & -\dfrac{\partial_tA_{k}(t,\xi)}{A_k(t,\xi)}\approx_{\delta} -\dfrac{\partial_tA^*_{k}(t,\xi)}{A^*_k(t,\xi)},\quad -\dfrac{\partial_tA_{l}(t,\eta)}{A_l(t,\eta)}\approx_{\delta} -\dfrac{\partial_tA^*_{l}(t,\eta)}{A^*_l(t,\eta)},
 \end{align*}
we get by \eqref{est:7.4JiaHao-4}  that
 \begin{align*}
    &\big|(\dot{A}_{k}^{*}/A_{k}^*)(t,\xi)\big|\approx_{\delta} \big|(\dot{A}_{k}/A_{k})(t,\xi)\big|\\
    &\lesssim  \big|(\dot{A}_{l}/A_{l})(t,\eta)\big|\mathrm{e}^{4\sqrt{\delta}|k-l,\xi-\eta|^{1/2}} \approx_{\delta} \big|(\dot{A}^*_{l}/A^*_{l})(t,\eta)\big|\mathrm{e}^{4\sqrt{\delta}|k-l,\xi-\eta|^{1/2}},
 \end{align*}
which gives \eqref{est:par-t-A*-2}.
\end{proof}

\section{Weighted bilinear estimates(I)}
The following lemmas from \cite{IJ}(see Lemma 8.1-8.6) are devoted to nonlinear estimates.

\begin{lemma}\label{lem:8.1JiaHao}
  (i) Assume that $m,m_1,m_2: \mathbb{R}\rightarrow \mathbb{C}$ are symbols satisfying
  \begin{align}\label{con:8.1JiaHao-1}
    |m(\xi)|& \lesssim|m_1(\xi-\eta)||m_2(\eta)|\big\{\langle \xi-\eta\rangle^{-2}+ \langle \eta\rangle^{-2}\big\}
  \end{align}
  for any $\xi,\eta\in\mathbb{R}$. If $M,M_1,M_2$ are the operators defined by these symbols, then
  \begin{align}\label{est:8.1JiaHao-1}
     &\|M(gh)\|_{L^2(\mathbb{R})}\lesssim  \|M_1g\|_{L^2(\mathbb{R})}\|M_2h\|_{L^2(\mathbb{R})}.
  \end{align}

 \noindent (ii) Assume that $m,m_2:\mathbb{Z}\times\mathbb{R}\rightarrow \mathbb{C}$ and $m_1: \mathbb{R}\rightarrow \mathbb{C}$ are symbols satisfying
  \begin{align}\label{con:8.1JiaHao-2}
    |m(k,\xi)|& \lesssim |m_1(\xi-\eta)||m_2(k,\eta)|\big\{\langle \xi-\eta\rangle^{-2}+ \langle k,\eta\rangle^{-2}\big\}
  \end{align}
  for any $\xi,\eta\in\mathbb{R},$ $k\in \mathbb{Z}$. If $M,M_1,M_2$ are the operators defined by these symbols, then
  \begin{align}\label{est:8.1JiaHao-2}
     &\|M(gh)\|_{L^2(\mathbb{T}\times\mathbb{R})}\lesssim  \|M_1g\|_{L^2(\mathbb{R})}\|M_2h\|_{L^2(\mathbb{T}\times\mathbb{R})}.
  \end{align}

\noindent  (iii) Assume that $m,m_1,m_2:\mathbb{Z}\times\mathbb{R}\rightarrow \mathbb{C}$ are symbols satisfying
  \begin{align}\label{con:8.1JiaHao-3}
    |m(k,\xi)|& \lesssim |m_1(k-l,\xi-\eta)||m_2(l,\eta)|\big\{\langle k-l,\xi-\eta\rangle^{-2}+ \langle l,\eta\rangle^{-2}\big\}
  \end{align}
  for any $\xi,\eta\in\mathbb{R},$ $l,k\in \mathbb{Z}$. If $M,M_1,M_2$ are the operators defined by these symbols, then
  \begin{align}\label{est:8.1JiaHao-3}
     &\|M(gh)\|_{L^2(\mathbb{T}\times\mathbb{R})}\lesssim  \|M_1g\|_{L^2(\mathbb{T}\times\mathbb{R})}\|M_2h\|_{L^2(\mathbb{T}\times\mathbb{R})}.
  \end{align}
\end{lemma}

\begin{lemma}\label{lem:8.2JiaHao}
  For any $t\geq1,\ \alpha\in[0,4],\ \xi,\eta\in\mathbb{R}$, and $Y\in\{NR,R\}$, we have
  \begin{align}\label{est:8.2JiaHao-1}
    \langle\xi\rangle^{-\alpha}A_{Y}(t,\xi)\lesssim_{\delta} \langle\xi-\eta\rangle^{-\alpha} A_{Y}(t,\xi-\eta)\langle \eta\rangle^{-\alpha} A_{Y}(t,\eta)\mathrm{e}^{-(\lambda(t)/20)\min(\langle \xi-\eta\rangle,\langle\eta\rangle)^{1/2}},
  \end{align}
  and
  \begin{align}\label{est:8.2JiaHao-2}
     &|(\dot{A}_{Y}/A_{Y})(t,\xi)|\lesssim_{\delta}\left\{|(\dot{A}_{Y}/A_{Y})(t,\xi-\eta)| +|(\dot{A}_{Y}/A_{Y})(t,\eta)| \right\}\mathrm{e}^{4\sqrt{\delta}\min(\langle\xi-\eta\rangle,\langle\eta\rangle)^{1/2}}.
  \end{align}
\end{lemma}

\begin{lemma}\label{lem:8.3JiaHao}
  For any $t\geq1,\ \xi,\eta\in\mathbb{R}$, and $k\in\mathbb{Z}$, we have
  \begin{align}\label{est:8.3JiaHao-1}
    A_{k}(t,\xi)\lesssim_{\delta}  A_{R}(t,\xi-\eta) A_{k}(t,\eta)\mathrm{e}^{-(\lambda(t)/20)\min(\langle \xi-\eta\rangle,\langle k,\eta\rangle)^{1/2}},
  \end{align}
  and
  \begin{align}\label{est:8.3JiaHao-2}
     &|(\dot{A}_{k}/A_{k})(t,\xi)|\lesssim_{\delta}\left\{|(\dot{A}_{R}/A_{R})(t,\xi-\eta)| +|(\dot{A}_{k}/A_{k})(t,\eta)| \right\}\mathrm{e}^{12\sqrt{\delta}\min(\langle\xi-\eta\rangle,\langle\eta\rangle)^{1/2}}.
  \end{align}
\end{lemma}

The following stronger estimates hold for the case when $\big((k,\xi),(l,\eta)\big)$ belongs to some range.

\begin{lemma}\label{lem:8.4JiaHao}
Assume $t\ge 1$ and let the sets $R_0,R_1,R_2, R_3$ be defined by \eqref{eq:R0}-\eqref{eq:R3} and $(\sigma,\rho)=(k-l,\xi-\eta)$. Assume  $\sigma\neq 0$. Then it holds that

  \begin{itemize}
    \item[(i)]  If $\big((k,\xi),(l,\eta)\big)\in R_0\bigcup R_1$, then
    \begin{align}\label{est:8.4JiaHao-1}
      \dfrac{|\rho/\sigma|+\langle t\rangle}{\langle t\rangle}\dfrac{\langle\rho\rangle/\sigma^2}{\langle t-\rho/\sigma\rangle^2} &|l A_k^2(t,\xi)-kA_l^2(t,\eta) |\\
      &\lesssim_{\delta}\sqrt{|(A_k\dot{A}_{k})(t,\xi)|} \sqrt{|(A_l\dot{A}_{l})(t,\eta)|}A_{\sigma}(t,\rho)
      \mathrm{e}^{-(\delta_0/200)\langle\sigma,\rho\rangle^{1/2}}.\nonumber
    \end{align}
    Moreover, if $\big((k,\xi),(l,\eta)\big)\in R_0$, the term $|l A_k^2(t,\xi)-kA_l^2(t,\eta) |$ in the above inequality can be replaced by $|l A_k^2(t,\xi)|+|kA_l^2(t,\eta) |$.
    \item[(ii)]
    If $\big((k,\xi),(l,\eta)\big)\in R_2$, then
    \begin{align}\label{est:8.4JiaHao-2}
      \dfrac{|\rho/\sigma|+\langle t\rangle}{\langle t\rangle}\dfrac{\langle\rho\rangle/\sigma^2}{\langle t-\rho/\sigma\rangle^2} &\big(|l A_k^2(t,\xi)|+|kA_l^{2}(t,\eta) |\big)\\
      &\lesssim_{\delta}\sqrt{|(A_k\dot{A}_{k})(t,\xi)|} \sqrt{|(A_{\sigma}\dot{A}_{\sigma})(t,\rho)|}A_{l}(t,\eta)
      \mathrm{e}^{-(\delta_0/200)\langle l,\eta\rangle^{1/2}}.\nonumber
    \end{align}
  \end{itemize}
\end{lemma}

\begin{lemma}\label{lem:8.5JiaHao}
Assume $t\ge 1$ and let the sets $R_0,R_1,R_2, R_3$ be defined by \eqref{eq:R0}-\eqref{eq:R3} and $(\sigma,\rho)=(k-l,\xi-\eta)$. Assume  $\sigma\neq 0$. Then it holds that

  \begin{itemize}
    \item[(i)]  If $\big((k,\xi),(l,\eta)\big)\in R_0\bigcup R_1$, then
    \begin{align}\label{est:8.5JiaHao-1}
      \dfrac{|\rho/\sigma|^2+\langle t\rangle^2}{|\sigma|\langle t\rangle^2}\dfrac{1}{\langle t-\rho/\sigma\rangle^2} &|\eta A_k^2(t,\xi)-\xi A_l^{2}(t,\eta) |\\
      &\lesssim_{\delta}\sqrt{|(A_k\dot{A}_{k})(t,\xi)|} \sqrt{|(A_l\dot{A}_{l})(t,\eta)|}A_{\sigma}(t,\rho)
      \mathrm{e}^{-(\delta_0/200)\langle\sigma,\rho\rangle^{1/2}}.\nonumber
    \end{align}
    Moreover, if $\big((k,\xi),(l,\eta)\big)\in R_0$, the term $|\eta A_k^2(t,\xi)-\xi A_l^2(t,\eta) |$ in the above inequality can be replaced by $|\eta A_k^2(t,\xi)|+|\xi A_l^2(t,\eta) |$.
    \item[(ii)]
    If $\big((k,\xi),(l,\eta)\big)\in R_2$, then
    \begin{align}\label{est:8.5JiaHao-2}
      \dfrac{|\rho/\sigma|^2+\langle t\rangle^2}{|\sigma|\langle t\rangle^2}\dfrac{1}{\langle t-\rho/\sigma\rangle^2} &\big(|\eta A_k^2(t,\xi)|+|\xi A_l^{2}(t,\eta) |\big)\\
      &\lesssim_{\delta}\sqrt{|(A_k\dot{A}_{k})(t,\xi)|} \sqrt{|(A_{\sigma}\dot{A}_{\sigma})(t,\rho)|}A_{l}(t,\eta)
      \mathrm{e}^{-(\delta_0/200)\langle l,\eta\rangle^{1/2}}.\nonumber
    \end{align}
  \end{itemize}
\end{lemma}
\begin{lemma}\label{lem:8.6JiaHao}
Assume $t\ge 1$ and let the sets $\Sigma_0, \Sigma_1, \Sigma_2, \Sigma_3$ be defined by \eqref{eq:set-Sigma} and $\rho=\xi-\eta$. It holds that
  \begin{itemize}
    \item[(i)]  If $\big((k,\xi),(l,\eta)\big)\in \Sigma_0\bigcup \Sigma_1$, then
    \begin{align}\label{est:8.6JiaHao-1}
      \dfrac{1}{\langle \rho\rangle\langle t\rangle+\langle \rho\rangle^{1/4}\langle t\rangle^{7/4}} &|\eta A_k^2(t,\xi)-\xi A_k^{2}(t,\eta) |\\
      &\lesssim_{\delta}\sqrt{|(A_k\dot{A}_{k})(t,\xi)|} \sqrt{|(A_k\dot{A}_{k})(t,\eta)|}A_{NR}(t,\rho)
      \mathrm{e}^{-(\delta_0/200)\langle\rho\rangle^{1/2}}.\nonumber
    \end{align}
    Moreover, if $\big((k,\xi),(l,\eta)\big)\in \Sigma_0$, the term $|\eta A_k^2(t,\xi)-\xi A_k^2(t,\eta) |$ in the above inequality can be replaced by $|\eta A_k^2(t,\xi)|+|\xi A_l^2(t,\eta) |$.
    \item[(ii)]
    If $\big((k,\xi),(l,\eta)\big)\in \Sigma_2$, then
    \begin{align}\label{est:8.6JiaHao-2}
      \dfrac{1}{\langle \rho\rangle\langle t\rangle+\langle \rho\rangle^{1/4}\langle t\rangle^{7/4}} &\big(|\eta A_k^2(t,\xi)|+|\xi A_k^2(t,\eta) |\big)\\
      &\lesssim_{\delta}\sqrt{|(A_k\dot{A}_{k})(t,\xi)|} \sqrt{|(A_{NR}\dot{A}_{NR})(t,\rho)|}A_{k}(t,\eta)
      \mathrm{e}^{-(\delta_0/200)\langle k,\eta\rangle^{1/2}}.\nonumber
    \end{align}
  \end{itemize}
\end{lemma}

\begin{lemma}\label{lem:8.7JiaHao}
  For any $t\geq 1,\ k\in \mathbb{Z}\setminus\{0\}$, and $\xi,\eta\in \mathbb{R}$, we have
  \begin{align}\label{est:8.7JiaHao-1}
     \dfrac{A_{NR}^2(t,\xi)}{|\dot{A}_{NR}(t,\xi)|}\langle t\rangle^{3/4}\langle\xi\rangle^{1/4} \lesssim_{\delta} A_k(t,\eta) \dfrac{\langle t\rangle\langle t-\eta/k\rangle^2}{\langle t\rangle+|\eta/k|}A_{-k}(t,\rho)\mathrm{e}^{-(\lambda(t)/20)[\min( \langle\rho\rangle,\langle\eta\rangle)+|k|]^{1/2}},
  \end{align}
  and
  \begin{align}\label{est:8.7JiaHao-2}
     |(\dot{A}_{NR}/A_{NR})(t,\xi)|\lesssim_{\delta}\left\{ |(\dot{A}_{k}/A_{k})(t,\eta)| +(\dot{A}_{-k}/A_{-k})(t,\rho)\right\} \mathrm{e}^{12\sqrt{\delta}[\min( \langle\rho\rangle,\langle\eta\rangle)+|k|]^{1/2}},
  \end{align}
  where $\rho=\xi-\eta$.
\end{lemma}

\begin{lemma}\label{lem:8.8JIaHao}
For any $t\geq1$ and $\xi,\eta\in\mathbb{R}$, we have
  \begin{align} \label{est:8.8JiaHao}
    \dfrac{A_{NR}^2(t,\xi)}{|\dot{A}_{NR}(t,\xi)|}\dfrac{\langle t\rangle^{3/4}}{\langle\xi\rangle^{3/4}}\lesssim_{\delta} \dfrac{A_{NR}^2(t,\eta)}{|\dot{A}_{NR}(t,\eta)|}\dfrac{\langle t\rangle^{3/4}}{\langle\eta\rangle^{3/4}}A_{NR}(t,\rho)\mathrm{e}^{- (\lambda(t)/40)\min(\langle\rho\rangle,\langle\eta\rangle)^{1/2}},
  \end{align}
  where $\rho=\xi-\eta$.
  \end{lemma}

  \begin{lemma}\label{lem:8.9JIaHao}
 Assume $t\ge 1$ and let the sets $S_0,S_1,S_2$ defined by \eqref{eq:Si}. It holds that
\begin{itemize}
  \item[(i)] If $(\xi,\eta)\in S_0\cup S_1, \alpha \in[0,4]$, and $Y\in\{NR,R\}$ then, with $\rho=\xi-\eta$,
      \begin{align}\label{est:8.9JiaHao-1}
      \begin{aligned}
    &\left|\eta A^2_{Y}(t,\xi)\langle\xi\rangle^{-\alpha}- \xi A_{Y}^2(t,\eta)\langle\eta\rangle^{-\alpha}\right|\\
    &\lesssim_{\delta}t^{1.6}|\rho|\dfrac{\sqrt{|(A_Y\dot{A}_Y)(t,\xi)|}}{\langle\xi\rangle^{\alpha/2}} \dfrac{\sqrt{|(A_Y\dot{A}_Y)(t,\eta)|}}{\langle\eta\rangle^{\alpha/2}}\cdot A_{NR}(t,\rho)\mathrm{e}^{-(\lambda(t)/40)\langle \rho\rangle^{1/2}}.
  \end{aligned}
  \end{align}
  \item[(ii)] If $(\xi,\eta)\in  S_2$, then
  \begin{align}\label{est:8.9JiaHao-2}
      \begin{aligned}
   \langle \eta\rangle {{A_{R}^2}}(t,\xi)\lesssim_{\delta}&t^{1.1}\langle \xi\rangle^{0.6}\sqrt{|(A_R\dot{A}_R)(t,\xi)|} \sqrt{|(A_{NR}\dot{A}_{NR})(t,\rho)|}\\
   &\times A_{R}(t,\eta)\mathrm{e}^{-(\lambda(t)/40)\langle \eta\rangle^{1/2}},
  \end{aligned}
  \end{align}
  and
    \begin{align}\label{est:8.9JiaHao-3}
      \begin{aligned}
   \langle \eta\rangle A_{NR}^2(t,\xi)\lesssim_{\delta}&t^{1.1}\langle \xi\rangle^{-0.4}\sqrt{|(A_{NR}\dot{A}_{NR})(t,\xi)|} \sqrt{|(A_{NR}\dot{A}_{NR})(t,\rho)|}\\
   &\times A_{NR}(t,\eta)\mathrm{e}^{-(\lambda(t)/40)\langle \eta\rangle^{1/2}}.
  \end{aligned}
  \end{align}
\end{itemize}

\end{lemma}

\if0
We also need the following new bounds on $A_Y$ with $Y\in \{NR,R,k\}$.

\begin{lemma}\label{lem:8.9JIaHao-mod}
 Let $\rho=\xi-\eta$, $Y\in\{NR,R\}$, $t\geq 1$. If $|\rho|\leq 1$, then it holds that
 \begin{align}\nonumber
    &\left|\eta A^2_{Y}(t,\xi)\langle\xi\rangle^{-\alpha}- \xi A_{Y}^2(t,\eta)\langle\eta\rangle^{-\alpha}\right|\lesssim_{\delta}t^{1.6}|\rho|\dfrac{\sqrt{|(A_Y\dot{A}_Y)(t,\xi)|}}{\langle\xi\rangle^{\alpha/2}} \dfrac{\sqrt{|(A_Y\dot{A}_Y)(t,\eta)|}}{\langle\eta\rangle^{\alpha/2}} .
  \end{align}
\end{lemma}

\begin{proof}
  Notice that by \eqref{est:7.2JiaHao-2},
    \begin{align*}
     &\dfrac{\big|\partial_{\xi}A_{Y}(t,\xi)\big|}{A_{Y}(t,\xi)}\lesssim_{\delta} \dfrac{\langle\xi\rangle^{1/2}+t}{\langle \xi\rangle +t}.
  \end{align*}
 And by \eqref{est:7.2JiaHao-3}, we have $A_{Y}(t,\xi)/A_{Y}(t,\eta)\approx_{\delta}1$ if $|\rho|\leq 1$. Then we get
  \begin{align}\label{est:8.9JiaHao-mod-prove1}
     & \left|\eta A^2_{Y}(t,\xi)\langle\xi\rangle^{-\alpha}- \xi A_{Y}^2(t,\eta)\langle\eta\rangle^{-\alpha}\right| \lesssim_{\delta}|\rho| \langle \xi\rangle^{1-\alpha}\dfrac{\langle\xi\rangle^{1/2}+t}{\langle \xi\rangle +t} A_{Y}(t,\xi)A_{Y}(t,\eta).
  \end{align}
  By \eqref{est:7.4JiaHao-1}, we have
  \begin{align}\label{est:8.9JiaHao-mod-prove2}
     & \sqrt{|(\dot{A}_Y/A_Y)(t,\xi)|} \sqrt{|(\dot{A}_Y/A_{Y})(t,\eta)|}\gtrsim_{\delta} \dfrac{\langle\xi\rangle^{1/2}}{\langle t\rangle^{1+\sigma_0}}.
  \end{align}
  Then the desired bound  follows from \eqref{est:8.9JiaHao-mod-prove1} and \eqref{est:8.9JiaHao-mod-prove2}.
\end{proof}\fi

 \begin{lemma}\label{lem:8.7JiaHao-mod}
  For any $t\geq 1,\ k\in \mathbb{Z}\setminus\{0\}$, and $\xi,\eta\in \mathbb{R}$, we have, with $\rho=\xi-\eta$,
  \begin{align}\label{est:8.7JiaHao-mod}
  \begin{aligned}
     \dfrac{A_{NR}^2(t,\xi)}{|\dot{A}_{NR}(t,\xi)|}\langle t\rangle^{3/4}\langle\xi\rangle^{1/4} \lesssim_{\delta}& A_k(t,\eta) \dfrac{\langle t\rangle\big(\langle t\rangle +|(k,\eta)|\big)\langle t-\eta/k\rangle^2}{\langle \eta/k\rangle^2}A^{*}_{-k}(t,\rho)\\
     &\times\mathrm{e}^{-(\lambda(t)/20)[\min( \langle\rho\rangle,\langle\eta\rangle)+|k|]^{1/2}}.
  \end{aligned}
  \end{align}
  \end{lemma}
  \begin{proof}
    Thanks to the facts that
    \begin{align*}
       & \dfrac{\langle t\rangle\langle t-\eta/k\rangle^2}{\langle t\rangle+|\eta/k|} \lesssim \dfrac{\langle t\rangle\big(\langle t\rangle +|(k,\eta)|\big)\langle t-\eta/k\rangle^2}{\langle\eta/k\rangle^2},\\
       &A_{-k}(t,\rho)\leq A_{-k}(t,\rho)\left(1+\dfrac{k^2+|\rho|}{\langle t\rangle^2}\right)^{1/2}=A^{*}_{-k}(t,\rho),
    \end{align*}
  the desired bound  follows from \eqref{est:8.7JiaHao-1}.
  \end{proof}

  \begin{lemma}\label{lem:8.3JiaHao-mod}
    For any $t\geq 1$, $\xi,\eta\in\mathbb{R}$ and $k\in\mathbb{Z}$, we have
    \begin{align}\nonumber
       & A_{k}(t,\xi)\lesssim_{\delta} \left(\dfrac{\mathbf{1}_{t\in I_{k,\xi}} |\xi|/k^2}{\langle t-\xi/k\rangle}+1\right) A_{NR}(t,\xi-\eta)A_{k}(t,\eta)\mathrm{e}^{-(\lambda(t)/20)\min (\langle\xi-\eta\rangle,\langle k,\eta\rangle)^{1/2}}.
    \end{align}
  \end{lemma}

  \begin{proof}
If $t\notin I_{k,\xi}$, then ${b_{k}(t,\xi)\thickapprox_{\delta} w_k(t,\xi)=w_{NR}(t,\xi)\thickapprox_{\delta} b_{NR}(t,\xi)}$ and by \eqref{est:7.2JiaHao-3},
    \begin{align*}
       &\dfrac{ b_{NR}(t,\xi-\eta)}{b_{k}(t,\xi)}\thickapprox_{\delta} \dfrac{ b_{NR}(t,\xi-\eta)}{b_{NR}(t,\xi)}\lesssim_{\delta}  \mathrm{e}^{\sqrt{\delta} |\eta|^{1/2}}.
    \end{align*}
    If $t\in I_{k,\xi}$. By \eqref{eq:(7.7)JiaHao} and \eqref{est:7.2JiaHao-1}, we have $b_{k}(t,\xi)\thickapprox_{\delta} b_{R}(t,\xi)$ and
    \beno
    b_{R}(t,\xi)\approx_{\delta}b_{NR}(t,\xi)\dfrac{\langle t-\xi/k\rangle}{|\xi|/k^2}.
    \eeno
     Hence, we get by \eqref{est:7.2JiaHao-3} that
    \begin{align*}
       & \dfrac{ b_{NR}(t,\xi-\eta)}{b_{k}(t,\xi)}\approx_{\delta} \dfrac{ b_{NR}(t,\xi-\eta)}{b_{R}(t,\xi)}\approx_{\delta} \dfrac{|\xi|/k^2}{\langle t-\xi/k\rangle}\dfrac{ b_{NR}(t,\xi-\eta)}{b_{NR}(t,\xi)}\lesssim_{\delta}  \dfrac{|\xi|/k^2}{\langle t-\xi/k\rangle} \mathrm{e}^{\sqrt{\delta} |\eta|^{1/2}}.
    \end{align*}
   Combining two cases, we arrive at
    \begin{align}\label{est:mod8.3-1}
      & \dfrac{ b_{NR}(t,\xi-\eta)}{b_{k}(t,\xi)} \lesssim_{\delta} \left(\dfrac{\mathbf{1}_{t\in I_{k,\xi}} |\xi|/k^2}{\langle t-\xi/k\rangle}+1\right)\mathrm{e}^{\sqrt{\delta} |\eta|^{1/2}}.
    \end{align}
    By \eqref{est:7.2JiaHao-3}, we have
    \begin{align}\label{est:mod8.3-2}
       & \dfrac{b_{k}(t,\eta)}{b_{k}(t,\xi)}\lesssim_{\delta} \mathrm{e}^{\sqrt{\delta} |\xi-\eta|^{1/2}}.
    \end{align}
    By \eqref{est:mod8.3-1}, \eqref{est:mod8.3-2} and the fact that $b_k(t,\eta), b_{NR}(t,\xi-\eta)\lesssim_{\delta} 1$, it holds that
    \begin{align}\label{est:mod8.3-3}
       & \dfrac{1}{b_{k}(t,\xi)}\lesssim_{\delta} \left(\dfrac{\mathbf{1}_{t\in I_{k,\xi}} |\xi|/k^2}{\langle t-\xi/k\rangle}+1\right) \dfrac{1}{b_{NR}(t,\xi-\eta)}\dfrac{1}{b_{k}(t,\eta)}\mathrm{e}^{\sqrt{ \delta} \min(\langle\xi-\eta\rangle,\langle \eta\rangle)^{1/2}}.
    \end{align}
By \eqref{eq:tr-3}, we have
    \begin{align*}
       &\mathrm{e}^{\lambda(t)\langle k,\xi\rangle^{\f12}+\sqrt{\delta}|k|^{\f12}} \lesssim_{\delta} \mathrm{e}^{\lambda(t)\langle \xi-\eta\rangle^{\f12}} \mathrm{e}^{\lambda(t)\langle k,\eta\rangle^{\f12}+\sqrt{\delta}|k|^{\f12}} \mathrm{e}^{-\f{\lambda(t)}{15}\min (\langle\xi-\eta\rangle,\langle k,\eta\rangle)^{\f12}},\\
       &\mathrm{e}^{\lambda(t)\langle k,\xi\rangle^{\f12}}\mathrm{e}^{\sqrt{\delta}\langle\xi\rangle^{\f12}} \lesssim_{\delta}\mathrm{e}^{\lambda(t)\langle\xi-\eta\rangle^{\f12}}\mathrm{e}^{\sqrt{\delta}\langle\xi-\eta\rangle^{1/2}} \cdot \mathrm{e}^{\lambda(t)\langle k,\eta\rangle^{\f12}}\mathrm{e}^{\sqrt{\delta}\langle\eta\rangle^{\f12}}  \mathrm{e}^{-\f{\lambda(t)}{15}\min (\langle\xi-\eta\rangle,\langle k,\eta\rangle)^{\f12}},
    \end{align*}
which along with \eqref{est:mod8.3-3} and the definition of $A_k$ and $A_{NR}$ implies
    \begin{align*}
       & A_{k}(t,\xi)\lesssim_{\delta} \left(\dfrac{\mathbf{1}_{t\in I_{k,\xi}} |\xi|/k^2}{\langle t-\xi/k\rangle}+1\right) A_{NR}(t,\xi-\eta)A_{k}(t,\eta)\mathrm{e}^{-(\lambda(t)/20)\min (\langle\xi-\eta\rangle,\langle k,\eta\rangle)^{1/2}}.
    \end{align*}
   \end{proof}

    \begin{lemma}\label{lem:A12k}
    Let
    \begin{align}
       & A_{k}^{(1)}(t,\xi)=\dfrac{ A_{k}(t,\xi)\langle t-\xi/k\rangle}{\langle t-\xi/k\rangle+|\xi/k^2|}\quad \text{for}\ k\neq 0;\quad A_{0}^{(1)}(t,\xi)=A_{0}(t,\xi),\label{eq:A1k}\\
       &A_{k}^{(2)}(t,\xi)= A_{k}(t,\xi)\langle t-\xi/k\rangle\quad \text{for}\ k\neq 0;\quad A_{0}^{(2)}(t,\xi)=A_{0}(t,\xi)(t+|\xi|). \label{eq:A2k}
    \end{align}
 If $|(l,\eta)|\geq |(k,\xi)|/8$, then it holds that for $j=1,2$,
    \begin{align}\nonumber
       & \dfrac{A^{(j)}_k(t,\xi)}{A_{l}^{(j)}(t,\eta)}\lesssim_{\delta} \mathrm{e}^{0.9\lambda(t)\langle k-l,\xi-\eta\rangle^{1/2}}.
    \end{align}
  \end{lemma}

 \begin{proof}
Recall the definition of $w_{R}(t,\xi)$ and $w_{k}(t,\xi)$:
\begin{align*}
   &w_{k}(t,\xi)=\left\{
   \begin{aligned}
   &w_{NR}(t,\xi)\dfrac{1+\delta^2|t-\xi/k|}{1+\delta^2|\xi|/(8k^2)},\quad \text{if}\ t\in I_{k,\xi}\ \text{and}\ |t-\xi/k|\leq |\xi|/(8k^2),\\
   &w_{NR}(t,\xi),\qquad\qquad\qquad\qquad \text{otherwise}.
   \end{aligned}\right.
\end{align*}
Hence, if $|\xi|\geq \delta^{-10},\ 1\leq|k|\leq \lfloor\sqrt{\delta^3|\xi|}\rfloor,\ \xi/k\geq 0$, we have
\begin{align*}
   w_{k}(t,\xi)\dfrac{|\xi|/k^2+\langle t-\xi/k\rangle}{\langle t-\xi/k\rangle}\approx_{\delta} w_{NR}(t,\xi).
\end{align*}
On the other hands, if  {$|\xi|\leq \delta^{-10}$} or $|k|\geq \lfloor\sqrt{\delta^3|\xi|}\rfloor+1$ or $\xi/k\leq 0$,  it holds that
\begin{align*}
   & w_{k}(t,\xi)=w_{NR}(t,\xi),\quad \dfrac{|\xi|/k^2+\langle t-\xi/k\rangle}{\langle t-\xi/k\rangle}\approx_{\delta} 1,
\end{align*}
which gives
\begin{align*}
w_{k}(t,\xi)\dfrac{|\xi|/k^2+\langle t-\xi/k\rangle}{\langle t-\xi/k\rangle}\approx_{\delta} w_{NR}(t,\xi).
\end{align*}
Thus, it always holds that
 \begin{align}\label{est:A12k-prove-wk}
  w_{k}(t,\xi)\dfrac{|\xi|/k^2+\langle t-\xi/k\rangle}{\langle t-\xi/k\rangle}\approx_{\delta} w_{NR}(t,\xi).
\end{align}

Thanks to the definition of $A_{NR}$ and $A_k$, we get by  \eqref{est:A12k-prove-wk} and \eqref{est:7.2JiaHao-1} that
\begin{align*}
   &A_{k}^{(1)}(t,\xi)= A_{k}(t,\xi)\dfrac{\langle t-\xi/k\rangle}{|\xi|/k^2+\langle t-\xi/k\rangle}\\
   &\approx_{\delta}\mathrm{e}^{\lambda(t)\langle k,\xi\rangle^{1/2}}
   \left(\dfrac{\mathrm{e}^{\sqrt{\delta}\langle\xi\rangle^{1/2}}}{w_{k}(t,\xi)}\dfrac{|\xi|/k^2+\langle t-\xi/k\rangle}{\langle t-\xi/k\rangle}+ \mathrm{e}^{\sqrt{\delta}|k|^{1/2}}\dfrac{\langle t-\xi/k\rangle}{|\xi|/k^2+\langle t-\xi/k\rangle}\right)\\
   &\approx_{\delta}
   \dfrac{\mathrm{e}^{\lambda(t)\langle k,\xi\rangle^{1/2}+\sqrt{\delta}\langle\xi\rangle^{1/2}}}{w_{NR}(t,\xi)}+ \dfrac{\mathrm{e}^{\lambda(t)\langle k,\xi\rangle^{1/2}+\sqrt{\delta}|k|^{1/2}}\langle t-\xi/k\rangle}{|\xi|/k^2+\langle t-\xi/k\rangle}\\
   &\approx_{\delta}
   \dfrac{\mathrm{e}^{\lambda(t)\langle k,\xi\rangle^{1/2}+\sqrt{\delta}\langle\xi\rangle^{1/2}}}{w_{NR}(t,\xi)}+ \dfrac{\mathrm{e}^{\lambda(t)\langle k,\xi\rangle^{1/2}+\sqrt{\delta}|k|^{1/2}}\langle t-\xi/k\rangle\mathbf{1}_{\{\langle\xi\rangle\leq |k|\}}}{|\xi|/k^2+\langle t-\xi/k\rangle},
\end{align*}
as $w_{NR}(t,\xi)\leq 1 $. Hence, if $|(l,\eta)|\geq |(k,\xi)|/8$, by using \eqref{est:7.1JiaHao-1} and \eqref{eq:tr-3}, we have ($k,l\neq 0$)
\begin{align*}
   \dfrac{A_{k}^{(1)}(t,\xi)}{A_{l}^{(1)}(t,\eta)}\lesssim_{\delta}& \dfrac{\mathrm{e}^{\lambda(t)\langle k,\xi\rangle^{1/2}+\sqrt{\delta}\langle\xi\rangle^{1/2}}}{\mathrm{e}^{ \lambda(t)\langle l,\eta\rangle^{1/2}+\sqrt{\delta}\langle\eta\rangle^{1/2}}} \dfrac{w_{NR}(t,\eta)}{w_{NR}(t,\xi)} \nonumber\\
   &+\dfrac{\mathrm{e}^{\lambda(t)\langle k,\xi\rangle^{1/2}+\sqrt{\delta}|k|^{1/2}}}{\mathrm{e}^{\lambda(t)\langle l,\eta\rangle^{1/2}+\sqrt{\delta}|l|^{1/2}}}\dfrac{\langle t-\xi/k\rangle\mathbf{1}_{\{\langle\xi\rangle\leq |k|\}}}{|\xi|/k^2+\langle t-\xi/k\rangle}\dfrac{|\eta|/l^2+\langle t-\eta/l\rangle}{\langle t-\eta/l\rangle}\nonumber\\
   \lesssim_{\delta}&\mathrm{e}^{0.85\lambda(t)\langle k-l,\xi-\eta\rangle^{1/2}},
\end{align*}
here we used the fact that
\begin{align*}
   & \dfrac{\langle t-\xi/k\rangle\mathbf{1}_{\{\langle\xi\rangle\leq |k|\}}}{|\xi|/k^2+\langle t-\xi/k\rangle}\dfrac{|\eta|/l^2+\langle t-\eta/l\rangle}{\langle t-\eta/l\rangle}
   \leq \mathbf{1}_{\{\langle\xi\rangle\leq |k|\}}(|\eta|/l^2+1)\leq
   \dfrac{|\eta|/l^2}{\langle\xi\rangle/k^2}+1\lesssim \langle k-l,\xi-\eta\rangle^{3}.
\end{align*}

Since
\begin{align*}
   A_0(t,\xi)=&\mathrm{e}^{\lambda(t)\langle\xi\rangle^{1/2}} \left(\dfrac{\mathrm{e}^{\sqrt{\delta}\langle\xi\rangle^{1/2}}}{b_{NR}
   (t,\xi)}+1\right)\approx_{\delta} \mathrm{e}^{\lambda(t)\langle\xi\rangle^{1/2}} \dfrac{\mathrm{e}^{\sqrt{\delta}\langle\xi\rangle^{1/2}}}{b_{NR}
   (t,\xi)}=\dfrac{\mathrm{e}^{(\lambda(t)+\sqrt{\delta})\langle \xi\rangle^{1/2}}}{w_{NR}(t,\xi)},
\end{align*}
we get by \eqref{est:7.1JiaHao-1} and \eqref{eq:tr-3} (if $|\eta|\geq |\xi/8|$) that
\begin{align*}
  \dfrac{A_{0}^{(1)}(t,\xi)}{A_{0}^{(1)}(t,\eta)}\lesssim_{\delta} & \mathrm{e}^{(\lambda(t)+\sqrt{\delta})(\langle \xi\rangle^{1/2}-\langle\eta\rangle^{1/2})}
  \dfrac{w_{NR}(t,\eta)}{w_{NR}(t,\xi)} \lesssim_{\delta} \mathrm{e}^{0.85\lambda(t)\langle\xi-\eta\rangle^{1/2}}.
\end{align*}
By \eqref{est:7.1JiaHao-1} and \eqref{eq:tr-3} again, we have (for $l\neq 0$ and $|(l,\eta)|\geq |\xi/8|$)
\begin{align*}
   \dfrac{A_{0}^{(1)}(t,\xi)}{A_{l}^{(1)}(t,\eta)}\lesssim_{\delta}& \mathrm{e}^{\lambda(t)(\langle \xi\rangle^{1/2}-\langle l,\eta\rangle^{1/2})}\mathrm{e}^{\sqrt{\delta}(\langle \xi\rangle^{1/2}-\langle \eta\rangle^{1/2})}
  \dfrac{w_{NR}(t,\eta)}{w_{NR}(t,\xi)}\lesssim_{\delta} \mathrm{e}^{0.85\lambda(t)\langle l,\xi-\eta\rangle^{1/2}}.
\end{align*}
For $k\neq 0$ and $|\eta|\geq |(k,\xi)|/8$, we get by using \eqref{est:7.1JiaHao-1}, \eqref{eq:tr-3} and $w_{NR}(t,\eta)\leq 1$ that
\begin{align*}
   \dfrac{A_{k}^{(1)}(t,\xi)}{A_{0}^{(1)}(t,\eta)}\lesssim_{\delta}& \mathrm{e}^{\lambda(t)(\langle k,\xi\rangle^{1/2}-\langle \eta\rangle^{1/2})}\mathrm{e}^{\sqrt{\delta}(\langle \xi\rangle^{1/2}-\langle \eta\rangle^{1/2})}
  \dfrac{w_{NR}(t,\eta)}{w_{NR}(t,\xi)}\\
  &+\mathrm{e}^{\lambda(t)(\langle k,\xi\rangle^{1/2}-\langle\eta\rangle^{1/2})} \mathrm{e}^{ \sqrt{\delta}(|k|^{1/2}-\langle\eta\rangle^{1/2})}w_{NR}(t,\eta)\\
  \lesssim_{\delta} &\mathrm{e}^{0.85\lambda(t)\langle k,\xi-\eta\rangle^{1/2}}.
\end{align*}
Therefore, if $k,l\in\mathbb{Z}$, $|(l,\eta)|\geq |(k,\xi)|/8$, we arrive at
 \begin{align}\label{eq:A1k/A1l}
   \dfrac{A_{k}^{(1)}(t,\xi)}{A_{l}^{(1)}(t,\eta)}\lesssim_{\delta}
   &\mathrm{e}^{0.85\lambda(t)\langle k-l,\xi-\eta\rangle^{1/2}}.
\end{align}
This finishes the proof  for the case of $j=1$.

Notice that
\begin{align*}
   & \dfrac{A_{k}^{(2)}(t,\xi)}{A_{l}^{(2)}(t,\eta)} =\dfrac{A_{k}^{(1)}(t,\xi)}{A_{l}^{(1)}(t,\eta)}\times \dfrac{A_{k}^{(2)}(t,\xi)}{A_{k}^{(1)}(t,\xi)}\times \dfrac{A_{l}^{(1)}(t,\eta)}{A_{l}^{(2)}(t,\eta)}.
\end{align*}
Thus, it suffices to prove that
\begin{align}\label{eq:A1k/A1l-1}
  \dfrac{A_{k}^{(2)}(t,\xi)}{A_{k}^{(1)}(t,\xi)}\times \dfrac{A_{l}^{(1)}(t,\eta)}{A_{l}^{(2)}(t,\eta)} \lesssim\langle k-l,\xi-\eta\rangle^{3}\lesssim_{\delta}\mathrm{e}^{0.05\lambda(t)\langle k-l,\xi-\eta\rangle^{1/2}},
\end{align}
for $k,l\in\mathbb{Z}$, $|(l,\eta)|\geq |(k,\xi)|/8$.

For $k=l=0$, we have
\begin{align*}
   \dfrac{A_{0}^{(2)}(t,\xi)}{A_{0}^{(1)}(t,\xi)}\times \dfrac{A_{0}^{(1)}(t,\eta)}{A_{0}^{(2)}(t,\eta)}\lesssim \dfrac{t+|\xi|}{t+|\eta|}\lesssim \langle \xi-\eta\rangle.
\end{align*}
For $k=0,\ l\neq 0$, we have
\begin{align*}
   \dfrac{A_{0}^{(2)}(t,\xi)}{A_{0}^{(1)}(t,\xi)}\times \dfrac{A_{l}^{(1)}(t,\eta)}{A_{l}^{(2)}(t,\eta)}&\lesssim \dfrac{t+|\xi|}{\langle t-\eta/l\rangle +|\eta/l^2|}\lesssim |l|+l^2|\xi|/\langle\eta\rangle \lesssim \langle l,\xi-\eta\rangle^3.
\end{align*}
For $k\neq0,\ l= 0$,  we have
\begin{align*}
   \dfrac{A_{k}^{(2)}(t,\xi)}{A_{k}^{(1)}(t,\xi)}\times \dfrac{A_{0}^{(1)}(t,\eta)}{A_{0}^{(2)}(t,\eta)}&\lesssim  \dfrac{\langle t-\xi/k\rangle +|\xi/k^2|}{t+|\eta|} \lesssim \dfrac{t +\langle \xi/k\rangle +|\xi/k^2|}{t+|\eta|}\\
   &\lesssim 1+|\xi|/\langle\eta\rangle \lesssim \langle \xi-\eta\rangle.
\end{align*}
For $k\neq 0,l\neq 0$, we have
\begin{align*}
  \dfrac{A_{k}^{(2)}(t,\xi)}{A_{k}^{(1)}(t,\xi)}\times \dfrac{A_{l}^{(1)}(t,\eta)}{A_{l}^{(2)}(t,\eta)} =&\dfrac{\langle t-\xi/k\rangle+|\xi/k^2|}{\langle t-\eta/l\rangle +|\eta/l^2|}\lesssim \dfrac{\langle t-\xi/k\rangle}{\langle t-\eta/l\rangle +|\eta/l^2|} +\dfrac{|\xi|}{\langle\eta\rangle}\dfrac{l^2}{k^2} \\
  \lesssim& \dfrac{|\xi/k-\eta/l|}{\langle t-\eta/l\rangle +|\eta/l^2|} +\left(\dfrac{|\xi-\eta|}{\langle\eta\rangle}+1\right) \left(\dfrac{|k-l|}{|k|}+1\right)^2\\
  \lesssim& \dfrac{|\xi-\eta||l|+|k-l||\eta|}{|kl| \big(\langle t-\eta/l\rangle +|\eta/l^2|\big)} +\langle k-l,\xi-\eta\rangle^3\\
  \lesssim & \langle k-l,\xi-\eta\rangle^3.
\end{align*}

Combining the above cases, we prove \eqref{eq:A1k/A1l-1}. Thus, we prove the case of $j=2$.
\end{proof}

\begin{lemma}\label{lem:8.3JiaHao-mod1}
    Let $A_{k}^{(j)}$($j=1,2$) be defined by \eqref{eq:A1k} and \eqref{eq:A2k}. For any $t\geq 1$, $\xi,\eta\in\mathbb{R}$ and $k,l\in\mathbb{Z}$, we have
\begin{align}\label{est:8.3JiaHao-mod3}
    \begin{aligned}
       A_{k}^{(j)}(t,\xi)\lesssim_{\delta}& \mathbf{1}_{\{|(k-l,\xi-\eta)|\leq|(l,\eta)|\}} A_{k-l}(t,\xi-\eta)A_{l}^{(j)}(t,\eta)\mathrm{e}^{-(\lambda(t)/20) \langle k-l,\xi-\eta\rangle^{1/2}}\\
       &+\mathbf{1}_{\{|(l,\eta)|\leq|(k-l,\xi-\eta)|\} } A_{k-l}^{(j)}(t,\xi-\eta)A_{l}(t,\eta)\mathrm{e}^{-(\lambda(t)/20) \langle l,\eta\rangle^{1/2}},
    \end{aligned}
\end{align}
    and
\begin{align}\label{est:8.3JiaHao-mod2}
       &|(\dot{A}_k/A_k)(t,\xi)|\lesssim_{\delta}\big(|(\dot{A}_{k-l}/A_{k-l}) (t,\xi-\eta)| +|(\dot{A}_l/A_l)(t,\eta)|\big)\mathrm{e}^{ 4\sqrt{\delta}\min(|l,\eta|,|k-l,\xi-\eta|)^{1/2}}.
\end{align}
\end{lemma}

\begin{proof}
The inequality \eqref{est:8.3JiaHao-mod2} follows from \eqref{est:7.4JiaHao-4}. If $|(l,\eta)|\geq |(k-l,\xi-\eta)|$, then we have $|(l,\eta)|\geq |(k,\xi)|/2$. Thus, we get by Lemma \ref {lem:A12k} that
    \begin{align*}
       &\dfrac{A_{k}^{(j)}(t,\xi)}{A_{l}^{(j)}(t,\eta)}\lesssim_{\delta} \mathrm{e}^{0.9\lambda(t)|(k-l,\xi-\eta)|^{1/2}}.
    \end{align*}
    Thanks to $A_{k-l}(t,\xi-\eta)\geq \mathrm{e}^{\lambda(t)\langle k-l,\xi-\eta\rangle^{1/2}}$, we get
\begin{align*}
       A_{k}^{(j)}(t,\xi)\lesssim_{\delta}& A_{k-l}(t,\xi-\eta)A_{l}^{(j)}(t,\eta)\mathrm{e}^{-(\lambda(t)/20)\langle k-l,\xi-\eta\rangle^{1/2}}.
\end{align*}
If $|(k-l,\xi-\eta)|\geq |(l,\eta)|$, we can similarly obtain
\begin{align*}
    A_{k}^{(j)}(t,\xi)\lesssim_{\delta} A_{k-l}^{(j)}(t,\xi-\eta)A_{l}(t,\eta)\mathrm{e}^{-(\lambda(t)/20)\langle l,\eta\rangle^{1/2}}.
\end{align*}
This proves  \eqref{est:8.3JiaHao-mod3}.
\end{proof}

\section{Weighted bilinear estimates(II)}

In this appendix, we establish some weighted bilinear estimates on the new weight $A^*_k$.
Let us recall the following notations:
\begin{align*}
   R_0=\Big\{&\big((k,\xi),(l,\eta)\big)\in(\mathbb{Z}\times \mathbb{R})^2: \min(\langle k,\xi\rangle,\langle l,\eta\rangle,\langle k-l,\xi-\eta\rangle)\geq\\& \dfrac{\langle k,\xi\rangle+\langle l,\eta\rangle+\langle k-l,\xi-\eta\rangle}{20}\Big\},\nonumber\\
   R_1=\Big\{&\big((k,\xi),(l,\eta)\big)\in(\mathbb{Z}\times \mathbb{R})^2: \langle k-l,\xi-\eta\rangle\leq \dfrac{\langle k,\xi\rangle+\langle l,\eta\rangle+\langle k-l,\xi-\eta\rangle}{10}\Big\},\\
   R_2=\Big\{&\big((k,\xi),(l,\eta)\big)\in(\mathbb{Z}\times \mathbb{R})^2: \langle l,\eta\rangle\leq \dfrac{\langle k,\xi\rangle+\langle l,\eta\rangle+\langle k-l,\xi-\eta\rangle}{10}\Big\},
\end{align*}
and for $i\in \{0,1,2,3\}$,
\begin{align*}
  \Sigma_{i}=\big\{\big((k,\xi),(l,\eta)\big)\in R_i: k=l\big\}.
\end{align*}

The following lemma is an analogue of Lemma \ref{lem:8.4JiaHao}.

\begin{lemma}\label{lem:A*R0R1R2-1}
Assume $t\ge 1$ and let  $(\sigma,\rho)=(k-l,\xi-\eta)$. Assume $\sigma\neq 0$. It holds that
  \begin{itemize}
    \item[(i)]  If $\big((k,\xi),(l,\eta)\big)\in R_0\bigcup R_1$, then
    \begin{align}\label{eq:A*R0R1bounds}
      \dfrac{|\rho/\sigma|+\langle t\rangle}{\langle t\rangle}\dfrac{\langle\rho\rangle/\sigma^2}{\langle t-\rho/\sigma\rangle^2} &\big|l A_k^*(t,\xi)^2-kA_l^{*}(t,\eta)^2\big|\\
      &\lesssim_{\delta}\sqrt{|(A_k^{*}\dot{A}_{k}^{*})(t,\xi)|} \sqrt{|(A_l^{*}\dot{A}_{l}^{*})(t,\eta)|}A_{\sigma}(t,\rho)
      \mathrm{e}^{-(\delta_0/201)\langle\sigma,\rho\rangle^{1/2}}.\nonumber
    \end{align}
    \item[(ii)]
    If $\big((k,\xi),(l,\eta)\big)\in R_2$, then
    \begin{align}\label{eq:A*R2bounds}
      \dfrac{|\rho/\sigma|+\langle t\rangle}{\langle t\rangle}\dfrac{\langle\rho\rangle/\sigma^2}{\langle t-\rho/\sigma\rangle^2} &\big(|l A_k^*(t,\xi)^2|+|kA_l^{*}(t,\eta)^2 |\big)\\
      &\lesssim_{\delta}\sqrt{|(A_k^{*}\dot{A}_{k}^{*})(t,\xi)|} \sqrt{|(A_{\sigma}\dot{A}_{\sigma})(t,\rho)|}A_{l}^*(t,\eta)
      \mathrm{e}^{-(\delta_0/201)\langle l,\eta\rangle^{1/2}}.\nonumber
    \end{align}
  \end{itemize}
\end{lemma}

\begin{proof}
  \textbf{Step 1.} Assume that $((k,\xi),(l,\eta))\in R_0$. \smallskip

  In this case, we have $\langle k,\xi\rangle+\langle l,\eta\rangle\lesssim\langle \sigma,\rho\rangle$.
   We get by Lemma \ref{lem:8.4JiaHao} that
  \begin{align*}
     &\dfrac{|\rho/\sigma|+\langle t\rangle}{\langle t\rangle}\dfrac{\langle\rho\rangle/\sigma^2}{\langle t-\rho/\sigma\rangle^2}(|l A_k(t,\xi)^2|+|kA_l(t,\eta)^2|) |\\
     &\lesssim_{\delta} \sqrt{|(A_k\dot{A}_{k})(t,\xi)|} \sqrt{|(A_l\dot{A}_{l})(t,\eta)|}A_{\sigma}(t,\rho)
      \mathrm{e}^{-(\delta_0/200)\langle\sigma,\rho\rangle^{1/2}}.
  \end{align*}
  Then by $A_j\leq A_j^*$, $ |\dot{A}_{j}|\lesssim|\dot{A}_{j}^*|$ for $j=k,l$ (see \eqref{est:par-t-A*-1}) and $\langle k,\xi\rangle+\langle l,\eta\rangle\lesssim\langle \sigma,\rho\rangle$, we have
  \begin{align*}
     &\dfrac{|\rho/\sigma|+\langle t\rangle}{\langle t\rangle}\dfrac{\langle\rho\rangle/\sigma^2}{\langle t-\rho/\sigma\rangle^2}\big(|l A_k^*(t,\xi)^2|+|kA_l^{*}(t,\eta)^2 |\big)\\
     &\leq \dfrac{|\rho/\sigma|+\langle t\rangle}{\langle t\rangle}\dfrac{\langle\rho\rangle/\sigma^2}{\langle t-\rho/\sigma\rangle^2}(|l A_k^2(t,\xi)|+|kA_l^2(t,\eta)|) \left(1+\dfrac{l^2+|\eta|}{\langle t\rangle^2}+\dfrac{k^2+|\xi|}{\langle t\rangle^2}\right)\\
     &\lesssim_{\delta}\sqrt{|(A_k\dot{A}_{k})(t,\xi)|} \sqrt{|(A_l\dot{A}_{l})(t,\eta)|}A_{\sigma}(t,\rho)
      \mathrm{e}^{-(\delta_0/200)\langle\sigma,\rho\rangle^{1/2}}(\langle k,\xi\rangle^2+\langle l,\eta\rangle^2)\\
      &\lesssim_{\delta}\sqrt{|(A_k^{*}\dot{A}_{k}^{*})(t,\xi)|} \sqrt{|(A_l^{*}\dot{A}_{l}^{*})(t,\eta)|}A_{\sigma}(t,\rho)
      \mathrm{e}^{-(\delta_0/201)\langle\sigma,\rho\rangle^{1/2}}.
  \end{align*}

  \noindent\textbf{Step 2.} Assume that $\big((k,\xi),(l,\eta)\big)\in R_1$. \smallskip

  We write
  \begin{align*}
     & \big[l A_k^*(t,\xi)^2-kA_l^{*}(t,\eta)^2 \big]=\mathcal{M}_1+\mathcal{M}_2 ,
  \end{align*}
  with
  \begin{align*}
   \mathcal{M}_1  &=\big[l A_k^2(t,\xi)-kA_l^2(t,\eta) \big] \left(1+\dfrac{l^2+|\eta|}{\langle t\rangle^2}\right),\\
   \mathcal{M}_2 &=l A_k^2(t,\xi) \left[\dfrac{k^2-l^2}{\langle t\rangle^2}+\dfrac{|\xi|-|\eta|}{\langle t\rangle^2} \right].
   %\mathcal{M}_3&:=kl\left[\dfrac{k A^2_k(t,\xi)}{\langle \xi\rangle^2}-\dfrac{lA_l^{2}(t,\eta) }{\langle \eta\rangle^2}\right].
  \end{align*}
  In order to prove Lemma \ref{lem:A*R0R1R2-2}, we also define
    \begin{align*}
     & \mathcal{M}_2'=\eta A_k^2(t,\xi) \left[\dfrac{k^2-l^2}{\langle t\rangle^2}+\dfrac{|\xi|-|\eta|}{\langle t\rangle^2} \right].
  \end{align*}
Then it suffices to prove that for $i=1,2$,
  \begin{align}\dfrac{|\rho/\sigma|+\langle t\rangle}{\langle t\rangle}\dfrac{\langle\rho\rangle/\sigma^2\cdot|\mathcal{M}_i| }{\langle t-\rho/\sigma\rangle^2} \lesssim_{\delta} \sqrt{|(A_k^{*}\dot{A}_{k}^{*})(t,\xi)|} \sqrt{|(A_l^{*}\dot{A}_{l}^{*})(t,\eta)|}A_{\sigma}(t,\rho)
      \mathrm{e}^{-(2\delta_0/401)\langle\sigma,\rho\rangle^{\f12}}\label{eq:Tibound-0}.
  \end{align}
  Moreover, we  have
    \begin{align}\label{eq:Tibound-0M2}
   \dfrac{1}{\langle t\rangle^2}|\mathcal{M}_2'| &\lesssim_{\delta} \sqrt{|(A_k^{*}\dot{A}_{k}^{*})(t,\xi)|} \sqrt{|(A_l^{*}\dot{A}_{l}^{*})(t,\eta)|}A_{\sigma}(t,\rho)
      \mathrm{e}^{-(\delta_0/200)\langle\sigma,\rho\rangle^{1/2}}.
  \end{align}
  For $\mathcal{M}_1$, we get by Lemma \ref{lem:8.4JiaHao} that
  \begin{align*}
     &\dfrac{|\rho/\sigma|+\langle t\rangle}{\langle t\rangle}\dfrac{\langle\rho\rangle/\sigma^2}{\langle t-\rho/\sigma\rangle^2}|l A_k(t,\xi)^2-kA_l(t,\eta)^2 |\\
     &\lesssim_{\delta} \sqrt{|(A_k\dot{A}_{k})(t,\xi)|} \sqrt{|(A_l\dot{A}_{l})(t,\eta)|}A_{\sigma}(t,\rho)
      \mathrm{e}^{-(\delta_0/200)\langle\sigma,\rho\rangle^{1/2}}.
  \end{align*}
  Then by \eqref{est:par-t-A*-1} and \eqref{est:weight-more-1}, we have
  \begin{align*}
     &\dfrac{|\rho/\sigma|+\langle t\rangle}{\langle t\rangle}\dfrac{\langle\rho\rangle/\sigma^2}{\langle t-\rho/\sigma\rangle^2}|\mathcal{M}_1|\\
     &\leq \dfrac{|\rho/\sigma|+\langle t\rangle}{\langle t\rangle}\dfrac{\langle\rho\rangle/\sigma^2}{\langle t-\rho/\sigma\rangle^2}\big|l A_k(t,\xi)^2-kA_l(t,\eta)^2 \big| \left(1+\dfrac{l^2+|\eta|}{\langle t\rangle^2}\right)\\
     &\lesssim_{\delta}\dfrac{|\rho/\sigma|+\langle t\rangle}{\langle t\rangle}\dfrac{\langle\rho\rangle/\sigma^2}{\langle t-\rho/\sigma\rangle^2}\big|l A_k(t,\xi)^2-kA_l(t,\eta)^2 \big| \left(1+\dfrac{k^2+|\xi|}{\langle t\rangle^2}\right)^{\f12}\left(1+\dfrac{l^2+|\eta|}{\langle t\rangle^2}\right)^{\f12}\langle \sigma, \rho\rangle\\
     &\lesssim_{\delta} \sqrt{|(A_k\dot{A}_{k})(t,\xi)|} \left(1+\dfrac{k^2+|\xi|}{\langle t\rangle^2}\right)^{\f12} \sqrt{|(A_l\dot{A}_{l})(t,\eta)|}A_{\sigma}(t,\rho) \left(1+\dfrac{l^2+|\eta|}{\langle t\rangle^2}\right)^{\f12}\\ &\quad\times
      \langle \sigma, \rho\rangle\mathrm{e}^{-(\delta_0/200)\langle\sigma,\rho\rangle^{1/2}}\\
     & \lesssim_{\delta}\sqrt{|(A_k^{*}\dot{A}_{k}^{*})(t,\xi)|} \sqrt{|(A_l^{*}\dot{A}_{l}^{*})(t,\eta)|}A_{\sigma}(t,\rho)
      \mathrm{e}^{-(2\delta_0/401)\langle\sigma,\rho\rangle^{1/2}}.
  \end{align*}
This proves \eqref{eq:Tibound-0} for $i=1$.

 It is easy to see that
  \begin{align}\label{est:weight-more-2}
     &\left|\dfrac{k^2-l^2}{\langle t\rangle^2}+\dfrac{|\xi|-|\eta|}{\langle t\rangle^2} \right|\lesssim \dfrac{(k-l)^2+|k||k-l|}{\langle t\rangle^2}+\dfrac{|\xi-\eta|}{\langle t\rangle^2}\nonumber\\
     &\lesssim \langle k-l,\xi-\eta\rangle^2 \dfrac{ \langle k\rangle}{\langle t\rangle^2}\lesssim\langle k-l,\xi-\eta\rangle^2\left( 1+\dfrac{k^2+|\xi|}{\langle t\rangle^2}\right)^{\frac{1}{2}} \langle t\rangle^{-1} .
  \end{align}
  Then by \eqref{est:par-t-A*-1}, \eqref{est:weight-more-2} and \eqref{eq:A*R0R1bounds-R1}, we have
  \begin{align*}
    &\dfrac{1}{\langle t\rangle^2}(|\mathcal{M}_2|+|\mathcal{M}_2'|)\\
    &\lesssim \dfrac{1}{\langle t\rangle^2}(|l A_k(t,\xi)^2|+|\eta A_k(t,\xi)^2|)\langle k-l,\xi-\eta\rangle^2\left( 1+\dfrac{k^2+|\xi|}{\langle t\rangle^2}\right)^{\frac{1}{2}} \langle t\rangle^{-1}\\
    &\lesssim_{\delta} \sqrt{|(A_k\dot{A}_{k})(t,\xi)|} \sqrt{|(A_l^*\dot{A}_{l}^*)(t,\eta)|}A_{\sigma}(t,\rho)
      \mathrm{e}^{-(\delta_0/100)\langle\sigma,\rho\rangle^{1/2}}\langle\sigma,\rho\rangle^{2}\left( 1+\dfrac{k^2+|\xi|}{\langle t\rangle^2}\right)^{\frac{1}{2}}\\
    &\lesssim_{\delta} \sqrt{|(A_k^*\dot{A}_{k}^*)(t,\xi)|} \sqrt{|(A_l^*\dot{A}_{l}^*)(t,\eta)|}A_{\sigma}(t,\rho)
      \mathrm{e}^{-(\delta_0/110)\langle\sigma,\rho\rangle^{1/2}},
  \end{align*}
  Then the bounds \eqref{eq:Tibound-0} and \eqref{eq:Tibound-0M2} follow from the fact that
  \beno
  \frac{|\rho/\sigma|+\langle t\rangle}{\langle t\rangle}\frac{\langle\rho\rangle/\sigma^2}{\langle t-\rho/\sigma\rangle^2}\lesssim \langle t\rangle^{-2}\mathrm{e}^{(\delta_0/300)\langle\sigma,\rho\rangle^{1/2}}.
  \eeno

  \smallskip

\noindent\textbf{Step 3.}  Assume that $\big((k,\xi),(l,\eta)\big)\in R_2$. \smallskip

By Lemma \ref{lem:8.4JiaHao}, we have
  \begin{align}\label{eq:Ak-8.4-1}
     \dfrac{|\rho/\sigma|+\langle t\rangle}{\langle t\rangle}\dfrac{\langle\rho\rangle/\sigma^2}{\langle t-\rho/\sigma\rangle^2} &\big(|l A_k(t,\xi)^2|+|kA_l(t,\eta)^2|\big)\\
      &\lesssim_{\delta}\sqrt{|A_k\dot{A}_{k}(t,\xi)|} \sqrt{|A_{\sigma}\dot{A}_{\sigma}(t,\rho)|}A_{l}(t,\eta)
      \mathrm{e}^{-(\delta_0/200)\langle l,\eta\rangle^{1/2}}.\nonumber
  \end{align}
  It is obvious that
  \begin{align*}
     &1+\dfrac{l^2+|\eta|}{\langle t\rangle^2}\lesssim\left( 1+\dfrac{l^2+|\eta|}{\langle t\rangle^2}\right)^{\f12}\langle l,\eta\rangle\leq\left( 1+\dfrac{l^2+|\eta|}{\langle t\rangle^2}\right)^{\f12}\left(1+\dfrac{k^2+|\xi|}{\langle t\rangle^2}\right)^{\f12}\langle l,\eta\rangle.
  \end{align*}
  This along with \eqref{est:par-t-A*-1} and  \eqref{eq:Ak-8.4-1}  gives
  \begin{align}\label{eq:A*R2bounds-A*l}
  \begin{aligned}
   \dfrac{|\rho/\sigma|+\langle t\rangle}{\langle t\rangle}\dfrac{\langle\rho\rangle/\sigma^2}{\langle t-\rho/\sigma\rangle^2} |kA^*_l(t,\eta)^2|\lesssim_{\delta}&\sqrt{|(A^*_k\dot{A}^*_{k})(t,\xi)|} \sqrt{|(A_{\sigma}\dot{A}_{\sigma})(t,\rho)|}\\
   &\times A^*_{l}(t,\eta)
      \mathrm{e}^{-(\delta_0/201)\langle l,\eta\rangle^{\f12}}.
  \end{aligned}
\end{align}
Thus, it remains to prove that
  \begin{align}\label{eq:A*R2bounds-prove}
   &\dfrac{|\rho/\sigma|+\langle t\rangle}{\langle t\rangle}\dfrac{\langle\rho\rangle/\sigma^2}{\langle t-\rho/\sigma\rangle^2} |lA^*_k(t,\xi)^2|\nonumber\\
   &\lesssim_{\delta}\sqrt{|(A^*_k\dot{A}^*_{k})(t,\xi)|} \sqrt{|(A_{\sigma}\dot{A}_{\sigma})(t,\rho)|}A^*_{l}(t,\eta)
      \mathrm{e}^{-(\delta_0/201)\langle l,\eta\rangle^{1/2}}.
  \end{align}
If $k^2+|\xi|\leq\langle t\rangle^2$, then
\begin{align}\label{eq:<1}
     & 1+\dfrac{k^2+|\xi|}{\langle t\rangle^2}\lesssim 1\lesssim\left( 1+\dfrac{k^2+|\xi|}{\langle t\rangle^2}\right)^{\f12}
  \left( 1+\dfrac{l^2+|\eta|}{\langle t\rangle^2}\right)^{\f12},
  \end{align}
  and \eqref{eq:A*R2bounds-prove} follows from \eqref{eq:Ak-8.4-1} and \eqref{est:par-t-A*-1}. Now we assume $k^2+|\xi|\geq\langle t\rangle^2$. We get by \eqref{est:weight-more-1} that
  \begin{align}\label{eq:A*R2bounds-A*k}
     & 1+\dfrac{k^2+|\xi|}{\langle t\rangle^2}\lesssim\left( 1+\dfrac{k^2+|\xi|}{\langle t\rangle^2}\right)^{\f12}
  \left( 1+\dfrac{l^2+|\eta|}{\langle t\rangle^2}\right)^{\f12}\left( 1+\dfrac{\rho^2+|\sigma|}{\langle t\rangle^2}\right)^{\f12}.
  \end{align}
  Thanks to \eqref{est:par-t-A*-1} and \eqref{eq:A*R2bounds-A*k}, we only need to prove that
 \begin{align}\label{eq:A*R2bounds-prove-11}
   &\bigg\langle \dfrac{|\sigma|+|\rho|^{1/2}}{\langle t\rangle}\bigg\rangle\dfrac{|\rho/\sigma|+\langle t\rangle}{\langle t\rangle}\dfrac{\langle\rho\rangle/\sigma^2}{\langle t-\rho/\sigma\rangle^2} |lA_k(t,\xi)^2|\nonumber\\
   &\lesssim_{\delta}\sqrt{|(A_k\dot{A}_{k})(t,\xi)|} \sqrt{|(A_{\sigma}\dot{A}_{\sigma})(t,\rho)|}A_{l}(t,\eta)
      \mathrm{e}^{-(\delta_0/200)\langle l,\eta\rangle^{1/2}}.
  \end{align}
  If $\rho^2+|\sigma|\leq\langle t\rangle^2 $, then \eqref{eq:A*R2bounds-prove-11} follows from \eqref{eq:Ak-8.4-1}. If $\rho^2+|\sigma|\geq\langle t\rangle^2 $, then \eqref{eq:A*R2bounds-prove-11} follows from $k^2+|\xi|\geq\langle t\rangle^2$, $\langle\rho\rangle/\sigma^2\leq\langle\rho/\sigma\rangle/|\sigma|$ and \eqref{eq:A*R0R1bounds-R1-1}.
\end{proof}

\begin{lemma}\label{lem:A*R0R1R2-1-R1}
Assume $t\ge 1$ and let $(\sigma,\rho)=(k-l,\xi-\eta)$. Suppose $\sigma \neq 0$. If $\big((k,\xi),(l,\eta)\big)\in R_1$, then it holds that
\begin{align}\label{eq:A*R0R1bounds-R1}
      \dfrac{|l A_k(t,\xi)^2|+|\eta A_k(t,\xi)^2|}{\langle t\rangle^3}
\lesssim_{\delta}&\sqrt{|(A_k\dot{A}_{k})(t,\xi)|} \sqrt{|(A_l^{*}\dot{A}_{l}^{*})(t,\eta)|}A_{\sigma}(t,\rho)
      \mathrm{e}^{-(\delta_0/100)\langle\sigma,\rho\rangle^{1/2}}.
    \end{align}
If $\big((k,\xi),(l,\eta)\big)\in R_2$, $ k^2+|\xi|\geq \langle t\rangle^2$, $ \sigma^2+|\rho|\geq \langle t\rangle^2$, then
    \begin{align}\label{eq:A*R0R1bounds-R1-1}
     &\dfrac{|\sigma|+|\rho|^{1/2}}{\langle t\rangle}\dfrac{|\rho/\sigma|+\langle t\rangle}{\langle t\rangle|\sigma|}\dfrac{\langle\rho/\sigma\rangle}{\langle t-\rho/\sigma\rangle^2}( |lA_k(t,\xi)^2|+|\eta A_k(t,\xi)^2|)\\
   &\lesssim_{\delta}\sqrt{|(A_k\dot{A}_{k})(t,\xi)|} \sqrt{|(A_{\sigma}\dot{A}_{\sigma})(t,\rho)|}A_{l}(t,\eta)
      \mathrm{e}^{-(\delta_0/100)\langle l,\eta\rangle^{1/2}}\nonumber
    \end{align}
\end{lemma}

\begin{proof}
  If $\big((k,\xi),(l,\eta)\big)\in R_1$, then $|(k-l,\xi-\eta)|\leq|(l,\eta)|$, and by \eqref{est:8.3JiaHao-mod3}
   for $j=1$, we have
\begin{align*}
       A_{k}^{(1)}(t,\xi)\lesssim_{\delta}& A_{\sigma}(t,\rho)A_{l}^{(1)}(t,\eta)\mathrm{e}^{-(\lambda(t)/20) \langle\sigma,\rho\rangle^{1/2}}.
\end{align*}
By \eqref{eq:A1k}, we have $A_{l}^{(1)}(t,\eta)\leq A_{l}(t,\eta)$ and
\begin{align}\label{eq:Ak1}
  A_{k}(t,\xi)=A_{k}^{(1)}(t,\eta)
\bigg(1+\mathbf{1}_{\{k\neq0\}}\dfrac{|\xi|/k^2}{\langle t-\xi/k\rangle}\bigg)\lesssim A_{k}^{(1)}(t,\eta)
\big(1+\mathbf{1}_{\{k\neq0,|\xi/k|\leq 2t\}}|\xi/k^2|\big).
\end{align}
%the definition of $ I_{t}^{*},\ I_{t}^{**}$ in \eqref{eq:I*} and \eqref{eq:I**}. We consider two cases.\smallskip
Thus,
\begin{align*}
       &A_{k}(t,\xi)\lesssim_{\delta} A_{\sigma}(t,\rho)A_{l}(t,\eta)\mathrm{e}^{-(\lambda(t)/20) \langle\sigma,\rho\rangle^{1/2}}(1+\mathbf{1}_{\{k\neq0,|\xi/k|\leq 2t\}}|\xi/k^2|),\\
       &\dfrac{|l A_k(t,\xi)^2|+|\eta A_k(t,\xi)^2|}{\langle t\rangle^3}\\ &\lesssim_{\delta}  \frac{|(l,\eta)|}{\langle t\rangle^3}A_{k}(t,\xi)A_{l}(t,\eta)A_{\sigma}(t,\rho)\mathrm{e}^{-(\lambda(t)/20)\langle\sigma,\rho\rangle^{1/2}}\big(1+\mathbf{1}_{\{k\neq0,|\xi/k|\leq 2t\}}|\xi/k^2|\big).
\end{align*}
Note that if $k\neq0, |\xi/k|\leq 2t$, then
\begin{align*}
       &{|(l,\eta)|^{1/2}}|\xi/k^2|\lesssim\langle\sigma,\rho\rangle^{1/2}{|( k,\xi)|^{1/2}}|\xi/k^2|\\&=\langle\sigma,\rho\rangle^{1/2}\langle\xi/k\rangle^{1/2}|k|^{1/2}|\xi/k||k|^{-1} \leq\langle\sigma,\rho\rangle^{1/2}\langle\xi/k\rangle^{3/2}\lesssim\langle\sigma,\rho\rangle^{1/2}\langle t\rangle^{3/2}.
\end{align*}We have\begin{align*}
       &\frac{|(l,\eta)|^{1/2}}{\langle t\rangle^{3/2}}(1+\mathbf{1}_{\{k\neq0,|\xi/k|\leq 2t\}}|\xi/k^2|)\lesssim\frac{|(l,\eta)|^{1/2}}{\langle t\rangle^{3/2}}+\langle\sigma,\rho\rangle^{1/2}\leq \bigg(\frac{|(l,\eta)|^{1/2}}{\langle t\rangle}+1\bigg)\langle\sigma,\rho\rangle^{1/2}\\ &\lesssim\bigg(\frac{|l|+|\eta|}{\langle t\rangle^2}+1\bigg)^{1/2}\langle\sigma,\rho\rangle^{1/2}\leq\bigg(\frac{|l|^2+|\eta|}{\langle t\rangle^2}+1\bigg)^{1/2}\langle\sigma,\rho\rangle^{1/2},
\end{align*}
and
\begin{align*}
       &\dfrac{|l A_k(t,\xi)^2|+|\eta A_k(t,\xi)^2|}{\langle t\rangle^3}\\ &\lesssim_{\delta}  \frac{|(l,\eta)|^{1/2}}{\langle t\rangle^{3/2}}A_{k}(t,\xi)A_{l}(t,\eta)A_{\sigma}(t,\rho)\mathrm{e}^{-(\lambda(t)/20)\langle\sigma,\rho\rangle^{1/2}}\bigg(\frac{|l|^2+|\eta|}{\langle t\rangle^2}+1\bigg)^{1/2}\langle\sigma,\rho\rangle^{1/2}\\&=\frac{|(l,\eta)|^{1/2}}{\langle t\rangle^{3/2}}A_{k}(t,\xi)A_{l}^*(t,\eta)A_{\sigma}(t,\rho)\mathrm{e}^{-(\lambda(t)/20)\langle\sigma,\rho\rangle^{1/2}}\langle\sigma,\rho\rangle^{1/2}
       \\ &\lesssim_{\delta}\frac{\langle l,\eta\rangle^{1/4}\langle k,\xi\rangle^{1/4}}{\langle t\rangle^{3/2}}A_{k}(t,\xi)A_{l}^*(t,\eta)A_{\sigma}(t,\rho)\mathrm{e}^{-(\lambda(t)/21)\langle\sigma,\rho\rangle^{1/2}}.
\end{align*}
Then the bound \eqref{eq:A*R0R1bounds-R1} follows from \eqref{est:7.4JiaHao-2} and \eqref{est:par-t-A*}.

Now we assume   $\big((k,\xi),(l,\eta)\big)\in R_2$, $ k^2+|\xi|\geq \langle t\rangle^2$, $ \sigma^2+|\rho|\geq \langle t\rangle^2$. Then $|(l,\eta)|\leq|(k-l,\xi-\eta)|$, and by \eqref{est:8.3JiaHao-mod3}
   for $j=1$, we have
\begin{align*}
       A_{k}^{(1)}(t,\xi)\lesssim_{\delta}& A_{\sigma}^{(1)}(t,\rho)A_{l}(t,\eta)\mathrm{e}^{-(\lambda(t)/20) \langle l,\eta\rangle^{1/2}}.
\end{align*}
By \eqref{eq:A1k}, we have $A_{\sigma}^{(1)}(t,\rho)\leq A_{\sigma}(t,\rho)$ and \eqref{eq:Ak1}.
Note that if $k\neq0,|\xi/k|\leq 2t $, as $ k^2+|\xi|\geq \langle t\rangle^2$, then $(k^2+|\xi|)|\xi/k^2|=|\xi|+|\xi/k|^2\leq|\xi|+|2t|^2$,
\begin{align}\label{eq2}
       &|\xi/k^2|\leq\frac{|\xi|+|2t|^2}{k^2+|\xi|}\leq 1+\frac{|2t|^2}{k^2+|\xi|}\leq 1+\frac{|2t|^2}{\langle t\rangle^2}\leq 5,
\end{align}and \eqref{eq:Ak1} becomes $A_{k}(t,\xi)\lesssim A_{k}^{(1)}(t,\eta)$.
Thus,
\begin{align}\nonumber
       &A_{k}(t,\xi)\lesssim_{\delta} A_{\sigma}(t,\rho)A_{l}(t,\eta)\mathrm{e}^{-(\lambda(t)/20) \langle l,\eta\rangle^{1/2}},\\ \label{eq:Ak2}&|l A_k(t,\xi)^2|+|\eta A_k(t,\xi)^2| \lesssim_{\delta}  {|(l,\eta)|}A_{k}(t,\xi)A_{l}(t,\eta)A_{\sigma}(t,\rho)\mathrm{e}^{-(\lambda(t)/20)\langle l,\eta\rangle^{1/2}}.
\end{align}

 \textbf{Case 1.} $ |t-\rho/\sigma|\geq \frac{|\rho|}{10|\sigma|}$. Then $ \langle t-\rho/\sigma\rangle \thickapprox \langle t\rangle+|\rho/\sigma|,$ and \begin{align*}
     &\dfrac{|\sigma|+|\rho|^{1/2}}{\langle t\rangle}\dfrac{|\rho/\sigma|+\langle t\rangle}{\langle t\rangle|\sigma|}\dfrac{\langle\rho/\sigma\rangle}{\langle t-\rho/\sigma\rangle^2}\lesssim\dfrac{|\sigma|+|\rho|^{1/2}}{\langle t\rangle}\dfrac{|\rho/\sigma|+\langle t\rangle}{\langle t\rangle|\sigma|}\dfrac{\langle\rho/\sigma\rangle}{(|\rho/\sigma|+\langle t\rangle)^2}\\ &\leq\dfrac{|\sigma|+|\rho|^{1/2}}{\langle t\rangle}\dfrac{1}{\langle t\rangle|\sigma|}=\dfrac{|\sigma|+|\rho|^{1/2}}{\langle t\rangle^2|\sigma|}\lesssim\dfrac{\langle\rho\rangle^{1/2}}{\langle t\rangle^2}.
    \end{align*}
    Then by \eqref{eq:Ak2}, we have
    \begin{align*}
     &\dfrac{|\sigma|+|\rho|^{1/2}}{\langle t\rangle}\dfrac{|\rho/\sigma|+\langle t\rangle}{\langle t\rangle|\sigma|}\dfrac{\langle\rho/\sigma\rangle}{\langle t-\rho/\sigma\rangle^2}( |lA_k(t,\xi)^2|+|\eta A_k(t,\xi)^2|)\\
   &\lesssim_{\delta}\dfrac{\langle\rho\rangle^{1/2}}{\langle t\rangle^2}{|(l,\eta)|}A_{k}(t,\xi)A_{l}(t,\eta)A_{\sigma}(t,\rho)\mathrm{e}^{-(\lambda(t)/20)\langle l,\eta\rangle^{1/2}}\\
   &\lesssim_{\delta}\dfrac{\langle\rho\rangle^{1/4}\langle\xi\rangle^{1/4}}{\langle t\rangle^2}A_{k}(t,\xi)A_{l}(t,\eta)A_{\sigma}(t,\rho)\mathrm{e}^{-(\lambda(t)/21)\langle l,\eta\rangle^{1/2}}.
    \end{align*}
  Thus, the bound \eqref{eq:A*R0R1bounds-R1-1} follows from \eqref{est:7.4JiaHao-2}.\smallskip

  \textbf{Case 2.} $ |t-\rho/\sigma|\leq \frac{|\rho|}{10|\sigma|}$. Then $0<t<2|\rho/\sigma|<4t$. Similar to the proof of \eqref{eq2}, as $ \sigma^2+|\rho|\geq \langle t\rangle^2$, we have $|\rho/\sigma^2|\leq 5$. Then\begin{align*}
     &\dfrac{|\sigma|+|\rho|^{1/2}}{\langle t\rangle}\dfrac{|\rho/\sigma|+\langle t\rangle}{\langle t\rangle|\sigma|}\dfrac{\langle\rho/\sigma\rangle}{\langle t-\rho/\sigma\rangle^2}\lesssim\dfrac{|\sigma|+|\rho|^{1/2}}{\langle t\rangle}\dfrac{\langle t\rangle}{\langle t\rangle|\sigma|}\dfrac{\langle t\rangle}{\langle t-\rho/\sigma\rangle^2}\\ &=\dfrac{|\sigma|+|\rho|^{1/2}}{|\sigma|}\dfrac{1}{\langle t-\rho/\sigma\rangle^2}\lesssim\dfrac{1}{\langle t-\rho/\sigma\rangle^2}\leq\dfrac{1}{\langle t-\rho/\sigma\rangle}.
    \end{align*}
    Then by \eqref{eq:Ak2}, we have
    \begin{align*}
     &\dfrac{|\sigma|+|\rho|^{1/2}}{\langle t\rangle}\dfrac{|\rho/\sigma|+\langle t\rangle}{\langle t\rangle|\sigma|}\dfrac{\langle\rho/\sigma\rangle}{\langle t-\rho/\sigma\rangle^2}( |lA_k(t,\xi)^2|+|\eta A_k(t,\xi)^2|)\\
   &\lesssim_{\delta}\dfrac{1}{\langle t-\rho/\sigma\rangle}{|(l,\eta)|}A_{k}(t,\xi)A_{l}(t,\eta)A_{\sigma}(t,\rho)\mathrm{e}^{-(\lambda(t)/20)\langle l,\eta\rangle^{1/2}}.
    \end{align*}Then the bound \eqref{eq:A*R0R1bounds-R1-1} follows from \eqref{est:8.40} and \eqref{est:7.4JiaHao-4}.
\end{proof}

The following lemma is an analogue of Lemma \ref{lem:8.5JiaHao}.

\begin{lemma}\label{lem:A*R0R1R2-2}
Assume $t\ge 1$ and  let  $(\sigma,\rho)=(k-l,\xi-\eta)$. Suppose $\sigma\neq 0$. Then it holds that
  \begin{itemize}
    \item[(i)]  If $\big((k,\xi),(l,\eta)\big)\in R_0\bigcup R_1$, then
    \begin{align}\label{eq:A*R0R1bounds-1}
    \begin{aligned}
      \dfrac{|\rho/\sigma|^2+\langle t\rangle^2}{|\sigma|\langle t\rangle^2}\dfrac{1}{\langle t-\rho/\sigma\rangle^2} &|\eta A_k^*(t,\xi)^2-\xi A_l^{*}(t,\eta)^2 |\\
      &\lesssim_{\delta}\sqrt{|(A_k^{*}\dot{A}_{k}^{*})(t,\xi)|} \sqrt{|(A_l^{*}\dot{A}_{l}^{*})(t,\eta)|}A_{\sigma}(t,\rho)
      \mathrm{e}^{-(\delta_0/201)\langle\sigma,\rho\rangle^{1/2}}.
   \end{aligned}
    \end{align}
    \item[(ii)]
    If $\big((k,\xi),(l,\eta)\big)\in R_2$, then
    \begin{align}\label{eq:A*R2bounds-1}
    \begin{aligned}
      \dfrac{|\rho/\sigma|^2+\langle t\rangle^2}{|\sigma|\langle t\rangle^2}\dfrac{1}{\langle t-\rho/\sigma\rangle^2} &\big(|\eta A_k^*(t,\xi)^2|+|\xi A_l^{*}(t,\eta)^2 |\big)\\
      &\lesssim_{\delta}\sqrt{|(A_k^{*}\dot{A}_{k}^{*})(t,\xi)|} \sqrt{|(A_{\sigma}\dot{A}_{\sigma})(t,\rho)|}A_{l}^*(t,\eta)
      \mathrm{e}^{-(\delta_0/201)\langle l,\eta\rangle^{1/2}}.
    \end{aligned}
    \end{align}
  \end{itemize}
\end{lemma}
\begin{proof}

  \textbf{Step 1.} If $\big((k,\xi),(l,\eta)\big)\in R_0$, the proof is exactly as in \textbf{Step 1} of the proof of \eqref{eq:A*R0R1bounds}. We omit the details.\smallskip

\noindent\textbf{Step 2.} Assume that $\big((k,\xi),(l,\eta)\big)\in R_1$ and we now prove \eqref{eq:A*R0R1bounds-1}. \smallskip

We write
  \begin{align*}
      &\eta A_k^*(t,\xi)^2-\xi A_l^{*}(t,\eta)^2 =\mathcal{M}_1'+\mathcal{M}_2',
   \end{align*}
   where
   \begin{align*}
      & \mathcal{M}_1'=\left(\eta A_k^2(t,\xi)-\xi A_l^{2}(t,\eta)\right)\left(1 +\dfrac{l^2+|\eta|}{\langle \eta\rangle^2}\right),\\
      &\mathcal{M}_2'=\eta A^2_k(t,\xi)\left[\dfrac{k^2-l^2}{\langle t\rangle^2} +\dfrac{|\xi|-|\eta|}{\langle t\rangle^2}\right].
   \end{align*}
 It suffices to prove that  for $i=1,2$,
    \begin{align}\label{eq:A*R0R1bounds-M12'}
      & \dfrac{|\rho/\sigma|^2+\langle t\rangle^2}{|\sigma|\langle t\rangle^2}\dfrac{|\mathcal{M}_i'|}{\langle t-\rho/\sigma\rangle^2}\lesssim_{\delta} \sqrt{|(A_k^{*}\dot{A}_{k}^{*})(t,\xi)|} \sqrt{|(A_l^{*}\dot{A}_{l}^{*})(t,\eta)|}A_{\sigma}(t,\rho)
      \mathrm{e}^{-(\delta_0/201)\langle\sigma,\rho\rangle^{\f12}}.
   \end{align}
  By Lemma \ref{lem:8.5JiaHao}, we have
  \begin{align*}
     &\dfrac{|\rho/\sigma|^2+\langle t\rangle^2}{|\sigma|\langle t\rangle^2}\dfrac{1}{\langle t-\rho/\sigma\rangle^2} |\eta A_k(t,\xi)^2- \xi A_l(t,\eta)^2|\\
     &\lesssim_{\delta} \sqrt{|(A_k\dot{A}_{k})(t,\xi)|} \sqrt{|(A_l\dot{A}_{l})(t,\eta)|}A_{\sigma}(t,\rho)
      \mathrm{e}^{-(\delta_0/200)\langle\sigma,\rho\rangle^{1/2}}.
  \end{align*}
  Then by  \eqref{est:weight-more-1} and \eqref{est:par-t-A*-1} in Lemma \ref{lem:par-t-A*}, we have
  \begin{align*}
     &\dfrac{|\rho/\sigma|^2+\langle t\rangle^2}{|\sigma|\langle t\rangle^2}\dfrac{1}{\langle t-\rho/\sigma\rangle^2} |\mathcal{M}_1'|\\
     &\leq \dfrac{|\rho/\sigma|^2+\langle t\rangle^2}{|\sigma|\langle t\rangle^2}\dfrac{1}{\langle t-\rho/\sigma\rangle^2} \big|\eta A_k(t,\xi)^2- \xi A_l(t,\eta)^2 \big| \left(1+\dfrac{l^2+|\eta|}{\langle t\rangle^2}\right)\\
     &\lesssim_{\delta}\dfrac{|\rho/\sigma|^2+\langle t\rangle^2}{|\sigma|\langle t\rangle^2}\dfrac{1}{\langle t-\rho/\sigma\rangle^2} \big|\eta A_k(t,\xi)^2-\xi A_l(t,\eta)^2\big| \left(1+\dfrac{k^2+|\xi|}{\langle t\rangle^2}\right)^{\f12}\left(1+\dfrac{l^2+|\eta|}{\langle t\rangle^2}\right)^{\f12}\langle \sigma, \rho\rangle\\
     &\lesssim_{\delta} \sqrt{|(A_k\dot{A}_{k})(t,\xi)|} \left(1+\dfrac{k^2+|\xi|}{\langle t\rangle^2}\right)^{\f12} \sqrt{|(A_l\dot{A}_{l})(t,\eta)|}A_{\sigma}(t,\rho) \left(1+\dfrac{l^2+|\eta|}{\langle t\rangle^2}\right)^{\f12}
      \\
     &\quad\times\langle \sigma, \rho\rangle\mathrm{e}^{-(\delta_0/200)\langle\sigma,\rho\rangle^{1/2}}\\
      &\lesssim_{\delta}\sqrt{|(A_k^{*}\dot{A}_{k}^{*})(t,\xi)|} \sqrt{|(A_l^{*}\dot{A}_{l}^{*})(t,\eta)|}A_{\sigma}(t,\rho)
      \mathrm{e}^{-(2\delta_0/401)\langle\sigma,\rho\rangle^{1/2}}.
  \end{align*}
This proves \eqref{eq:A*R0R1bounds-M12'} for $i=1$.

  The inequality \eqref{eq:A*R0R1bounds-M12'} for $i=2$ can be deduced from \eqref{eq:Tibound-0M2} and the fact
   \begin{align*}
      & \dfrac{|\rho/\sigma|^2+\langle t\rangle^2}{|\sigma|\langle t\rangle^2}\dfrac{1}{\langle t-\rho/\sigma\rangle^2} \lesssim_{\delta} \dfrac{1}{\langle t\rangle^2}\mathrm{e}^{\delta\langle \sigma,\rho\rangle^{1/2}}.
   \end{align*}

\noindent\textbf{Step 3.} Assume that $\big((k,\xi),(l,\eta)\big)\in R_2$ and  we  prove \eqref{eq:A*R2bounds-1}.\smallskip

 By Lemma \ref{lem:8.5JiaHao}, we have
  \begin{align}\label{eq:Ak-8.5-1}
     \dfrac{|\rho/\sigma|^2+\langle t\rangle^2}{|\sigma|\langle t\rangle^2}\dfrac{1}{\langle t-\rho/\sigma\rangle^2}&\big(|\eta A_k(t,\xi)^2|+|\xi A_l(t,\eta)^2|\big)\\
      &\lesssim_{\delta}\sqrt{|(A_k\dot{A}_{k})(t,\xi)|} \sqrt{|(A_{\sigma}\dot{A}_{\sigma})(t,\rho)|}A_{l}(t,\eta)
      \mathrm{e}^{-(\delta_0/200)\langle l,\eta\rangle^{1/2}}.\nonumber
  \end{align}
  It is obvious that
  \begin{align*}
     &1+\dfrac{l^2+|\eta|}{\langle t\rangle^2}\lesssim\left( 1+\dfrac{l^2+|\eta|}{\langle t\rangle^2}\right)^{\f12}\langle l,\eta\rangle\leq\left( 1+\dfrac{l^2+|\eta|}{\langle t\rangle^2}\right)^{\f12}\left(1+\dfrac{k^2+|\xi|}{\langle t\rangle^2}\right)^{\f12}\langle l,\eta\rangle.
  \end{align*}
  This along with \eqref{est:par-t-A*-1}, \eqref{eq:Ak-8.5-1} and the definition of $A_{l}^*(t,\eta)$ gives
  \begin{align}\label{eq:xiA*R2bounds-A*l}
   \dfrac{|\rho/\sigma|^2+\langle t\rangle^2}{|\sigma|\langle t\rangle^2}\dfrac{|\xi A^*_l(t,\eta)^2|}{\langle t-\rho/\sigma\rangle^2} &\lesssim_{\delta}\sqrt{|(A^*_k\dot{A}^*_{k})(t,\xi)|} \sqrt{|(A_{\sigma}\dot{A}_{\sigma})(t,\rho)|}A^*_{l}(t,\eta)
      \mathrm{e}^{-(\delta_0/201)\langle l,\eta\rangle^{\f12}}.
  \end{align}

Therefore, for \eqref{eq:A*R2bounds-1}, it remains to prove that
  \begin{align}\label{eq:xiA*R2bounds-prove}
   &\dfrac{|\rho/\sigma|^2+\langle t\rangle^2}{|\sigma|\langle t\rangle^2}\dfrac{|\eta A^*_k(t,\xi)^2|}{\langle t-\rho/\sigma\rangle^2}
   \lesssim_{\delta}\sqrt{|(A^*_k\dot{A}^*_{k})(t,\xi)|} \sqrt{|(A_{\sigma}\dot{A}_{\sigma})(t,\rho)|}A^*_{l}(t,\eta)
      \mathrm{e}^{-(\delta_0/201)\langle l,\eta\rangle^{\f12}}.
  \end{align}
If $k^2+|\xi|\leq\langle t\rangle^2$, then \eqref{eq:xiA*R2bounds-prove} follows from  \eqref{eq:Ak-8.5-1}, \eqref{eq:<1} and \eqref{est:par-t-A*-1}. Now we assume $k^2+|\xi|\geq\langle t\rangle^2$.
  Thanks to \eqref{est:par-t-A*-1} and \eqref{eq:A*R2bounds-A*k}, we only need to prove that
  \begin{align}\label{eq:xiA*R2bounds-prove-11}
  \begin{aligned}
   &\bigg\langle \dfrac{|\sigma|+|\rho|^{1/2}}{\langle t\rangle}\bigg\rangle\dfrac{|\rho/\sigma|^2+\langle t\rangle^2}{|\sigma|\langle t\rangle^2}\dfrac{|\eta A_k(t,\xi)^2|}{\langle t-\rho/\sigma\rangle^2}\\
   &\lesssim_{\delta}\sqrt{|(A_k\dot{A}_{k})(t,\xi)|} \sqrt{|(A_{\sigma}\dot{A}_{\sigma})(t,\rho)|}A_{l}(t,\eta)
      \mathrm{e}^{-(\delta_0/200)\langle l,\eta\rangle^{1/2}}.
  \end{aligned}
  \end{align}
  If $\rho^2+|\sigma|\leq\langle t\rangle^2$, then \eqref{eq:xiA*R2bounds-prove-11} follows from  \eqref{eq:Ak-8.5-1}. If $\rho^2+|\sigma|\geq\langle t\rangle^2$, then \eqref{eq:xiA*R2bounds-prove-11} follows from $k^2+|\xi|\geq\langle t\rangle^2$, \eqref{eq:A*R0R1bounds-R1-1} and
  \beno
  \dfrac{|\rho/\sigma|^2+\langle t\rangle^2}{|\sigma|\langle t\rangle^2}\leq \dfrac{|\rho/\sigma|+\langle t\rangle}{|\sigma|\langle t\rangle}\dfrac{|\rho/\sigma|+\langle t\rangle}{\langle t\rangle}\lesssim \dfrac{|\rho/\sigma|+\langle t\rangle}{|\sigma|\langle t\rangle}\langle\rho/\sigma\rangle.
 \eeno
\end{proof}

The following lemma is an analogue of Lemma \ref{lem:8.6JiaHao}.

\begin{lemma}\label{lem:A*S0S1S2-1}
  Assume  $t\geq1$ and let $\rho=\xi-\eta$. It holds that
  \begin{itemize}
    \item[(i)]  If $\big((k,\xi),(l,\eta)\big)\in \Sigma_0\bigcup \Sigma_1$, then
    \begin{align}\label{eq:A*S0S1bounds-1}
      \dfrac{1}{\langle \rho\rangle\langle t\rangle+\langle \rho\rangle^{1/4}\langle t\rangle^{7/4}} &|\eta A_k^*(t,\xi)^2-\xi A_k^{*}(t,\eta)^2 |\\
      &\lesssim_{\delta}\sqrt{|(A_k^{*}\dot{A}_{k}^{*})(t,\xi)|} \sqrt{|(A_k^{*}\dot{A}_{k}^{*})(t,\eta)|}A_{NR}(t,\rho)
      \mathrm{e}^{-(\delta_0/201)\langle\rho\rangle^{1/2}}.\nonumber
    \end{align}
    \item[(ii)]
    If $\big((k,\xi),(l,\eta)\big)\in \Sigma_2$, then
    \begin{align}\label{eq:A*S2bounds-1}
      \dfrac{1}{\langle \rho\rangle\langle t\rangle+\langle \rho\rangle^{1/4}\langle t\rangle^{7/4}} &\big(|\eta A_k^*(t,\xi)^2|+|\xi A_k^{*}(t,\eta)^2 |\big)\\
      &\lesssim_{\delta}\sqrt{|(A_k^{*}\dot{A}_{k}^{*})(t,\xi)|} \sqrt{|(A_{NR}\dot{A}_{NR})(t,\rho)|}A_{k}^*(t,\eta)
      \mathrm{e}^{-(\delta_0/201)\langle k,\eta\rangle^{1/2}}.\nonumber
    \end{align}
  \end{itemize}
\end{lemma}

\begin{proof}
  \textbf{Step 1.} Assume that $\big((k,\xi),(l,\eta)\big)\in\Sigma_0\cup\Sigma_1$ and we  prove \eqref{eq:A*S0S1bounds-1}.
  \smallskip

{By the symmetry, we assume $|\xi|\leq|\eta|$}. We write
  \begin{align*}
     & [\eta A_k^*(t,\xi)^2-\xi A_k^{*}(t,\eta)^2 ]=\mathcal{T}_1''+\mathcal{T}_2'',
  \end{align*}
  where
  \begin{align*}
     \mathcal{T}_1''= [\eta A_k(t,\xi)^2-\xi A_k(t,\eta)^2 ]\left(1+\dfrac{k^2+|\eta|}{\langle t\rangle^2}\right),
     \quad\mathcal{T}_2''=\eta A^2_k(s,\xi)\dfrac{|\xi|-|\eta|}{\langle t\rangle^2}.
  \end{align*}
  For $\mathcal{T}_1''$, we get by Lemma \ref{lem:8.6JiaHao}, \eqref{est:weight-more-1} and \eqref{est:par-t-A*-1} that
  \begin{align*}
     \dfrac{|\mathcal{T}_1''|}{\langle \rho\rangle\langle t\rangle+\langle \rho\rangle^{1/4}\langle t\rangle^{7/4}}\lesssim_{\delta}&\sqrt{|(A_k\dot{A}_{k})(t,\xi)|} \sqrt{|(A_k\dot{A}_{k})(t,\eta)|}A_{NR}(t,\rho)
      \mathrm{e}^{-(\delta_0/200)\langle\rho\rangle^{1/2}}\\
      &\times \left(1+\dfrac{k^2+|\eta|}{\langle t\rangle^2}\right)\\
      \lesssim_{\delta}&\sqrt{|(A_k\dot{A}_{k})(t,\xi)|} \sqrt{|(A_k\dot{A}_{k})(t,\eta)|}A_{NR}(t,\rho)
      \mathrm{e}^{-(\delta_0/201)\langle\rho\rangle^{1/2}}\\
      &\times \left(1+\dfrac{k^2+|\eta|}{\langle t\rangle^2}\right)^{\f12}\left(1+\dfrac{k^2+|\xi|}{\langle t\rangle^2}\right)^{\f12} \\
      \lesssim_{\delta}&\sqrt{|(A^*_k\dot{A}^*_{k})(t,\xi)|} \sqrt{|(A^*_k\dot{A}^*_{k})(t,\eta)|}A_{NR}(t,\rho)
      \mathrm{e}^{-(\delta_0/201)\langle\rho\rangle^{1/2}}.
  \end{align*}

  Thus, it suffices to prove that
  \begin{align}\label{eq:A*S0S1bounds-prove-1}
     &\dfrac{1}{\langle t\rangle^{7/4}}|\mathcal{T}_2''|\lesssim_{\delta} \sqrt{|(A^*_k\dot{A}^*_{k})(t,\xi)|} \sqrt{|(A^*_k\dot{A}^*_{k})(t,\eta)|}A_{NR}(t,\rho)
      \mathrm{e}^{-(\delta_0/201)\langle\rho\rangle^{1/2}}.
  \end{align}

\if0Similar to \eqref{est:weight-more-2}, we have
\begin{align}\label{est:weight-more-3}
  \left|\dfrac{\xi^2}{\langle t\rangle^4+\langle t\rangle^{\sigma_1}|\xi|} -\dfrac{\eta^2}{\langle t\rangle^4+\langle t\rangle^{\sigma_1}|\eta|}\right|& \lesssim_{\delta} \langle \rho\rangle^{2}\langle \xi\rangle^{-1}\left(1+\dfrac{k^2}{\langle t\rangle^2}+\dfrac{\xi^2}{\langle t\rangle^4+\langle t\rangle^{\sigma_1}|\xi|}\right).
\end{align}
  If $(k,\xi)\notin I_{t}^{**}$ or $(k,\xi),(k,\eta)\in I_t^{*}$,\fi
  We get by  \eqref{est:7.2JiaHao-3} that
  %\eqref{est:weight-more-1}, \eqref{est:weight-more-3} and \eqref{est:bl/bk-1} that \eqref{eq:tr-3} with $\beta=1/2$,
  \begin{align*}
    |\mathcal{T}_2''| =&|\eta|\mathrm{e}^{2\lambda(t)\langle k,\xi\rangle^{1/2}}\left[\dfrac{\mathrm{e}^{\sqrt{\delta}\langle \xi\rangle^{1/2}}}{b_k(t,\xi)}+\mathrm{e}^{\sqrt{\delta}|k|^{1/2}}\right]^2 \left|\dfrac{|\xi|-|\eta|}{\langle t\rangle^2}\right|\\
    \lesssim_{\delta} &|\eta|\mathrm{e}^{2\lambda(t)\langle k,\xi\rangle^{1/2}}\left[\dfrac{\mathrm{e}^{\sqrt{\delta}\langle \xi\rangle^{1/2}}}{b_k(t,\xi)}+
    \mathrm{e}^{\sqrt{\delta}|k|^{1/2}}\right] \left[\dfrac{\mathrm{e}^{\sqrt{\delta}\langle \xi\rangle^{1/2}}}{b_k(t,\eta)}+
    \mathrm{e}^{\sqrt{\delta}|k|^{1/2}}\right]\mathrm{e}^{2\sqrt{\delta}\langle \rho\rangle^{1/2}}\dfrac{|\rho|}{\langle t\rangle^2}\\
    %&\times\left|\dfrac{\xi^2}{\langle t\rangle^4+\langle t\rangle^{\sigma_1}|\xi|} -\dfrac{\eta^2}{\langle t\rangle^4+\langle t\rangle^{\sigma_1}|\eta|}\right|\\
    \lesssim_{\delta} &\langle \eta\rangle^{1/2}\langle \xi\rangle^{1/2}\mathrm{e}^{2\lambda(t)\langle k,\xi\rangle^{1/2}}\left[\dfrac{\mathrm{e}^{\sqrt{\delta}\langle \xi\rangle^{1/2}}}{b_k(t,\xi)}+\mathrm{e}^{\sqrt{\delta}|k|^{1/2}}\right] \left[\dfrac{\mathrm{e}^{\sqrt{\delta}\langle
     \xi\rangle^{1/2}}}{b_k(t,\eta)}+\mathrm{e}^{\sqrt{\delta}|k|^{1/2}}\right] \frac{\mathrm{e}^{3\sqrt{\delta}\langle \rho\rangle^{1/2}}}{\langle t\rangle^2}\\
   % &\times \langle \rho\rangle^{2}\langle \xi\rangle^{-1}\left(1+\dfrac{k^2+|\xi|}{\langle t\rangle^2}\right)\\
    \lesssim_{\delta} &\mathrm{e}^{2\lambda(t)\langle k,\xi\rangle^{1/2}}\left[\dfrac{\mathrm{e}^{\sqrt{\delta}\langle \xi\rangle^{1/2}}}{b_k(t,\xi)}+\mathrm{e}^{\sqrt{\delta}|k|^{1/2}}\right] \left[\dfrac{\mathrm{e}^{\sqrt{\delta}\langle \xi\rangle^{1/2}}}{b_k(t,\eta)}+\mathrm{e}^{\sqrt{\delta}|k|^{1/2}}\right] \mathrm{e}^{3\sqrt{\delta}\langle \rho\rangle^{1/2}}\\
    &\times \left(1+\dfrac{k^2+|\xi|}{\langle t\rangle^2}\right)^{1/2}\left(1+\dfrac{k^2+|\eta|}{\langle t\rangle^2}\right)^{1/2}.
  \end{align*}
  Using  $A_{NR}(t,\rho)\geq \mathrm{e}^{\lambda(t)\langle \rho\rangle^{1/2}}$, {$|\xi|\leq|\eta|$} and the definition of $A^*_{k}(t,\xi),\ A_{k}^*(t,\eta)$, we infer that
  \begin{align*}
     & \dfrac{1}{\langle t\rangle^{7/4}}|\mathcal{T}_2''|\lesssim_{\delta}\dfrac{1}{\langle t\rangle^{7/4}}A_{k}^*(t,\xi)A_{k}^*(t,\eta)A_{NR}(t,\rho)\mathrm{e}^{-(\lambda (t)/20)\langle \rho\rangle^{1/2}}.
  \end{align*}
The bound \eqref{eq:A*S0S1bounds-prove-1} then follows from {\eqref{est:par-t-A*}}.
\smallskip

\noindent \textbf{Step 2.} Assume that $\big((k,\xi),(l,\eta)\big)\in \Sigma_2$ and  we prove {\eqref{eq:A*S2bounds-1}}.
\smallskip

 Assuming {$\sigma\in\mathbb{Z}\setminus\{0\}$ and} $\big((k,\xi),(l,\eta)\big)\in\Sigma_2$, we get by Lemma \ref{lem:8.6JiaHao} that
  \begin{align*}
     \dfrac{|\xi A_k^{*}(t,\eta)^2 |}{\langle \rho\rangle\langle t\rangle+\langle \rho\rangle^{1/4}\langle t\rangle^{7/4}} \lesssim_{\delta}&\sqrt{|(A_k\dot{A}_{k})(t,\xi)|} \sqrt{|(A_{NR}\dot{A}_{NR})(t,\rho)|}A_{k}(t,\eta)\mathrm{e}^{-(\delta_0/200)\langle k,\eta\rangle^{1/2}}\\
     &\times \left(1+\dfrac{k^2+|\eta|}{\langle t\rangle^{2}}\right)\\
     \lesssim_{\delta}&\sqrt{|(A_k\dot{A}_{k})(t,\xi)|} \sqrt{|(A_{NR}\dot{A}_{NR})(t,\rho)|}A_{k}(t,\eta)\mathrm{e}^{-(\delta_0/200)\langle k,\eta\rangle^{1/2}}\\
     &\times \langle k,\eta\rangle^2\\
     \lesssim_{\delta}&\sqrt{|(A_k\dot{A}_{k})(t,\xi)|} \sqrt{|(A_{NR}\dot{A}_{NR})(t,\rho)|}A_{k}(t,\eta)\mathrm{e}^{-(\delta_0/201)\langle k,\eta\rangle^{1/2}}.
  \end{align*}
  Thanks to \eqref{est:par-t-A*-1}, we deduce that $A_k\leq A_k^*$, $|\dot{A}_{k}|\lesssim_{\delta}|\dot{A}^*_{k}|$, and
  \begin{align*}
     \dfrac{|\xi A_k^{*}(t,\eta)^2 |}{\langle \rho\rangle\langle t\rangle+\langle \rho\rangle^{1/4}\langle t\rangle^{7/4}} \lesssim_{\delta}&
     %\sqrt{|(A_k\dot{A}_{k})(t,\xi)|} \sqrt{|(A_{NR}\dot{A}_{NR})(t,\rho)|}A_{k}(t,\eta)\mathrm{e}^{-(\delta_0/201)\langle k,\eta\rangle^{1/2}}\\
     %&\times \left(1+\dfrac{k^2}{\langle t\rangle^{2}}+ \dfrac{\eta^2}{\langle t\rangle^4+\langle t\rangle^{\sigma_1}|\eta|}\right)^{\f12}
     %\left(1+\dfrac{k^2}{\langle t\rangle^{2}}+ \dfrac{\xi^2}{\langle t\rangle^4+\langle t\rangle^{\sigma_1}|\xi|}\right)^{\f12}\\ \lesssim_{\delta}&
     \sqrt{|(A^*_k\dot{A}^*_{k})(t,\xi)|} \sqrt{|(A_{NR}\dot{A}_{NR})(t,\rho)|}A^*_{k}(t,\eta)\mathrm{e}^{-(\delta_0/201)\langle k,\eta\rangle^{1/2}}.
\end{align*}

It remains to prove the harder inequality:
\begin{align}\label{eq:A*S2bounds-prove-1}
  \dfrac{|\eta A_k^{*}(t,\xi)^2 |}{\langle \rho\rangle\langle t\rangle+\langle \rho\rangle^{1/4}\langle t\rangle^{7/4}} \lesssim_{\delta}&\sqrt{|(A^*_k\dot{A}^*_{k})(t,\xi)|} \sqrt{|(A_{NR}\dot{A}_{NR})(t,\rho)|}A^*_{k}(t,\eta)\mathrm{e}^{-(\delta_0/201)\langle k,\eta\rangle^{1/2}}.
\end{align}
By {Lemma \ref{lem:8.6JiaHao}}, \eqref{est:weight-more-1}
%$1+\frac{k^2+|\xi|}{\langle t\rangle^2}\leq1+\frac{k^2+|\eta|}{\langle t\rangle^2}+\frac{\langle \rho\rangle}{\langle t\rangle^2}$
and \eqref{est:par-t-A*-1}, it suffices to prove that
\begin{align}\label{eq:A*S2bounds-prove-2}
  \dfrac{\langle \rho\rangle}{\langle t\rangle^2}\dfrac{|\eta A_k(t,\xi)^2 |}{\langle \rho\rangle\langle t\rangle+\langle \rho\rangle^{1/4}\langle t\rangle^{7/4}} \lesssim_{\delta}&\sqrt{|(A_k\dot{A}_{k})(t,\xi)|} \sqrt{|(A_{NR}\dot{A}_{NR})(t,\rho)|}A_{k}(t,\eta)\mathrm{e}^{-\f{\delta_0}{201}\langle k,\eta\rangle^{\f12}}.
\end{align}
By Lemma \ref{lem:8.3JiaHao-mod} (see also \eqref{est:Ak-A0-1}), we have
\begin{align*}
    \dfrac{\langle \rho\rangle}{\langle t\rangle^2}\dfrac{|\eta A_k(t,\xi)^2 |}{\langle \rho\rangle\langle t\rangle+\langle \rho\rangle^{1/4}\langle t\rangle^{7/4}}\leq\dfrac{|\eta A_k(t,\xi)^2 |}{\langle t\rangle^3}\lesssim_{\delta}\dfrac{1}{\langle t\rangle^{2}}A_k(t,\xi)A_{NR}(t,\rho)A_{k}(t,\eta)\mathrm{e}^{-(\lambda(t)/20)\langle k,\eta\rangle^{1/2}},
\end{align*}
and \eqref{eq:A*S2bounds-prove-2} follows from \eqref{est:7.4JiaHao-1} and \eqref{est:7.4JiaHao-2}.
\end{proof}

 The following lemma shows the estimates on weights when $|\rho|$ is small.

\begin{lemma}\label{lem:rholeq1}
  Assume $t\ge 1$ and let $\rho=\xi-\eta$. If $|\rho|\leq 1$, then it holds that
  \begin{align}\label{eq:rholeq1}
      \dfrac{1}{\langle t\rangle^{7/4}} &|\eta A_k(t,\xi)^2-\xi A_k(t,\eta)^2| \lesssim_{\delta} |\rho|\sqrt{|(A_k\dot{A}_{k})(t,\xi)|} \sqrt{|(A_k\dot{A}_{k})(t,\eta)|},
    \end{align}
    and
    \begin{align}\label{eq:rholeq1*}
      \dfrac{1}{\langle t\rangle^{7/4}} &|\eta A_k^{*}(t,\xi)^2-\xi A_k^*(t,\eta)^{2}| \lesssim_{\delta}|\rho|\sqrt{|(A_k^*\dot{A}^*_{k})(t,\xi)|} \sqrt{|(A_k^*\dot{A}^*_{k})(t,\eta)|}.
    \end{align}
\end{lemma}

\begin{proof}
We {first} prove \eqref{eq:rholeq1}. We write
  \begin{align*}
     &\eta A_k(t,\xi)^2-\xi A_k(t,\eta)^2=\mathcal{T}_1''+\mathcal{T}_2'',
  \end{align*}
  where
  \begin{align}
     & \mathcal{T}''_1= \left(\eta \mathrm{e}^{2\lambda(t)\langle k,\xi\rangle^{1/2}}-\xi \mathrm{e}^{2\lambda(t)\langle k,\eta\rangle^{1/2}}\right)\left[\dfrac{\mathrm{e}^{\sqrt{\delta}\langle\xi\rangle^{1/2}} }{b_{k}(t,\xi)}+\mathrm{e}^{\sqrt{\delta}|k|^{1/2}}\right]^2,\label{eq:T1''}\\
     &\mathcal{T}''_2=\xi\mathrm{e}^{2\lambda(t)\langle k,\eta\rangle^{1/2}}\left[\dfrac{\mathrm{e}^{\sqrt{\delta}\langle\xi\rangle^{1/2}} }{b_{k}(t,\xi)}- \dfrac{\mathrm{e}^{\sqrt{\delta}\langle\eta\rangle^{1/2}} }{b_{k}(t,\eta)}\right] \left[\dfrac{\mathrm{e}^{\sqrt{\delta}\langle\xi\rangle^{1/2}} }{b_{k}(t,\xi)}+ \dfrac{\mathrm{e}^{\sqrt{\delta}\langle\eta\rangle^{1/2}} }{b_{k}(t,\eta)}+2\mathrm{e}^{\sqrt{\delta}|k|^{1/2}}\right].  \label{eq:T2''}
  \end{align}
  Thus, it remains to prove that for $i=1,2$,
  \begin{align}\label{eq:Ti''bound}
     & \dfrac{|\mathcal{T}_i''|}{\langle t\rangle^{7/4}}  \lesssim_{\delta} |\rho|\sqrt{|(A_k\dot{A}_{k})(t,\xi)|} \sqrt{|(A_k\dot{A}_{k})(t,\eta)|}.
  \end{align}

 We first estimate $\mathcal{T}_1''$. We have
  \begin{align*}
     & \left|\eta \mathrm{e}^{2\lambda(t)\langle k,\xi\rangle^{1/2}}-\xi \mathrm{e}^{2\lambda(t)\langle k,\eta\rangle^{1/2}}\right| \lesssim |\rho| \left(1+\dfrac{|\xi|}{\langle k,\xi\rangle^{1/2}}\right)\left[ \mathrm{e}^{2\lambda(t)\langle k,\xi\rangle^{1/2}}+ \mathrm{e}^{2\lambda(t)\langle k,\eta\rangle^{1/2}}\right],
  \end{align*}
which gives
  \begin{align}\label{eq:T1''-prove1}
      &  \dfrac{|\mathcal{T}_1''|}{\langle t\rangle^{\f74}} \lesssim_{\delta} \dfrac{|\rho|}{\langle t\rangle^{\f74}} \left(1+\dfrac{|\xi|}{\langle k,\xi\rangle^{\f12}}\right)\left[ \mathrm{e}^{2\lambda(t)\langle k,\xi\rangle^{\f12}}+ \mathrm{e}^{2\lambda(t)\langle k,\eta\rangle^{\f12}}\right]\left[\dfrac{\mathrm{e}^{\sqrt{\delta}\langle\xi\rangle^{ \f12}} }{b_{k}(t,\xi)}+ 2\mathrm{e}^{\sqrt{\delta}|k|^{\f12}}\right]^2.
   \end{align}
%   If $(k,\xi)\notin I_t^{**}$ with $I_{t}^{**}$  defined by \eqref{eq:I**},
   Then we  get by  \eqref{est:7.2JiaHao-3} that
   \begin{align}\label{eq:T1''-prove2}
       \dfrac{|\mathcal{T}_1''|}{\langle t\rangle^{\f74}}&\lesssim_{\delta} \dfrac{|\rho|}{\langle t\rangle^{\f74}} \left(1+\dfrac{|\xi|}{\langle k,\xi\rangle^{\f12}}\right)\left[ \mathrm{e}^{2\lambda(t)\langle k,\xi\rangle^{\f12}}+ \mathrm{e}^{2\lambda(t)\langle k,\eta\rangle^{\f12}}\right]\left[\dfrac{\mathrm{e}^{\sqrt{\delta}\langle\xi\rangle^{\f12}} }{b_{k}(t,\xi)}+ 2\mathrm{e}^{\sqrt{\delta}|k|^{\f12}}\right]\nonumber\\
       &\qquad \times \left[\dfrac{\mathrm{e}^{\sqrt{\delta}\langle\xi\rangle^{\f12}} }{b_{k}(t,\eta)}+ 2\mathrm{e}^{\sqrt{\delta}|k|^{\f12}}\right]\nonumber\\
       &\lesssim_{\delta} |\rho|\dfrac{\langle\xi\rangle^{1/2} }{\langle t\rangle^{7/4}} A_{k}(t,\xi)A_{k}(t,\eta).
   \end{align}
   Then the bound \eqref{eq:Ti''bound} for $i=1$ follows from \eqref{eq:T1''-prove2} and \eqref{est:7.4JiaHao-2}.

\if0   On the other hand, if $(k,\xi)\in I_t^{**}$ then we use \eqref{est:bl/bk-2} and \eqref{eq:T1''-prove1}
      \begin{align}\label{eq:T1''-prove3}
       \dfrac{|\mathcal{T}_1''|}{\langle t\rangle^{\f74}}
       &\lesssim_{\delta} |\rho|\dfrac{\langle\xi\rangle^{1/2} }{\langle t\rangle^{7/4}} \dfrac{|\xi/k^2|}{\langle t-\xi/k\rangle} A_{k}(t,\xi)A_{k}(t,\eta) \lesssim_{\delta} \dfrac{|\rho|}{\langle t\rangle^{1/4}\langle t-\xi/k\rangle} A_{k}(t,\xi)A_{k}(t,\eta).
   \end{align}
   Then the bound \eqref{eq:Ti''bound}($i=1$) follows from \eqref{eq:T1''-prove3}, \eqref{est:7.4JiaHao-2}, \eqref{est:7.4JiaHao-4} and \eqref{eq:(7.8)JiaHao} as before. \fi\smallskip

Now we estimate $\mathcal{T}_2''$. By \eqref{est:7.2JiaHao-2} and \eqref{est:7.2JiaHao-3}, we have (for $|\rho|\leq 1$)
\begin{align*}
   |\xi|\left|\dfrac{\mathrm{e}^{\sqrt{\delta}\langle\xi\rangle^{1/2}} }{b_{k}(t,\xi)}- \dfrac{\mathrm{e}^{\sqrt{\delta}\langle\eta\rangle^{1/2}} }{b_{k}(t,\eta)}\right| &\lesssim_{\delta}|\rho|\left(\dfrac{\langle\xi\rangle}{L_{\kappa}(t,\xi)}+\langle\xi\rangle^{1/2}\right)\cdot \min\left(\dfrac{\mathrm{e}^{\sqrt{\delta}\langle\eta\rangle^{1/2}}}{b_k(t,\eta)}, \dfrac{\mathrm{e}^{\sqrt{\delta}\langle\xi\rangle^{1/2}}}{b_k(t,\xi)}\right)
   \\
   & \lesssim_{\delta}|\rho|\left(\langle \xi\rangle\dfrac{\langle \xi\rangle^{1/2}+\langle t\rangle}{\langle\xi\rangle +\langle t\rangle}\right)\min\left(\dfrac{\mathrm{e}^{\sqrt{\delta}\langle\eta\rangle^{1/2}}}{b_k(t,\eta)}, \dfrac{\mathrm{e}^{\sqrt{\delta}\langle\xi\rangle^{1/2}}}{b_k(t,\xi)}\right),
\end{align*}
  then, thank to $|\rho|\leq 1$, we get
  \begin{align}\label{eq:T2''-prove1}
     \dfrac{|\mathcal{T}_2''|}{\langle t\rangle^{\f74}} \lesssim_{\delta} & \dfrac{|\rho|\mathrm{e}^{2\lambda(t)\langle k,\xi\rangle^{\f12}}}{\langle t\rangle^{7/4}}\left(\langle \xi\rangle\dfrac{\langle \xi\rangle^{1/2}+\langle t\rangle}{\langle\xi\rangle +\langle t\rangle}\right)\left[\dfrac{\mathrm{e}^{\sqrt{\delta}\langle\xi\rangle^{\f12}} }{b_{k}(t,\xi)}+ 2\mathrm{e}^{\sqrt{\delta}|k|^{\f12}}\right]  \left[\dfrac{\mathrm{e}^{\sqrt{\delta}\langle\xi\rangle^{\f12}} }{b_{k}(t,\eta)}+ 2\mathrm{e}^{\sqrt{\delta}|k|^{\f12}}\right]\nonumber\\
    \lesssim_{\delta} & \dfrac{|\rho|\langle \xi\rangle^{1/2}}{\langle t\rangle^{5/4}} A_k(t,\xi)A_{k}(t,\eta).
   \end{align}
   Then the desired bound \eqref{eq:Ti''bound} for $i=2$ follows from \eqref{eq:T2''-prove1} and  \eqref{est:7.4JiaHao-2}.
   \smallskip

 Next we prove \eqref{eq:rholeq1*}. We write
  \begin{align*}
     &\eta A_k^*(t,\xi)^2-\xi A_k^*(t,\eta)^{2}=\mathcal{T}_1''\left(1+\dfrac{k^2}{\langle t\rangle^2}\right)+\mathcal{T}_3''+\mathcal{T}_2''\left(1+\dfrac{k^2+|\eta|}{\langle t\rangle^2}\right),
  \end{align*}
  where $\mathcal{T}''_1$ and $\mathcal{T}''_2$ are defined by \eqref{eq:T1''} and \eqref{eq:T2''} and
  \begin{align}
     & \mathcal{T}''_3= \xi\eta\left(\dfrac{\mathop{\text{sgn}}(\xi)\mathrm{e}^{2\lambda(t)\langle k,\xi\rangle^{1/2}}}{\langle t\rangle^2}-
     \dfrac{\mathop{\text{sgn}}(\eta)\mathrm{e}^{2\lambda(t)\langle k,\eta\rangle^{1/2}}}{\langle t\rangle^2}\right)
     \left[\dfrac{\mathrm{e}^{\sqrt{\delta}\langle\xi\rangle^{1/2}} }{b_{k}(t,\xi)}+
     \mathrm{e}^{\sqrt{\delta}|k|^{1/2}}\right]^2.\nonumber%\label{eq:T3''}
  \end{align}
   Let
    \begin{align*}
     & \mathcal{T}''_4= \mathcal{T}_1''\left(1+\dfrac{k^2}{\langle t\rangle^2}\right),\quad \mathcal{T}_5''=\mathcal{T}_2''
     \left(1+\dfrac{k^2+|\eta|}{\langle t\rangle^2}\right).%\label{eq:T45''}
  \end{align*}
  Then by Lemma \eqref{lem:par-t-A*}, \eqref{est:weight-more-1} and $|\rho|\leq 1$, it remains to prove that for $i=3,4,5$,
    \begin{align}\label{eq:Ti''bound345}
     & \dfrac{|\mathcal{T}_i''|}{\langle t\rangle^{7/4}}  \lesssim_{\delta} |\rho|\sqrt{|(A_k\dot{A}_{k})(t,\xi)|} \sqrt{|(A_k\dot{A}_{k})(t,\eta)|}\left(1+\dfrac{k^2+|\eta|}{\langle t\rangle^2}\right).
  \end{align}

    Thanks to (as $\rho=\xi-\eta$, if $|\xi|> |\rho|$ then $\mathop{\text{sgn}}(\xi)=\mathop{\text{sgn}}(\eta) $)
  \begin{align*}
\left|\dfrac{\mathop{\text{sgn}}(\xi)\mathrm{e}^{2\lambda(t)\langle k,\xi\rangle^{1/2}}}{\langle t\rangle^2}-
     \dfrac{\mathop{\text{sgn}}(\eta)\mathrm{e}^{2\lambda(t)\langle k,\eta\rangle^{1/2}}}{\langle t\rangle^2}\right| &\lesssim_{\delta} \left[ \mathrm{e}^{2\lambda(t)\langle k,\xi\rangle^{1/2}}+ \mathrm{e}^{2\lambda(t)\langle k,\eta\rangle^{1/2}}\right]\\&\quad\times\big(\mathbf{1}_{\{|\xi|\leq |\rho|\}}+
     |\rho|/\langle k,\xi\rangle^{\f12}\big)/{\langle t\rangle^2},
  \end{align*}
  we have\big(using $ |\xi\eta|(\mathbf{1}_{\{|\xi|\leq |\rho|\}}+
     |\rho|/\langle k,\xi\rangle^{\f12})\leq |\rho|(1+
     |\xi|/\langle k,\xi\rangle^{\f12})|\eta|$\big)
 \begin{align*}
      \dfrac{|\mathcal{T}_3''|}{\langle t\rangle^{\f74}} \lesssim_{\delta}& \dfrac{|\rho|}{\langle t\rangle^{\f74}} \left(1+\dfrac{|\xi|}{\langle k,\xi\rangle^{\f12}}\right)\left[ \mathrm{e}^{2\lambda(t)\langle k,\xi\rangle^{\f12}}+ \mathrm{e}^{2\lambda(t)\langle k,\eta\rangle^{\f12}}\right]\left[\dfrac{\mathrm{e}^{\sqrt{\delta}\langle\xi\rangle^{ \f12}} }{b_{k}(t,\xi)}+ 2\mathrm{e}^{\sqrt{\delta}|k|^{\f12}}\right]^2 \dfrac{|\eta|}{\langle t\rangle^2}.
  \end{align*}
 Then the desired bound \eqref{eq:Ti''bound345} ($i=3$) can be proved in a similar way as in the proof of \eqref{eq:Ti''bound} for $i=1$. The bound \eqref{eq:Ti''bound345} ($i=4$) directly follows from \eqref{eq:Ti''bound}($i=1$) and the definition of $\mathcal{T}''_4$.
   The  bound \eqref{eq:Ti''bound345} ($i=5$) directly follows from \eqref{eq:Ti''bound}($i=2$) and the definition of $\mathcal{T}''_5$.
\end{proof}

\section*{Acknowledgments}
P. Zhang is supported by National Key R$\&$D Program of China under grant
  2021YFA1000800 and K. C. Wong Education Foundation.
 He is also partially supported by National Natural Science Foundation of China under Grants  12288201 and 12031006.
 Z. Zhang is partially supported by  NSF of China  under Grant 12171010 and 12288101.

\end{CJK*}
\end{document}